\documentclass[12pt]{article}
\usepackage{mathrsfs,amsthm,graphicx,color,verbatim,bbm,amsmath,amsfonts,amssymb,newclude,nicefrac,amsfonts,graphicx,color,enumerate,hyperref,bm,titling}

\usepackage[english]{babel} 

\theoremstyle{plain}
\newtheorem{theorem}{Theorem}[section]
\newtheorem{lemma}[theorem]{Lemma}
\newtheorem{prop}[theorem]{Proposition}
\newtheorem{cor}[theorem]{Corollary}

\newtheorem{setting}[theorem]{Setting}
\newtheorem{definition}[theorem]{Definition}

\newcommand{\E}{\mathbb{E}}
\renewcommand{\P}{\mathbb{P}}
\newcommand{\Q}{\mathbb{Q}}
\newcommand{\R}{\mathbb{R}}
\newcommand{\N}{\mathbb{N}}
\newcommand{\smallsum}{\textstyle\sum}
\newcommand{\textint}{\textstyle\int}

\newcommand{\Exp}[1]{ \E \! \left[ #1 \right]}

\newcommand{\EXPP}[1]{ \E \big[ #1 \big]}
\newcommand{\EXPPP}[1]{ \E \Big[ #1 \Big]}
\newcommand{\EXPPPP}[1]{ \E \bigg[ #1 \bigg]}

\newcommand{\norm}[1]{ \left\| #1 \right\| }
\newcommand{\Norm}[1]{ \| #1 \| }
\newcommand{\Normm}[1]{ \big\| #1 \big\| }

\newcommand{\qandqShort}{\quad\text{and}\quad}

\newcommand{\qandq}{\qquad\text{and}\qquad}
\newcommand{\andq}{\text{and}\qquad}
\newcommand{\Image}{\operatorname{Im}}
\newcommand{\id}{\operatorname{id}}
\newcommand{\diag}{\operatorname{diag}}
\allowdisplaybreaks

\newcommand{\vertiii}[1]{{\left\vert\kern-0.25ex\left\vert\kern-0.25ex\left\vert {#1} 
		\right\vert\kern-0.25ex\right\vert\kern-0.25ex\right\vert}}
	\newcommand{\vertiiibig}[1]{{\big\vert\kern-0.25ex\big\vert\kern-0.25ex\big\vert {#1} 
			\big\vert\kern-0.25ex\big\vert\kern-0.25ex\big\vert}}
			\newcommand{\vertiiiStandard}[1]{{\vert\kern-0.25ex\vert\kern-0.25ex\vert {#1} 
				\vert\kern-0.25ex\vert\kern-0.25ex\vert}}
\newcommand{\HSNorm}[1]{\vertiii#1} 
 
\newcommand{\HSNormStandard}[1]{\vertiiiStandard#1} 
\newcommand{\HSNormIndex}[3]{{\vertiii#1}_{\operatorname{HS}(\R^{#3},\R^{#2})}} 
\makeatletter
\newcommand{\vast}{\bBigg@{3.5}}
\newcommand{\Vast}{\bBigg@{4}}
\makeatother

\begin{document}

\title{A proof that artificial neural networks overcome the curse of dimensionality in the numerical approximation of Black-Scholes partial differential equations}



\author{ Philipp Grohs$^{1,2}$,
Fabian Hornung$^{3,4}$,
Arnulf Jentzen$^{5,6,7}$,
\\ and
Philippe von Wurstemberger$^{8,9}$
\bigskip
\\
\small{$^1$Faculty of Mathematics and Research Platform Data Science,}\\
\small{University of Vienna, Austria, e-mail:  philipp.grohs@univie.ac.at}
\smallskip
\\
\small{$^2$Johann Radon Institute of Computational and Applied Mathematics,}\\
\small{Austrian Academy of Sciences, Austria, email: philipp.grohs@oeaw.ac.at}
\smallskip
\\
\small{$^3$Institute for Analysis, Karlsruhe Institute of Technology, Germany}
\smallskip
\\
\small{$^4$Seminar for Applied Mathematics, ETH Zurich,}\\
\small{Switzerland, e-mail: fabianhornung89@gmail.com}
\smallskip
\\
\small{$^5$School of Data Science and Shenzhen Research Institute of Big Data,}\\
\small{The Chinese University of Hong Kong, Shenzhen,}\\
\small{China, e-mail: ajentzen@cuhk.edu.cn}
\smallskip
\\
\small{$^6$Faculty of Mathematics and Computer Science, University of M\"unster,}\\
\small{Germany, e-mail: ajentzen@uni-muenster.de}
\smallskip
\\
\small{$^7$Seminar for Applied Mathematics, ETH Zurich,}\\
\small{Switzerland, e-mail: arnulf.jentzen@sam.math.ethz.ch}
\smallskip
\\
\small{$^8$Seminar for Applied Mathematics, ETH Zurich,}\\
\small{Switzerland, e-mail: philippe.vonwurstemberger@math.ethz.ch}
\smallskip
\\
\small{$^9$School of Data Science, The Chinese University of Hong Kong,}\\
\small{Shenzhen, China, e-mail: philippevw@cuhk.edu.cn}
} 


\vspace{-20cm}
\maketitle

\begin{abstract}
Artificial neural networks (ANNs) have very successfully been used in numerical simulations for a series of computational problems 
ranging from  image classification/image recognition, speech recognition, time series analysis, game intelligence, and computational advertising to numerical approximations of partial differential equations (PDEs).
 Such numerical simulations suggest that ANNs have the capacity to very efficiently approximate high-dimensional functions and, especially, 
 indicate that ANNs seem to admit the fundamental power to overcome the curse of dimensionality when approximating the high-dimensional functions appearing in the above named computational problems. 
 There are a series of rigorous mathematical approximation results for ANNs in the scientific literature. 
Some of them 
prove convergence without convergence rates and some 
even rigorously establish convergence rates but there are only a few special cases where mathematical results can rigorously explain the empirical success of ANNs when approximating high-dimensional functions. 
The key contribution of this article is to disclose that ANNs can efficiently approximate high-dimensional functions in the case of numerical approximations of Black-Scholes PDEs.  
 More precisely, this work reveals that the number of required parameters of an ANN to approximate the solution of the Black-Scholes PDE grows at most polynomially in both the reciprocal of the prescribed approximation accuracy $\varepsilon > 0$ and the PDE dimension $d \in \N$. We thereby prove, for the first time, that ANNs do indeed overcome the curse of dimensionality in the numerical approximation of Black-Scholes PDEs.
\end{abstract}

\tableofcontents

\section{Introduction}
\label{sect:intro}

Artificial neural networks (ANNs) (cf., e.g., Goodfellow et al.\ \cite{goodfellow2016deep}, McCulloch \& Pitts \cite{mcculloch1943logical}, Priddy \& Keller \cite{priddy2005artificial}, Schmidhuber \cite{schmidhuber2015deep}) have very successfully been used in numerical simulations for a series of computational problems 
ranging from  image classification/image recognition (cf., e.g., Huang et al.~\cite{huang2017densely}, Krizhevsky et al.~\cite{krizhevsky2012imagenet}, Simonyan \& Zisserman \cite{simonyan2014very}), speech recognition (cf., e.g., Dahl et al.~\cite{dahl2012context}, Hinton et al.~\cite{hinton2012deep}, Graves et al.~\cite{graves2013speech}, Wu et al.~\cite{wu2016stimulated}), time series analysis (cf., e.g., Goodfellow et al.~\cite{goodfellow2016deep}, LeCun et al.~\cite{lecun2015deep}), game intelligence (cf., e.g., Silver et al.~\cite{silver2016mastering,silver2017mastering}), and computational advertising to numerical approximations of partial differential equations (PDEs) 
(cf., e.g., \cite{Kolmogorov,beck2017machine,becker2018deep, weinan2017deep, weinan2018deep, fujii2017asymptotic, han2017solving, henry2017deep, khoo2017solving,lagaris1998artificial, mishra2018machine, nabian2018deep, raissi2018forward, sirignano2017dgm}).
Such numerical simulations suggest that ANNs have the capacity to very efficiently approximate high-dimensional functions. 
In particular, they indicate that ANNs seem to admit the fundamental power to resolve the curse of dimensionality 
(cf., e.g., Bellman \cite{bellman2013dynamic})
in the sense that the number of parameters of an ANN to approximate the high-dimensional functions appearing in the above named computational problems grows at most polynomially in both the reciprocal of  the prescibed accuracy $\varepsilon>0$ and the dimension $d \in \N$.
There are a series of rigorous mathematical approximation results for ANNs in the scientific literature 
(cf., e.g., \cite{bach2017breaking,Barron1993, Barron1994,blum1991approximation,bolcskei2017optimal,Burger2001235,CandesDiss,chen1995approximation,ChuXM1994networksforlocApprox,Cybenko1989,DeVore1997approxfeedforward,Eldan2016PowerofDepth,ellacott1994aspects,Funahashi1989183,hartman1990layered,Hornik1991251,hornik1993some,Hornik1989universalApprox,hornik1990universal,leshno1993multilayer,Mhaskar1995151,Mhaskar1996NNapprox,Mhaskar2016DeepVSShallow,NguyenThien1999687,park1991universal,perekrestenko2018universal,petersen2018topological,petersen2017optimal,pinkus1999approximation,Schmitt1999lowercomplbounds,ShaCC2015provableAppDNN,sirignano2017dgm,yarotsky2017error,yarotsky2018universal} and the references mentioned therein).
Some of them prove convergence without convergence rates and some even rigorously establish convergence rates but there are only a few special cases where mathematical results can rigorously explain the empirical success of ANNs when approximating high-dimensional functions. 

The key contribution of this article is to disclose that ANNs can efficiently approximate high-dimensional functions in the case of numerical approximations of Black-Scholes PDEs. 
More accurately, Theorem~\ref{cont_NN_approx} in Section~\ref{approximationSection} below, which is the main result of this paper, reveals that the number of required parameters of an ANN to approximate the solution of a Kolmogorov PDE with affine linear drift and affine linear diffusion functions grows at most polynomially in both the reciprocal of the prescribed approximation accuracy $\varepsilon > 0$ and the PDE dimension $d \in \N$.
In particular, we thereby prove, for the first time, that ANNs do indeed resolve the curse of dimensionality in the numerical approximation of Black-Scholes PDEs.
To illustrate the main result of this article (Theorem~\ref{cont_NN_approx} in Subsection \ref{subsectionContinuousANN} below), we now present in the following theorem a special case of Theorem~\ref{cont_NN_approx}.

\begin{theorem}
\label{thm:intro}
Let 
	$T, \mathfrak{c}, p \in (0,\infty)$,  
	$\alpha, \beta \in \R$,
	$\mathbf{a} \in C(\R, \R)$,
let $\mathbf{A}_d  \in C(\R^d, \R^d)$, $d \in \N$, satisfy for all 
	$d \in \N$,
	$x = (x_1,x_2, \ldots, x_d) \in \R^d$ 
that
$
	\mathbf{A}_d(x)
	=
	(\mathbf{a}(x_1), \mathbf{a}(x_2), \ldots, \mathbf{a}(x_d))
$,
let $\varphi_d \in C(\R^{d}, \R)$, $d \in \N$,
let
\begin{equation}
\begin{split}
\label{thm:intro:NeuralNetworks}
	\mathcal{N}
&=
	\cup_{\mathcal{L} \in \{2,3, \ldots \}}
	\cup_{ (l_0,l_1,\ldots, l_\mathcal{L}) \in ((\N^{\mathcal{L}}) \times \{ 1 \} ) }
	\left(
		\times_{k = 1}^\mathcal{L} (\R^{l_k \times l_{k-1}} \times \R^{l_k})
	\right),
\end{split}
\end{equation}
let 
$
	\mathcal{P}\colon \mathcal{N} \to \N
$ 
and 
$
	\mathcal{R} 
\colon 
	\mathcal{N} 
\to 
	\cup_{d = 1}^\infty C(\R^d, \R)
$ 
be the functions which satisfy
for all 
	$ \mathcal{L} \in \{2, 3, \ldots \}$, 
	$ (l_0,l_1,\ldots, l_\mathcal{L}) \in ((\N^{\mathcal{L}}) \times \{ 1 \}) $, 
	$
		\Phi 
	=
		((W_1, B_1), \ldots,\allowbreak (W_\mathcal{L}, B_\mathcal{L}))
		\allowbreak\in  
		( \times_{k = 1}^\mathcal{L} \allowbreak(\R^{l_k \times l_{k-1}} \times \R^{l_k}))
	$,
	$x_0 \in \R^{l_0}, x_1 \in \R^{l_1}, \ldots, x_{\mathcal{L}-1} \in \R^{l_{\mathcal{L}-1}}$ with $\forall \, k \in \N \cap (0,\mathcal{L}) \colon x_k = \mathbf{A}_{l_k}(W_k x_{k-1} + B_k)$
that
\begin{equation}
	\mathcal{R}(\Phi) \in C(\R^{l_0}, \R),
	\qquad
	( \mathcal{R}(\Phi) ) (x_0) = W_\mathcal{L} x_{\mathcal{L}-1} + B_\mathcal{L},
\end{equation}
and
$
	\mathcal{P}(\Phi)
	=
	\sum_{k = 1}^\mathcal{L} l_k(l_{k-1} + 1) 
$,
for every $d\in\N$ let $\left\| \cdot \right\|_{\R^d} \colon \R^d \to [0,\infty)$ be the $d$-dimensional Euclidean norm,
and let $(\phi_{d, \delta})_{d \in \N, \, \delta \in (0,1]} \subseteq \mathcal{N}$ satisfy for all 
	$d \in \N$, 
	$\delta \in (0,1]$, 
	$x \in \R^{d}$ 
that
	$\mathcal{P}(\phi_{d, \delta}) \leq \mathfrak{c} \, d^{\mathfrak{c}} \delta^{-{\mathfrak{c}}} $,
	$\mathcal{R}(\phi_{d, \delta}) \in C(\R^{d}, \R)$,
	$	\left| 
	( \mathcal{R}(\phi_{d, \delta}) )  (x)
	\right| 
	\leq 
	\mathfrak{c} \, d^{\mathfrak{c}} (1+\| x \|_{\R^d}^{\mathfrak{c}})
	$, and
\begin{equation}
	\label{cont_NN_approxSimple:ass1}
	\left| 
	\varphi_d(x) - ( \mathcal{R}(\phi_{d, \delta}) )(x)
	\right|
	\leq 
	\mathfrak{c} \, d^{\mathfrak{c}} \, \delta \, (1+ \| x \|_{\R^d}^{\mathfrak{c}}).
\end{equation}
Then  
\begin{enumerate}[(i)]
\item \label{cont_NN_approxSimple:item1}
there exist unique continuous functions $u_d\colon [0,T]\allowbreak \times \R^{d} \to \R$, $d \in \N$, which satisfy for all 
	$d \in \N$, 
	$x \in \R^{d}$ 
that $u_d(T,x) = \varphi_d(x)$,
which satisfy for all 
	$d \in \N$ 
that
$
	\inf_{q \in (0,\infty)} \allowbreak
	\sup_{(t, x) \in [0, T] \times \R^d} \allowbreak
	\frac{ | u_d(t, x) | }{ 1 + \norm{x}_{\R^d}^q }
	\allowbreak<
	\infty
$,
and which satisfy for all 
	$d \in \N$ 
that $u_d$ 
is a viscosity solution of
\begin{equation}
\label{intro_thm:BS}
\big(\tfrac{\partial u_d}{\partial t}\big)(t,x) 
	+
	\left[
  	\sum_{i = 1}^d 
	    \tfrac{ \beta^2 |x_i|^2}{2} 
  	    \big( \tfrac{\partial^2 u_d}{\partial x_i^2 } \big)(t,x)
    \right]
    +
	\left[
    \sum_{i = 1}^d 
	    \alpha
	     x_i
	    \big( \tfrac{\partial u_d}{\partial x_i } \big)(t,x)
	\right]
=
	0
\end{equation}
for $(t,x) = (t, x_1, x_2, \ldots, x_d) \in (0,T) \times \R^{d}$
and
		
\item \label{cont_NN_approxSimple:item2}
there exist $\mathfrak{C} \in (0,\infty)$, $(\psi_{d, \varepsilon})_{d \in \N, \,\varepsilon \in (0,1]} \subseteq \mathcal{N}$ such that
for all $d \in \N$, $\varepsilon \in (0,1]$ it holds that
$
\mathcal{P}(\psi_{d, \varepsilon}) 
\leq
\mathfrak{C}  \, d^{\mathfrak{C}} \, \varepsilon^{-\mathfrak{C}}
$,
$
\mathcal{R}(\psi_{d, \varepsilon}) \in C(\R^{d}, \R)
$,
and
\begin{equation}
\label{intro_thm:concl} 
\left[
\int_{[0,1]^d}  
\left|
u_d(0,x) - ( \mathcal{R}(\psi_{d, \varepsilon}) ) (x)
\right|^p \,
dx
\right]^{\nicefrac{1}{p}} 
\leq
\varepsilon.
\end{equation}
\end{enumerate}
\end{theorem}

Theorem~\ref{thm:intro} is an immediate consequence of 
	Corollary~\ref{cont_NN_approxSimple} in Subsection \ref{subsectionContinuousANN} below 
	and the auxiliary results in Lemma~\ref{BS_properties} and Corollary~\ref{BS_endvalue} in Section~\ref{blackScholesSubsection} below.
Corollary~\ref{cont_NN_approxSimple}, in turn, is a consequence of Theorem~\ref{cont_NN_approx}, the main result of this article.
We now provide some explanations and interpretations for the objects appearing in Theorem~\ref{thm:intro} above.
The set $\mathcal{N}$ in \eqref{thm:intro:NeuralNetworks} corresponds to the set of all fully-connected artificial neural networks. 
We thus describe for every $ \mathcal{L} \in \{2, 3, \ldots \}$, $ (l_0,l_1,\ldots, l_{\mathcal{L}}) \in \N^{\mathcal{L}} \times \{1\}$
a neural network 
$\Phi = ((W_1, B_1), \ldots,\allowbreak (W_\mathcal{L}, B_\mathcal{L})) \in ( \times_{k = 1}^\mathcal{L} \allowbreak(\R^{l_k \times l_{k-1}} \times \R^{l_k}))  \subseteq \mathcal{N}$ 
(with $\mathcal{L} + 1$ layers, $\mathcal{L} - 1$ hidden layers, $l_0$ neurons in the input layer, $l_1$ neurons in the first hidden layer, $l_2$ neurons in the second hidden layer, ..., $l_{ \mathcal{L} - 1 }$ neurons in the $(\mathcal{L} - 1)$-th hidden layer, and $l_{ \mathcal{L} } = 1$ neuron in the output layer)
as a tuple consisting of its weight matrices $W_k \in \R^{l_k \times l_{k-1}}$, $k \in \{1,2, \ldots, \mathcal{L} \}$, and its bias vectors $B_k \in \R^{l_k}$, $k \in \{1,2, \ldots, \mathcal{L} \}$. 
For every artificial neural network $\Phi \in \mathcal{N}$ in Theorem~\ref{thm:intro} the number $ \mathcal{P}( \Phi )$ consequently corresponds to the number of parameters used to describe the artificial neural network $\Phi$.
Moreover, for every artificial neural network $ \Phi \in \mathcal{N} $ in Theorem~\ref{thm:intro} the function $\mathcal{R}( \Phi )$ is the realization of the artificial neural network $\Phi$ which is defined using the activation function $\mathbf{a} \in C(\R, \R)$.
The equation in display \eqref{intro_thm:BS} in Theorem~\ref{thm:intro} is usually referred to as Black-Scholes equations in the financial literature. The functions $\varphi_d$, $d \in \N$, in Theorem~\ref{thm:intro} are usually called the payoff functions in the Black-Scholes model and are the terminal values of the PDE in \eqref{intro_thm:BS}.
Roughly speaking, the conclusion in item~\eqref{cont_NN_approxSimple:item2} of Theorem~\ref{thm:intro} states that the solutions $u_d \in  C([0,T] \times \R^{d} , \R)$, $d \in \N$, of the PDE in \eqref{intro_thm:BS} can be approximated on the unit cube by ANNs without the curse of dimensionality. 
This can be seen as a partial theoretical justification of the performance of deep learning algorithms for the approximation of solutions of Black-Scholes PDEs, such as, e.g., the algorithm proposed in  Beck et al.\ \cite{Kolmogorov} (see, e.g., Section 4.4 in Beck et al.\ \cite{Kolmogorov} for numerical results for the Black-Scholes PDE).


Numerical approximation algorithms for high-dimensional PDEs have been exceedingly studied in the scientific literature (cf., e.g., \cite{ChassagneuxRungeKutta,CohenSchwab,CohenDeVore,weinan2017linear,geiss2014decoupling, Gobet2014,GobetStratified, GobetRegression,GrahamTalay, hutzenthaler2018overcoming,PardouxPeng,PetersdorffSchwab,warin2018monte,warin2018nesting} and the references mentioned therein). In particular, it is well-known that Monte Carlo methods (cf., e.g., \cite{GrahamTalay} and the references mentioned therein) like the Monte Carlo Euler method (cf., e.g., \cite[Section 7.5 in Chapter 7]{GrahamTalay}) and the multilevel Monte Carlo Euler method (cf., e.g., Creutzig et al.\ \cite{Creutzig}, Giles \cite{Giles}, Heinrich \cite{HeinrichOriginal}, \cite{HeinrichOverview}, and the references mentioned therein) do overcome the curse of dimensionality at one particular fixed point in space.
A key contribution of Theorem~\ref{thm:intro} above is to provide an approximation result not only at a single space point but on the entire $d$-dimensional unit cube $[0,1]^d$ (cf.\ display \eqref{intro_thm:concl} above and, e.g., Beck et al.~\cite{Kolmogorov}). 

Our proofs of Theorem~\ref{thm:intro} and Theorem~\ref{cont_NN_approx}, respectively, construct suitable ANN approximations using probabilistic arguments. 
More formally, our proofs of Theorem~\ref{thm:intro} and Theorem~\ref{cont_NN_approx}, respectively, employ -- besides other arguments -- 
\begin{enumerate}[(i)]
\item
the Feynman-Kac formula for viscosity solutions of Kolmogorov PDEs 
(cf.\ Proposition~\ref{viscosity_existence} in Subsection~\ref{SectionExistenceOfViscosity} below, 
\eqref{quantitative:eq01} in the proof of Proposition~\ref{quantitative} in Subsection~\ref{SectionQuantitativeError} below, and, 
e.g., Hairer et al.~\cite[Corollary 4.17]{HairerHutzenthalerJentzen15}),

\item
Monte-Carlo approximations for the expected value in the Feynman-Kac formula 
(cf.\ Corollary~\ref{mc_Lp_error2}  in Subsection~\ref{SectionMonteCarlo} below and 
\eqref{quantitative:eq07} in the proof of Proposition~\ref{quantitative} in Subsection~\ref{SectionQuantitativeError} below),

\item
the fact that the solution of an affine linear stochastic differential equation (SDE) depends affine linearly on the initial value (cf.\ Proposition~\ref{affine_solutions_of_SDEs} in Subsection~\ref{SectionAffineSDE} below and \eqref{quantitative:setting2} in the proof of Proposition~\ref{quantitative} in Subsection~\ref{SectionQuantitativeError} below),  and

\item \label{intro:item4}
an argument assuring the existence of a suitable elementary event on an artificial probability space constructed in the proof (cf.\ Proposition~\ref{construction_realization} in Subsection~\ref{SectionRealization} below and \eqref{quantitative:ChoiceOfRealization} in the proof of Proposition~\ref{quantitative} in Subsection~\ref{SectionQuantitativeError} below).

\end{enumerate}
The last argument, item~\eqref{intro:item4} above, is the crucial ingredient which builds the bridge between the probabilistic arguments of the proof and the deterministic conclusion.

There are several directions for further research arising from this work. 
After a first preprint of this work has come out, a few research articles exploring such directions and extending the results of this paper have appeared.
For example, \cite{grohs2019space, grohs2019deep, hutzenthaler2019proof, jentzen2018proof,Reisinger19rectified} establish results similar to Theorem~\ref{cont_NN_approx} for different kinds of differential equations.
More precisely,
	Jentzen et al.\ \cite{jentzen2018proof} considers PDEs with nonlinear but Lipschitz continuous drift coefficients, 
	Hutzenthaler et al.\ \cite{hutzenthaler2019proof} considers nonlinear PDEs, 
	Reisinger at al.\ \cite{Reisinger19rectified} considers value functions of zero-sum games,	
	Grohs et al.\ \cite{grohs2019deep}  considers PDEs which can be approximated by a suitable Monte Carlo method, and
	Grohs et al.\  \cite{grohs2019space} considers ordinary differential equations and measures the error on the entire space-time domain.
Another direction of further research is to consider different ways to measure the error between the exact solution and the ANN approximation. 
In particular, Elbr\"achter et al.\ \cite{Elbrachter19dnn} measure the error with the supremum norm and furthermore are able to improve the speed of convergence by using additional structure on a special subclass of Black-Scholes PDEs.
Moreover, Kutyniok et al.\ \cite{Kutyniok19theoretical} prove that certain parametric maps can be approximated by ANNs without the curse of dimensionality.

The remainder of this article is organized as follows. 
In Section 2 we supply several auxiliary results on Monte Carlo approximations (Subsection~\ref{SectionMonteCarlo}), affine functions (Subsection~\ref{SectionAffine}), SDEs (Subsections~\ref{SectionMomentEstimates}--\ref{SectionAffineSDE}), and viscosity solutions for PDEs (Subsection~\ref{SectionExistenceOfViscosity}). 
These auxiliary results are then used in Section \ref{approximationSection} to establish that ANNs can approximate the solutions of Kolmogorov PDEs with affine linear drift and diffusion functions without suffering under the curse of dimensionality.
In particular, we prove in Theorem~\ref{cont_NN_approx} in Section \ref{approximationSection} the main approximation result of this article. 
In Section~\ref{blackScholesSubsection} we illustrate the application of Theorem~\ref{cont_NN_approx} in the case of the Black-Scholes PDE with different payoff functions such as, 
	basket call options in Subsection~\ref{SectionCall}, 
	basket put options in Subsection~\ref{SectionPut}, 
	call on max options in Subsection~\ref{SectionCallOnMax}, and 
	call on min options in Subsection~\ref{SectionCallOnMin}.

\section{Probabilistic and analytic preliminaries}

In this section we provide several basic and in parts well-known auxiliary results on Monte Carlo approximations (Subsection~\ref{SectionMonteCarlo}), affine functions (Subsection~\ref{SectionAffine}), SDEs (Subsections~\ref{SectionMomentEstimates}--\ref{SectionAffineSDE}), and viscosity solutions for PDEs (Subsection~\ref{SectionExistenceOfViscosity}).


\subsection{Monte Carlo approximations}\label{SectionMonteCarlo}


In this subsection we employ Kahane-Khintchine-type estimates from the literature (cf., e.g., Hyt\"onen et al.\ \cite[Theorem 6.2.4 in Subsection 6.2b]{AnaInBS2}) 
to present the known $L^p$-Monte Carlo estimate in Corollary~\ref{mc_Lp_error2} below. 
Corollary~\ref{mc_Lp_error2} is an immediate consequence of Proposition~\ref{mc_Lp_error} and  Lemma~\ref{Hytoenen_estimate} below. We include Proposition~\ref{mc_Lp_error} and  Lemma~\ref{Hytoenen_estimate} for completeness but refer to the literature for their proofs.

\begin{lemma}
\label{L2_monte_carlo}
Let $n \in \N$,
let $ ( \Omega, \mathcal{F}, \P ) $ be a probability space, 
and 
let 
$ 
	X_i  \colon \Omega \to \R 
$, $ i \in \{1, 2, \dots, n\} $,
be i.i.d.\ random variables with 
$
	\E\big[ | X_1 | \big] < \infty
$.
Then 
\begin{equation}
\left(\EXPP{  \vert
	\Exp{X_1}
	-
	\tfrac{1}{n}
	\big(
	\smallsum_{ i = 1 }^{ n } X_{ i }
	\big)
	\vert^2  }\right)^{\!1/2}
= n^{-1/2}
\left(
\E\big[ 
\vert
X_1 
-
\E[ X_1 ]
\vert^2 \,
\big]
\right)^{\! 1 / 2 }.
\end{equation}
\end{lemma}

\begin{proof}[Proof of Lemma~\ref{L2_monte_carlo}]
Note that 
the hypothesis that  
$ 
X_i\colon \Omega\to \R$, $i \in \{1, 2, \dots, n\}
$,
are i.i.d.\ random variables
assures that
\begin{equation}
\begin{split}
	&\Exp{  \left|
		\Exp{X_1}
		-
		\tfrac{1}{n}
		\big(
			\smallsum_{ i = 1 }^{ n } X_{ i }
		\big)
	\right|^2  } \\
&=
	\tfrac{1}{n^2} 
	\left[
		\smallsum_{ i,j= 1 }^{ n }
		\Exp{  ( \Exp{X_i} - X_{ i } ) ( \Exp{X_j} - X_{ j } )  }	 
	\right] \\
&=
	\tfrac{1}{n^2} 
	\big( 
		n \,\EXPP{  | \Exp{X_1} - X_{ 1 } |^2 }
	\big)
=
	n^{-1} 
		\EXPP{ | X_1 - \Exp{X_1} \!| ^2}.
\end{split}
\end{equation}
The proof of Lemma~\ref{L2_monte_carlo} is thus completed.
\end{proof}

\begin{definition}\label{def:Kahane}
	Let $p,q\in (0,\infty).$ Then we denote by $\mathfrak{K}_{p,q}\in [0,\infty]$ the extended real number  given by
	\begin{align}\label{eq:Kahane}
	&\mathfrak{K}_{p,q}=\nonumber\\&\sup\!\left\{c\in [0,\infty)\colon \!\left[\begin{aligned}
	&\exists\, \text{$\R$-Banach space $(E,\norm{\cdot}_E)$}\colon
	\\&\exists\, \text{probability space $(\Omega,\mathcal{F},\P)$}\colon
	\\&\exists\, \text{$\P$-Rademacher family $r_j\colon \Omega\to \{-1,1\}, j\in\N$}\colon
	\\&\exists\, k\in\N\colon \exists\, x_1,x_2,\dots,x_k\in E\backslash\{0\}\colon
	\\& \left(\mathbb{E}\big[\Vert\smallsum_{j=1}^k r_j x_j\Vert_{E}^p\big]\right)^{\!1/p}=c \left(\mathbb{E}\big[\Vert\smallsum_{j=1}^k r_j x_j\Vert_{E}^q\big]\right)^{\!1/q}
	\end{aligned}\! \right]\!
	\right\}
	\end{align}
	and we call $\mathfrak{K}_{p,q}$ the
	$(p,q)$-Kahane-Khintchine constant.
\end{definition}


The next result, Proposition~\ref{mc_Lp_error} below, is, e.g., proved as Corollary 5.12 in Cox et al.\ \cite{CoxHutzenthalerJentzenNervenWelti17}. 

\begin{prop}
\label{mc_Lp_error}
Let $p \in [2,\infty)$, $d, n \in \N$, 
let $\norm{\cdot} \colon \R^d \to [0,\infty)$ be the $d$-dimensional Euclidean norm,  
let 
$
  \mathfrak{K}_{p,2} \in (0,\infty)
$
be the $(p, 2)$-Kahane-Khintchine constant (cf. Definition \ref{def:Kahane}), 
let $ ( \Omega, \mathcal{F}, \P ) $ be a probability space, 
and 
let $ X_i  \colon \Omega \to \R^d $,
$ i \in \{1, 2, \dots, n\} $,
be i.i.d.\ random variables with 
$
	\E\big[ \| X_1 \| \big] < \infty
$.
Then 
\begin{equation}
\left(\EXPP{  \Vert
	\Exp{X_1}
	-
	\tfrac{1}{n}
	\big(
	\smallsum_{ i = 1 }^{ n } X_{ i }
	\big)
	\Vert^p  }\right)^{\!1/p}
  \leq
  \frac{ 2 \, \mathfrak{K}_{p,2} }{
  	\sqrt{n}
  } 
    \left(
    \E\big[ 
      \|
        X_1 
        -
        \E[ X_1 ]
      \|^p \,
    \big]
    \right)^{\! 1 / p }
  .
\end{equation}
\end{prop}

%
%

The next result, Lemma~\ref{Hytoenen_estimate} below, is, e.g., proved as  Theorem 6.2.4 in Subsection 6.2b in Hyt\"onen et al.\ \cite{AnaInBS2}.

\begin{lemma}
\label{Hytoenen_estimate}
Let $ p, q \in [1,\infty)$ satisfy that $q < p$ and let $\mathfrak{K}_{p,q}$ be the $(p,q)$-Kahane-Khintchine constant (cf. Definition \ref{def:Kahane}). 
Then 
\begin{equation}
	\mathfrak{K}_{p,q} \leq \sqrt{ \frac{p-1}{q-1}}.
\end{equation}
\end{lemma}

%

\begin{cor}
	\label{lem:kahane-const}
	Let $ p \in (2,\infty)$ and let $\mathfrak{K}_{p,2}$ be the $(p,2)$-Kahane-Khintchine constant (cf. Definition \ref{def:Kahane}). 
	Then
	\begin{equation}
		\mathfrak{K}_{p,2} \leq \sqrt{p-1}.
	\end{equation}
\end{cor}

\begin{proof}[Proof of Corollary~\ref{lem:kahane-const}]
Note that Lemma~\ref{Hytoenen_estimate} assures that
\begin{equation}
	\mathfrak{K}_{p,2}
\leq 
	\sqrt{ \frac{p-1}{2-1}}
=
	\sqrt{p-1}.
\end{equation}
The proof of Corollary~\ref{lem:kahane-const} is thus completed.
\end{proof}

\begin{cor}
\label{mc_Lp_error2}
Let $p \in [2,\infty)$, $d, n \in \N$, 
let $\norm{\cdot} \colon \R^d \to [0,\infty)$ be the $d$-dimensional Euclidean norm,  
let $ ( \Omega, \mathcal{F}, \P ) $ be a probability space, 
and 
let $ X_i  \colon \Omega \to \R^d $,
$ i \in \{1, 2, \dots, n\} $,
be i.i.d.\ random variables with 
$
	\E\big[ \| X_1 \| \big] < \infty
$.
Then
\begin{equation}
\left(\EXPP{  \Vert
	\Exp{X_1}
	-
	\tfrac{1}{n}
	\big(
	\smallsum_{ i = 1 }^{ n } X_{ i }
	\big)
	\Vert^p  }\right)^{\!1/p}
  \leq
   2 \left[\frac{p-1}{n}\right]^{\!1/2}
       \left(
   \E\big[ 
   \|
   X_1 
   -
   \E[ X_1 ]
   \|^p \,
   \big]
   \right)^{\! 1 / p }.
\end{equation}
\end{cor}

\begin{proof}[Proof of Corollary~\ref{mc_Lp_error2}]
Throughout this proof assume without loss of generality that $p > 2$ (cf.\ Lemma~\ref{L2_monte_carlo}).
Note that Proposition	~\ref{mc_Lp_error} and Corollary~\ref{lem:kahane-const} demonstrate that
\begin{equation}
\begin{split}
\left(\EXPP{  \Vert
	\Exp{X_1}
	-
	\tfrac{1}{n}
	\big(
	\smallsum_{ i = 1 }^{ n } X_{ i }
	\big)
	\Vert^p  }\right)^{\!1/p}
  &\leq
  \frac{ 2 \, \mathfrak{K}_{p,2} }{
 	\sqrt{n}
 } 
  \left( \E\big[ 
\|
X_1 
-
\E[ X_1 ]
\|^p \,
\big]
\right)^{\!1 / p }
 \\
&\leq
   \frac{ 2 \sqrt{p-1} }{
 	\sqrt{n}
 } 
  \left( \E\big[ 
\|
X_1 
-
\E[ X_1 ]
\|^p \,
\big]
\right)^{\!1 / p }
    \\&=     2 \left[\frac{p-1}{n}\right]^{\!1/2}
  \left( \E\big[ 
\|
X_1 
-
\E[ X_1 ]
\|^p \,
\big]
\right)^{\! 1 / p }.
\end{split}
\end{equation}
The proof of Corollary~\ref{mc_Lp_error2} is thus completed.
\end{proof}

\subsection{Properties of affine functions}\label{SectionAffine}

This subsection recalls in Lemmas~\ref{affine_property}--\ref{affine_representation} and Corollaries~\ref{affine_characterization}--\ref{linear_growth_affine_HilbertSchmidt} a few well-known properties for affine functions. For the sake of completeness we include in this subsection also proofs for Lemmas~\ref{affine_property}--\ref{affine_representation} and Corollaries~\ref{affine_characterization}--\ref{linear_growth_affine_HilbertSchmidt}.

\begin{lemma}
	\label{affine_property}
	Let $d, m \in \N$, $A \in \R^{m \times d}$, $b \in \R^m$ and 
	let $\varphi \colon \R^d \to \R^m$ be the function which satisfies 
	for all $x \in \R^d$ that
	\begin{equation}
	\label{affine_property:ass1}
	\varphi(x) = Ax+b.
	\end{equation}
	Then
	it holds 
	for all $x,y \in \R^d$, $\lambda \in \R$ that
	\begin{equation}
	\varphi(\lambda x + y) + \lambda \varphi(0) 
	= 
	\lambda\varphi(x) + \varphi(y).
	\end{equation}
\end{lemma}

\begin{proof}[Proof of Lemma~\ref{affine_property}]
	Observe that \eqref{affine_property:ass1} assures that
	for all $x,y \in \R^d$, $\lambda \in \R$ it holds that
	\begin{equation}
	\begin{split}
	\varphi(\lambda x + y) + \lambda \varphi(0) 
	&=
	A (\lambda x + y) + b + \lambda ( A \cdot 0 + b ) \\
	&=
	\lambda (A x + b) + A y + b
	=
	\lambda\varphi(x) + \varphi(y).
	\end{split}
	\end{equation}
	The proof of Lemma~\ref{affine_property} is thus completed.
\end{proof}

\begin{lemma}
	\label{affine_representation}
	Let $d, m \in \N$, 
	$e_1, e_2, \ldots, e_d \in \R^d$ satisfy 
	$
	e_1 = (1, 0, \ldots, 0)$,
	$e_2 = (0, 1, 0, \ldots, 0),
	\ldots$,
	$e_d = (0, \ldots, 0, 1)
	$,
	let $\varphi = (\varphi_1, \varphi_2, \ldots, \varphi_m) \colon \R^d \to \R^m$ be a function which satisfies 
	for all $x,y \in \R^d$, $\lambda \in \R$ that
	\begin{equation}
	\label{affine_representation:ass1}
	\varphi(\lambda x + y) + \lambda \varphi(0) 
	= 
	\lambda\varphi(x) + \varphi(y),
	\end{equation}
	and let $A \in \R^{m \times d}$, $b \in \R^m$ satisfy
	$b = \varphi(0)$
	and
	\begin{equation}
	\label{affine_representation:ass2}
	\begin{split}
	A
	&=
	\begin{pmatrix}
	\varphi_1(e_1) - \varphi_1(0) 		&\varphi_1(e_2) - \varphi_1(0)  	&\ldots 	&\varphi_1(e_d) - \varphi_1(0)	\\
	\varphi_2(e_1) - \varphi_2(0) 		&\varphi_2(e_2) - \varphi_2(0)  	&\ldots 	&\varphi_2(e_d) - \varphi_2(0)	\\
	\vdots 												&\vdots  											&\ddots 	&\vdots 										\\
	\varphi_m(e_1) - \varphi_m(0) 		&\varphi_m(e_2) - \varphi_m(0)  	&\ldots 	&\varphi_m(e_d) - \varphi_m(0)
	\end{pmatrix} \\
	&=
	\bigg(
	\varphi(e_1) - \varphi(0) 
	\Big| \,
	\varphi(e_2) - \varphi(0) 
	\Big|
	\cdots
	\Big| \,
	\varphi(e_d) - \varphi(0)
	\bigg).
	\end{split}
	\end{equation}
	Then
	it holds
	for all $x \in \R^d$ that
	\begin{equation}
	\varphi(x) = Ax+b.
	\end{equation}
\end{lemma}

\begin{proof}[Proof of Lemma~\ref{affine_representation}]
	First, note that \eqref{affine_representation:ass1} implies that
	for all $x,y \in \R^d$, $\lambda \in \R$ it holds that
	\begin{equation}
	\varphi(\lambda x + y) 
	= 
	\lambda(\varphi(x) - \varphi(0)) + \varphi(y).
	\end{equation}
	This, induction, and \eqref{affine_representation:ass2} assure that 
	for all $x = (x_1,x_2, \ldots, x_d) \in \R^d$ it holds that
	\begin{equation}
	\begin{split}
	\varphi(x)
	&=	
	\varphi
	\!\left(
	\smallsum\limits_{i=1}^d x_i e_i
	\right) 
	=
	\varphi
	\!\left(
	x_1 e_1 
	+
	\smallsum\limits_{i=2}^d x_i e_i
	\right) \\
	&=
	x_1 (\varphi(e_1) - \varphi(0)) 
	+ 
	\varphi
	\!\left(
	\smallsum\limits_{i=2}^d x_i e_i
	\right) \\
	&=
	\left[
	\smallsum\limits_{i=1}^{\min\{1,d\}} x_i (\varphi(e_i) - \varphi(0))	
	\right]
	+ 
	\varphi \!\left(\smallsum\limits_{i=\min\{1,d\}+1}^d x_i e_i\right) \\
	&=
	\left[
	\sum_{i=1}^{\min\{2,d\}} x_i (\varphi(e_i) - \varphi(0))	
	\right]
	+ 
	\varphi \!\left(\smallsum\limits_{i=\min\{2,d\}+1}^d x_i e_i\right) \\
	&= \dots \\
	&=
	\left[
	\smallsum\limits_{i=1}^{\min\{d,d\}} x_i (\varphi(e_i) - \varphi(0))	
	\right]
	+ 
	\varphi \!\left(\smallsum\limits_{i=\min\{d,d\}+1}^d x_i e_i\right) \\
	&=
	\left[
	\smallsum\limits_{i=1}^{d} x_i (\varphi(e_i) - \varphi(0))	
	\right]
	+ 
	\varphi (0) \\
	&=
	A x + b.
	\end{split}
	\end{equation}
	The proof of Lemma~\ref{affine_representation} is thus completed.
\end{proof}

\begin{cor}
	\label{affine_characterization}
	Let $d, m \in \N$ and let $\varphi \colon \R^d \to \R^m$ be a function. 
	Then the following two statements are equivalent:
	\begin{enumerate}[(i)]
		\item \label{affine_characterization:item1} 
		There exist $A \in \R^{m \times d}$, $b \in \R^m$ such that 
		for all $x \in \R^d$ it holds that
		\begin{equation}
		\varphi(x) = Ax+b.
		\end{equation}
		
		\item \label{affine_characterization:item2}  
		It holds 
		for all $x,y \in \R^d$, $\lambda \in \R$ that
		\begin{equation}
		\varphi(\lambda x + y) + \lambda \varphi(0) 
		=
		\lambda\varphi(x) + \varphi(y).
		\end{equation}
	\end{enumerate}
\end{cor}

\begin{proof}[Proof of Corollary~\ref{affine_characterization}]
	Note that Lemma~\ref{affine_property} establishes that 
	(\eqref{affine_characterization:item1} $\Rightarrow$ \eqref{affine_characterization:item2}).
	In addition, observe that Lemma~\ref{affine_representation} demonstrates that
	(\eqref{affine_characterization:item2} $\Rightarrow$ \eqref{affine_characterization:item1}).
	The proof of Corollary \ref{affine_characterization} is thus completed.
\end{proof}

\begin{cor}
	\label{linear_growth_affine}
	Let $d, m \in \N$, 
	let $\varphi \colon \R^d \to \R^m$ be a function which satisfies
	for all $x,y \in \R^d$, $\lambda \in \R$ that
	\begin{equation}
	\varphi(\lambda x + y) + \lambda \varphi(0) 
	=
	\lambda\varphi(x) + \varphi(y),
	\end{equation}
	and
	for every $k \in \N$ 
	let $\left\| \cdot \right\|_{\R^k} \colon \R^k \to [0,\infty)$ be the $k$-dimensional Euclidean norm.
	Then 
	there exists $c \in [0,\infty)$ such that
	for all $x, y \in \R^d$ it holds that
	\begin{equation}
	\label{linear_growth_affine:concl1}
	\norm{\varphi(x)}_{\R^m} 
	\leq 
	c( 1 + \norm{x}_{\R^d})
	\quad \text{and} \quad
	\norm{\varphi(x) - \varphi(y)}_{\R^m} 
	\leq	
	c \norm{x - y}_{\R^d}.
	\end{equation}
\end{cor}

\begin{proof}[Proof of Corollary~\ref{linear_growth_affine}]
	Throughout this proof 
	let $A \in \R^{m \times d}$, $b \in \R^m$ satisfy
	for all $x \in \R^d$ that
	\begin{equation}
	\label{linear_growth_affine:setting1}
	\varphi(x) = Ax+b
	\end{equation}
	(cf.\ Corollary~\ref{affine_characterization})
	and let $c \in [0,\infty)$ be given by
	\begin{equation}
	\label{linear_growth_affine:setting2}
	c
	=
	\max \left\{
	\left[\sup\nolimits_{v \in \R^d \backslash \{ 0 \}} \tfrac{ \norm{Av}_{\R^m}  }{ \norm{v}_{\R^d}}\right], \norm{b}_{\R^m}
	\right\}.
	\end{equation}
	Note that \eqref{linear_growth_affine:setting1} and \eqref{linear_growth_affine:setting2} assure that 
	for all $x \in \R^d$ it holds that
	\begin{equation}
	\label{linear_growth_affine:eq1}
	\begin{split}
	\norm{\varphi(x)}_{\R^m}
	&= 
	\norm{Ax+b}_{\R^m} 
	\leq
	\norm{Ax}_{\R^m}  + \norm{b}_{\R^m} \\
	&\leq
	\left[ 
	\sup\nolimits_{v \in \R^d \backslash \{ 0 \}} \tfrac{ \norm{Av}_{\R^m}  }{ \norm{v}_{\R^d}}
	\right]  
	\norm{x}_{\R^d} 
	+ 
	\norm{b}_{\R^m} 
	\leq	
	c ( \norm{x}_{\R^d}  + 1).
	\end{split}
	\end{equation}
	Furthermore, observe that \eqref{linear_growth_affine:setting1} and \eqref{linear_growth_affine:setting2} imply that 
	for all $x, y \in \R^d$ it holds that
	\begin{equation}
	\begin{split}
	\norm{\varphi(x) - \varphi(y)}_{\R^m}
	&= 
	\norm{(Ax+b) - (Ay+b)}_{\R^m} 
	=
	\norm{A(x - y)}_{\R^m} \\
	&\leq
	\left[ 
	\sup\nolimits_{v \in \R^d \backslash \{ 0 \}} \tfrac{ \norm{Av}_{\R^m}  }{ \norm{v}_{\R^d}}
	\right]  
	\norm{x-y}_{\R^d} \\
	&\leq	
	c \norm{x-y}_{\R^d}.
	\end{split}
	\end{equation}
	Combining this and \eqref{linear_growth_affine:eq1} establishes \eqref{linear_growth_affine:concl1}.
	The proof of Corollary~\ref{linear_growth_affine} is thus completed.
\end{proof}

\begin{cor}
	\label{linear_growth_affine_HilbertSchmidt}
	Let $d, k,m \in \N$, 
	let $\sigma \colon \R^d \to \R^{k\times m}$ be a function which satisfies
	for all $x,y \in \R^d$, $\lambda \in \R$ that
	\begin{equation}\label{affineConditionHilbertSchmidt}
	\sigma(\lambda x + y) + \lambda \sigma(0) 
	=
	\lambda\sigma(x) + \sigma(y),
	\end{equation}
	let $\left\| \cdot \right\| \colon \R^d \to [0,\infty)$ be the $d$-dimensional Euclidean norm, and let $\HSNorm{\cdot}\colon$ $\R^{k\times m} \to [0,\infty)$ be the Hilbert-Schmidt norm on $\R^{k\times m}$.
	Then 
	there exists $c \in [0,\infty)$ such that
	for all $x, y \in \R^d$ it holds that
	\begin{equation}
	\label{linear_growth_affine_HilbertSchmidt:concl1}
	\HSNorm{{\sigma(x)}}
	\leq 
	c( 1 + \norm{x})
	\qandq
	\HSNorm{{\sigma(x)-\sigma(y)}}
	\leq	
	c \norm{x - y}.
	\end{equation}
\end{cor}
\begin{proof}[Proof of Corollary~\ref{linear_growth_affine_HilbertSchmidt}]
	Throughout this proof for every $\mathfrak{d}\in\N$ let $\left\| \cdot \right\|_{\R^\mathfrak{d}} \colon \R^\mathfrak{d} \allowbreak\to [0,\infty)$ be the $\mathfrak{d}$-dimensional Euclidean norm, let $e_1, e_2, $$\ldots, e_m \in \R^m$ satisfy
	$
	e_1 = (1, 0, \ldots, 0)$,
	$e_2 = (0, 1, 0, \ldots, 0),
	\ldots$,
	$e_m = (0, \ldots, 0, 1)
	$,
	and let $\varphi \colon \R^d \to \R^{(m k)}$ be the function which satisfies
	for all $x\in \R^d$ that 
	\begin{equation}\label{defPhiMatrix}
	\varphi(x)=\begin{pmatrix} (\sigma(x))e_1 \\(\sigma(x))e_2\\\ldots \\(\sigma(x))e_m\end{pmatrix}.
	\end{equation}
	Note that \eqref{affineConditionHilbertSchmidt} and \eqref{defPhiMatrix} ensure that for all $x,y \in \R^d$, $\lambda \in \R$ it holds that
	\begin{equation}
	\begin{split}
	&\varphi(\lambda x + y) + \lambda \varphi(0) 
	= \begin{pmatrix} (\sigma(\lambda x + y))e_1 \\(\sigma(\lambda x + y))e_2\\\ldots \\(\sigma(\lambda x + y))e_m\end{pmatrix}+\lambda \begin{pmatrix} (\sigma(0))e_1 \\(\sigma(0))e_2\\\ldots \\(\sigma(0))e_m\end{pmatrix}
	\\&= \begin{pmatrix} (\sigma(\lambda x + y))e_1+\lambda (\sigma(0))e_1 \\(\sigma(\lambda x + y))e_2+\lambda (\sigma(0))e_2\\\ldots \\(\sigma(\lambda x + y))e_m+\lambda (\sigma(0))e_m\end{pmatrix}
	= \begin{pmatrix} \left[\sigma(\lambda x + y)+\lambda \sigma(0)\right]e_1 \\\left[\sigma(\lambda x + y)+\lambda \sigma(0)\right]e_2\\\ldots \\\left[\sigma(\lambda x + y)+\lambda \sigma(0)\right]e_m\end{pmatrix}	
	\\&= \begin{pmatrix} \left[\lambda \sigma(x) + \sigma(y)\right]e_1 \\\left[\lambda \sigma(x) + \sigma(y)\right]e_2\\\ldots \\\left[\lambda \sigma(x) + \sigma(y)\right]e_m\end{pmatrix}	
	=\lambda \begin{pmatrix} (\sigma(x))e_1 \\(\sigma(x))e_2\\\ldots \\(\sigma(x))e_m\end{pmatrix}+\begin{pmatrix} (\sigma(y))e_1 \\(\sigma(y))e_2\\\ldots \\(\sigma(y))e_m\end{pmatrix}	
	=\lambda\varphi(x) + \varphi(y).
	\end{split}
	\end{equation}
	This and Corollary~\ref{linear_growth_affine}
	(with $d=d$, $m=mk$, $\varphi=\varphi$ 
	in the notation of Corollary~\ref{linear_growth_affine})
	imply that there exists $c\in [0,\infty)$ such that for all $x, y \in \R^d$ it holds that
	\begin{equation}
	\label{linear_growth_affineMatrix:concl1}
	\norm{\varphi(x)}_{\R^{(m k)}} 
	\leq 
	c( 1 + \norm{x}_{\R^d})
	\qandqShort
	\norm{\varphi(x) - \varphi(y)}_{\R^{(m k)}} 
	\leq	
	c \norm{x - y}_{\R^d}.
	\end{equation} 
	Furthermore, note that for all $x,y \in \R^d$ it holds that 
	\begin{equation}
	\HSNorm{{\sigma(x)}}^2=\sum_{j=1}^m \norm{[\sigma(x)]e_j}_{\R^k}^2=\norm{\varphi(x)}_{\R^{(m k)}}^2
	\end{equation}
	and 
	\begin{equation}
	\HSNorm{{\sigma(x)-\sigma(y)}}^2=\sum_{j=1}^m \norm{[\sigma(x)-\sigma(y)]e_j}_{\R^k}^2=\norm{\varphi(x)-\varphi(y)}_{\R^{(m k)}}^2.
	\end{equation}	
	Combining this with \eqref{linear_growth_affineMatrix:concl1} ensures that for all $x, y \in \R^d$ it holds that
	\begin{equation}
	\HSNorm{{\sigma(x)}}=\norm{\varphi(x)}_{\R^{(m k)}} 
	\leq 
	c( 1 + \norm{x}_{\R^d})
	\end{equation}
	and
	\begin{equation}
	\HSNorm{{\sigma(x)-\sigma(y)}}=\norm{\varphi(x) - \varphi(y)}_{\R^{(m k)}} 
	\leq	
	c \norm{x - y}_{\R^d}.
	\end{equation} 
	The proof of Corollary~\ref{linear_growth_affine_HilbertSchmidt} is thus completed.
\end{proof}

\subsection[A priori estimates for stochastic differential equations (SDEs)]{A priori estimates for solutions of stochastic differential equations}\label{SectionMomentEstimates}

In this subsection we establish in Proposition~\ref{moments_of_solution_of_SDE} below an elementary a priori estimate for solutions of SDEs with at most linearly growing coefficient functions (see \eqref{moments_of_solution_of_SDE:ass1} in Proposition~\ref{moments_of_solution_of_SDE} below for details). 
Our proof of Proposition~\ref{moments_of_solution_of_SDE} employs 
	the Gronwall integral inequality (see Lemma~\ref{gronwall} below), 
	a special case of Minkowksi's integral inequality (see Lemma~\ref{Minkowski} below), and 
	the Burkholder-Davis-Gundy type inequality in Da Prato \& Zabczyk \cite[Lemma 7.7]{DaPratoZabczyk92} (see Lemma~\ref{lem-daPratoZabzcyk} below). 
The Gronwall inequality in Lemma~\ref{gronwall} is a well-known result in the scientific literature and its proof is therefore omitted (cf., e.g., Lemma 2.6 in Andersson et al.\ \cite{Andersson16existence} and Lemma 7.1.1 in Henry \cite{Henry81geometric}).



\begin{lemma}
\label{gronwall}
Let $\alpha,\beta, T\in [0,\infty)$ and 
let $f \colon [0,T] \to \R$ be a $\mathcal{B}([0, T]) / \mathcal{B}(\R)$-measurable function which satisfies for all $t \in [0,T]$ that $\int_0^T |f(s)| \, ds  < \infty$ and 
\begin{equation}
\label{gronwall:ass1}
	f(t) \leq  \alpha + \beta\int_0^t  f(s) \, ds.
\end{equation}
Then
it holds 
for all $t \in [0,T]$ that
\begin{equation}
	f(t) \leq \alpha e^{\beta t}.
\end{equation}
 \end{lemma}

The next result, Lemma~\ref{Minkowski} below, follows, e.g., from Garling \cite[Corollary 5.4.2]{Garling} or Jentzen \& Kloeden \cite[Corollary A.1 in Appendix A]{JentzenKloeden11}.

\begin{lemma}[Moments of pathwise integrals]
\label{Minkowski}
Let $T \in (0,\infty)$, $p \in [1,\infty)$, 
let $(\Omega, \mathcal{F},\P)$ be a probability space, 
and let $X \colon [0, T] \times \Omega \to [0,\infty)$ be a $(\mathcal{B}([0, T])\otimes\mathcal{F}) / \mathcal{B}([0,\infty))$-measurable function.
Then 
it holds 
for all $t \in [0, T]$ that
\begin{equation}
	\left(
		\Exp{ 
			\big|
				\textint_0^t  X_s  \, ds
			\big|^p 
		}
	\right)^{ \! \! \nicefrac{1}{p}}
\leq
	\int_0^t
	\left(
		\EXPP{
				 | X_s |^p
		}
	\right)^{\! \nicefrac{1}{p}}
	ds.
\end{equation}
\end{lemma}

The next result, Lemma~\ref{lem-daPratoZabzcyk} below, is, e.g., proved as Lemma 7.7 in Da Prato \& Zabczyk \cite{DaPratoZabczyk92}.

\begin{lemma}\label{lem-daPratoZabzcyk}
	Let $d,m \in \N$, $p \in [2,\infty)$, $T \in (0,\infty)$, 
	let $\norm{\cdot} \colon \R^d \to [0,\infty)$ be the $d$-dimensional Euclidean norm, 
	let $\HSNorm{\cdot}\colon \R^{d \times m} \to [0,\infty)$ be the Hilbert-Schmidt norm on $\R^{d \times m}$,
	let $(\Omega,\allowbreak \mathcal{F},\allowbreak \P,\allowbreak(\mathbbm{F}_t)_{t \in [0,T]})$ be a filtered probability space which fulfils the usual conditions, 
	let $W\colon [0,T]\times \Omega \to \R^m$ be a standard $(\mathbbm{F}_t)_{t \in [0,T]}$-Brownian motion, and let 
	$X\colon [0,T]\times \Omega\to \R^{d\times m}$ be an 
	$(\mathbbm{F}_t)_{t \in [0,T]}$-predictable
	stochastic process which satisfies 
	$\P\big(\textint_0^T \HSNorm{{X_s}}^2\,ds<\infty\big)=1.$
Then 
it holds
for all $t \in [0, T]$, $s\in[0,t]$ that
		\begin{equation}
	\left(\Exp{ 
		\norm{ 
			\int_s^t X_r \, dW_r 
		}^p 
	}
	\right)^{\!\nicefrac{1}{p}}
	\leq
	\left[\frac{p(p-1)}{2}\right]^{\nicefrac{1}{2}}
	\left[
	\int_s^t
	\big(
	\Exp{ 
		\HSNorm{ 
			{X_r}
		}^p 
	}
	\big)^{\!\nicefrac{2}{p}} 
	\, dr
	\right]^{\!\nicefrac{1}{2}}.
	\end{equation}
\end{lemma}


\begin{prop}
\label{moments_of_solution_of_SDE}
Let $d,m \in \N$, $p \in [2,\infty)$, $T, \mathfrak{m}_1, \mathfrak{m}_2, \mathfrak{s}_1, \mathfrak{s}_2 \in [0, \infty)$,  $\xi \in \R^d$,
let $\norm{\cdot} \colon \R^d \to [0,\infty)$ be the $d$-dimensional Euclidean norm, 
let $\HSNorm{\cdot} \colon \R^{d \times m} \to [0,\infty)$ be the Hilbert-Schmidt norm on $\R^{d \times m}$,
let $(\Omega,\allowbreak \mathcal{F},\allowbreak \P,\allowbreak(\mathbbm{F}_t)_{t \in [0,T]})$ be a filtered probability space which fulfils the usual conditions, 
let $W\colon [0,T]\times \Omega \to \R^m$ be a standard $(\mathbbm{F}_t)_{t \in [0,T]}$-Brownian motion, 
let $\mu \colon \R^d \to \R^d$ be  $\mathcal{B}(\R^d) / \mathcal{B}(\R^d)$-mea\-surable, let $\sigma \colon \R^d \to \R^{d \times m}$ be $\mathcal{B}(\R^d) / \mathcal{B}(\R^{d\times m})$-mea\-surable, assume 
for all $x \in \R^d$ that
\begin{equation}
\label{moments_of_solution_of_SDE:ass1}
	\norm{\mu(x)} 
\leq 
	\mathfrak{m}_1+ \mathfrak{m}_2\norm{x}
\qandq
	\HSNorm{{\sigma(x)}}
\leq 
	\mathfrak{s}_1 + \mathfrak{s}_2 \norm{x},
\end{equation}
and 
let $X  \colon [0,T]\times \Omega \to \R^d$ be an $(\mathbbm{F}_t)_{t \in [0,T]}$-adapted stochastic process with continuous sample paths which satisfies that 
for all $t \in [0,T]$  it holds $\P$-a.s.\  that
\begin{equation}
\label{moments_of_solution_of_SDE:ass2}
	X_t
= 
	\xi + \int_0^t \mu(X_s)\, ds + \int_0^t \sigma(X_s) \, dW_s.
\end{equation}
Then 
it holds
for all $t \in [0, T]$ that
\begin{equation}
\label{moments_of_solution_of_SDE:concl1}
\begin{split}
	&\left(
		\EXPP{ \Norm{X_t}^p  }
	\right)^{\!\nicefrac{1}{p}} \\
&\leq 
	\sqrt{2} 
	\Big(
		\Norm{\xi} + \mathfrak{m}_1 T + \mathfrak{s}_1  \sqrt{\tfrac{p(p - 1)T}{2}} 
	\Big)
	\exp{\!
	\bigg( \!
		\left[
			\mathfrak{m}_2 \sqrt{T} + \mathfrak{s}_2 \sqrt{\tfrac{p(p - 1)}{2}}
		\right]^2
		t
	\bigg)} \\
&\leq
	\sqrt{2} 
	\Big(
		\Norm{\xi} + \mathfrak{m}_1 T + \mathfrak{s}_1  p\sqrt{T} 
	\Big)
	\exp{\!
	\bigg( \!
		\left[
			\mathfrak{m}_2 \sqrt{T} + \mathfrak{s}_2 p
		\right]^2
		t
	\bigg)}.
\end{split}
\end{equation}
\end{prop}

\begin{proof}[Proof of Proposition~\ref{moments_of_solution_of_SDE}]
Throughout this proof assume without loss of generality that $T>0$ and
let $\tau_n\colon \Omega \to [0,T]$, $n\in\N$, be the functions which satisfy for every $n\in\N$ that 
\begin{equation}\label{moments_of_solution_of_SDE:defStoppingTimes}
	\tau_n=\inf(\{t\in[0,T]\colon \norm{X_t}> n\}\cup \{T\}).
\end{equation}
Note that the hypothesis that $X  \colon [0,T]\times \Omega \to \R^d$ is an 
$(\mathbbm{F}_t)_{t \in [0,T]}$-adapted stochastic process with continuous sample paths ensures that for all $t\in(0,T]$, $n\in\N$ it holds that 
\begin{equation}
\begin{split}
\{\tau_n<t\}&=\{\exists\, s\in [0,t)\colon\, \norm{X_s}> n\}\\
&=\{\exists\, s\in [0,t)\cap \Q \colon\, \norm{X_s}> n\}\\
&=\left(\cup_{s\in [0,t)\cap \Q} \{\norm{X_s}> n\}\right)\in \mathbbm{F}_t.
\end{split}
\end{equation}
This demonstrates that for all  $t\in [0,T)$,  $r\in (t,T]$, $n\in\N$ it holds that 
\begin{equation}
\{\tau_n\le t\}=\left(\cap_{k\in\N} \{\tau_n< t+\tfrac{1}{k}\}\right)=\left(\cap_{k\in\N,\, t+\nicefrac{1}{k}\le r} \{\tau_n< t+\tfrac{1}{k}\}\right)\in \mathbbm{F}_r.
\end{equation}
The hypothesis that $(\Omega,\allowbreak \mathcal{F},\allowbreak \P,\allowbreak(\mathbbm{F}_t)_{t \in [0,T]})$ fulfils the usual conditions
hence ensures that for all  $t\in [0,T)$, $n\in\N$ it holds that $\{\tau_n\le t\}\in \mathbbm{F}_t^+=\mathbbm{F}_t.$ Therefore, we obtain that for all $n\in\N$ it holds that $\tau_n$ is an $(\mathbbm{F}_t)_{t \in [0,T]}$-stopping time.
%
Moreover, observe that \eqref{moments_of_solution_of_SDE:ass2} and the triangle inequality assure that 
for all $t \in [0, T]$, $n\in\N$ it holds that
\begin{equation}
\label{moments_of_solution_of_SDE:eq1}
\begin{split}
	\left(
		\EXPP{ \Norm{X_{\min\{t,\tau_n\}}}^p  }
	\right)^{\!\nicefrac{1}{p}}
&\leq
	\norm{\xi}
	+ 
	\left(
		\Exp{ 
			\norm{
				\int_0^{\min\{t,\tau_n\}} \mu(X_s)\, ds 
			}^p \,
		}
	\right)^{ \! \! \nicefrac{1}{p}}\\&\quad
	+ 
	\left(
		\Exp{ 
			\norm{
				\int_0^{\min\{t,\tau_n\}} \sigma(X_s) \, dW_s 
			}^p \,
		}
	\right)^{ \! \! \nicefrac{1}{p}}.
\end{split}
\end{equation}
Next note that Lemma~\ref{Minkowski}, \eqref{moments_of_solution_of_SDE:ass1}, and the triangle inequality demonstrate that
for all $t \in [0, T]$, $n\in\N$ it holds that
\begin{equation}
\begin{split}
	\left(
		\Exp{ 
			\norm{
				\int_0^{\min\{t,\tau_n\}} \mu(X_s)\, ds 
			}^p \,
		}
	\right)^{ \! \! \nicefrac{1}{p}}
&\leq
		\left(
		\Exp{ 
			\left|
				\int_0^{\min\{t,\tau_n\}} \norm{\mu(X_s)}\, ds 
			\right|^p \,
		}
	\right)^{ \! \! \nicefrac{1}{p}} \\
&\leq
	\int_0^t
		\left(
			\EXPP{ 
				 \Norm{\mu(X_s)}^p \mathbbm{1}_{\{s\le \tau_n\}}
			}
		\right)^{ \!  \nicefrac{1}{p}}
	ds \\
	&\leq
	\int_0^t
	\left(
	\EXPP{ 
		\Norm{\mu(X_{\min\{s,\tau_n\}})}^p 
	}
	\right)^{ \!  \nicefrac{1}{p}}
	ds \\
	&\leq
	\int_0^t
	\left(
	\EXPP{
		\big(  \mathfrak{m}_1 + \mathfrak{m}_2 \Norm{X_{\min\{s,\tau_n\}}}  \big)^p 
	}
	\right)^{ \!  \nicefrac{1}{p}}
	ds \\
&\leq
		 \int_0^t
			\left[\mathfrak{m}_1 + 
			\mathfrak{m}_2 
			\left(
				\EXPP{
					 \Norm{X_{\min\{s,\tau_n\}}}^p 
				}
			\right)^{ \!  \nicefrac{1}{p}}\right]
		ds\\
&=
	 \mathfrak{m}_1 t
	 +
	 \mathfrak{m}_2
	 \left(
	 	\int_0^t
			\left(
				\EXPP{ 
					 \Norm{X_{\min\{s,\tau_n\}}}^p 
				}
			\right)^{ \!  \nicefrac{1}{p}}
		ds 
	\right).
\end{split}
\end{equation}
The Cauchy-Schwarz inequality hence proves that 
for all $t \in [0, T]$, $n\in\N$ it holds that
\begin{equation}
\label{moments_of_solution_of_SDE:eq2}
\begin{split}
	&\left(
		\Exp{ 
			\norm{
				\int_0^{\min\{t,\tau_n\}} \mu(X_s)\, ds 
			}^p \,
		}
	\right)^{ \! \! \nicefrac{1}{p}}\\
&\leq
	\mathfrak{m}_1 T
	+
	\mathfrak{m}_2
	\left[
		\int_0^t
		 	1^2 \,
		 ds
	\right]^{\nicefrac{1}{2}}
	\left[
		\int_0^t
		 	\left(
				\EXPP{\Norm{X_{\min\{s,\tau_n\}}}^p }
			\right)^{ \!  \nicefrac{2}{p}}
		ds
	\right]^{\nicefrac{1}{2}} \\
&\leq
	\mathfrak{m}_1 T
	+
	\mathfrak{m}_2
	\sqrt{T}
	\left[
		\int_0^t
			\left(
				\EXPP{\Norm{X_{\min\{s,\tau_n\}}}^p }
			\right)^{ \!  \nicefrac{2}{p}}
		ds
	\right]^{\nicefrac{1}{2}}.
\end{split}
\end{equation}
Moreover, note that  the hypothesis that $X  \colon [0,T]\times \Omega \to \R^d$ is an 
$(\mathbbm{F}_t)_{t \in [0,T]}$-adapted stochastic process with continuous sample paths
shows that $X\colon [0,T]\allowbreak\times \Omega \to \R^d$ is an 
$(\mathbbm{F}_t)_{t \in [0,T]}$-predictable
stochastic process.
 The fact that for every $n\in\N$ it holds that $([0,T]\times \Omega \ni (t,\omega)\mapsto \mathbbm{1}_{\{t\le \tau_n(\omega)\}} \in \{0,1\})$
is an $(\mathbbm{F}_t)_{t \in [0,T]}$-predictable
stochastic process (cf., e.g., Kallenberg \cite[Lemma 22.1]{Kallenberg1997}) and the hypothesis that $\sigma \colon \R^d \to \R^{d \times m}$ is a $\mathcal{B}(\R^d) / \mathcal{B}(\R^{d\times m})$-measurable function 
%
  hence ensure that for every $n\in\N$ it holds that 
  \begin{equation}\label{moments_of_solution_of_SDE:indicatorProcess}
  	([0,T]\times \Omega \ni (t,\omega)\mapsto \sigma(X_t(\omega))\mathbbm{1}_{\{t\le \tau_n(\omega)\}} \in \R^{d\times m})
  \end{equation}
is an $(\mathbbm{F}_t)_{t \in [0,T]}$-predictable
stochastic process. 
Combining this, 
 \eqref{moments_of_solution_of_SDE:ass1}, and \eqref{moments_of_solution_of_SDE:defStoppingTimes}
 with the hypothesis that 
 $X  \colon [0,T]\times \Omega \to \R^d$ has continuous sample paths demonstrates that for all $n\in\N\cap (\Norm{\xi},\infty)$ it holds  that 
\begin{equation}
\begin{split}
\int_0^T \HSNorm{{\sigma(X_s)\mathbbm{1}_{\{s\le \tau_n\}}}}^2 \,ds
&\le T \left[\sup_{s\in[0,\tau_n]}\HSNorm{{\sigma(X_s)}}^2\right]\\
&\le T \left[\sup_{s\in[0,\tau_n]}\,\left[(\mathfrak{s}_1+\mathfrak{s}_2\Norm{X_s})^2\right]\right]\\
&\le T(\mathfrak{s}_1+\mathfrak{s}_2n)^2
<\infty.
\end{split}
\end{equation}
Lemma \ref{lem-daPratoZabzcyk}, \eqref{moments_of_solution_of_SDE:indicatorProcess}, \eqref{moments_of_solution_of_SDE:ass1},
 and the triangle inequality therefore establish that
for all $t \in [0, T]$, $n\in\N\cap (\Norm{\xi},\infty)$ it holds that
\begin{equation}
\label{moments_of_solution_of_SDE:eq3}
\begin{split}
	&\left(
		\Exp{ 
			\norm{
				\int_0^{\min\{t,\tau_n\}} \sigma(X_s) \, dW_s 
			}^p \,
		}
	\right)^{ \! \! \nicefrac{1}{p}}\\
		&=\left(
	\Exp{ 
		\norm{
			\int_0^{t}  \sigma(X_s) \mathbbm{1}_{\{s\le \tau_n\}} \, dW_s 
		}^p \,
	}
	\right)^{ \! \! \nicefrac{1}{p}}\\
&\leq	
	\sqrt{\tfrac{p(p-1)}{2}}
	\left(
		\int_0^t
			\left(
				\Exp{ 
					\HSNorm{{ \sigma(X_s)}}^p \mathbbm{1}_{\{s\le \tau_n\}}
				}
			\right)^{\nicefrac{2}{p}}
		ds
	\right)^{ \! \! \nicefrac{1}{2}} \\
	&\leq	
	\sqrt{\tfrac{p(p-1)}{2}}
	\left(
	\int_0^t
	\left(
	\Exp{ 
		\HSNormStandard{{ \sigma(X_{\min\{s,\tau_n\}})}}^p
	}
	\right)^{\nicefrac{2}{p}}
	ds
	\right)^{ \! \! \nicefrac{1}{2}} \\
	&\leq	
	\sqrt{\tfrac{p(p-1)}{2}}
	\left(
	\int_0^t
	\left(
	\Exp{
		\big(  \mathfrak{s}_1 + \mathfrak{s}_2 \norm{X_{\min\{s,\tau_n\}}}  \big)^p
	}
	\right)^{\nicefrac{2}{p}}
	ds
	\right)^{ \! \! \nicefrac{1}{2}}\\
&\leq	
	\sqrt{\tfrac{p(p-1)}{2}}	
	\left(
		\int_0^t
			\left(
				\mathfrak{s}_1 +
				\mathfrak{s}_2
				\left(
					\EXPP{ \Norm{X_{\min\{s,\tau_n\}}}^p}
				\right)^{\nicefrac{1}{p}}
			\right)^{2}
		ds
	\right)^{ \! \! \nicefrac{1}{2}} \\
&\leq	
	\sqrt{\tfrac{p(p-1)}{2}}	
	\left(
		\mathfrak{s}_1
		\left[
			\int_0^t
				1^2 \,
			ds
		\right]^{ \nicefrac{1}{2}}
		+
		\mathfrak{s}_2
		\left[
			\int_0^t
				\left(
					\EXPP{  \Norm{X_{\min\{s,\tau_n\}}}^p}
				\right)^{\nicefrac{2}{p}}
			ds
		\right]^{ \nicefrac{1}{2}}
	\right) \\
&\leq
	\mathfrak{s}_1
	\sqrt{\tfrac{p(p-1)T}{2}}	
	+
	\mathfrak{s}_2
	\sqrt{\tfrac{p(p-1)}{2}}	
	\left[
		\int_0^t
			\left(
				\EXPP{ \Norm{X_{\min\{s,\tau_n\}}}^p}
			\right)^{\nicefrac{2}{p}}
		ds
	\right]^{ \nicefrac{1}{2}}.
\end{split}
\end{equation}
Combining this, \eqref{moments_of_solution_of_SDE:eq1}, and \eqref{moments_of_solution_of_SDE:eq2} proves that
for all $t \in [0, T]$, $n\in\N\cap (\Norm{\xi},\infty)$ it holds that
\begin{equation}
\begin{split}
	&\left(
		\EXPP{ \Norm{X_{\min\{t,\tau_n\}}}^p  }
	\right)^{\!\nicefrac{1}{p}} \\
&\leq
	\norm{\xi}
	+ 
	\mathfrak{m}_1 T
	+ 
	\mathfrak{s}_1
	\sqrt{\tfrac{p(p-1)T}{2}}	\\
&\qquad
	+ 
	\left(
		\mathfrak{m}_2 \sqrt{T}
		+
		\mathfrak{s}_2 \sqrt{\tfrac{p(p-1)}{2}}	
	\right)
	\left[
		\int_0^t
			\left(
				\EXPP{ \Norm{X_{\min\{s,\tau_n\}}}^p }
			\right)^{ \!  \nicefrac{2}{p}}
		ds
	\right]^{\nicefrac{1}{2}}.
\end{split}
\end{equation}
The fact that 
for all $x, y \in \R$ it holds that
$
	| x + y |^2 \leq 2 (x^2 + y^2)
$
therefore demonstrates that
for all $t \in [0, T]$, $n\in\N\cap (\Norm{\xi},\infty)$ it holds that
\begin{equation}
\label{moments_of_solution_of_SDE:eq4}
\begin{split}
	&\left(
		\EXPP{ \Norm{X_{\min\{t,\tau_n\}}}^p  }
	\right)^{\!\nicefrac{2}{p}} \\
&\leq
	2\left[
		\norm{\xi}
		+ 
		\mathfrak{m}_1 T
		+ 
		\mathfrak{s}_1
		\sqrt{\tfrac{p(p-1)T}{2}}
	\right]^2	\\
&\qquad
	+ 
	2
	\left[
		\mathfrak{m}_2 \sqrt{T}
		+
		\mathfrak{s}_2 \sqrt{\tfrac{p(p-1)}{2}}	
	\right]^2
	\left[
		\int_0^t
			\left(
				\EXPP{ \Norm{X_{\min\{s,\tau_n\}}}^p }
			\right)^{ \!  \nicefrac{2}{p}}
		ds
	\right].
\end{split}
\end{equation}
Next note that \eqref{moments_of_solution_of_SDE:defStoppingTimes} ensures for all $n\in\N\cap (\Norm{\xi},\infty)$ that 
\begin{equation}
\begin{split}
\int_0^T
\left(
\EXPP{ \Norm{X_{\min\{s,\tau_n\}}}^p }
\right)^{ \!  \nicefrac{2}{p}} 
ds 
&\leq
\int_0^T
\left(
\EXPP{ n^p }
\right)^{ \!  \nicefrac{2}{p}} 
ds =
T
n^2
< \infty.
\end{split}
\end{equation}
Combining this and \eqref{moments_of_solution_of_SDE:eq4} with  Lemma~\ref{gronwall} (with
$
	\alpha 
= 
	2\big[
		\Norm{\xi} + \mathfrak{m}_1 T +  \mathfrak{s}_1 \allowbreak\sqrt{p(p-1)T/2}
	\big]^2
$,
$
	\beta
=
	2
	\big[
		\mathfrak{m}_2\sqrt{T}
		+
		\mathfrak{s}_2 \sqrt{p(p-1)/{2}}	
	\big]^2
$,
$
T = T
$,
$
f
=
\big(
[0,T]\ni t\mapsto\left(
\EXPP{ \Norm{X_{\min\{t,\tau_n\}}}^p }
\right)^{ \! 2/p}  
\in \R\big)
$
in the notation of Lemma~\ref{gronwall})
demonstrates that 
for all $t \in [0, T]$, $n\in\N\cap (\Norm{\xi},\infty)$ it holds that
\begin{equation}
\begin{split}
	&\left(
		\EXPP{ \Norm{X_{\min\{t,\tau_n\}}}^p  }
	\right)^{\!\nicefrac{2}{p}} \\
&\leq
	2\left[
		\norm{\xi}
		+ 
		\mathfrak{m}_1 T
		+ 
		\mathfrak{s}_1
		\sqrt{\tfrac{p(p-1)T}{2}}
	\right]^2 
	\exp \!
	\left(
		2
		\left[
			\mathfrak{m}_2\sqrt{T}
			+
			\mathfrak{s}_2 \sqrt{\tfrac{p(p-1)}{2}}	
		\right]^2 t
	\right).
\end{split}
\end{equation}
Therefore, we obtain that for all $t \in [0, T]$, $n\in\N\cap (\Norm{\xi},\infty)$ it holds that
\begin{equation}\label{{moments_of_solution_of_SDE:estimateStoppedProcess}}
\begin{split}
	&\left(
		\EXPP{ \Norm{X_{\min\{t,\tau_n\}}}^p  }
	\right)^{\!\nicefrac{1}{p}} \\
&\leq
	\sqrt{2}
	\left[
		\norm{\xi}
		+ 
		\mathfrak{m}_1 T
		+ 
		\mathfrak{s}_1\sqrt{\tfrac{p(p-1)}{2}}
		\sqrt{T}
	\right] 
	\exp \!
	\left(
		\left[
			\mathfrak{m}_2 \sqrt{T}
			+
			\mathfrak{s}_2 \sqrt{\tfrac{p(p-1)}{2}}
		\right]^2 t
	\right).
\end{split}
\end{equation}
Furthermore, observe that \eqref{moments_of_solution_of_SDE:defStoppingTimes} and the fact that $X  \colon [0,T]\times \Omega \to \R^d$ is a stochastic process with continuous sample paths
ensure that for all $t\in[0,T]$ it holds  that $\lim_{n\to\infty}\min\{t,\tau_n\}=t$. Therefore, we obtain  that for all $t\in[0,T]$ it holds that
\begin{equation}
	\Norm{X_{t}}=\Norm{X_{(\lim_{n\to\infty}\min\{t,\tau_n\})}}=\Norm{\lim_{n\to\infty} X_{\min\{t,\tau_n\}}}
	=\lim_{n\to\infty} \Norm{ X_{\min\{t,\tau_n\}}}.
\end{equation}
Fatou's Lemma and \eqref{{moments_of_solution_of_SDE:estimateStoppedProcess}} hence imply for all $t\in[0,T]$ that 
\begin{equation}
\begin{split}
&\left(
\EXPP{ \Norm{X_{t}}^p  }
\right)^{\!\nicefrac{1}{p}}
=\left(
\EXPP{\lim_{n\to\infty} \Norm{X_{\min\{t,\tau_n\}}}^p  }
\right)^{\!\nicefrac{1}{p}}\\
&\le \left(\liminf_{n\to\infty}
\EXPP{\Norm{X_{\min\{t,\tau_n\}}}^p  }
\right)^{\!\nicefrac{1}{p}}\le \sup_{n\in\N\cap (\Norm{\xi},\infty)}\left(
\EXPP{\Norm{X_{\min\{t,\tau_n\}}}^p  }
\right)^{\!\nicefrac{1}{p}} \\
&\leq
\sqrt{2}
\left[
\norm{\xi}
+ 
\mathfrak{m}_1 T
+ 
\mathfrak{s}_1\sqrt{\tfrac{p(p-1)}{2}}
\sqrt{T}
\right] 
\exp 
\!\left(
\left[
\mathfrak{m}_2 \sqrt{T}
+
\mathfrak{s}_2 \sqrt{\tfrac{p(p-1)}{2}}
\right]^2 t
\right).
\end{split}
\end{equation}
The fact that 
$\sqrt{\tfrac{p(p-1)}{2}}	
\leq	
\sqrt{p^2 - p}
\leq
\sqrt{p^2}
=
p$
therefore establishes \eqref{moments_of_solution_of_SDE:concl1}.
The proof of Proposition~\ref{moments_of_solution_of_SDE} is thus completed.
\end{proof}


\subsection[SDEs with affine coefficient functions]{Stochastic differential equations with affine coefficient functions}\label{SectionAffineSDE}

In this subsection we establish in Proposition~\ref{affine_solutions_of_SDEs} elementary regularity properties for SDEs with affine coefficient functions. Our proof of Proposition~\ref{affine_solutions_of_SDEs}, roughly speaking, employs the elementary results in Lemma~\ref{affine_SDE_fixedpoint} and Proposition~\ref{affine_SDE_allpoints} (which are, loosely speaking, alleviated versions of Proposition~\ref{affine_solutions_of_SDEs}), the well-known fact that modifications of continuous stoch-astic processes are indistinguishable (cf.\ Lemma~\ref{mod_indist} below), the well-known fact that a modification of an adapted stochastic process is an adapted stochastic process (see Lemma~\ref{normal_adapted} below for details), and a version of the Kolmogorov-Chentsov theorem (see Lemma~\ref{Kolmorogov_chentsov} below for details).
Results similar to Lemma~\ref{Kolmorogov_chentsov} can, e.g., be found in Cox et al.\ \cite[Theorem 3.5 in Subsection 3.1]{CoxHutzenthalerJentzen14} and Mittmann \& Steinwart \cite[Theorem 2.1 in Section 2]{mittmannSteinwart}.
  For the sake of completeness we include in this subsection also proofs for  Lemmas~\ref{mod_indist} and \ref{normal_adapted}.
\begin{lemma}
 \label{affine_SDE_fixedpoint}
 Let $d\in \N$, $T \in (0,\infty)$, 
 let $(\Omega,\allowbreak \mathcal{F},\allowbreak\P,\allowbreak(\mathbbm{F}_t)_{t \in [0,T]})$ be a filtered probability space which fulfils the usual conditions, 
 let $W\colon [0,T]\times \Omega \to \R^d$ be a standard $(\mathbbm{F}_t)_{t \in [0,T]}$-Brownian motion, 
 let $\mu \colon \R^d \to \R^d$ and $\sigma \colon \R^d \to \R^{d \times d}$ be functions which satisfy 
 for all $x,y \in \R^d$, $\lambda \in \R$ that 
 \begin{equation}\label{affine_SDE_fixedpointMu}
 	\mu(\lambda x + y) + \lambda \mu(0) = \lambda \mu(x) + \mu(y)
 \end{equation}
 and 
 \begin{equation}\label{affine_SDE_fixedpointSigma}
 	\sigma(\lambda x + y) + \lambda \sigma(0) = \lambda \sigma(x) + \sigma(y),
 \end{equation} 
 and 
 let $X^x  \colon [0,T]\times \Omega \to \R^d$, $x \in \R^d$, be  $(\mathbbm{F}_t)_{t \in [0,T]}$-adapted stochastic processes with continuous sample paths which satisfy that 
 for all $x \in \R^d$, $t \in [0,T]$ it holds $\P$-a.s.\  that
 \begin{equation}
 \label{affine_SDE_fixedpoint:ass1}
   X_t^x 
 = 
   x + \int_0^t \mu(X_s^x)\, ds + \int_0^t \sigma(X_s^x) \, dW_s.
 \end{equation}
 Then 
 it holds 
 for all $t \in [0,T]$, $x,y \in \R^d$, $\lambda \in \R$ that
 \begin{equation}
 \label{affine_SDE_fixedpoint:concl1}
   \P\!\left( X_t^{\lambda x +y}+\lambda X_t^0 = \lambda X_t^x + X_t^y \right) = 1.
 \end{equation}
 \end{lemma}

\begin{proof}[Proof of Lemma~\ref{affine_SDE_fixedpoint}]
Throughout this proof 
let $x,y \in \R^d$, $\lambda \in \R$ and
let $Y  \colon [0,T]\times \Omega \to \R^d$ be the stochastic process which satisfies 
for all $t \in [0,T]$ that
\begin{equation}\label{affine_SDE_fixedpointY}
	Y_t = \lambda( X_t^x -  X_t^0 ) + X_t^y.
\end{equation}
Note that the hypothesis that 
for all $z \in \R^d$ it holds that 
$
	X^z \colon [0,T]\times \Omega \to \R^d
$
is an $(\mathbbm{F}_t)_{t \in [0,T]}$-adapted stochastic process with continuous sample paths assures that 
$Y$ is an $(\mathbbm{F}_t)_{t \in [0,T]}$-adapted stochastic process with continuous sample paths.
Moreover, observe that \eqref{affine_SDE_fixedpoint:ass1} and \eqref{affine_SDE_fixedpointY} ensure that 
for all $t \in [0,T]$ it holds $\P$-a.s.\  that
\begin{equation}
\label{affine_SDE_fixedpoint:eq1}
\begin{split}
	Y_t 
&= 
	\lambda (X^x_t -  X_t^0 ) + X_t^y \\
&=
	\lambda 
	\Big(
		\!\left[x + \textint_0^t \mu(X_s^x)\, ds + \int_0^t \sigma(X_s^x) \, dW_s \right]  \\
&\quad
		- 
		\left[
			0 + \textint_0^t \mu(X_s^0)\, ds + \int_0^t \sigma(X_s^0) \, dW_s 
		\right] \!
	\Big) \\
&\quad
	+ 
	\left[y + \textint_0^t \mu(X_s^y)\, ds + \int_0^t \sigma(X_s^y) \, dW_s \right]\\
&=
	\lambda x + y 
	+
	\textint_0^t \left[\lambda \left( \mu(X_s^x) -  \mu(X_s^0) \right) + \mu(X_s^y) \right] ds \\
&\quad
	+ 
	\textint_0^t \left[ \lambda \left( \sigma(X_s^x) -  \sigma(X_s^0) \right) + \sigma(X_s^y)\right]  dW_s
	.
\end{split}
\end{equation}
In addition, note that \eqref{affine_SDE_fixedpointMu} and \eqref{affine_SDE_fixedpointSigma} ensure that
for all $\nu \in \{\mu, \sigma \}$, $a, b, c \in \R^d$, $\lambda \in \R$ it holds that
\begin{equation}
\label{affine_SDE_fixedpoint:eq2}
\begin{split}
	\lambda \left( \nu(a) -  \nu(b) \right) + \nu(c)
&=
	\lambda  \nu(a)  + \nu(c) - \lambda \nu(b) \\
&=
	\nu(\lambda a + c)  + \lambda \nu(0) - \lambda \nu(b) \\
&=
	(- \lambda) \nu(b) + \nu(\lambda a + c)  + \lambda \nu(0)  \\
&=
	 \nu((- \lambda)b + \lambda a + c) + (- \lambda) \nu(0)  + \lambda \nu(0)  \\
&=
	 \nu(\lambda (a - b) + c).
\end{split}
\end{equation}
Combining this with \eqref{affine_SDE_fixedpoint:eq1} implies that
for all $t \in [0,T]$ it holds $\P$-a.s.\  that
\begin{equation}
\label{affine_SDE_fixedpoint:eq3}
\begin{split}
	Y_t 
&=
	\lambda x + y 
	+
	\textint_0^t  \mu\big(\lambda (X_s^x - X_s^0) + X_s^y\big) \, ds \\
&\quad
	+ 
	\textint_0^t \sigma\big(\lambda (X_s^x - X_s^0) + X_s^y\big) \, dW_s \\
&=
	\lambda x + y 
	+
	\textint_0^t  \mu(Y_s) \, ds 
	+ 
	\textint_0^t \sigma(Y_s) \, dW_s.
\end{split}
\end{equation}
The fact that 
for all $t \in [0,T]$ it holds $\P$-a.s.\  that
\begin{equation}
\begin{split}
	X^{\lambda x + y}_t 
=
	\lambda x + y 
	+
	\textint_0^t  \mu(X^{\lambda x + y}_s) \, ds 
	+ 
	\textint_0^t \sigma(X^{\lambda x + y}_s) \, dW_s,
\end{split}
\end{equation}
Corollary~\ref{linear_growth_affine}, Corollary~\ref{linear_growth_affine_HilbertSchmidt}, and, e.g.,  Da Prato \& Zabczyk \cite[Item (i) in Theorem 7.4]{DaPratoZabczyk92} (cf., e.g., Klenke~\cite[Theorem 26.8]{Klenke14}) hence demonstrate that  
for all $t \in [0,T]$ it holds that
\begin{equation}
	\P\!
	\left( 
		X_t^{\lambda x +y} = Y_t
	\right) 
= 
	1.
\end{equation}
This and \eqref{affine_SDE_fixedpointY} imply that
for all $t \in [0,T]$ it holds that
\begin{equation}
\begin{split}
	\P\!
	\left( 
		X_t^{\lambda x +y}+\lambda X_t^0 = \lambda X_t^x + X_t^y 
	\right)
&= 
	\P\!
	\left( 
		X_t^{\lambda x +y} = \lambda (X_t^x  - X_t^0 ) + X_t^y 
	\right) \\
&= 
	\P\!
	\left( 
		X_t^{\lambda x +y} = Y_t
	\right)
=
	1.
\end{split}
\end{equation}
The proof of Lemma~\ref{affine_SDE_fixedpoint} is thus completed.
\end{proof}
 

\begin{lemma}[Modifications of continuous random fields are indistinguishable]
	\label{mod_indist}
	Let $d \in \N$, 
	let $(E, \delta)$ be a separable metric space,
	let $(\Omega, \mathcal{F},\P)$ be a probability space,
	let $X, Y \colon E\times \Omega \to \R^d$ be random fields, assume
	for all $\omega \in \Omega$ that
		\begin{equation}\label{mod_indist_continuity}
	(E \ni e \mapsto X_{e}(\omega) \in \R^d),\, (E \ni e \mapsto Y_{e}(\omega) \in \R^d) \in C(E, \R^d),
	\end{equation}
	and 
	assume 
	for all $e \in E$ that 
	$
	\P ( X_e = Y_e ) = 1
	$.
	Then 
	\begin{enumerate}[(i)]
		\item \label{mod_indist:item1}
		it holds that
		$
		\{   \forall \, e \in E \colon X_e = Y_e \} \in \mathcal{F}
		$
		and
		\item \label{mod_indist:item2}
		it holds that
		$
		\P \!  \left(  \forall \, e \in E \colon X_e = Y_e \right) 
		= 
		1
		$.
	\end{enumerate}
\end{lemma}

\begin{proof}[Proof of Lemma~\ref{mod_indist}]
	Throughout this proof assume without loss of generality~that $E\neq \emptyset$,
	let $(e_n)_{n \in \N} \subseteq E$ satisfy that
	\begin{equation}\label{mod_indist_density}
	\overline{
		\{ e_n \in E \colon n \in \N \}
	}
	=
	E,
	\end{equation}
	and let $\mathcal{N}\subseteq \Omega$ satisfy that 
	\begin{equation}\label{mod_indistNullSet}
	\mathcal{N}=\cup_{n\in\N} \{ X_{e_n} \neq Y_{e_n} \}.
	\end{equation}
	Note that the fact that $X$ and $Y$ are random fields assures that 
	for all $e \in E$  it holds that
	\begin{equation}
	\{ X_e = Y_e \} = \{ X_e - Y_e = 0\} \in \mathcal{F}.
	\end{equation}
	Hence, we obtain that
	\begin{equation}\label{mod_indistIntersectionMeasurable}
	(\cap_{n \in \N} \{ X_{e_n} = Y_{e_n} \}) \in \mathcal{F}.
	\end{equation}
	Combining this and \eqref{mod_indistNullSet} implies that
	\begin{equation}\label{mod_indistNMeasurable}
	\mathcal{N}=\left[\Omega\backslash\left(\cap_{n \in \N} \{ X_{e_n} = Y_{e_n} \}\right)\right] \in \mathcal{F}.
	\end{equation}
	Moreover, observe that the hypothesis that
	for all $e \in E$ it holds that 
	$
	\P ( X_e = Y_e ) = 1
	$ ensures that for all $n\in\N$ it holds that $
	\P ( X_{e_n} \neq Y_{e_n} ) =0.
	$ 
	Therefore, we obtain that
	\begin{equation}\label{mod_indistNisNullSet}
	\P(\mathcal{N})\le \sum_{n=1}^\infty \P ( X_{e_n} \neq Y_{e_n} ) =0.
	\end{equation}
	Next note that \eqref{mod_indist_density} implies that for every $v\in E$ there exists a strictly increasing function $n_v\colon \N\to\N$ such that 
	$\limsup_{k\to\infty}\delta(e_{n_v(k)},v)=0.$ 
	Combining this with \eqref{mod_indist_continuity} ensures that for every $v\in E$  there exists a strictly increasing function $n_v\colon \N\to\N$ such that for every
		$\omega\in  \{ \forall\, k\in\N\colon X_{e_k} = Y_{e_k} \}$ it holds that $\limsup_{k\to\infty}\delta(e_{n_v(k)},v)=0$ and  
		\begin{equation}
		X_v(\omega)=\lim_{k\to\infty} X_{e_{n_v(k)}}(\omega)=\lim_{k\to\infty} Y_{e_{n_v(k)}}(\omega)=Y_v(\omega).
		\end{equation}
	This and \eqref{mod_indistIntersectionMeasurable} demonstrate that
	%
	%
	%
	%
	%
	%
	\begin{equation}
	\label{mod_indist:eq1}
	\begin{split}
	\left\{  \forall \, e \in E \colon X_e = Y_e \right\}
	&=
	\left\{  \forall \, n \in \N \colon  X_{e_n} = Y_{e_n} \right\} \\
	&=
	(\cap_{n \in \N} \{  X_{e_n} = Y_{e_n}  \})
	\in 
	\mathcal{F}.
	\end{split}
	\end{equation}
	This proves item~\eqref{mod_indist:item1}. 
	Combining \eqref{mod_indistNMeasurable} and \eqref{mod_indistNisNullSet} hence implies that
	\begin{equation}
	\begin{split}
	\P \!  \left(  \forall \, e \in E \colon X_e = Y_e \right) 
	=
	\P \!  \left(  
	\cap_{n \in \N} \left\{  X_{e_n} = Y_{e_n} \right \}
	\right)  =\P(\Omega\backslash \mathcal{N})=1-\P(\mathcal{N})=1.
	\end{split}
	\end{equation}
	This establishes item~\eqref{mod_indist:item2}.
	The proof of Lemma~\ref{mod_indist} is thus completed.
\end{proof}

\begin{prop}
 \label{affine_SDE_allpoints}
 Let $d\in \N$, $T \in (0,\infty)$, 
 let $(\Omega,\allowbreak \mathcal{F},\allowbreak\P,\allowbreak(\mathbbm{F}_t)_{t \in [0,T]})$ be a filtered probability space which fulfils the usual conditions, 
 let $W \colon [0,T]\times \Omega \to \R^d$ be a standard $(\mathbbm{F}_t)_{t \in [0,T]}$-Brownian motion, 
 let $\mu \colon \R^d \to \R^d$ and $\sigma \colon \R^d \to \R^{d \times d}$ be functions which satisfy 
 for all $x,y \in \R^d$, $\lambda \in \R$ that 
 \begin{equation}
 	   \mu(\lambda x + y) + \lambda \mu(0) = \lambda \mu(x) + \mu(y)
 \end{equation} 
 and 
 \begin{equation}
 	\sigma(\lambda x + y) + \lambda \sigma(0) = \lambda \sigma(x) + \sigma(y),
 \end{equation}
 let $X^x  \colon [0,T]\times \Omega \to \R^d$, $x \in \R^d$, be $(\mathbbm{F}_t)_{t \in [0,T]}$-adapted stochastic processes,
  assume for all $\omega \in \Omega$ that
 \begin{equation}\label{affine_SDE_allpointsContinuityAssumption}
 	 \left(\R^d \times [0,T]\ni  (x,t) \mapsto X_t^x(\omega) \in \R^d\right)\in C(\R^d\times[0,T],\R^d),
 \end{equation}
 and assume that 
 for all $x \in \R^d$, $t \in [0,T]$  it holds $\P$-a.s.\  that
 \begin{equation} 
   X_t^x 
 = 
   x + \int_0^t \mu(X_s^x)\, ds + \int_0^t \sigma(X_s^x) \, dW_s.
 \end{equation}
 Then
 \begin{enumerate}[(i)]
 \item \label{affine_SDE_allpoints:item1}
  it holds that
  \begin{equation}
  \left\{ 
     \forall\, x,y \in \R^d,  \lambda \in \R, t \in [0, T] \colon X_t^{\lambda x +y}+\lambda X_t^0 = \lambda X_t^x + X_t^y 
   \right\}
   \in \mathcal{F}
  \end{equation}
  and
 \item \label{affine_SDE_allpoints:item2}
 it holds that
 \begin{equation}
   \P\!
   \left( 
     \forall\, x,y \in \R^d,  \lambda \in \R, t \in [0, T] \colon X_t^{\lambda x +y}+\lambda X_t^0 = \lambda X_t^x + X_t^y 
   \right) 
 = 
   1.
 \end{equation} 
 \end{enumerate}
\end{prop}
 
\begin{proof}[Proof of Proposition~\ref{affine_SDE_allpoints}]
Throughout this proof let
$
	Y, Z  \colon (\R^d \times \R^d \times \R \times [0,T])  \times \Omega \to \R^d
$ 
be the random fields which satisfy 
for all $x,y \in \R^d$, $\lambda \in \R$, $t \in [0, T]$ that
\begin{equation}
\label{affine_SDE_allpoints:setting1}
	Y_{(x, y, \lambda, t)} = X_t^{\lambda x +y}+\lambda X_t^0 
\qandq
	Z_{(x, y, \lambda, t)} = \lambda X_t^x + X_t^y.
\end{equation}
Observe that Lemma~\ref{affine_SDE_fixedpoint} assures that
for all $x,y \in \R^d$, $\lambda \in \R$, $t \in [0, T]$ it holds that
\begin{equation}
\label{affine_SDE_allpoints:eq1}
\begin{split}
	\P \big( Y_{(x, y, \lambda, t)} = Z_{(x, y, \lambda, t)} \big)
=
	\P \big( X_t^{\lambda x +y}+\lambda X_t^0 = \lambda X_t^x + X_t^y  \big)
=
	1.
\end{split}
\end{equation}
Moreover, note that 
\eqref{affine_SDE_allpointsContinuityAssumption}
and the fact that
\begin{multline}
\left(\R^d \times \R^d \times \R\times [0,T] \ni (x, y, \lambda,t) \mapsto (\lambda x + y,t) \in \R^d\times [0,T]\right)\\\in C(\R^d \times \R^d \times \R\times [0,T],\R^d\times [0,T])
\end{multline}
demonstrate that 
for all  $\omega \in \Omega$ it holds that
\begin{equation}
		Y_{\cdot}(\omega), Z_{\cdot}(\omega) \in C(\R^d \times \R^d \times \R \times [0,T], \R^d).
\end{equation}
Combining this, \eqref{affine_SDE_allpoints:eq1}, and Lemma~\ref{mod_indist} 
(with
$d = d$,
$E = \R^d \times \R^d \times \R\times [0,T]$,
$(\Omega, \mathcal{F},\P) = (\Omega, \mathcal{F},\P)$,
$X = Y$,
$Y = Z$
in the notation of Lemma~\ref{mod_indist}) proves that
\begin{equation}
	\left\{ 
		\forall\, x,y \in \R^d,  \lambda \in \R, t \in [0, T] \colon Y_{(x, y, \lambda, t)} = Z_{(x, y, \lambda, t)} 
	\right\}
	\in \mathcal{F} 
\end{equation}
and
\begin{equation}
	\P \!  \left(  \forall \, x,y \in \R^d,  \lambda \in \R, t \in [0, T] \colon Y_{(x, y, \lambda, t)} = Z_{(x, y, \lambda, t)}  \right) 
= 
	1.
\end{equation}
This and \eqref{affine_SDE_allpoints:setting1} demonstrate that
\begin{equation}
\begin{split}
	&\left\{ 
		\forall\, x,y \in \R^d,  \lambda \in \R, t \in [0, T] \colon X_t^{\lambda x +y}+\lambda X_t^0 = \lambda X_t^x + X_t^y 
	\right\} \\
&=
	\left\{ 
		\forall\, x,y \in \R^d,  \lambda \in \R, t \in [0, T] \colon Y_{(x, y, \lambda, t)} = Z_{(x, y, \lambda, t)} 
	\right\}
	\in \mathcal{F} 
\end{split}
\end{equation}
and
\begin{equation}
\begin{split}
	&\P \!  \left(  \forall \, x,y \in \R^d,  \lambda \in \R, t \in [0, T] \colon X_t^{\lambda x +y}+\lambda X_t^0 = \lambda X_t^x + X_t^y   \right)  \\
&=
	\P \!  \left(  \forall \, x,y \in \R^d,  \lambda \in \R, t \in [0, T] \colon Y_{(x, y, \lambda, t)} = Z_{(x, y, \lambda, t)}  \right) 
= 
	1.
\end{split}
\end{equation}
This establishes items~\eqref{affine_SDE_allpoints:item1}--\eqref{affine_SDE_allpoints:item2}.
The proof of Proposition~\ref{affine_SDE_allpoints} is thus completed.
\end{proof}

\begin{lemma}[Modifications of adapted processes are adapted]
\label{normal_adapted}
Let $d\in \N$, $T \in (0,\infty)$,
let $(\Omega,\allowbreak \mathcal{F},\allowbreak\P,\allowbreak(\mathbbm{F}_t)_{t \in [0,T]})$ be a filtered probability space which fulfils the usual conditions,
let $X, Y \colon [0, T] \times \Omega \to \R^d$ be stochastic processes, assume that
$X$ is an $(\mathbbm{F}_t)_{t \in [0,T]}$-adapted stochastic process, and assume for all $t\in[0,T]$ that 
$
	\P(X_t = Y_t) = 1.
$
Then 
$Y$ is an $(\mathbbm{F}_t)_{t \in [0,T]}$-adapted stochastic process.
\end{lemma}

\begin{proof}[Proof of Lemma~\ref{normal_adapted}]
Throughout this proof let $t \in [0,T]$.
Note that the hypothesis that $
\P(X_t = Y_t) = 1
$ ensures that
\begin{equation}\label{normal_adaptedNullSet}
	\P(X_t \neq Y_t) = 0.
\end{equation}
This and the hypothesis 
that $(\Omega,\allowbreak \mathcal{F},\allowbreak\P,\allowbreak(\mathbbm{F}_t)_{t \in [0,T]})$ is a filtered probability space which fulfils the usual conditions
imply that
$
	\{ X_t \neq Y_t \} \in \mathbbm{F}_0 \subseteq \mathbbm{F}_t
$.
Hence, we obtain that 
\begin{equation}
\label{normal_adapted:eq1}
	\{ X_t = Y_t \} = \Omega \backslash \{ X_t \neq Y_t \} \in \mathbbm{F}_t.
\end{equation}
Moreover, observe that \eqref{normal_adaptedNullSet}
demonstrates that 
for all $B \in \mathcal{B}(\R^d)$ it holds that
\begin{equation}
	\P \! \left(
		\{ Y_t \in B \} \cap \{ X_t \neq Y_t \}
	\right) 
\leq
	\P \! \left(
		 X_t \neq Y_t 
	\right)  
=
	0.
\end{equation}
The hypothesis 
that $(\Omega,\allowbreak \mathcal{F},\allowbreak\P,\allowbreak(\mathbbm{F}_t)_{t \in [0,T]})$ is a filtered probability space which fulfils the usual conditions
 therefore implies that for all $B \in \mathcal{B}(\R^d)$ it holds that
 \begin{equation}
 	(\{ Y_t \in B \} \cap \{ X_t \neq Y_t \}) \in \mathbbm{F}_0 \subseteq \mathbbm{F}_t.
 \end{equation}
Combining this with the hypothesis that $X$ is an $(\mathbbm{F}_t)_{t \in [0,T]}$-adapted stochastic process and \eqref{normal_adapted:eq1} demonstrates that
for all $B \in \mathcal{B}(\R^d)$ it holds that
\begin{equation}
\begin{split}
\{ Y_t \in B \} 
&= 
\left( \{ Y_t \in B \} \cap \{ X_t = Y_t \} \right)
\cup
\left( \{ Y_t \in B \} \cap \{ X_t \neq Y_t \} \right) \\
&=
\left( \{ X_t \in B \} \cap  \{ X_t = Y_t \}\right)
\cup
\left( \{ Y_t \in B \} \cap \{ X_t \neq Y_t \} \right)
\in 
\mathbbm{F}_t.
\end{split}
\end{equation}
This establishes that 
$Y$ is an $(\mathbbm{F}_t)_{t \in [0,T]}$-adapted stochastic process.
The proof of Lemma~\ref{normal_adapted} is thus completed.
\end{proof}

\begin{lemma}[A version of the Kolmogorov-Chentsov theorem]
\label{Kolmorogov_chentsov}
Let $d,k \in \N$, $p \in (d, \infty)$, $\alpha \in (\nicefrac{d}{p},\infty)$,  for every $\mathfrak{d}\in\N$ let $\left\| \cdot \right\|_{\R^\mathfrak{d}} \colon \R^\mathfrak{d} \allowbreak\to [0,\infty)$ be the $\mathfrak{d}$-dimensional Euclidean norm,
let $(\Omega, \mathcal{F},\P)$ be a probability space,
let $D \subseteq \R^d$ be a non-empty set, and
let $X \colon D \times \Omega \to \R^k$ be a random field which satisfies 
for all $n \in \N$ that
\begin{multline}\label{Kolmorogov_chentsov:ass1}
\sup\!\Bigg(\Bigg\{\frac{ 
	\left(
	\EXPP{ \Norm{ X_v - X_w }_{\R^k}^p }
	\right)^{\nicefrac{1}{p}}
}
{ \norm{v - w}_{\R^d}^\alpha }\colon v, w \in D \cap [-n, n]^d,\, v \neq w\Bigg\}
\\\cup
\big\{\left(
\EXPP{ \Norm{ X_v }_{\R^k}^p }
\right)^{\nicefrac{1}{p}}
\colon v\in D \cap [-n, n]^d\big\}\cup\{0\} \Bigg)<\infty.
\end{multline}
Then
there exists a random field
$
	Y \colon D \times \Omega \to \R^k
$
which satisfies 
  \begin{enumerate}[(i)]
	\item \label{Kolmorogov_chentsov:item1}
	that for all $\omega \in \Omega$ it holds that
$
	(D \ni v \mapsto Y_v(\omega) \in \R^k) \in C(D, \R^k)
$
	and
	\item  \label{Kolmorogov_chentsov:item2}
	that for all $v \in D$ it holds that
	$
		\P ( X_v = Y_v ) = 1.
	$
\end{enumerate}
\end{lemma}

\begin{proof}[Proof of Lemma~\ref{Kolmorogov_chentsov}]
Throughout this proof let   
$g_n\colon D \cap [-n, n]^d\to \allowbreak L^p(\Omega;\allowbreak\R^k)$, $n\in\N$, be functions which satisfy that for all $n\in\N$, $v\in D \cap [-n, n]^d$ 
it holds $\P$-a.s.\ that $X_v=g_n(v)$ (cf.~\eqref{Kolmorogov_chentsov:ass1}), let $\mathfrak{c}\in [0,\infty)$ be a real number which satisfies that for all $n\in\N$, $v, w \in D \cap [-n, n]^d$ it holds that
\begin{equation}\label{Kolmorogov_chentsov:HoelderConstant}
		\left(
		\EXPP{ \Norm{ X_v - X_w }_{\R^k}^p }
		\right)^{\nicefrac{1}{p}}\le \mathfrak{c} \norm{v - w}_{\R^d}^\alpha
\end{equation}
(cf.~\eqref{Kolmorogov_chentsov:ass1}), let $\mathfrak{a}=\min\{\alpha,1\}$, and let 
$
(\cdot)^{+} \colon \R \to [0,\infty)
$ 
be the function which satisfies for all $q \in \R$ that 
$
(q)^{+} = \max \{ q , 0 \}
$. 
Note that for all $n\in\N$, $v, w \in D \cap [-n, n]^d$ it holds that 
\begin{equation}
\begin{split}
	\norm{v - w}_{\R^d}^{(\alpha-1)^+}
&\le (\Norm{v}_{\R^d}+\norm{w}_{\R^d})^{(\alpha-1)^+}
\le \big(n\sqrt{d}+n\sqrt{d}\big)^{(\alpha-1)^+}
\\&= \big(2 n\sqrt{d}\big)^{(\alpha-1)^+}.
\end{split}
\end{equation}
Combining this and \eqref{Kolmorogov_chentsov:HoelderConstant} with the fact that $\alpha-\mathfrak{a}=(\alpha-1)^+\ge 0$ ensures that for all $n\in\N$, $v, w \in D \cap [-n, n]^d$ it holds that 
\begin{equation}\label{Kolmorogov_chentsov:HoelderConstantTwo}
\begin{split}
\left(
\EXPP{ \Norm{ X_v - X_w }_{\R^k}^p }
\right)^{\nicefrac{1}{p}}
&\le \mathfrak{c} \norm{v - w}_{\R^d}^\alpha
= \mathfrak{c} \norm{v - w}_{\R^d}^{\mathfrak{a}} \norm{v - w}_{\R^d}^{\alpha-\mathfrak{a}}
\\&=\mathfrak{c} \norm{v - w}_{\R^d}^{\mathfrak{a}} \norm{v - w}_{\R^d}^{(\alpha-1)^+}
\\&\le \mathfrak{c} \norm{v - w}_{\R^d}^{\mathfrak{a}} \big(2 n\sqrt{d}\big)^{(\alpha-1)^+}.
\end{split}
\end{equation}
This and \eqref{Kolmorogov_chentsov:ass1} imply that for all $n\in\N$ it holds that
\begin{multline}\label{Kolmorogov_chentsov:ass1Proof}
\sup\!\Bigg(\Bigg\{\frac{ 
	\left(
	\EXPP{ \Norm{ X_v - X_w }_{\R^k}^p }
	\right)^{\nicefrac{1}{p}}
}
{ \norm{v - w}_{\R^d}^\mathfrak{a} }\colon v, w \in D \cap [-n, n]^d,\, v \neq w\Bigg\}
\\\cup
\big\{\left(
\EXPP{ \Norm{ X_v }_{\R^k}^p }
\right)^{\nicefrac{1}{p}}
\colon v\in D \cap [-n, n]^d\big\}\cup\{0\} \Bigg)<\infty.
\end{multline}
Therefore, we obtain that for all $n\in\N$ it holds that $g_n$ is a globally bounded and globally $\mathfrak{a}$-H\"older continuous function.
Mittmann \& Steinwart \cite[Theorem 2.2]{mittmannSteinwart} hence ensures that for every $n\in\N$ there is a globally bounded and globally $\mathfrak{a}$-H\"older continuous function $G_n\colon \R^d\to L^p(\Omega;\R^k)$ which satisfies for all $v\in D \cap [-n, n]^d$  that $G_n(v)=g_n(v)$.
This assures that there exist random fields $\xi_n\colon \R^d \times \Omega\to \R^k$, $n\in\N$, which satisfy 
\begin{enumerate}[(a)]
	\item that for all $n\in\N$, $v\in \R^d$ it holds $\P$-a.s.\ that $(\xi_n)_v=G_n(v)$ and 
	\item that for all $n\in\N$ it holds that
	\begin{equation}
	\label{Kolmorogov_chentsov:extensionHoelder1}
	\begin{split}
	&\sup_{\substack{v, w \in [-n,n+1)^d,\\ v \neq w}}
	\frac{ 
		\left(
		\EXPP{ \Norm{ (\xi_n)_v - (\xi_n)_w }_{\R^k}^p }
		\right)^{\nicefrac{1}{p}}
	}
	{ \norm{v - w}_{\R^d}^\mathfrak{a} }\\
	&\le 
	\sup_{\substack{v, w \in \R^d,\\ v \neq w}}
	\frac{ 
		\left(
		\EXPP{ \Norm{ (\xi_n)_v - (\xi_n)_w }_{\R^k}^p }
		\right)^{\nicefrac{1}{p}}
	}
	{ \norm{v - w}_{\R^d}^\mathfrak{a} }
	< 
	\infty	
	\end{split}
	\end{equation}
	\begin{equation}
	\label{Kolmorogov_chentsov:extensionHoelder2}
	\andq
	\sup_{v \in[-n,n+1)^d}
	\left(
	\EXPP{ \Norm{ (\xi_n)_v}_{\R^k}^p }
	\right)^{\nicefrac{1}{p}}
	\le \sup_{v\in \R^d}
	\left(
	\EXPP{ \Norm{ (\xi_n)_v}_{\R^k}^p }
	\right)^{\nicefrac{1}{p}}
	< 
	\infty.
	\end{equation}
\end{enumerate}
Combining this and, e.g., Revuz \& Yor \cite[Theorem 2.1 in Section 2 in Chapter I]{revuz1999} (with $X=\xi_n$, $\gamma=p$, $d=d$, $\varepsilon=\mathfrak{a} p-d$ in the notation of \cite[Theorem 2.1 in Section 2 in Chapter I]{revuz1999})
ensures that there exist random fields
$
	Y_n \colon [-n, n]^d \times \Omega \to \R^k
$, $n\in\N$, 
which satisfy 
\begin{enumerate}[(A)]
	\item that for all $n\in\N$,  $\omega \in \Omega$ it holds that
	\begin{equation}
	\label{Kolmorogov_chentsov:eq1}
	\left( [-n,n]^d \ni v \mapsto (Y_n)_v(\omega) \in \R^k\right) \in C([-n, n]^d, \R^k)
	\end{equation}
	and
	\item that  for all $n\in\N$,  $v \in  [-n, n]^d$ it holds that $\P ( (Y_n)_v = (\xi_n)_v ) = 1.$ 
\end{enumerate}
%
%
The fact that for all $n\in\N$, $v\in D \cap [-n, n]^d$ it holds that $\P ( X_v = (\xi_n)_v ) = 1$ therefore implies that
for all $n\in\N$, $v \in D \cap [-n, n]^d$ it holds that
\begin{equation}
\label{Kolmorogov_chentsov:eq2}
\P ( (Y_n)_v = X_v )=1.
\end{equation}
This assures that
for all $n\in\N, m\in\N\cap [1,n], v \in D \cap [-m, m]^d$ it holds that
\begin{equation}
	\P ( \{(Y_n)_v = X_v\}\cap  \{(Y_m)_v = X_v\}) = 1.
\end{equation}
The fact that for all $n\in\N, m\in\N\cap [1,n], v \in D \cap [-m, m]^d$ it holds that 
\begin{equation}
	\{(Y_n)_v = X_v\}\cap  \{(Y_m)_v = X_v\}\subseteq \{(Y_n)_v = (Y_m)_v\}
\end{equation}
therefore demonstrates that 
for all $n\in\N, m\in\N\cap [1,n], v \in D \cap [-m, m]^d$ it holds that
\begin{equation}
	\P ( (Y_n)_v = (Y_m)_v ) = 1.
\end{equation}
Combining this with \eqref{Kolmorogov_chentsov:eq1} and Lemma~\ref{mod_indist}
(with $d=k$, $E=D\cap [-m,m]^d$ for $m\in \N\cap [1,n]$, $n\in\N$ in the notation of Lemma~\ref{mod_indist})
 establishes that
for all $n\in \N,$ $m\in\N\cap [1,n]$ it holds that
\begin{equation}\label{Kolmorogov_chentsovYnYk}
	\P \big( \forall \, v \in D \cap [-m, m]^d \colon (Y_n)_v = (Y_m)_v \big) 
= 
	1.
\end{equation}
Next let $\Pi \in \mathcal{F}$ be the event given by
\begin{equation}
\label{Kolmorogov_chentsov:eq3}
\Pi
=
\left\{
\forall \, n \in \N,	\, m \in \N\cap [1,n],	\, v \in D \cap [-m, m]^d
\colon
(Y_n)_v = (Y_m)_v
\right\}.
\end{equation}
Observe that \eqref{Kolmorogov_chentsov:eq3} and \eqref{Kolmorogov_chentsovYnYk} show that
\begin{equation}
\label{Kolmorogov_chentsov:eq4}
\begin{split}
	\P ( \Pi )
=
	\P \!\left(
		\cap_{n=1}^\infty\cap_{m=1}^n
		\left\{ \forall \, v \in D \cap [-m, m]^d \colon (Y_n)_v = (Y_m)_v \right\}
	\right)
=
	1.
\end{split}
\end{equation}
Moreover, note that \eqref{Kolmorogov_chentsov:eq3} ensures that there exists a unique random field
$
Z \colon D \times \Omega \to \R^k
$
 which satisfies 
\begin{enumerate}[(I)]
	\item \label{Kolmorogov_chentsov:I} that for all $\omega \in \Pi$, $n \in \N$, $v \in D \cap [-n, n]^d$ it holds that $Z_v (\omega) = (Y_n)_v(\omega)$ and
	\item \label{Kolmorogov_chentsov:II} that for $\omega\in \Omega \backslash \Pi$, $v\in D$ it holds that $Z_v (\omega) = 0$.
\end{enumerate}
Observe that \eqref{Kolmorogov_chentsov:I}, \eqref{Kolmorogov_chentsov:eq2}, and \eqref{Kolmorogov_chentsov:eq4} demonstrate that 
for all $n \in \N$, $v \in D \cap [-n, n]^d$ it holds that
\begin{equation}
\begin{split}
	\P ( Z_v = X_v )&=\P ( \{Z_v = X_v\}\cap \Pi )
\\&=
\P ( \{(Y_n)_v = X_v\}\cap \Pi ) 
= \P ( (Y_n)_v = X_v) 
= 
1.
\end{split}
\end{equation}
This shows that 
for all $v \in D$ it holds that
\begin{equation}
\label{Kolmorogov_chentsov:eq6}
	\P ( Z_v = X_v )
= 
	1.
\end{equation}
Moreover, observe that \eqref{Kolmorogov_chentsov:eq1} and \eqref{Kolmorogov_chentsov:I} imply that
for all $\omega \in \Pi$, $n \in \N$ it holds that
\begin{equation}
\begin{split}
	&\left(D \cap [-n, n]^d \ni v \mapsto Z_v(\omega) \in \R^k\right)  \\
&=
	\left(D \cap [-n, n]^d \ni v \mapsto (Y_n)_v(\omega) \in \R^k\right) \in C(D \cap [-n, n]^d, \R^k).
\end{split}
\end{equation}
This ensures that
for all $\omega \in \Pi$ it holds that
\begin{equation}
\label{Kolmorogov_chentsov:eq7}
\begin{split}
	\left(D \ni v \mapsto Z_v(\omega) \in \R^k\right) 
\in 
	C(D, \R^k).
\end{split}
\end{equation}
In addition, note that \eqref{Kolmorogov_chentsov:II} assures that 
for all $\omega \in \Omega \backslash \Pi$ it holds that
\begin{equation}
\label{Kolmorogov_chentsov:eq8}
\begin{split}
	\left(D \ni v \mapsto Z_v(\omega) \in \R^k\right) 
=
	\left(D \ni v \mapsto 0 \in \R^k\right) 
\in 
	C(D, \R^k).
\end{split}
\end{equation}
Combining this and \eqref{Kolmorogov_chentsov:eq7} demonstrates that
for all $\omega \in \Omega$ it holds that
\begin{equation}
\begin{split}
	\left(D \ni v \mapsto Z_v(\omega) \in \R^k\right) 
\in 
	C(D, \R^k).
\end{split}
\end{equation}
This and \eqref{Kolmorogov_chentsov:eq6} complete the proof of Lemma~\ref{Kolmorogov_chentsov}.
\end{proof}

  \begin{prop}
  \label{affine_solutions_of_SDEs}
  Let $d\in \N$, $T \in (0,\infty)$,
  let $(\Omega,\allowbreak \mathcal{F},\allowbreak\P,\allowbreak(\mathbbm{F}_t)_{t \in [0,T]})$ be a filtered probability space which fulfils the usual conditions, 
 let $W \colon [0,T]\times \Omega \to \R^d$ be a standard $(\mathbbm{F}_t)_{t \in [0,T]}$-Brownian motion,
 and 
  let $\mu \colon \R^d \to \R^d$ and $\sigma \colon \R^d \to \R^{d \times d}$ be functions which satisfy for all $x,y \in \R^d$, $\lambda \in \R$ that 
  \begin{equation} \label{affine_solutions_of_SDEs:AffineLinearOne}
  	 \mu(\lambda x + y) + \lambda \mu(0) = \lambda \mu(x) + \mu(y)
  \end{equation}
  and 
  \begin{equation}\label{affine_solutions_of_SDEs:AffineLinearTwo}
  	\sigma(\lambda x + y) + \lambda \sigma(0) = \lambda \sigma(x) + \sigma(y).
  \end{equation}
  Then there exist up to indistinguishability unique  $(\mathbbm{F}_t)_{t \in [0,T]}$-adapted stochastic processes with continuous sample paths $X^x  \colon [0,T]\times \Omega \to \R^d$, $x \in \R^d$, which satisfy
  \begin{enumerate}[(i)]
   \item \label{affine_solutions_of_SDEs:item1}
   that for all $x \in \R^d$, $t \in [0,T]$ it holds $\P$-a.s.\ that
   \begin{equation}\label{affine_solutions_of_SDEsEquationX}
    X_t^x = x + \int_0^t \mu(X_s^x)\, ds + \int_0^t \sigma(X_s^x) \, dW_s 
   \end{equation}
   and
    \item  \label{affine_solutions_of_SDEs:item2}
    that for all  $x,y \in \R^d$, $\lambda \in \R$, $t \in [0,T]$,  $\omega \in \Omega$ it holds that
    \begin{equation}
    	X_t^{\lambda x+y}(\omega) + \lambda X_t^{0}(\omega) = \lambda X_t^{x}(\omega) + X_t^{y}(\omega).
    \end{equation}
  \end{enumerate}
 \end{prop}
 
\begin{proof}[Proof of Proposition~\ref{affine_solutions_of_SDEs}]
Throughout this proof 
let $p \in (2(d+1), \infty)$,
for every $k \in \N$ let $\left\| \cdot \right\|_{\R^k} \colon \R^k \to [0,\infty)$ be the $k$-dimensional Euclidean norm,
let $\HSNormIndex{\cdot}{d}{d}\colon \R^{d \times d} \to [0,\infty)$ be the Hilbert-Schmidt norm on $\R^{d\times d}$,
let $\mathfrak{m}, \mathfrak{s} \in (0,\infty)$ satisfy 
for all $x \in \R^d$ that
\begin{equation}
\label{affine_solutions_of_SDEs:settingMu}
\norm{\mu(x)}_{\R^d} 
\leq 
\mathfrak{m}( 1 + \norm{x}_{\R^d})
\qandq
\norm{\mu(x) - \mu(y)}_{\R^d} 
\leq	
\mathfrak{m} \norm{x - y}_{\R^d}
\end{equation}
and
\begin{equation}
\label{affine_solutions_of_SDEs:settingSigma}
\HSNormIndex{{\sigma(x)}}{d}{d}
\leq 
\mathfrak{s}( 1 + \norm{x}_{\R^d})
\qandqShort
\HSNormIndex{{\sigma(x)- \sigma(y)}}{d}{d}
\leq	
\mathfrak{s} \norm{x - y}_{\R^d}
\end{equation}
(cf. Corollary~\ref{linear_growth_affine} and Corollary~\ref{linear_growth_affine_HilbertSchmidt}),
and let $C \in (0,\infty)$ be given by
\begin{equation}\label{affine_solutions_of_SDEs:setting2}
	C
=
	4d\sqrt{2}
	\big(
		1+ \mathfrak{m} T + \mathfrak{s}  p\sqrt{T} 
	\big)
	\exp{\!
	\Big( \!
		\big[
			\mathfrak{m} \sqrt{T} + \mathfrak{s} p
		\big]^2
		T
	\Big)}.
\end{equation}
Note that \eqref{affine_solutions_of_SDEs:settingMu}, \eqref{affine_solutions_of_SDEs:settingSigma}, e.g., Jentzen \& Kloeden \cite[Theorem 5.1]{JentzenKloeden11} (with $T=T$, $(\Omega,\allowbreak \mathcal{F},\allowbreak\P,\allowbreak(\mathbbm{F}_t)_{t \in [0,T]})\allowbreak=(\Omega,\allowbreak \mathcal{F},\allowbreak\P,\allowbreak(\mathbbm{F}_t)_{t \in [0,T]})$, $H=\R^d$, $\norm{\cdot}_H=\norm{\cdot}_{\R^d}$, $U=\R^d$, $\norm{\cdot}_U=\norm{\cdot}_{\R^d}$, $Q v=v$, $(W_t)_{t\in[0,T]}=(W_t)_{t\in[0,T]}$, $D(A)=\R^d$, $Av=0$, $\eta=1$, $\alpha=0$, $\delta=0$, $F(v)=\mu(v)$, $\beta=0$, $B(v)u=\sigma(v)u$, $\gamma=0$, $p=4$, $\xi=(\Omega\ni \omega\mapsto x\in\R^d)$ for $u,v,x\in\R^d$ in the notation of \cite[Theorem 5.1]{JentzenKloeden11}) (cf., e.g.,  Da Prato \& Zabczyk \cite[Item (i) in Theorem 7.4]{DaPratoZabczyk92} and Klenke~\cite[Theorem 26.8]{Klenke14}), 
and, e.g., the Kolmogorov-Chentsov type theorem in Lemma \ref{Kolmorogov_chentsov} 
(with $d=1$, $k=d$, $p=4$, $\alpha=1/2$, $D=[0,T]$ in the notation of Lemma \ref{Kolmorogov_chentsov})
 assure that
there exist $(\mathbbm{F}_t)_{t \in [0,T]}$-adapted stochastic processes with continuous sample paths $X^x  \colon [0,T]\times \Omega \to \R^d$, $x \in \R^d$, which satisfy
that for all $x \in \R^d$, $t \in [0,T]$ it holds $\P$-a.s.\ that
\begin{equation}
\label{affine_solutions_of_SDEs:eq01}
	X_t^x = x + \int_0^t \mu(X_s^x)\, ds + \int_0^t \sigma(X_s^x) \, dW_s.
\end{equation}
Observe that \eqref{affine_solutions_of_SDEs:eq01} and Lemma~\ref{affine_SDE_fixedpoint} prove that
for all $t \in [0,T]$, $x,y \in \R^d$, $\lambda \in \R$ it holds $\P$-a.s.\ that
\begin{equation}
\label{affine_solutions_of_SDEs:eq02}
	X_t^{\lambda x +y}+\lambda X_t^0 = \lambda X_t^x + X_t^y .
\end{equation}
This implies that 
for all $t \in [0,T]$, $v \in \R^d\backslash\{0\}$ it holds $\P$-a.s.\ that
\begin{equation}
\begin{split}
X^{v}_t 
&=
X^{\norm{v}_{\R^d} \frac{v}{\Norm{v}_{\R^d}}}_t 
=
\Big(X^{\norm{v}_{\R^d} \frac{v}{\Norm{v}_{\R^d}}}_t +\Norm{v}_{\R^d} X_t^0\Big)-\Norm{v}_{\R^d} X_t^0
\\&=	\Big(\Norm{v}_{\R^d} X^{\frac{v}{\Norm{v}_{\R^d}}}_t+X_t^0\Big)-\Norm{v}_{\R^d} X_t^0
=
\Norm{v}_{\R^d} \Big( X^{\frac{v}{\Norm{v}_{\R^d}}}_t  -  X^{0}_t \Big) + X^{0}_t.
\end{split}
\end{equation}
Combining this and \eqref{affine_solutions_of_SDEs:eq02} 
(with
$t = t$,
$x = y$, 
$y = x$, 
$\lambda = -1$
for $t\in [0,T]$, $x,y\in \R^d$
in the notation of  \eqref{affine_solutions_of_SDEs:eq02})
implies that
for all $t \in [0,T]$, $x,y \in \R^d$ with $x\neq y$ it holds that
\begin{equation}
\label{affine_solutions_of_SDEs:eq03}
\begin{split}
	&\left(
		\EXPP{  \Norm{ X^{x}_t - X^{y}_t }_{\R^d}^p }
	\right)^{\nicefrac{1}{p}}
=
	\left(
		\EXPP{ \Norm{ X^{x-y}_t - X^{0}_t }_{\R^d}^p }
	\right)^{\nicefrac{1}{p}} \\
&=
	\left(
		\EXPPP{ \Normm{\Norm{x-y}_{\R^d} \big( X^{\frac{x-y}{\Norm{x-y}_{\R^d}}}_t  -  X^{0}_t \big) + X^{0}_t - X^{0}_t }_{\R^d}^p }
	\right)^{ \! \nicefrac{1}{p}} \\
&=
	\Big(
		\EXPPP{ \Normm{ X^{\frac{x-y}{\Norm{x-y}_{\R^d}}}_t  -  X^{0}_t }_{\R^d}^p }
	\Big)^{ \! \nicefrac{1}{p}}
	\Norm{x-y}_{\R^d}.
\end{split}
\end{equation}
In addition, observe that \eqref{affine_solutions_of_SDEs:settingMu}, \eqref{affine_solutions_of_SDEs:settingSigma}, \eqref{affine_solutions_of_SDEs:eq01}, Proposition~\ref{moments_of_solution_of_SDE}
(with 
$d=d$, $m=d$, $p=p$, $T=T$, $\mathfrak{m}_1=\mathfrak{m}$, $\mathfrak{m}_2=\mathfrak{m}$, $\mathfrak{s}_1=\mathfrak{s}$, $\mathfrak{s}_2=\mathfrak{s}$ 
in the notation of Proposition~\ref{moments_of_solution_of_SDE}),  and the triangle inequality assure that 
for all $t \in [0,T]$ it holds that
\begin{equation}
\begin{split}
	&\sup_{v \in \R^d, \Norm{ v }_{\R^d} = 1}
	\left(
		\EXPP{ \Norm{ X^{v}_t  -  X^{0}_t }_{\R^d}^p }
	\right)^{ \! \nicefrac{1}{p}}\\
&\leq
	\left(
		\EXPP{ \Norm{ X^{0}_t }_{\R^d}^p }
	\right)^{ \! \nicefrac{1}{p}}
	+
	\sup_{v \in \R^d, \Norm{ v }_{\R^d} = 1}
	\left(
		\EXPP{ \Norm{ X^{v}_t }_{\R^d}^p }
	\right)^{ \! \nicefrac{1}{p}} \\
&\leq
	\sqrt{2} 
	\big(
		\mathfrak{m} T + \mathfrak{s}  p \sqrt{T} 
	\big)
	\exp\!
	\Big( \!
		\big[
			\mathfrak{m} \sqrt{T} + \mathfrak{s} p
		\big]^2
		T
	\Big) \\
&\quad
	+
	\sup_{v \in \R^d, \Norm{ v }_{\R^d} = 1}\left[
	\sqrt{2} 
	\big(
		\Norm{v} + \mathfrak{m} T + \mathfrak{s}  p\sqrt{T} 
	\big)
	\exp{\!
	\Big( \!
		\big[
			\mathfrak{m} \sqrt{T} + \mathfrak{s} p
		\big]^2
		T
	\Big)} \right]\\
&\leq
	2\sqrt{2}	
	\big(
		1+ \mathfrak{m} T + \mathfrak{s}  p\sqrt{T} 
	\big)
	\exp{\!
	\Big( \!
		\big[
			\mathfrak{m} \sqrt{T} + \mathfrak{s} p
		\big]^2
		T
	\Big)}.
\end{split}
\end{equation}
This, \eqref{affine_solutions_of_SDEs:setting2}, \eqref{affine_solutions_of_SDEs:eq03}, and the fact that
for all $n \in \N$, $x=(x_1,x_2,\dots, x_d),y=(y_1,y_2,\dots, y_d) \in [-n, n]^d$ it holds that
\begin{equation}
\begin{split}
\Norm{x-y}_{\R^d}
&=	
\Norm{x-y}_{\R^d}^{\nicefrac{1}{2}} \Norm{x-y}_{\R^d}^{\nicefrac{1}{2}}
\\&= \big[[\vert x_1-y_1\vert^2+\ldots+\vert x_d-y_d\vert^2]^{\nicefrac{1}{2}}\big]^{\nicefrac{1}{2}} \Norm{x-y}_{\R^d}^{\nicefrac{1}{2}}
\\&\leq
 \big[[d(2n)^2]^{\nicefrac{1}{2}}\big]^{\nicefrac{1}{2}}  \Norm{x-y}_{\R^d}^{\nicefrac{1}{2}}
\\&=
\big[2n \sqrt{d}\big]^{\nicefrac{1}{2}}  \Norm{x-y}_{\R^d}^{\nicefrac{1}{2}}
\leq
2dn \Norm{x-y}_{\R^d}^{\nicefrac{1}{2}}
\end{split}
\end{equation}
assure that 
for all $t \in [0,T]$, $n \in \N$, $x,y \in [-n, n]^d$ with $x\neq y$ it holds that
\begin{equation}
\label{affine_solutions_of_SDEs:eq04}
\begin{split}
	&\left(
		\EXPP{  \Norm{ X^{x}_t - X^{y}_t }_{\R^d}^p }
	\right)^{\nicefrac{1}{p}} 
	= 	\Big(
	\EXPPP{ \Normm{ X^{\frac{x-y}{\Norm{x-y}_{\R^d}}}_t  -  X^{0}_t }_{\R^d}^p }
	\Big)^{ \! \nicefrac{1}{p}}
	\Norm{x-y}_{\R^d}
	\\&\le \left[\sup_{v \in \R^d, \Norm{ v }_{\R^d} = 1}
	\left(
	\EXPP{ \Norm{ X^{v}_t  -  X^{0}_t }_{\R^d}^p }
	\right)^{ \! \nicefrac{1}{p}}\right] 
	\Norm{x-y}_{\R^d}
	\\
&\leq
	2\sqrt{2}	
	\big(
		1+ \mathfrak{m} T + \mathfrak{s}  p\sqrt{T} 
	\big)
	\exp{\!
	\Big( \!
		\big[
			\mathfrak{m} \sqrt{T} + \mathfrak{s} p
		\big]^2
		T
	\Big)}
	\Norm{x-y}_{\R^d}  \\
&\leq	
	2\sqrt{2}	
	\big(
		1+ \mathfrak{m} T + \mathfrak{s}  p\sqrt{T} 
	\big)
	\exp{\!
	\Big( \!
		\big[
			\mathfrak{m} \sqrt{T} + \mathfrak{s} p
		\big]^2
		T
	\Big)}
	2dn
	\Norm{x-y}_{\R^d}^{\nicefrac{1}{2}} \\
&=
	nC  \Norm{x-y}_{\R^d}^{\nicefrac{1}{2}}.
\end{split}
\end{equation}
Moreover, note that \eqref{affine_solutions_of_SDEs:eq01} and the triangle inequality imply that 
for all $x \in \R^d$, $s, t \in [0, T]$ with $s \leq t$ it holds that
\begin{equation}
\label{affine_solutions_of_SDEs:eq05}
\begin{split}
	&\left(
		\EXPP{  \Norm{ X^{x}_t - X^{x}_s }_{\R^d}^p }
	\right)^{\nicefrac{1}{p}} \\
&=
	\bigg(
		\EXPPP{ 
			\Normm{ 
				x + \textint_0^t \mu(X_u^x)\, du + \int_0^t \sigma(X^x_u) \, dW_u \\
&\qquad
				- 
				\left( 
					x + \textint_0^s \mu(X_u^x)\, du + \int_0^s \sigma(X^x_u) \, dW_u
				\right)\!
			}_{\R^d}^p 
		}
	\bigg)^{\nicefrac{1}{p}} \\
&=
	\Big(
		\EXPP{ 
			\Norm{ 
				\textint_s^t \mu(X_u^x)\, du + \int_s^t \sigma(X^x_u) \, dW_u 
			}_{\R^d}^p 
		}
	\Big)^{\nicefrac{1}{p}} \\
&\leq
	\Big(
		\EXPP{ 
			\Norm{ 
				\textint_s^t \mu(X_u^x)\, du
			}_{\R^d}^p 
		}
	\Big)^{\nicefrac{1}{p}} 
	+
	\Big(
		\EXPP{ 
			\Norm{ 
				\int_s^t \sigma(X^x_u) \, dW_u 
			}_{\R^d}^p 
		}
	\Big)^{\nicefrac{1}{p}}.
\end{split}
\end{equation}
Furthermore, observe that Lemma~\ref{Minkowski}, Proposition~\ref{moments_of_solution_of_SDE}, \eqref{affine_solutions_of_SDEs:settingMu}, and the fact that
for all $s, t \in [0, T]$ it holds that
$
	| t - s | 
= 
	| t - s |^{\nicefrac{1}{2}} | t - s |^{\nicefrac{1}{2}} 
\leq 
	\sqrt{T} | t - s |^{\nicefrac{1}{2}}
$
ensure that 
for all $n \in \N$, $x \in [-n,n]^d$, $s, t \in [0, T]$ with  $s \leq t$ it holds that
\begin{equation}
\label{affine_solutions_of_SDEs:eq06}
\begin{split}
		&\left(
	\Exp{ 
		\norm{
			\int_s^t \mu(X_u^x)\, du 
		}_{\R^d}^p \,
	}
	\right)^{ \! \! \nicefrac{1}{p}}
		\leq 	\left(
	\Exp{ 
		\left|
		\int_s^t \norm{\mu(X_u^x)}_{\R^d}\, du 
		\right|^p \,
	}
	\right)^{ \! \! \nicefrac{1}{p}} \\
&\leq
	\int_s^t
	\big(
		\EXPP{ 
			\Norm{ 
				 \mu(X_u^x)
			}_{\R^d}^p 
		}
	\big)^{\nicefrac{1}{p}} 	
	\, du 
\leq
	\int_s^t
	\mathfrak{m}  
	\big( 
		1 +
		\big(
			\EXPP{ 
				\Norm{ 
					 X_u^x
				}_{\R^d}^p 
			}
		\big)^{\nicefrac{1}{p}} 	
	\big)
	\, du \\
&\leq
	\int_s^t
	\mathfrak{m}  
	\Big[ 
		1 +
		\sqrt{2} 
		\big(
			\Norm{x} + \mathfrak{m} T + \mathfrak{s}  p\sqrt{T} 
		\big)
		\exp{\!
		\Big( \!
			\big[
				\mathfrak{m} \sqrt{T} + \mathfrak{s} p
			\big]^2
			T
		\Big)}
	\Big]
	\, du \\
&=
	\mathfrak{m}  
	\Big[ 
		1 +
		\sqrt{2} 
		\big(
			n\sqrt{d} + \mathfrak{m} T + \mathfrak{s}  p\sqrt{T} 
		\big)
		\exp{\!
		\Big( \!
			\big[
				\mathfrak{m} \sqrt{T} + \mathfrak{s} p
			\big]^2
			T
		\Big)}
	\Big]
	(t-s) \\
&\leq
	\mathfrak{m}  
	(1 + \sqrt{2})
	n d
	\big(
		1 + \mathfrak{m} T + \mathfrak{s}  p\sqrt{T} 
	\big)
	\exp{\!
	\Big( \!
		\big[
			\mathfrak{m} \sqrt{T} + \mathfrak{s} p
		\big]^2
		T
	\Big)}
	| t - s | \\
&= \left(\frac{1 + \sqrt{2}}{4\sqrt{2}}\right) \mathfrak{m} n C | t - s |
\leq
	\mathfrak{m}  n C
	\sqrt{T}
	| t - s |^{\nicefrac{1}{2}}.
\end{split}
\end{equation}
Moreover, note that the fact that for all $x\in\R^d$ it holds that $X^x  \colon [0,T]\times \Omega \to \R^d$ is an 
$(\mathbbm{F}_t)_{t \in [0,T]}$-adapted stochastic process with continuous sample paths and \eqref{affine_solutions_of_SDEs:settingSigma} ensure that
it holds $\P$-a.s. that 
\begin{equation}
\begin{split}
	&\int_0^T \HSNormIndex{{\sigma(X_s^x)}}{d}{d}^2 \,ds
\le T \left[\sup_{s\in[0,T]}\HSNormIndex{{\sigma(X_s^x)}}{d}{d}^2\right]\\
&\le T \left[\sup_{s\in[0,T]}\,\left[\mathfrak{m}^2(1+\Norm{X_s^x}_{\R^d})^2\right]\right]
\le 2T\mathfrak{m}^2 \left(1+\sup_{s\in[0,T]}\Norm{X_s^x}_{\R^d}^2\right)
<\infty.
\end{split}
\end{equation}
Lemma \ref{lem-daPratoZabzcyk},
Proposition~\ref{moments_of_solution_of_SDE},
the fact that for all $x\in\R^d$ it holds that $X^x\colon [0,T]\times \Omega \to \R^{d}$ is an 
$(\mathbbm{F}_t)_{t \in [0,T]}$-adapted stochastic process with continuous sample paths,
 and \eqref{affine_solutions_of_SDEs:settingSigma} hence demonstrate that 
for all $n \in \N$, $x \in [-n,n]^d$, $s, t \in [0, T]$ with  $s \leq t$ it holds that
\begin{equation}
\begin{split}
	&\Big(
		\EXPP{ 
			\Norm{ 
				\textint_s^t \sigma(X^x_u) \, dW_u 
			}_{\R^d}^p 
		}
	\Big)^{\nicefrac{1}{p}} \\
&\leq
\left(\frac{p(p-1)}{2}\right)^{\nicefrac{1}{2}}
\bigg(
\int_s^t
\big(
\EXPP{ 
	\HSNormIndex{{	\sigma(X^x_u)}}{d}{d}^p
}
\big)^{\!\nicefrac{2}{p}} 
\, du
\bigg)^{\nicefrac{1}{2}} \\
&\leq
p
\bigg(
\int_s^t
\Big[
\big(
\EXPP{ 
		\HSNormIndex{{	\sigma(X^x_u)}}{d}{d}^p
}
\big)^{\!\nicefrac{1}{p}} 
\Big]^2
\, du
\bigg)^{\nicefrac{1}{2}} \\
&\leq
	p
	\bigg(
		\int_s^t
		\Big[
		\mathfrak{s}
		\left(
		1
		+ 
		\big(
			\EXPP{ 
				\Norm{ 
						X^x_u
				}_{\R^d}^p 
			}
		\big)^{\!\nicefrac{1}{p}} \right)
		\Big]^2
		\, du
	\bigg)^{\nicefrac{1}{2}} \\
&\leq
	p \mathfrak{s}
	\bigg[
		\left(
			\int_s^t 1^2\, du
		\right)^{\nicefrac{1}{2}}
		+
		\bigg(
		\int_s^t
		\Big[
		\big(
			\EXPP{ 
				\Norm{ 
						X^x_u
				}_{\R^d}^p 
			}
		\big)^{\!\nicefrac{1}{p}} 
		\Big]^2
		\, du
		\bigg)^{\!\!\nicefrac{1}{2}} \,
	\bigg] \\
&\leq
	p \mathfrak{s}
	\bigg[
		| t-s | ^{\nicefrac{1}{2}} \\
&\quad
		+
		\bigg(
		\int_s^t
		\Big[
			\sqrt{2} 
			\big(
				\Norm{x} + \mathfrak{m} T + \mathfrak{s}  p\sqrt{T} 
			\big)
			\exp{\!
			\Big( \!
				\big[
					\mathfrak{m} \sqrt{T} + \mathfrak{s} p
				\big]^2
				T
			\Big)}
		\Big]^2
		\, du
		\bigg)^{\!\!\nicefrac{1}{2}} \,
	\bigg] \\
&\leq
	p \mathfrak{s}
	| t-s | ^{\nicefrac{1}{2}}
	\bigg[
		1
		+
			\sqrt{2} 
			\big(
				n\sqrt{d} + \mathfrak{m} T + \mathfrak{s}  p\sqrt{T} 
			\big)
			\exp{\!
			\Big( \!
				\big[
					\mathfrak{m} \sqrt{T} + \mathfrak{s} p
				\big]^2
				T
			\Big)}
	\bigg] \\
&\leq
	p \mathfrak{s}
	| t-s | ^{\nicefrac{1}{2}}
	(1 + \sqrt{2}) n\sqrt{d}
	\big(
		1 + \mathfrak{m} T + \mathfrak{s}  p\sqrt{T} 
	\big)
	\exp{\!
	\Big( \!
		\big[
			\mathfrak{m} \sqrt{T} + \mathfrak{s} p
		\big]^2
		T
	\Big)}\\
&= p \mathfrak{s} n C | t-s | ^{\nicefrac{1}{2}} \left(\frac{1 + \sqrt{2}}{4\sqrt{2d}}\right) 
\leq
	p \mathfrak{s} n C
	| t-s | ^{\nicefrac{1}{2}}.
\end{split}
\end{equation}
Combining this with \eqref{affine_solutions_of_SDEs:eq05} and  \eqref{affine_solutions_of_SDEs:eq06} establishes that
for all $n \in \N$, $x \in [-n,n]^d$, $s, t \in [0, T]$ it holds that
\begin{equation}
\label{affine_solutions_of_SDEs:eq07}
\begin{split}
	\left(
		\EXPP{  \Norm{ X^{x}_t - X^{x}_s }_{\R^d}^p }
	\right)^{\nicefrac{1}{p}} 
&\leq
	\mathfrak{m}  
	n C
	\sqrt{T}
	| t - s |^{\nicefrac{1}{2}}
	+
	p \mathfrak{s} n C 
	| t-s | ^{\nicefrac{1}{2}} \\
&=
	n C 
	( \mathfrak{m}  \sqrt{T} + p \mathfrak{s} )
	| t-s | ^{\nicefrac{1}{2}}.
\end{split}
\end{equation}
Moreover, observe that the fact that 
for all $a, b \in [0,\infty)$ it holds that
$
	a + b \leq \sqrt{2} (a^2 + b^2)^{\nicefrac{1}{2}}
$
ensures that  
for all $a, b \in [0,\infty)$ it holds that
\begin{equation}
	\sqrt{a} +\sqrt{b} 
\leq 
	\sqrt{2} (a + b)^{\nicefrac{1}{2}}
\leq
	\sqrt{2} \big( \sqrt{2} (a^2 + b^2)^{\nicefrac{1}{2}} \big)^{\nicefrac{1}{2}}
\leq
	2 \big( (a^2 + b^2)^{\nicefrac{1}{2}} \big)^{\nicefrac{1}{2}}.
\end{equation}
This, \eqref{affine_solutions_of_SDEs:eq04}, and \eqref{affine_solutions_of_SDEs:eq07} 
demonstrate that
for all $n \in \N$, $x, y \in [-n,n]^d$, $s, t \in [0, T]$ it holds that
\begin{equation}
\begin{split}
 	&\left(
		\EXPP{  \Norm{ X^{x}_t - X^{y}_s }_{\R^d}^p }
	\right)^{\nicefrac{1}{p}} \\
&\leq
	 \left(
		\EXPP{  \Norm{ X^{x}_t - X^{y}_t}_{\R^d}^p }
	\right)^{\nicefrac{1}{p}}
	+
	 \left(
		\EXPP{  \Norm{ X^{y}_t - X^{y}_s }_{\R^d}^p }
	\right)^{\nicefrac{1}{p}} \\
&\leq
	n C \Norm{x-y}_{\R^d}^{\nicefrac{1}{2}}
	+
	n C 
	( \mathfrak{m}  \sqrt{T} + p \mathfrak{s} )
	| t-s | ^{\nicefrac{1}{2}} \\
&\leq
	n C
	( 1 + \mathfrak{m}  \sqrt{T} + p \mathfrak{s} )
	(\Norm{x-y}_{\R^d}^{\nicefrac{1}{2}} + | t-s | ^{\nicefrac{1}{2}}) \\
&\leq
	n C
	( 1 + \mathfrak{m}  \sqrt{T} + p \mathfrak{s} )
	2
	\left[
		(\Norm{x-y}_{\R^d}^{2} + | t-s | ^{2})^{\nicefrac{1}{2}}
	\right]^{\nicefrac{1}{2}}\\
&=
	2 n C
	( 1 + \mathfrak{m}  \sqrt{T} + p \mathfrak{s} )
			\norm{(x,t)-(y,s)
		}_{\R^{d+1}}^{\nicefrac{1}{2}}.
\end{split}
\end{equation}
Hence, we obtain for all $n \in \N$ that
\begin{equation}
\label{affine_solutions_of_SDEs:eq08}
\begin{split}
	&\sup_{ \substack{ (x,t), (y,s) \in(\R^d \times [0, T]) \cap [-n,n]^{d+1}, \\ (x,t) \neq (y,s)}}
	\left(\frac{ 
		\left(
			\EXPP{  \Norm{ X^{x}_t - X^{y}_s }_{\R^d}^p }
		\right)^{\nicefrac{1}{p}}
	}
	{
	\norm{(x,t)-(y,s)
	}_{\R^{d+1}}^{\nicefrac{1}{2}}
	} \right)\\
&\leq
	2 n C
	( 1 + \mathfrak{m}  \sqrt{T} + p \mathfrak{s} )
<
	\infty.
\end{split}
\end{equation}
In addition, note that \eqref{affine_solutions_of_SDEs:settingMu},  \eqref{affine_solutions_of_SDEs:settingSigma}, and Proposition~\ref{moments_of_solution_of_SDE}  assure that
for all $n \in \N$ it holds that
\begin{equation}
\begin{split}
	&\sup_{ (x,t) \in (\R^d \times [0, T]) \cap [-n,n]^{d+1}}
		\left[\left(
			\EXPP{  \Norm{ X^{x}_t }_{\R^d}^p }
		\right)^{\nicefrac{1}{p}}\right]\\
&\leq
	\sup_{ (x,t) \in (\R^d \times [0, T]) \cap [-n,n]^{d+1}}
	\left[	\sqrt{2} 
		\big(
			\Norm{x} + \mathfrak{m} T + \mathfrak{s}  p\sqrt{T} 
		\big)
		\exp{\!
		\Big( \!
			\big[
				\mathfrak{m} \sqrt{T} + \mathfrak{s} p
			\big]^2
			T
		\Big)}  \right] \\
&\leq
		\sqrt{2} 
		\big(
			n \sqrt{d} + \mathfrak{m} T + \mathfrak{s}  p\sqrt{T} 
		\big)
		\exp{\!
		\Big( \!
			\big[
				\mathfrak{m} \sqrt{T} + \mathfrak{s} p
			\big]^2
			T
		\Big)}
<
	\infty.
\end{split}
\end{equation}
Lemma~\ref{Kolmorogov_chentsov} 
(with
$d = d+1$, $k=d$, $p = p$, 
$\alpha = \nicefrac{1}{2}$,
$(\Omega, \mathcal{F},\P)=(\Omega, \mathcal{F},\P)$, 
$D = \R^d \times [0, T]$,
$X = ((\R^d\times [0,T])\times \Omega\ni ((x,t),\omega)\mapsto X^x_t(\omega)\in \R^d)$
in the notation of Lemma~\ref{Kolmorogov_chentsov})
and \eqref{affine_solutions_of_SDEs:eq08} hence prove that
there exist stochastic processes
$
	Y^x \colon [0, T] \times \Omega \to \R^d
$, $x \in \R^d$, which satisfy
\begin{enumerate}[(I)]
	\item that for all $\omega \in \Omega$ it holds that
	\begin{equation}
	\label{affine_solutions_of_SDEs:eq09}
	\left( \R^d \times [0, T] \ni (x, t) \mapsto Y^x_t(\omega) \in \R^d\right) \in C(\R^d \times [0, T], \R^d)
	\end{equation}
	and 
	\item that for all $x \in \R^d$, $t \in [0, T] $ it holds that
	\begin{equation}
	\label{affine_solutions_of_SDEs:eq10}
	\P ( Y^x_t = X^x_t ) = 1.
	\end{equation}
\end{enumerate}
The fact that 
for all $x \in \R^d$ it holds that
$X^x$ is an $(\mathbbm{F}_t)_{t \in [0,T]}$-adapted stochastic process
and Lemma~\ref{normal_adapted} therefore ensure that
for all $x \in \R^d$ it holds that
$Y^x$ is an $(\mathbbm{F}_t)_{t \in [0,T]}$-adapted stochastic process.
Next note that \eqref{affine_solutions_of_SDEs:eq09}, \eqref{affine_solutions_of_SDEs:eq10}, and the fact that for all $x\in\R^d$ it holds that $X^x$ has continuous sample paths assure that for all $x\in\R^d$ it holds that
\begin{equation}
\label{affine_solutions_of_SDEsXYindist}
\P \big(\forall\,t\in [0,T] \colon Y^x_t = X^x_t \big)=\P \big(\forall\,t\in [0,T]\cap \Q \colon Y^x_t = X^x_t \big) = 1.
\end{equation}
Moreover, observe that \eqref{affine_solutions_of_SDEs:eq10} implies that for all $x \in \R^d$, $t \in [0,T]$  it holds $\P$-a.s.\  that
\begin{equation}
	\int_0^t \mu(X^x_s)\, ds=\int_0^t \mu(Y^x_s)\, ds\qandq
	\int_0^t \sigma(X^x_s) \, dW_s=\int_0^t \sigma(Y^x_s) \, dW_s.
\end{equation}
Combining this, \eqref{affine_solutions_of_SDEs:eq10}, and \eqref{affine_solutions_of_SDEs:eq01} ensures that for all $x \in \R^d$, $t \in [0,T]$  it holds $\P$-a.s.\  that
\begin{equation} 
\label{affine_solutions_of_SDEs:eq12}
\begin{split}
Y^x_t 
&=X^x_t=
x + \int_0^t \mu(X^x_s)\, ds + \int_0^t \sigma(X^x_s) \, dW_s\\
&=x + \int_0^t \mu(Y^x_s)\, ds + \int_0^t \sigma(Y^x_s) \, dW_s.
\end{split}
\end{equation}
Next let 
$\Pi \subseteq \Omega$ be the set given by
\begin{equation}
\label{affine_solutions_of_SDEs:eq11}
\begin{split}
\Pi 
&= 
\Big\{
\forall\, x,y \in \R^d,  \lambda \in \R, t \in [0, T] \colon \,
Y_t^{\lambda x +y} +\lambda Y_t^0 = \lambda Y_t^x + Y_t^y 
\Big\}.
\end{split}
\end{equation}
Combining \eqref{affine_solutions_of_SDEs:AffineLinearOne}, \eqref{affine_solutions_of_SDEs:AffineLinearTwo}, \eqref{affine_solutions_of_SDEs:eq09} \eqref{affine_solutions_of_SDEs:eq12}, and  \eqref{affine_solutions_of_SDEs:eq11} with Proposition~\ref{affine_SDE_allpoints} demonstrates that 
\begin{equation}\label{affine_solutions_of_SDEsPi}
	\Pi \in \mathcal{F}\qandq \P (\Pi ) = 1.
\end{equation}
This proves that there exist unique stochastic processes
$
	Z^x \colon [0, T] \times \Omega \to \R^d
$, $x \in \R^d$,
 which satisfy for all $x \in \R^d$, $t \in [0,T]$, $\omega \in \Omega$ that
	\begin{align}\label{affine_solutions_of_SDEs:eq13}
&Z_t^x(\omega)=\!\left\{\begin{aligned}
&Y_t^x(\omega)\colon\quad \omega\in \Pi\\&0\hspace{0.9cm}\colon\quad \omega\in \Omega\backslash \Pi.
\end{aligned}\! 
\right.
\end{align}
Observe that  \eqref{affine_solutions_of_SDEs:eq11} and \eqref{affine_solutions_of_SDEs:eq13} imply that 
for all $x,y \in \R^d$, $\lambda \in \R$, $t \in [0, T]$, $\omega \in \Pi$ it holds that
\begin{equation}
\label{affine_solutions_of_SDEs:eq14}
\begin{split}
	Z_t^{\lambda x +y} (\omega)+\lambda Z_t^0(\omega) 
&= 	
	Y_t^{\lambda x +y} (\omega)+\lambda Y_t^0(\omega) \\
&= 
	\lambda Y_t^x(\omega) + Y_t^y (\omega) 
= 
	\lambda Z_t^x(\omega) + Z_t^y (\omega).
\end{split}
\end{equation}
Moreover, note that \eqref{affine_solutions_of_SDEs:eq13} assures that
for all $x,y \in \R^d$, $\lambda \in \R$, $t \in [0, T]$, $\omega \in \Omega \backslash \Pi$ it holds that
\begin{equation}
\begin{split}
	Z_t^{\lambda x +y} (\omega)+\lambda Z_t^0(\omega) 
= 	
	0 + 0 
= 
	\lambda Z_t^x(\omega) + Z_t^y (\omega).
\end{split}
\end{equation}
Combining this with \eqref{affine_solutions_of_SDEs:eq14} demonstrates that
for all $x,y \in \R^d$, $\lambda \in \R$, $t \in [0, T]$, $\omega \in \Omega $ it holds that
\begin{equation}
\label{affine_solutions_of_SDEs:eq15}
\begin{split}
	Z_t^{\lambda x +y} (\omega)+\lambda Z_t^0(\omega) 
= 
	\lambda Z_t^x(\omega) + Z_t^y (\omega).
\end{split}
\end{equation}
Furthermore, observe that \eqref{affine_solutions_of_SDEs:eq13} ensures for all $x\in\R^d$ that 
\begin{equation}
	\Pi\subseteq \{\forall \, t \in [0, T] \colon Z^x_t = Y^x_t\}.
\end{equation}
This and \eqref{affine_solutions_of_SDEsPi} show for all $x\in\R^d$ that
\begin{equation}\label{affine_solutions_of_SDEs:ZandYindist}
	\P \big( \forall \, t \in [0, T] \colon Z^x_t = Y^x_t\big) = 1.
\end{equation}
The fact that for all $x\in\R^d$ it holds that $Y^x$ is an $(\mathbbm{F}_t)_{t \in [0,T]}$-adapted stochastic process with continuous sample paths, Lemma \ref{normal_adapted},
and \eqref{affine_solutions_of_SDEs:eq13}  therefore imply that for all $x\in\R^d$ it holds that $Z^x$ is an $(\mathbbm{F}_t)_{t \in [0,T]}$-adapted stochastic process with continuous sample paths. Combining this and  \eqref{affine_solutions_of_SDEs:eq12}  with \eqref{affine_solutions_of_SDEs:ZandYindist} demonstrates
that
for all $x \in \R^d$, $t \in [0,T]$  it holds $\P$-a.s.\  that
\begin{equation} 
\label{affine_solutions_of_SDEs:eq16}
Z^x_t 
= 
x + \int_0^t \mu(Z^x_s)\, ds + \int_0^t \sigma(Z^x_s) \, dW_s.
\end{equation}
This,
\eqref{affine_solutions_of_SDEs:settingMu}, \eqref{affine_solutions_of_SDEs:settingSigma},
and, e.g., Jentzen \& Kloeden \cite[Theorem 5.1]{JentzenKloeden11} (with $T=T$, $(\Omega,\allowbreak \mathcal{F},\allowbreak\P,\allowbreak(\mathbbm{F}_t)_{t \in [0,T]})=(\Omega,\allowbreak \mathcal{F},\allowbreak\P,\allowbreak(\mathbbm{F}_t)_{t \in [0,T]})$, $H=\R^d$, $\norm{\cdot}_H=\norm{\cdot}_{\R^d}$,  $U=\R^d$, $\norm{\cdot}_U=\norm{\cdot}_{\R^d}$, $Q v=v$, $(W_t)_{t\in[0,T]}=(W_t)_{t\in[0,T]}$, $D(A)=\R^d$, $Av=0$, $\eta=1$, $\alpha=0$, $\delta=0$, $F(v)=\mu(v)$, $\beta=0$, $B(v)u=\sigma(v)u$, $\gamma=0$, $p=2$, $\xi=(\Omega\ni \omega\mapsto x\in\R^d)$ for $u,v,x\in\R^d$ in the notation of \cite[Theorem 5.1]{JentzenKloeden11})
(cf., e.g., Da Prato \& Zabczyk \cite[Item (i) in Theorem 7.4]{DaPratoZabczyk92} and Klenke~\cite[Theorem 26.8]{Klenke14}) establish that for all $x \in \R^d$ and all $(\mathbbm{F}_t)_{t \in [0,T]}$-adapted stochastic processes with continuous sample paths $V  \colon [0,T]\times \Omega \to \R^d$ 
which satisfy
that for all $t \in [0,T]$ it holds $\P$-a.s.\ that
\begin{equation}
	V_t = x + \int_0^t \mu(V_s)\, ds + \int_0^t \sigma(V_s) \, dW_s
\end{equation}
it holds that
$
\forall \, t \in [0,T] \colon\,	\P(  V_t = Z^x_t) = 1
$. Lemma \ref{mod_indist} hence demonstrates that for all $x\in\R^d$ and all $(\mathbbm{F}_t)_{t \in [0,T]}$-adapted stochastic processes with continuous sample paths $V  \colon [0,T]\times \Omega \to \R^d$ 
which satisfy
that for all $t \in [0,T]$ it holds $\P$-a.s.\ that
\begin{equation}
V_t = x + \int_0^t \mu(V_s)\, ds + \int_0^t \sigma(V_s) \, dW_s
\end{equation}
it holds that
$
	\P( \forall \, t \in [0,T] \colon V_t = Z^x_t) = 1
$.
Combining this with \eqref{affine_solutions_of_SDEs:eq15} and  \eqref{affine_solutions_of_SDEs:eq16} completes the proof of Proposition~\ref{affine_solutions_of_SDEs}.
\end{proof}


\subsection[Viscosity solutions for partial differential equations (PDEs)]{Viscosity solutions for partial differential equations}\label{SectionExistenceOfViscosity}s

In this subsection we apply results on viscosity solutions for PDEs from the literature (cf., e.g., Crandall et al.\ \cite{UsersGuideViscosity}, Crandall \& Lions \cite{ViscosityOriginalPaper}, and Hairer et al.\ \cite[Subsections 4.3--4.4]{HairerHutzenthalerJentzen15}) 
to establish in Proposition~\ref{viscosity_existence} and Corollary~\ref{viscosity_affine_existence} the existence, uniqueness, and regularity results for viscosity solutions which we need for our proofs of the ANN approximation results.
We recall that the notion of viscosity solutions allows to extend the classical notion of solutions of second order PDEs to functions which are merely continuous and not necessarily $C^2$.
Our proof of Proposition~\ref{viscosity_existence} employs the following well-known result, Lemma~\ref{viscosity_lyapunov} below, on the existence of a suitable Lyapunov-type function $V \colon \R^d \to \R$ where $d \in \N$ under the coercivity-type hypothesis in \eqref{viscosity_lyapunovCoercivity}. 
For the sake of completeness we include in this subsection also a proof of Lemma~\ref{viscosity_lyapunov}.

\begin{lemma}
	\label{viscosity_lyapunov}
	Let $d,m \in \N$, $p \in [4, \infty)$,
	let $\langle \cdot, \cdot \rangle \colon \R^d \times \R^d \to \R$ be the $d$-dimensional Euclidean scalar product,
	let $\left\| \cdot \right\| \colon \R^d \to [0,\infty)$ be the $d$-dimensional Euclidean norm, 
	let $\HSNorm{\cdot} \colon \R^{d \times m} \to [0,\infty)$ be the Hilbert-Schmidt norm on $\R^{d\times m}$,
	let 
	$\mu \colon \R^d \to \R^d$ and
	$\sigma \colon \R^d \to \R^{d \times m}$
	be functions which satisfy that
	\begin{equation}\label{viscosity_lyapunovCoercivity}
			 \left[ \sup_{x \in \R^d} \frac{\langle x, \mu(x) \rangle}{(1 + \norm{x}^2)} \right] 
		+
		\left[  \sup_{x \in \R^d} \frac{ \HSNorm{{\sigma (x)}}}{(1 + \norm{x})}   \right] 
		<
		\infty,
	\end{equation}
	and let $V \colon \R^d \to \R$ be the function which satisfies 
	for all $x \in \R^d$ that
	$
	V(x) = 1 + \norm{x}^p
	$.
	Then
	\begin{enumerate}[(i)]
		\item \label{viscosity_lyapunov:item1}
		it holds that 
		$V \in C^2(\R^d, (0,\infty))$ 
		and
		\item \label{viscosity_lyapunov:item2}
		there exists $\rho \in (0,\infty)$ such that 
		for all $x \in \R^d$ it holds that
		\begin{equation}
		\left< \mu(x) , (\nabla V) (x) \right>
		+
		\operatorname{Trace}\! \big( 
		\sigma(x)[\sigma(x)]^{\ast}(\operatorname{Hess} V )(x)
		\big)  
		\leq
		\rho V(x).
		\end{equation}
	\end{enumerate}
\end{lemma}

\begin{proof}[Proof of Lemma~\ref{viscosity_lyapunov}]
	Throughout this proof let $\mu_i\colon \R^d \to \R$, $i \in \{1, 2, \ldots,\allowbreak d \}$, and  $\sigma_{i,j} \colon \R^d \to \R$, $i\in \{1, 2, \ldots, d \}$, $j \in \{1, 2, \ldots, m \}$, satisfy 
	for all $x \in \R^d$ that
	$
	\mu(x) = (\mu_i(x))_{i \in \{1, 2, \ldots, d \}}
	$
	and
	$
	\sigma(x) = (\sigma_{i,j}(x))_{i\in \{1, 2, \ldots, d \}, j \in \{1, 2, \ldots, m \}}
	$
	and let $c \in [0,\infty)$ satisfy 
	for all $x \in \R^d$ that
	\begin{equation}
	\label{viscosity_lyapunov:setting1}
	\langle x, \mu(x) \rangle
	\leq
	c( 1 + \norm{x}^2)
	\qandq
	\HSNorm{{\sigma (x)}} 
	\leq	
	c( 1 + \norm{x}).
	\end{equation}
	Note that the fact that 
	for all $x = (x_1, x_2, \ldots, x_d) \in \R^d$, $i \in \{1, 2, \ldots, d \}$ it holds that
	$
	V(x) = 1+\big[\sum_{i = 1}^d |x_i|^2 \big]^{\nicefrac{p}{2}}
	$
	assures that
	for all $x = (x_1, x_2, \ldots, x_d) \in \R^d$, $i \in \{1, 2, \ldots, d \}$ it holds that
	\begin{equation}
	\label{viscosity_lyapunov:eq1}
	\big(   \tfrac{ \partial }{ \partial x_i} V   \big) (x) 
	= 
	\frac{p}{2} \left[\smallsum_{i = 1}^d |x_i|^2 \right]^{\nicefrac{p}{2}-1}2x_i
	=
	p  \norm{x}^{p-2} x_i.
	\end{equation}
	This ensures that
	for all $x = (x_1, x_2, \ldots, x_d) \in \R^d$, $i,j \in \{1, 2, \ldots, d \}$ it holds that
	\begin{equation}
	\label{viscosity_lyapunov:eq2}
	\big(   \tfrac{ \partial^2  }{ \partial x_j \partial x_i} V  \big) (x) 
	= 
	\begin{cases}
	p (p-2)  \norm{x}^{p-4} x_i x_j &: i \neq j \\
	p (p-2)  \norm{x}^{p-4} |x_i|^2 + p  \norm{x}^{p-2} &: i = j
	\end{cases}.
	\end{equation}
	Combining this and \eqref{viscosity_lyapunov:eq1} proves item~\eqref{viscosity_lyapunov:item1}.
	Next observe that \eqref{viscosity_lyapunov:eq1} and  \eqref{viscosity_lyapunov:eq2} demonstrate that
	for all $x = (x_1, x_2, \ldots, x_d) \in \R^d$ it holds that
	\begin{equation}
	\begin{split}
	&\left< \mu(x) , (\nabla V) (x) \right>
	+
	\operatorname{Trace}\! \big( 
	\sigma(x)[\sigma(x)]^{\ast}(\operatorname{Hess} V )(x)
	\big)   \\
	&=
	\left[
	\smallsum_{i = 1}^d  \mu_i(x)  \big(   \tfrac{ \partial }{ \partial x_i}V   \big) (x) 
	\right]
	+
	\left[
	\smallsum_{i,j= 1}^d  \smallsum_{k= 1}^m  
	\sigma_{i,k}(x) \sigma_{j,k}(x) 
	\big(   \tfrac{ \partial^2  }{ \partial x_i \partial x_j}V   \big) (x)
	\right]  \\
	&=
	\left[
	\smallsum_{i = 1}^d  
	\mu_i(x) \,
	p  \norm{x}^{p-2} x_i
	\right]
	+
	\left[
	\smallsum_{i = 1}^d  \smallsum_{k = 1}^m 
	\sigma_{i,k}(x) \sigma_{i,k}(x) \,
	p  \norm{x}^{p-2}
	\right] \\
	&\quad
	+
	\left[
	\smallsum_{i,j= 1}^d  \smallsum_{k= 1}^m  
	\sigma_{i,k}(x) \sigma_{j,k}(x) \,
	p (p-2)  \norm{x}^{p-4} x_i x_j 
	\right] \\
	&=	
	p  \norm{x}^{p-2}
	\langle x, \mu(x) \rangle
	+
	p  \norm{x}^{p-2}
	\left[
	\smallsum_{i = 1}^d  \smallsum_{k = 1}^m 
	\vert \sigma_{ik}(x)\vert^2
	\right] \\
	&\quad
	+ p (p-2)  \norm{x}^{p-4}
	\left[
\smallsum_{i,j= 1}^d  \smallsum_{k= 1}^m  
\sigma_{i,k}(x) \sigma_{j,k}(x) \,
 x_i x_j 
\right]\\
	&=	
p  \norm{x}^{p-2}
\langle x, \mu(x) \rangle
+
p  \norm{x}^{p-2}
\HSNorm{{\sigma(x)}}^2 \\
&\quad
+ p (p-2)  \norm{x}^{p-4}
(x^* \sigma(x)[\sigma(x)]^* x)
\\	&\le	
p  \norm{x}^{p-2}
\langle x, \mu(x) \rangle
+
p  \norm{x}^{p-2}
\HSNorm{{\sigma(x)}}^2
+ p (p-2)  \norm{x}^{p-2}
\HSNorm{{\sigma(x)}}^2
\\	&=	
p  \norm{x}^{p-2}
\langle x, \mu(x) \rangle
+ p (p-1)  \norm{x}^{p-2}
\HSNorm{{\sigma(x)}}^2 .
	\end{split}
	\end{equation}
	This and \eqref{viscosity_lyapunov:setting1} ensure that
	for all $x \in \R^d$ it holds that
	\begin{equation}
	\begin{split}
	&\left< \mu(x) , (\nabla V) (x) \right>
	+
	\operatorname{Trace}\! \big( 
	\sigma(x)[\sigma(x)]^{\ast}(\operatorname{Hess} V )(x)
	\big)   \\
	&\leq	
	p  \norm{x}^{p-2}
	c (1 + \norm{x}^2)
	+
	p (p-1)  \norm{x}^{p-2}
	 c^2 (1 + \norm{x})^2 \\
	&=
	(pc  + p (p-1)  c^2) \norm{x}^{p-2}  
	+
	 2 p (p-1)  c^2 \norm{x}^{p-1} 
	+
	(pc + p (p-1)  c^2) \norm{x}^{p} \\
		&\le
	(pc  + p (p-1)  c^2) (1+\norm{x}^{p})  
	+
	 2 p (p-1)  c^2 (1+\norm{x}^{p})\\ 
	&\quad+
	(pc + p (p-1)  c^2) (1+\norm{x}^{p}) \\
			&=
	(2pc  + 4p (p-1)  c^2) (1+\norm{x}^{p})
	=
	(2pc  + 4p (p-1)  c^2) V(x).    
	\end{split}
	\end{equation}
	This establishes item~\eqref{viscosity_lyapunov:item2}.
	The proof of Lemma~\ref{viscosity_lyapunov} is thus completed.
\end{proof}

\begin{prop}[Existence and uniqueness of viscosity solutions]
	\label{viscosity_existence}
	Let $d,m \in \N$, $c \in [0,\infty)$,
	let $\langle \cdot, \cdot \rangle \colon \R^d \times \R^d \to \R$ be the $d$-dimensional Euclidean scalar product,
	let $\left\| \cdot \right\| \colon \R^d \to [0,\infty)$ be the $d$-dimensional Euclidean norm, 
	let $\HSNorm{\cdot} \colon \R^{d \times m} \to [0,\infty)$ be the Hilbert-Schmidt norm on $\R^{d\times m}$,
	let 
	$\varphi \colon \R^d \to \R$ be a continuous and at most polynomially growing function, and
	let
	$\mu \colon \R^d \to \R^d$ and
	$\sigma \colon \R^d \to \R^{d \times m}$
	be functions which satisfy 
	for all $x,y \in \R^d$ that
	\begin{equation}
	\label{viscosity_existence:ass1}
	\langle x, \mu(x) \rangle \leq c(1 + \norm{x}^2),
	\qquad
	\HSNorm{{\sigma (x)}} \leq c(1 + \norm{x}),
	\end{equation}
	\begin{equation}
	\label{viscosity_existence:ass2}
	\andq
	\norm{\mu(x) - \mu(y)} + \HSNorm{{\sigma(x) - \sigma(y)}} \leq c \norm{x-y}.
	\end{equation}
	Then
	\begin{enumerate}[(i)]
		\item \label{viscosity_existence:item1}
		there exists a continuous function $u\colon [0,\infty) \times \R^d \to \R$ which satisfies
		for all $x \in \R^d$ that 
		$
		u(0,x) = \varphi(x)
		$, 
		which satisfies for all $T \in (0,\infty)$ that
$
\inf_{q \in (0, \infty)} 
\sup_{(t, x) \in [0, T] \times \R^d} 
\frac{ | u(t, x) | }{ 1 + \norm{x}^q }
<
\infty,
$
		and which satisfies that $u|_{(0,\infty) \times \R^d}$  is a viscosity solution of
		\begin{equation}
		\begin{split}
		(\tfrac{\partial }{\partial t} u)(t,x) 
		&= 
		\tfrac{1}{2} 
		\operatorname{Trace}\! \big( 
		\sigma(x)[\sigma(x)]^{\ast}(\operatorname{Hess}_x u )(t,x)
		\big)  
		+
		\langle (\nabla_x u)(t,x),\mu(x)\rangle
		\end{split}
		\end{equation}
		for $(t,x) \in (0,\infty) \times \R^d$,
		
		\item \label{viscosity_existence:item2}
		for all $T \in (0,\infty)$ it holds that
		$u|_{(0,T) \times \R^d}$ is a viscosity solution of
		\begin{equation}
		\begin{split}
		(\tfrac{\partial }{\partial t}u)(t,x) 
		&= 
		\tfrac{1}{2} 
		\operatorname{Trace}\! \big( 
		\sigma(x)[\sigma(x)]^{\ast}(\operatorname{Hess}_x u )(t,x)
		\big)  
		+
		\langle (\nabla_x u)(t,x),\mu(x)\rangle
		\end{split}
		\end{equation}
		for $(t,x) \in (0,T) \times \R^d$,
		
		\item \label{viscosity_existence:item3}
		for all $T \in (0,\infty)$ and
		all continuous functions $v\colon \allowbreak [0,T] \times \R^d \to \R$ which satisfy
		for all $x \in \R^d$ that 
		$
		v(0,x) = \varphi(x)
		$, 
		which satisfy that
			$\inf_{q \in (0, \infty)} 
			\sup_{(t, x) \in [0, T] \times \R^d} 
			\frac{ | v(t, x) | }{ 1 + \norm{x}^q }
			<
			\infty,$
		and which satisfy that $v|_{(0,T) \times \R^d}$ is a viscosity solution of
		\begin{equation}
		\begin{split}
		(\tfrac{\partial }{\partial t}v)(t,x) 
		&= 
		\tfrac{1}{2} 
		\operatorname{Trace}\! \big( 
		\sigma(x)[\sigma(x)]^{\ast}(\operatorname{Hess}_x v )(t,x)
		\big)  
		+
		\langle (\nabla_x v)(t,x),\mu(x)\rangle
		\end{split}
		\end{equation}
		for $(t,x) \in (0,T) \times \R^d$
		it holds that
		$v = u|_{[0, T] \times \R^d}$, 
		and
		
		\item \label{viscosity_existence:item4}
		for every $T \in (0,\infty)$, $x \in \R^d$,
		every filtered probability space $(\Omega,\allowbreak \mathcal{F},\allowbreak\P,\allowbreak(\mathbbm{F}_t)_{t \in [0,T]})$ which fulfils the usual conditions,
		every standard $(\mathbbm{F}_t)_{t \in [0,T]}$-Brownian motion $W \colon [0,T] \times \Omega \to \R^m$, and
		every $(\mathbbm{F}_t)_{t \in [0,T]}$-adapted stochastic process with continuous sample paths $X \colon [0,T] \times \Omega \to \R^d$ which satisfies
		that for all $t \in [0,T]$ it holds	$\P$-a.s.\ that
		\begin{equation}
		X_t 
		= 
		x 
		+
		\int_0^t \mu(X_s) \, ds
		+
		\int_0^t \sigma(X_s) \, dW_s
		\end{equation}
		it holds that
		\begin{equation}
		u(T,x)
		=
		\EXPP{\varphi(X_T)}. 
		\end{equation} 
	\end{enumerate}
\end{prop}

\begin{proof}[Proof of Proposition~\ref{viscosity_existence}]
	Throughout this proof let $C_n \subseteq \R^d$, $n \in \N$, be the sets which satisfy for all $n \in \N$ that
	$
	C_n = \{x \in \R^d \colon \norm{x} > n \}
	$
	and for every $p \in (0,\infty)$ let $V_p \colon \R^d \to (0,\infty)$ be the function which satisfies 
	for all $x \in \R^d$ that
	$
	V_p(x) = 1 + \norm{x}^p
	$.
	Note that the hypothesis that $\varphi$ is a continuous and at most polynomially growing function, \eqref{viscosity_existence:ass1}, \eqref{viscosity_existence:ass2}, and
	Hairer et al.\ \cite[Corollary 4.17]{HairerHutzenthalerJentzen15}
	demonstrate that 
	there exists a continuous function $u\colon [0, \infty) \times \R^d \to \R$ which satisfies
	for all $x \in \R^d$ that 
	$u(0,x) = \varphi(x)$,
	which satisfies for all $T \in (0,\infty)$ that
	\begin{equation}\label{viscosity_existenceGrowthu}
			\inf_{q \in (0, \infty)} 
		\sup_{(t, x) \in [0, T] \times \R^d} 
		\frac{ | u(t, x) | }{ 1 + \norm{x}^q }
		<
		\infty,
	\end{equation}
	and which satisfies that $u|_{(0,\infty) \times \R^d}$ is a viscosity solution of
	\begin{equation}
	\begin{split}
	(\tfrac{\partial }{\partial t}u)(t,x) 
	&= 
	\tfrac{1}{2} 
	\operatorname{Trace}\! \big( 
	\sigma(x)[\sigma(x)]^{\ast}(\operatorname{Hess}_x u )(t,x)
	\big)  
	+
	\langle (\nabla_x u)(t,x),\mu(x)\rangle
	\end{split}
	\end{equation}
	for $(t,x) \in (0,\infty) \times \R^d$.
	This establishes items~\eqref{viscosity_existence:item1}--\eqref{viscosity_existence:item2}.
	For the next step 
	let $T \in (0,\infty)$ and 
	let 
	$v \colon \allowbreak [0,T] \times \R^d \to \R$ 
	be a continuous function which satisfies
	for all $x \in \R^d$ that 
	$
	v(0,x) = \varphi(x)
	$, 
	which satisfies that
	\begin{equation}\label{viscosity_existenceGrowthv}
			\inf_{q \in (0, \infty)} 
		\sup_{(t, x) \in [0, T] \times \R^d} 
		\frac{ | v(t, x) | }{ 1 + \norm{x}^q }
		<
		\infty,
	\end{equation}
	and which satisfies that $v|_{(0,T) \times \R^d}$ is a viscosity solution of
	\begin{equation}
	\label{viscosity_existence:eq1}
	\begin{split}
	(\tfrac{\partial }{\partial t}v)(t,x) 
	&= 
	\tfrac{1}{2} 
	\operatorname{Trace}\! \big( 
	\sigma(x)[\sigma(x)]^{\ast}(\operatorname{Hess}_x v )(t,x)
	\big)  
	+
	\langle (\nabla_x v)(t,x),\mu(x)\rangle
	\end{split}
	\end{equation}
	for $(t,x) \in (0,T) \times \R^d$.
	Observe that \eqref{viscosity_existenceGrowthu} and \eqref{viscosity_existenceGrowthv} show that 
		\begin{equation}
			\inf_{q \in [3, \infty)} 
			\sup_{(t, x) \in [0, T] \times \R^d} 
			\frac{ | u(t, x) | }{ 1 + \norm{x}^q }
			<
			\infty
	\quad \text{and}\quad
			\inf_{q \in [3, \infty)} 
			\sup_{(t, x) \in [0, T] \times \R^d} 
			\frac{ | v(t, x) | }{ 1 + \norm{x}^q }
			<
			\infty.
		\end{equation}
Therefore, we obtain that there exists $p\in [3,\infty)$ such that 
	\begin{equation}
	\sup_{(t, x) \in [0, T] \times \R^d} 
	\frac{ | u(t, x) | }{ 1 + \norm{x}^p }
	<
	\infty
	\qandq
	\sup_{(t, x) \in [0, T] \times \R^d} 
	\frac{ | v(t, x) | }{ 1 + \norm{x}^p }
	<
	\infty.
	\end{equation}
	Next note that the fact that there exists $r_0>0$ such that the function $\big([r_0,\infty)\ni r\mapsto \tfrac{1 + r^p }{ 1 + r^{p+1} }\big)$ is monotonically decreasing ensures that there exists $n_0\in\N$ such that for all $n\in\N\cap [n_0,\infty)$, $x\in C_n$ it holds that 
	\begin{align*}
		\tfrac{1 + \norm{x}^p }{ 1 + \norm{x}^{p+1} }\le \tfrac{1 + n^p }{ 1 + n^{p+1} }.
	\end{align*}
	Therefore, we obtain that for all $w \in \{u, v \}$ it holds that
	\begin{equation}
	\label{viscosity_existence:eq2}
	\begin{split}
	&\lim_{n \to \infty} 
	\sup_{(t, x) \in [0, T] \times C_n} 
	\tfrac{ | w(t, x) | }{ V_{p+1}(x) } \\
	&=
	\lim_{n \to \infty} 
	\sup_{(t, x) \in [0, T] \times C_n} 
	\left(\tfrac{ | w(t, x) | }{1 + \norm{x}^p }
	\tfrac{1 + \norm{x}^p }{ 1 + \norm{x}^{p+1} } \right)\\
	&\leq
	\left[
	\sup_{(t, x) \in [0, T] \times \R^d} 
	\tfrac{ | w(t, x) | }{ 1 + \norm{x}^p }
	\right]
	\lim_{n \to \infty} 
	\sup_{(t, x) \in [0, T] \times C_n} 
	\tfrac{1 + \norm{x}^p }{ 1 + \norm{x}^{p+1} } \\
	&\leq
	\left[
	\sup_{(t, x) \in [0, T] \times \R^d} 
	\tfrac{ | w(t, x) | }{ 1 + \norm{x}^p }
	\right]
	\lim_{n \to \infty} 
	\tfrac{1 + n^p }{ 1 + n^{p+1} }
	=
	0.
	\end{split}
	\end{equation}
	In addition, note that \eqref{viscosity_existence:ass1}, the fact that $p+1 \in [4,\infty)$, and  Lemma~\ref{viscosity_lyapunov} 
	(with $d=d$, $m=m$, $p=p+1$, $\mu=\mu$, $\sigma=2^{-1/2}\sigma$, $V=V_{p+1}$ in the notation of Lemma~\ref{viscosity_lyapunov})	
	prove that 
	there exists $\rho \in (0,\infty)$ such that 
	for all $x \in \R^d$ it holds that
	$V_{p+1} \in C^2(\R^d, (0,\infty))$ 
	and
	\begin{equation}
	\left< \mu(x) , (\nabla V_{p+1}) (x) \right>
	+\tfrac{1}{2}
	\operatorname{Trace}\! \big( 
	\sigma(x)[\sigma(x)]^{\ast}(\operatorname{Hess} V_{p+1} )(x)
	\big)  
	\leq
	\rho V_{p+1}(x).
	\end{equation}
	Combining this, 
	\eqref{viscosity_existence:ass2}, item~\eqref{viscosity_existence:item2}, 
	the fact that 
	for all $x \in \R^d$ it holds that
	$
	u(0,x) = \varphi(x) = v(0,x)
	$, 
	\eqref{viscosity_existence:eq1}, and  \eqref{viscosity_existence:eq2} 
	with 
	Hairer et al.\ \cite[Corollary 4.14]{HairerHutzenthalerJentzen15} 
	(with
	$T = T$, 
	$d = d$,
	$m = m$, 
	$\rho = \rho$, 
	$O = \R^d$, 
	$\varphi = \varphi$, 
	$v = 0$, 
	$\mu = \mu$, 
	$\sigma =2^{-1/2}\sigma$, 
	$V = V_{p+1}$
	in the notation of  Hairer et al.\ \cite[Corollary 4.14]{HairerHutzenthalerJentzen15})
	establishes that
	$v = u|_{[0, T] \times \R^d}$.
	This proves item~\eqref{viscosity_existence:item3}.
	It thus remains to prove item~\eqref{viscosity_existence:item4}. 
	For this 
	let $T \in (0,\infty)$, $x \in \R^d$,
	let $(\Omega,\allowbreak \mathcal{F},\allowbreak\P,\allowbreak(\mathbbm{F}_t)_{t \in [0,T]})$ be a filtered probability space which fulfils the usual conditions,
	let $W \colon [0,\infty) \times \Omega \to \R^m$ be a standard $ (\mathbbm{F}_t)_{t \in [0,\infty)}$-Brownian motion, and
	let $X \colon [0,T] \times \Omega \to \R^d$ be an $(\mathbbm{F}_t)_{t \in [0,T]}$-adapted stochastic process with continuous sample paths which satisfies
	that for all $t \in [0,T]$ it holds	$\P$-a.s.\ that
	\begin{equation}
	X_t 
	= 
	x 
	+
	\int_0^t \mu(X_s) \, ds
	+
	\int_0^t \sigma(X_s) \, dW_s.
	\end{equation}
	Observe that \eqref{viscosity_existence:ass2} and  Klenke~\cite[Theorem 26.8]{Klenke14}
	assure that there exists an $(\mathbbm{F}_t)_{t \in [0,\infty)}$-adapted stochastic process with continuous sample paths $Y \colon [0,\infty) \allowbreak\times \Omega \to \R^d$ which satisfies
	that for all $t \in [0,\infty)$ it holds $\P$-a.s.\ that	
	\begin{equation}
Y_t 
= 
x 
+
\int_0^t \mu(Y_s) \, ds
+
\int_0^t \sigma(Y_s) \, dW_s.
\end{equation}
Moreover, note that, e.g., Jentzen \& Kloeden \cite[Theorem 5.1]{JentzenKloeden11} (with $T=T$, $(\Omega,\allowbreak \mathcal{F},\allowbreak\P,\allowbreak(\mathbbm{F}_t)_{t \in [0,T]})=(\Omega,\allowbreak \mathcal{F},\allowbreak\P,\allowbreak(\mathbbm{F}_t)_{t \in [0,T]})$, $H=\R^d$, $U=\R^m$, $Q u=u$, $(W_t)_{t\in[0,T]}=(W_t)_{t\in[0,T]}$, $D(A)=\R^d$, $Av=0$, $\eta=1$, $\alpha=0$, $\delta=0$, $F(v)=\mu(v)$, $\beta=0$, $B(v)u=\sigma(v)u$, $\gamma=0$, $p=2$, $\xi=(\Omega\ni \omega\mapsto x\in\R^d)$ for $v,x\in\R^d$, $u\in\R^m$ in the notation of \cite[Theorem 5.1]{JentzenKloeden11}) implies that for all $s \in [0,T]$ it holds that $\P (Y_s = X_s) = 1$.
	This, item~\eqref{viscosity_existence:item1}, and the Feynman-Kac formula in Hairer et al.\ \cite[Corollary 4.17]{HairerHutzenthalerJentzen15} demonstrate that
	\begin{equation}
	u(T,x)
	=
	\EXPP{\varphi(Y_T)}
	=
	\EXPP{\varphi(X_T)}.
	\end{equation} 
	This implies item~\eqref{viscosity_existence:item4}.
	The proof of Proposition~\ref{viscosity_existence} is thus completed.
\end{proof}

\begin{cor}
	\label{viscosity_affine_existence}
	Let $d,m \in \N$, $T \in (0,\infty)$, 
	let $\left\| \cdot \right\| \colon \R^d \to [0,\infty)$ be the $d$-dimensional Euclidean norm, 
	let 
	$\varphi \colon \R^d \to \R$ be a continuous and at most polynomially growing function, 
	and
	let 
	$\mu \colon \R^d \to \R^d$ and
	$\sigma \colon \R^d \to \R^{d \times m}$
	be functions which satisfy 
	for all $x,y \in \R^d$, $\lambda \in \R$ that
\begin{equation}\label{viscosity_affine_existenceAffineCoefficients}
\begin{split}
	\mu(\lambda x+y) + \lambda \mu(0) = \lambda\mu(x)+\mu(y)\\ \andq 
\sigma(\lambda x+y) + \lambda \sigma(0) = \lambda\sigma(x)+\sigma(y).
\end{split}
\end{equation}
	Then
	\begin{enumerate}[(i)]
		\item \label{viscosity_affine_existence:item1}
		there exists a unique continuous function $u\colon [0,T]\allowbreak \times \R^d \to \R$ 
		which satisfies that
		$
		\inf_{q \in (0,\infty)} \allowbreak
		\sup_{(t, x) \in [0, T] \times \R^d} \allowbreak
		\frac{ | u(t, x) | }{ 1 + \norm{x}^q }
		<
		\infty
		$,
		which satisfies
		for all $x \in \R^d$ that 
		$
		u(0,x) = \varphi(x)
		$, 
		and which satisfies that $u|_{(0,T) \times \R^d}$ is a viscosity solution of
		\begin{equation}\label{viscosity_affine_existence:PDE}
		\begin{split}
		(\tfrac{\partial }{\partial t}u)(t,x) 
		&= 
		\tfrac{1}{2} 
		\operatorname{Trace}\! \big( 
		\sigma(x)[\sigma(x)]^{\ast}(\operatorname{Hess}_x u )(t,x)
		\big)  
		+
		\langle (\nabla_x u)(t,x),\mu(x)\rangle
		\end{split}
		\end{equation}
		for $(t,x) \in (0,T) \times \R^d$ and
		
		\item \label{viscosity_affine_existence:item2}
		for every $x \in \R^d$,
				every filtered probability space $(\Omega,\allowbreak \mathcal{F},\allowbreak\P,\allowbreak(\mathbbm{F}_t)_{t \in [0,T]})$ which fulfils the usual conditions,
		every standard $(\mathbbm{F}_t)_{t \in [0,T]}$-Brownian motion $W \colon [0,T] \times \Omega \to \R^m$, and
		every $(\mathbbm{F}_t)_{t \in [0,T]}$-adapted stochastic process with continuous sample paths $X \colon [0,T] \times \Omega \to \R^d$ which satisfies
		that for all $t \in [0,T]$ it holds	$\P$-a.s.\ that
		$
		X_t 
		= 
		x 
		+
		\int_0^t \mu(X_s) \, ds
		+
		\int_0^t \sigma(X_s) \, dW_s
$
		it holds that
		\begin{equation}
		u(T,x)
		=
		\EXPP{\varphi(X_T)}. 
		\end{equation} 
	\end{enumerate}
\end{cor}

\begin{proof}[Proof of Corollary~\ref{viscosity_affine_existence}]
	Throughout this proof 
	let $\langle \cdot, \cdot \rangle \colon \R^d \times \R^d \to \R$ be the $d$-dimensional Euclidean scalar product and let $\HSNorm{\cdot} \colon \R^{d \times m} \to [0,\infty)$ be the Hilbert-Schmidt norm on $\R^{d\times m}$.
	Note that 
	\eqref{viscosity_affine_existenceAffineCoefficients}, Corollary~\ref{linear_growth_affine}, and Corollary~\ref{linear_growth_affine_HilbertSchmidt} assure that
	there exists $\kappa \in (0,\infty)$ such that 
	for all $x, y \in \R^d$ it holds that
	\begin{equation}
	\label{viscosity_affine_existence:eq1}
	\norm{\mu(x)} +\HSNorm{{\sigma (x)}}
	\leq 
	\kappa  ( 1 + \norm{x})
	\end{equation}
	and
	\begin{equation}
	\label{viscosity_affine_existence:eq2}
	\norm{\mu(x) - \mu(y)} + \HSNorm{{\sigma (x)- \sigma(y)}}  \leq \kappa \norm{x-y}.
	\end{equation}
	This ensures that 
	for all $x \in \R^d$ it holds that
	\begin{equation}
	\label{viscosity_affine_existence:eq3}
	\begin{split}
	\langle x, \mu(x) \rangle 
	&\leq 
	\norm{x}\norm{\mu(x)}
	\leq
	\norm{x}\kappa( 1 + \norm{x}) 
	=
	\kappa( \norm{x} + \norm{x}^2) \\
	&\leq
	\kappa( 1 + \norm{x}^2 + \norm{x}^2)
	\leq
	2\kappa( 1 + \norm{x}^2).
	\end{split}
	\end{equation}
	Combining this, the hypothesis that $\varphi$ is a continuous function, and \eqref{viscosity_affine_existence:eq1}--\eqref{viscosity_affine_existence:eq2} 
	with 
	items~\eqref{viscosity_existence:item1}--\eqref{viscosity_existence:item3} in Proposition~\ref{viscosity_existence} 
	(with
	$ d = d $,
	$ m = m $,
	$ c = 2 \kappa$,
	$ \varphi = \varphi $,
	$ \mu = \mu $,
	$ \sigma = \sigma $
	in the notation of Proposition~\ref{viscosity_existence})
	proves item~\eqref{viscosity_affine_existence:item1}.
	Moreover, note that item~\eqref{viscosity_existence:item4} in Proposition~\ref{viscosity_existence} and item~\eqref{viscosity_affine_existence:item1} establish item~\eqref{viscosity_affine_existence:item2}.
	The proof of Corollary~\ref{viscosity_affine_existence} is thus completed.
\end{proof}

\section[Artificial neural network (ANN) approximations]{Artificial neural network approximations}\label{approximationSection}

\subsection{Construction of a realization on the artificial probability space}\label{SectionRealization}

In Theorem~\ref{cont_NN_approx} in Subsection~\ref{subsectionContinuousANN} below we establish that the number of required parameters of an ANN to approximate the solution of the Black-Scholes PDE grows at most polynomially in both the reciprocal of the prescribed approximation accuracy $ \varepsilon > 0 $ and the PDE dimension $ d \in \N $. An important ingredient in our proof of Theorem~\ref{cont_NN_approx} is an artificial probability space on which we establish the existence of a suitable realization with the desired approximation properties. In this subsection we provide, roughly speaking, in the elementary result in Proposition~\ref{construction_realization} below on a very abstract level the argument for the existence of such a realization on the artificial probability space. Proposition~\ref{construction_realization} is an immediate consequence from the elementary result in Corollary~\ref{construction_realization_help} below. Corollary~\ref{construction_realization_help}, in turn, follows from the following elementary lemma, Lemma \ref{construction_realization_helpLemma}. 

\begin{lemma}
	\label{construction_realization_helpLemma}
	Let $\varepsilon \in \R$, 
	let $ ( \Omega, \mathcal{F}, \P ) $ be a probability space, and
	let $ X \colon \Omega \to \R $ be a random variable which satisfies 
	that 	$
	\P(  X  > \varepsilon) = 1.
	$
	Then 
	\begin{enumerate}[(i)]
		
		\item \label{construction_realization_helpLemma:item1}
		it holds that $\Exp{\max\{-X,0\}}<\infty$ and
		
		\item \label{construction_realization_helpLemma:item2}
		it holds that $\Exp{  X  } > \varepsilon.$
	\end{enumerate}
\end{lemma}

\begin{proof}[Proof of Lemma~\ref{construction_realization_helpLemma}]
	Observe that the hypothesis that 	$
	\P(  X  > \varepsilon) = 1
	$  establishes item \eqref{construction_realization_helpLemma:item1}. Next note that 
	the hypothesis that 	$
	\P( X > \varepsilon) = 1
	$
	implies that 
	\begin{equation}
	0<1=\P( X > \varepsilon)=\P\!\left(\cup_{n\in\N}\big\{X \ge \varepsilon+\tfrac{1}{n}\big\}\right)=\lim_{n\to\infty} \P( X \ge \varepsilon+\tfrac{1}{n}).
	\end{equation}
	Hence, we obtain that there exists $\delta \in (0,\infty)$ such that 
	$
	\label{construction_realization_help:setting1}
	\P( X \geq \varepsilon + \delta) 
	> 
	0
	$.
	The hypothesis that 
	$
	\P( X > \varepsilon) = 1
	$
	therefore ensures that
	\begin{equation}
	\begin{split}
	\Exp{ X }
	&\geq
	(\varepsilon + \delta) \, \P( X \geq \varepsilon + \delta )
	+
	\varepsilon \,\P( X \in  (\varepsilon, \varepsilon + \delta)) \\
	&=
	\delta \, \P( X \geq \varepsilon + \delta )
	+
	\varepsilon \, \P( X > \varepsilon) 
	=
	\delta \, \P( X \geq \varepsilon + \delta )
	+
	\varepsilon 
	>
	\varepsilon.		 
	\end{split}
	\end{equation}This establishes item \eqref{construction_realization_helpLemma:item2}.
	The proof of Lemma~\ref{construction_realization_helpLemma} is thus completed.
\end{proof}

\begin{cor}
	\label{construction_realization_help}
	Let $\varepsilon \in [0,\infty)$, 
	let $ ( \Omega, \mathcal{F}, \P ) $ be a probability space, and
	let $ X \colon \Omega \to \R $ be a random variable which satisfies 
	that 	$
	\P( | X | > \varepsilon) = 1.
	$
	Then 
	\begin{equation}
	\Exp{ | X | } > \varepsilon.
	\end{equation}
\end{cor}

\begin{proof}[Proof of Corollary~\ref{construction_realization_help}]	
	Note that item \eqref{construction_realization_helpLemma:item2} in Lemma~\ref{construction_realization_helpLemma} 
	(with
	$\varepsilon=\varepsilon$, $X=\vert X\vert$
	in the notation of Lemma~\ref{gronwall}) ensures that $\Exp{ | X | } > \varepsilon$.
	The proof of Corollary~\ref{construction_realization_help} is thus completed. 
\end{proof}

\begin{prop}
	\label{construction_realization}
	Let $\varepsilon \in [0,\infty)$, 
	let $ ( \Omega, \mathcal{F}, \P ) $ be a probability space, and
	let $ X \colon \Omega \to \R $ be a random variable which satisfies that
	\begin{equation}\label{construction_realizationHypothesis}
	\Exp{ | X | } \leq \varepsilon.
	\end{equation}
	Then
	$\P( | X | \leq \varepsilon)>0.$
\end{prop}

\begin{proof}[Proof of Proposition~\ref{construction_realization}]
	Combining Corollary~\ref{construction_realization_help}  and \eqref{construction_realizationHypothesis} completes the proof of Proposition~\ref{construction_realization}.
%
\end{proof}

\subsection{Approximation error estimates}\label{SectionQuantitativeError}
\begin{samepage}
\begin{prop}
\label{quantitative}
Let $d, n \in \N$, $p\in [2,\infty)$, $T \in (0,\infty)$, $c, \varepsilon, L\in [0,\infty)$, $\mathbf{v}, \mathbf{w} \in [1/p,\infty)$, 
	let $\langle \cdot, \cdot \rangle \colon \R^d \times \R^d \to \R$ be the $d$-dimensional Euclidean scalar product,
let $\left\| \cdot \right\| \colon \R^d \to [0,\infty)$ be the $d$-dimensional Euclidean norm, 
let $\HSNorm{\cdot} \colon \R^{d \times d} \to [0,\infty)$ be the Hilbert-Schmidt norm on $\R^{d\times d}$,
let $\nu \colon \mathcal{B}(\R^d) \to [0,1]$ be a probability measure,
let $\varphi \colon \R^d \to \R$ be a continuous function, 
let $\phi \colon \R^d \to \R$ be a $\mathcal{B}(\R^d)/ \mathcal{B}(\R)$-measurable function which satisfies
for all $ x \in \R^d$ that
\begin{equation}
\label{quantitative:ass1}
	\left| 
		\phi(x)
	\right| 
\leq 
	c \, (1+\| x \|^{\mathbf{v}})
\qandq
	\left| 
		\varphi(x)-\phi(x)
	\right|
\leq 
	\varepsilon (1+ \| x \|^{\mathbf{w}}),
\end{equation}
and 
let $\mu \colon \R^d \to \R^d$ and $\sigma \colon \R^d \to \R^{d \times d}$ be functions which satisfy 
for all $x,y \in \R^d$, $\lambda \in \R$ that 
\begin{equation}\label{quantitative:AffineOne}
	\mu(\lambda x+y) + \lambda \mu(0) = \lambda\mu(x)+\mu(y),
\end{equation}
\begin{equation}\label{quantitative:AffineTwo}
	\sigma(\lambda x+y) + \lambda \sigma(0) = \lambda\sigma(x)+\sigma(y),
\end{equation}
and
$
		\norm{\mu(x)} + \HSNorm{{\sigma(x)}} 
\leq  
	L  (1 + \norm{x}).
$
Then 
\begin{enumerate}[(i)]

\item \label{quantitative:item1}
there exists a unique continuous function $u\colon [0,T]\allowbreak \times \R^d \to \R$
which satisfies that
$\inf_{q \in (0,\infty)} \allowbreak
\sup_{(t, x) \in [0, T] \times \R^d} \allowbreak
\frac{ | u(t, x) | }{ 1 + \norm{x}^q }
<
\infty,$
 which satisfies
for all $x \in \R^d$ that 
$
u(0,x) = \varphi(x)
$, 
and which satisfies that $u|_{(0,T) \times \R^d}$ is a viscosity solution of
\begin{equation}
\begin{split}
	(\tfrac{\partial }{\partial t}u)(t,x) 
&= 
	\tfrac{1}{2} 
	\operatorname{Trace}\! \big( 
		\sigma(x)[\sigma(x)]^{\ast}(\operatorname{Hess}_x u )(t,x)
	\big)  
	+
	\langle (\nabla_x u)(t,x),\mu(x)\rangle
\end{split}
\end{equation}
for $(t,x) \in (0,T) \times \R^d$ and

\item \label{quantitative:item2}
there exist $A_1,A_2,\ldots,A_{n} \in \R^{d \times d}$, $b_1,b_2,\ldots,b_{n} \in \R^d$ such that
\begin{equation}
\label{quantitative:concl1} 
\begin{split}
	&\left[
		\int_{\R^d}  
		\left|
			u(T,x) - \tfrac{1}{n}\big[ \!  \smallsum_{i=1}^{n}\, \phi( A_i x + b_i)  \big] 
		\right|^p \,
		\nu(dx)
	\right]^{\nicefrac{1}{p}} \\
&\leq 
		\varepsilon 
	\Bigg(
		1 +
		2^{\mathbf{w}/2} 
		\exp{\!
		\left( 
			\big[
				\sqrt{T}+\max\{2, \mathbf{w} \}
			\big]^2
			L^2 \, T \mathbf{w}
		\right)} \\
		&\qquad \cdot
		\left[
			L 
			\big(
				T + \max\{2, \mathbf{w} \}\sqrt{T} 
			\big)
			+
			\Big[ \!
				\textint_{\R^d}  
					\Norm{x}^{\mathbf{w} p} \,
				\nu(dx)
			\Big]^{ \!\nicefrac{1}{(\mathbf{w}p)}}
		\right]^\mathbf{w}
	\Bigg) \\
&\quad 
	+n^{-1/2} 4\,c\,(p-1)^{1/2}	
	\Bigg( 
		1
		+
		2^{\mathbf{v}/2} 
		\exp{ \!
		\left( 
			\big[
				\sqrt{T} + \max\{2, \mathbf{v}p\}
			\big]^2
			L^2 \, T \mathbf{v}
		\right)} \\
	&\qquad \cdot	
		\left[
			L
			\big(
				T +  \max\{2, \mathbf{v}p\} \sqrt{T} 
			\big)		
			+
			\Big[
				\textint_{\R^d} 
					\Norm{x}^{\mathbf{v}p} \,
				\nu(dx)
			\Big]^{ \nicefrac{1}{(\mathbf{v}p)}}
		\right]^\mathbf{v} 
	\Bigg).
\end{split}
\end{equation}
\end{enumerate}
\end{prop}
\end{samepage}

\begin{proof}[Proof of Proposition~\ref{quantitative}]
Throughout this proof 
let $e_j \in \R^d$, $j \in \{1,2,\ldots, d \}$, be given by
$
e_1 = (1, 0, \ldots, 0),
e_2 = (0, 1, 0, \ldots, 0),
\ldots,
e_d = (0, \ldots, 0, 1)
$,
let $m \colon (0,\infty) \to [2,\infty)$ be the function which satisfies for all $z \in (0,\infty)$ that 
$
	m(z) = \max \{2, z \}
$,
let $(\Omega,\allowbreak \mathcal{F},\allowbreak \P,\allowbreak(\mathbbm{F}_t)_{t \in [0,T]})$ be a filtered probability space which fulfils the usual conditions,
let $W^{i} \colon [0,T] \times \Omega \to \R^d$, $i \in \N$, be independent standard $(\mathbbm{F}_t)_{t \in [0,T]}$-Brownian motions, 
let $X^{i, x} \colon [0,T] \times \Omega \to \R^d$, $i \in \N$, $x \in \R^d$, be $(\mathbbm{F}_t)_{t \in [0,T]}$-adapted stochastic processes with continuous sample paths which satisfy 
that for all $i \in \N$, $x \in \R^d$, $t \in [0,T]$ it holds $\P$-a.s.\ that
\begin{equation}
\label{quantitative:setting1}
	X^{i, x}_t 
= 
	x 
	+
	\int_0^t \mu(X^{i, x}_s) \, ds
	+
	\int_0^t \sigma(X^{i, x}_s) \, dW^{i}_s
\end{equation}
and that for all $i \in \N$, $t \in [0,T]$, $\omega \in \Omega$, $\lambda \in \R$, $x, y \in \R^d$ it holds that
\begin{equation}
\label{quantitative:setting2}
	X^{i, \lambda x + y}_t(\omega) 
	+ 
	\lambda X^{i, 0}_t(\omega) 
=
	\lambda X^{i, x}_t(\omega) 
	+
	 X^{i, y}_t(\omega)
\end{equation}
(cf.~Proposition~\ref{affine_solutions_of_SDEs}), 
and let $\mathscr{A}_{i} \colon \Omega \to \R^{d \times d}$,  $i \in \N$, and $\mathscr{B}_{i} \colon \Omega \to \R^{d}$,  $i \in \N$, be the random variables which satisfy that
for all $i \in \N$, $\omega \in \Omega$ it holds that
$
	\mathscr{B}_{i} (\omega)
=
	X^{i, 0}_T (\omega) 
$
and 
\begin{multline} 
\label{quantitative:setting3}
	\mathscr{A}_{i}(\omega) \\
=
	\bigg(
		X^{i, e_1}_T (\omega) - X^{i, 0}_T (\omega)
		\Big| \,
		X^{i, e_2}_T (\omega) - X^{i, 0}_T (\omega)
		\Big|
		\cdots
		\Big| \,
		X^{i, e_d}_T (\omega) - X^{i, 0}_T (\omega)
	\bigg).
\end{multline}
Observe that \eqref{quantitative:ass1} assures for all $x\in\R^d$ that 
\begin{equation}
\begin{split}
\vert 
\varphi(x)\vert&\leq 
\left| \phi(x)
\right| +\left| 
\varphi(x)-\phi(x)
\right|
\leq 
c  (1+\| x \|^{\mathbf{v}})+
\varepsilon (1+ \| x \|^{\mathbf{w}})
\\&\leq 
c  (2+\| x \|^{\max\{\mathbf{v},\mathbf{w}\}})+
\varepsilon (2+\| x \|^{\max\{\mathbf{v},\mathbf{w}\}})
\\&\leq  2(c+\varepsilon) (1+\| x \|^{\max\{\mathbf{v},\mathbf{w}\}}).
\end{split}
\end{equation}
Therefore, we obtain that $\varphi$ is an at most polynomially growing function.
This, \eqref{quantitative:AffineOne}, \eqref{quantitative:AffineTwo}, and item~\eqref{viscosity_affine_existence:item1} in Corollary~\ref{viscosity_affine_existence}
(with
$ d = d $,
$ m=d $,
$ T = T $,
$ \varphi = \varphi $,
$ \mu = \mu $,
$ \sigma = \sigma $
in the notation of Corollary~\ref{viscosity_affine_existence})
demonstrate that
there exists a unique continuous function $u\colon [0,T] \times \R^d \to \R$ which satisfies
for all $x \in \R^d$ that 
$
	u(0,x) = \varphi(x)
$, 
which satisfies that
$
	\inf_{q \in (0,\infty)} 
	\sup_{(t, x) \in [0, T] \times \R^d} 
	\frac{ | u(t, x) | }{ 1 + \norm{x}^q }
<
	\infty
$,
and which satisfies that $u|_{(0,T) \times \R^d}$ is a viscosity solution of
\begin{equation}
\begin{split}
	(\tfrac{\partial }{\partial t}u)(t,x) 
&= 
	\tfrac{1}{2} 
	\operatorname{Trace}\! \big( 
		\sigma(x)[\sigma(x)]^{\ast}(\operatorname{Hess}_x u )(t,x)
	\big)  
	+
	\langle (\nabla_x u)(t,x),\mu(x)\rangle
\end{split}
\end{equation}
for $(t,x) \in (0,T) \times \R^d$.
This proves item~\eqref{quantitative:item1}.
Next note that item~\eqref{viscosity_affine_existence:item2} in Corollary~\ref{viscosity_affine_existence}, item~\eqref{quantitative:item1}, and \eqref{quantitative:setting1} establish that
for all $x \in \R^d$ it holds that
\begin{equation}
\label{quantitative:eq01}
	u(T,x)
=
	\EXPP{\varphi(X^{1, x}_T)}. 
\end{equation}
Moreover, observe that \eqref{quantitative:setting2}, \eqref{quantitative:setting3}, and Lemma~\ref{affine_representation}
(with
$d = d$,
$m = d$, 
$e_j=e_j$,
$\varphi(x)=X^{i, x}_T(\omega)$
for $j\in \{1,2,\dots,d\}$, $i \in \N$, $x \in \R^d$, $\omega \in \Omega$
in the notation of Lemma~\ref{affine_representation})
 demonstrate that
for all $i \in \N$, $x \in \R^d$, $\omega \in \Omega$ it holds that
\begin{equation}
\label{quantitative:eq02}
	X^{i, x}_T(\omega) 
= 
	\mathscr{A}_{i}(\omega) x +\mathscr{B}_{i}(\omega).
\end{equation}
This ensures that 
for all $i \in \N$, $\omega \in \Omega$ it holds that the function
\begin{equation}
\R^d \ni x \mapsto X^{i, x}_T(\omega) \in \R^d
\end{equation}
is continuous.
Combining this and the fact that 
for all $i \in \N$, $x \in \R^d$ it holds that $X^{i, x}_T \colon \Omega \to \R^d$ is $\mathcal{F} / \mathcal{B}(\R^d)$-measurable 
with Beck et al.\ \cite[Lemma 2.4]{KolmogorovArxiv} 
(with
$E = \R^d$,
$\mathcal{E} = \R^d$, 
$\Omega=\Omega$,
$X(x,\omega)=X^{i, x}_T(\omega)$
for $i \in \N$, $x \in \R^d$, $\omega \in \Omega$
in the notation of \cite[Lemma 2.4]{KolmogorovArxiv})
establishes that
for all $i \in \N$ it holds that the function
\begin{equation}
\label{quantitative:eq03}
	\R^d \times \Omega \ni (x, \omega) 
\mapsto 
	X^{i, x}_T(\omega) \in \R^d
\end{equation}
is $(\mathcal{B}(\R^d) \otimes \mathcal{F}) / \mathcal{B}(\R^d)$-measurable.
This, \eqref{quantitative:ass1}, and the triangle inequality assure that
for all $x \in \R^d$ it holds that
\begin{equation}
\label{quantitative:eq04}
\begin{split}
	&\left[
		\int_{\R^d}  
		\left|
			\EXPP{\varphi(X^{1, x}_T)}
			- 
			\EXPP{\phi (X^{1, x}_T)}
		\right|^p \,
		\nu(dx)
	\right]^{\nicefrac{1}{p}} \\
&=
	\left[
		\int_{\R^d}  
		\left|
			\EXPP{\varphi(X^{1, x}_T)
			-
			\phi (X^{1, x}_T)}
		\right|^p \,
		\nu(dx)
	\right]^{\nicefrac{1}{p}} \\
&\leq
	\left[
		\int_{\R^d}  
		\left(
			\Exp{ 
			\big|
				\varphi(X^{1, x}_T) - \phi (X^{1, x}_T) 
			\big| }
		\right)^p \,
		\nu(dx)
	\right]^{\nicefrac{1}{p}} \\
&\leq	
	\left[
		\int_{\R^d}  
		\left(
			\Exp{ 
			\varepsilon 
			\big(
				1+ \Normm{ X^{1, x}_T }^{\mathbf{w}}
			\big) }
		\right)^p \,
		\nu(dx)
	\right]^{\nicefrac{1}{p}} \\
&=	
	\varepsilon 
	\left[
		\int_{\R^d}  
		\left(
			1
			+
			\Exp{ \Normm{ X^{1, x}_T }^{\mathbf{w}} }
		\right)^p \,
		\nu(dx)
	\right]^{\nicefrac{1}{p}} \\
&\leq	
	\varepsilon 
	\left(
		1 +
		\left[
			\int_{\R^d}  
			\left(
				\Exp{ \Normm{ X^{1, x}_T }^{\mathbf{w}} }
			\right)^p \,
			\nu(dx)
		\right]^{\nicefrac{1}{p}}
	\right).
\end{split}
\end{equation}
Furthermore, observe that Jensen's inequality, 
the hypothesis that 
for all $x \in \R^d$ it holds that
$
		\norm{\mu(x)} + \HSNorm{{\sigma(x)}} 
\leq  
	L  (1 + \norm{x})
$,
\eqref{quantitative:setting1}, and
Proposition~\ref{moments_of_solution_of_SDE}
(with
$d = d$,
$m=d$,
$p = m(z)$, 
$\mathfrak{m}_1=\mathfrak{m}_2=\mathfrak{s}_1=\mathfrak{s}_2=L$,
$T = T$,
$\xi = x$,
$\mu = \mu$,
$\sigma = \sigma$ 
for $z \in (0,\infty)$, $x \in \R^d$
in the notation of Proposition~\ref{moments_of_solution_of_SDE})
prove that 
for all $x \in \R^d$, $z \in (0,\infty)$ it holds that
\begin{equation}
\label{quantitative:eq05}
\begin{split}
	&\Exp{ \Normm{ X^{1, x}_T }^z} 
= 
	\Exp{ \Big[ \Normm{ X^{1, x}_T }^{m(z)} \Big]^{\nicefrac{z}{m(z)}} }
\leq 
	\left(
		\Exp{ \Normm{ X^{1, x}_T }^{m(z)} }
	\right)^{\nicefrac{z}{m(z)}}\\
&\leq
	\left[
		\sqrt{2} 
		\big(
			\Norm{x} +L T + L  m(z)\sqrt{T} 
		\big) 
		\exp{ \!
		\Big( 
			\big[
				L\sqrt{T}+L m(z)
			\big]^2 \,
			T
		\Big)}
	\right]^{z} \\
&=
	2^{z/2} 
	\Big[
		\Norm{x} + L \big(   T + m(z)\sqrt{T}   \big) 
	\Big]^z
	\exp{\!
	\Big( 
		\big[
			 \sqrt{T} + m(z)
		\big]^2
		L^2 \, T z
	\Big)}.
\end{split}
\end{equation}
The fact that $\mathbf{w}p\ge 1$ and the triangle inequality hence prove that
for all $x \in \R^d$ it holds that
\begin{equation}
\begin{split}
	&\left[
		\int_{\R^d}  
		\left(
			\Exp{ \Normm{ X^{1, x}_T }^{\mathbf{w}} }
		\right)^p \,
		\nu(dx)
	\right]^{\nicefrac{1}{p}} \\
&\leq
	\Bigg[
		\int_{\R^d}  
		\bigg[
			2^{\mathbf{w}/2} 
			\Big[
				\Norm{x} +L \big(   T + m(\mathbf{w})\sqrt{T}   \big)
			\Big]^\mathbf{w} \\
			&\qquad \cdot
			\exp{ \!
			\Big( 
				\big[
					 \sqrt{T} + m(\mathbf{w})
				\big]^2
				L^2 \, T \mathbf{w}
			\Big)}
		\bigg]^p \,
		\nu(dx)
	\Bigg]^{\nicefrac{1}{p}} \\
&=
	2^{\mathbf{w}/2} 
	\exp{ \!
	\Big( 
		\big[
			\sqrt{T} + m(\mathbf{w})
		\big]^2
		L^2 \, T \mathbf{w}
	\Big)} \\
	&\qquad \cdot
	\left[
		\Bigg[
			\int_{\R^d}  
			\Big[
				\Norm{x} + L \big(   T +  m(\mathbf{w})\sqrt{T}   \big)
			\Big]^{\mathbf{w} p}\,
			\nu(dx)
		\Bigg]^{\nicefrac{1}{(\mathbf{w}p)}}
	\right]^\mathbf{w} \\
&\leq
	2^{\mathbf{w}/2} 
	\exp{ \!
	\Big( 
		\big[
			\sqrt{T} + m(\mathbf{w})
		\big]^2
		L^2 \, T \mathbf{w}
	\Big)} \\
	&\qquad \cdot
	\left[
		\left[
			\int_{\R^d}  
				\Norm{x}^{\mathbf{w} p}\,
			\nu(dx)
		\right]^{\nicefrac{1}{(\mathbf{w}p)}}
		+ 
		L \big(   T +  m(\mathbf{w})\sqrt{T}   \big)
	\right]^\mathbf{w}.
\end{split}
\end{equation}
Combining this and \eqref{quantitative:eq04} demonstrates that
for all $x \in \R^d$ it holds that
\begin{equation}
\label{quantitative:eq06}
\begin{split}
	&\left[
		\int_{\R^d}  
		\left|
			\EXPP{\varphi(X^{1, x}_T)}
			- 
			\EXPP{\phi (X^{1, x}_T)}
		\right|^p \,
		\nu(dx)
	\right]^{\nicefrac{1}{p}} \\
&\leq
	\varepsilon 
	\Bigg(
		1 +
		2^{\mathbf{w}/2} 
		\exp{\!
		\Big( 
			\big[
				\sqrt{T} + m(\mathbf{w})
			\big]^2
			L^2 \, T \mathbf{w}
		\Big)} \\
		&\qquad \cdot
		\left[
			\left[
				\int_{\R^d}  
					\Norm{x}^{\mathbf{w} p} \,
				\nu(dx)
			\right]^{\nicefrac{1}{(\mathbf{w}p)}}
			+ 
			L \big(   T + m(\mathbf{w})\sqrt{T}   \big)
		\right]^\mathbf{w}
	\Bigg).
\end{split}
\end{equation}
Moreover, observe that the fact that  $W^{i} \colon [0,T] \times \Omega \to \R^d$, $i \in \N$, are independent Brownian motions ensures for every $x\in\R^d$ that   $X^{i, x}_T \colon \Omega \to \R^d$, $i \in \N$, are i.i.d.\ random variables (cf., e.g., Beck et al. \cite[Theorem 2.8]{KolmogorovArxiv} and Klenke \cite[Theorem 15.8]{Klenke14}).
Combining this, \eqref{quantitative:ass1}, \eqref{quantitative:eq05}, and the hypothesis that $\phi \colon \R^d \to \R$ is a $\mathcal{B}(\R^d)/ \mathcal{B}(\R)$-measurable function proves for every $x\in\R^d$ that $\phi(X^{i, x}_T) \colon \Omega \to \R$, $i \in \{1,2,\dots, n\}$, are i.i.d. random variables which satisfy for every $x\in\R^d$ that 
\begin{equation}
	\EXPP{\vert{\phi (X^{1, x}_T)}\vert}\le c \, \left(1+\Exp{ \Normm{ X^{1, x}_T }^\mathbf{v}} \right)<\infty.
\end{equation} 
This, H\"older's inequality, Fubini's theorem, and Corollary~\ref{mc_Lp_error2}
(with
$ p = p $,
$ d = 1 $,
$ n = n $,
$ X_i = \phi (X^{i, x}_T) $ for $i \in \{1,2,\dots, n\}$, $x\in\R^d$
in the notation of Corollary~\ref{mc_Lp_error2})
 demonstrate that 
\begin{equation}
\label{quantitative:eq07}
\begin{split}
	&\Exp{
	\left[
		\int_{\R^d}  
		\left|
			\EXPP{\phi (X^{1, x}_T)}
			- 
			\tfrac{1}{n} \big[ \! \smallsum_{i=1}^{n}\, \phi(X^{i, x}_T)  \big] 
		\right|^p \,
		\nu(dx)
	\right]^{\nicefrac{1}{p}}
	} \\
&\leq
	\left(
	\Exp{
		\int_{\R^d}  
		\left|
			\EXPP{\phi (X^{1, x}_T)}
			- 
			\tfrac{1}{n} \! \left[\smallsum_{i=1}^{n}\, \phi(X^{i, x}_T)\right] 
		\right|^p \,
		\nu(dx)
	}
	\right)^{ \! \! \nicefrac{1}{p}} \\
&=
	\left(
		\int_{\R^d} 
		\Exp{
			\left|
				\EXPP{\phi (X^{1, x}_T)}
				- 
				\tfrac{1}{n} \! \left[\smallsum_{i=1}^{n}\, \phi(X^{i, x}_T)\right] 
			\right|^p \,
		}
		\nu(dx)
	\right)^{ \! \! \nicefrac{1}{p}} \\
&\leq
	\left(
		\int_{\R^d} 
		\left[
			\frac{2(p-1)^{1/2}}{n^{1/2}}
		\right]^p
		\Exp{
			\big|
				\phi(X^{1, x}_T)
				- 
				\EXPP{\phi (X^{1, x}_T)}
			\big|^p \,
		}
		\nu(dx)
	\right)^{ \! \! \nicefrac{1}{p}} \\
&=
			\frac{2(p-1)^{1/2}}{n^{1/2}}
	\left(
		\int_{\R^d} 
		\Exp{
			\big|
				\phi(X^{1, x}_T)
				- 
				\EXPP{\phi (X^{1, x}_T)}
			\big|^p \,
		}
		\nu(dx)
	\right)^{ \! \! \nicefrac{1}{p}}.
\end{split}
\end{equation}
Moreover, observe that H\"older's inequality demonstrates that for all $x\in \R^d$ it holds that
\begin{equation}
	\EXPPP{
		\big|
		\EXPP{\phi (X^{1, x}_T)}
		\big|^p}
	=\big|
	\EXPP{\phi (X^{1, x}_T)}
	\big|^p
	\le \EXPP{\big|
	\phi (X^{1, x}_T)
	\big|^p}.
\end{equation}
The triangle inequality, H\"older's inequality, \eqref{quantitative:ass1}, and \eqref{quantitative:eq07} hence imply that
\begin{equation}
\label{quantitative:eq08}
\begin{split}
	&\Exp{
	\left[
		\int_{\R^d}  
		\left|
			\EXPP{\phi (X^{1, x}_T)}
			- 
			\tfrac{1}{n} \! \left[\smallsum_{i=1}^{n}\, \phi(X^{i, x}_T)\right] 
		\right|^p \,
		\nu(dx)
	\right]^{\nicefrac{1}{p}}
	} \\
&\leq
			\frac{2(p-1)^{1/2}}{n^{1/2}}
	\Bigg[
		\left(
			\int_{\R^d} 
			\Exp{
				\big|
					\phi(X^{1, x}_T)
				\big|^p \,
			}
			\nu(dx)
		\right)^{ \! \! \nicefrac{1}{p}} \\
	&\hspace{3cm}+
		\left(
			\int_{\R^d} 
			\EXPPP{
				\big|
					\EXPP{\phi (X^{1, x}_T)}
				\big|^p \,
			}
			\nu(dx)
		\right)^{ \! \! \nicefrac{1}{p}}
	\Bigg] \\
&\leq
			\frac{4(p-1)^{1/2}}{n^{1/2}}
	\left(
		\int_{\R^d} 
		\Exp{
			\big|
				\phi(X^{1, x}_T)
			\big|^p \,
		}
		\nu(dx)
	\right)^{ \! \! \nicefrac{1}{p}} \\
&\leq
			\frac{4(p-1)^{1/2}}{n^{1/2}}
	\left(
		\int_{\R^d} 
		\Exp{
			\big|
				c  (1+\| X^{1, x}_T \|^{\mathbf{v}}   )
			\big|^p 
		}
		\nu(dx)
	\right)^{ \! \! \nicefrac{1}{p}} \\
&=
			\frac{4c(p-1)^{1/2}}{n^{1/2}}
	\left(
		\int_{\R^d} 
		\Exp{
			\big(
				 1+\| X^{1, x}_T \|^{\mathbf{v}}
			\big)^p 
		}
		\nu(dx)
	\right)^{ \! \! \nicefrac{1}{p}} \\
&\leq
\frac{4c(p-1)^{1/2}}{n^{1/2}}
	\left( 
		1
		+
		\left[
		\int_{\R^d} 
		\Exp{
				 \| X^{1, x}_T \|^{\mathbf{v}p}
		}
		\nu(dx)
		\right]^{ \nicefrac{1}{p}}
	\right).
\end{split}
\end{equation}
This, \eqref{quantitative:eq05}, the triangle inequality, and the fact that $\mathbf{v}p\ge 1$ assure that
\begin{equation}
\begin{split}
	&\Exp{
	\left[
		\int_{\R^d}  
		\left|
			\EXPP{\phi (X^{1, x}_T)}
			- 
			\tfrac{1}{n} \! \left[\smallsum_{i=1}^{n}\, \phi(X^{i, x}_T)\right] 
		\right|^p \,
		\nu(dx)
	\right]^{\nicefrac{1}{p}}
	} \\
&\leq
\frac{4c(p-1)^{1/2}}{n^{1/2}}
	\Bigg( 
		1
		+
		\bigg[
			\int_{\R^d} 
			2^{(\mathbf{v}p)/2} 
			\Big[
				\Norm{x} + L \big(   T + m(\mathbf{v}p)\sqrt{T}   \big) 
			\Big]^{\mathbf{v}p}
			\\
			&\qquad \cdot
			\exp{\!
			\Big( 
				\big[
					\sqrt{T} + m(\mathbf{v}p)
				\big]^2
				L^2 \, T \mathbf{v}p
			\Big)}		
			\nu(dx)
		\bigg]^{ \nicefrac{1}{p}}
	\Bigg) \\
&=
\frac{4c(p-1)^{1/2}}{n^{1/2}}
	\Bigg( 
		1
		+
		2^{\mathbf{v}/2} 
		\exp{\!
		\Big( 
			\big[
				\sqrt{T} + m(\mathbf{v}p)
			\big]^2
			L^2 \, T \mathbf{v}
		\Big)} \\
	&\qquad \cdot	
		\left[		
			\bigg[
				\int_{\R^d} 
				\Big[
					\Norm{x} + L \big(   T + m(\mathbf{v}p)\sqrt{T}   \big) 
				\Big]^{\mathbf{v}p}
				\nu(dx)
			\bigg]^{ \nicefrac{1}{(\mathbf{v}p)}}
		\right]^\mathbf{v}
	\Bigg) \\
&\leq
\frac{4c(p-1)^{1/2}}{n^{1/2}}
	\Bigg( 
		1
		+
		2^{\mathbf{v}/2} 
		\exp{\!
		\Big( 
			\big[
				\sqrt{T} + m(\mathbf{v}p)
			\big]^2
			L^2 \, T \mathbf{v}
		\Big)} \\
	&\qquad \cdot	
		\left[		
			\left[
				\int_{\R^d} 
					\Norm{x}^{\mathbf{v}p} \,
				\nu(dx)
			\right]^{ \nicefrac{1}{(\mathbf{v}p)}}
			+ 
			L \big(   T + m(\mathbf{v}p)\sqrt{T}   \big) 
		\right]^\mathbf{v}
	\Bigg).
\end{split}
\end{equation}
Proposition~\ref{construction_realization} hence demonstrates that
there exists $\bm\omega \in \Omega$ such that 
\begin{equation}\label{quantitative:ChoiceOfRealization}
\begin{split}
	&\left[
		\int_{\R^d}  
		\left|
			\EXPP{\phi (X^{1, x}_T)}
			- 
			\tfrac{1}{n} \! \left[\smallsum_{i=1}^{n}\, \phi(X^{i, x}_T(\bm\omega))\right] 
		\right|^p \,
		\nu(dx)
	\right]^{\nicefrac{1}{p}}\\
&\leq
\frac{4c(p-1)^{1/2}}{n^{1/2}}
	\Bigg( 
		1
		+
		2^{\mathbf{v}/2} 
		\exp{\!
		\Big( 
			\big[
				\sqrt{T} + m(\mathbf{v}p)
			\big]^2
			L^2 \, T \mathbf{v}
		\Big)} \\
	&\qquad \cdot	
		\left[		
			\left[
				\int_{\R^d} 
					\Norm{x}^{\mathbf{v}p} \,
				\nu(dx)
			\right]^{ \nicefrac{1}{(\mathbf{v}p)}}
			+ 
			L \big(   T + m(\mathbf{v}p)\sqrt{T}   \big) 
		\right]^\mathbf{v}
	\Bigg).
\end{split}
\end{equation}
Combining this, the triangle inequality, \eqref{quantitative:eq01}, \eqref{quantitative:eq02}, and \eqref{quantitative:eq06} establishes that there exists $\bm\omega \in \Omega$ such that 
\begin{equation}
\begin{split}
	&\left[
		\int_{\R^d}  
		\left|
			u(T,x) 
			- 
			\tfrac{1}{n} 
			\Big[ \!
				\smallsum_{i=1}^{n}\, 
				\phi ( 
					\mathscr{A}_{i}(\bm\omega) x +  \mathscr{B}_{i}(\bm\omega) 
					)
			\Big] 
		\right|^p \,
		\nu(dx)
	\right]^{\nicefrac{1}{p}} \\
&=
	\left[
		\int_{\R^d}  
		\left|
			\EXPP{\varphi(X^{1, x}_T)}
			- 
			\tfrac{1}{n} \!
			\left[
				\smallsum_{i=1}^{n}\, 
				\phi ( 
					X^{i, x}_T(\bm\omega) 
					)
			\right] 
		\right|^p \,
		\nu(dx)
	\right]^{\nicefrac{1}{p}}\\
&\leq
	\left[
		\int_{\R^d}  
		\left|
			\EXPP{\varphi(X^{1, x}_T)}
			- 
			\EXPP{\phi (X^{1, x}_T)}
		\right|^p \,
		\nu(dx)
	\right]^{\nicefrac{1}{p}} \\
&\quad	+
	\left[
		\int_{\R^d}  
		\left|
			\EXPP{\phi (X^{1, x}_T)}
			- 
			\tfrac{1}{n} \! \left[\smallsum_{i=1}^{n}\, \phi(X^{i, x}_T(\bm\omega) )\right] 
		\right|^p \,
		\nu(dx)
	\right]^{\nicefrac{1}{p}}\\
&\leq
	\varepsilon 
	\Bigg(
		1 +
		2^{\mathbf{w}/2} 
		\exp{\!
		\Big( 
			\big[
				\sqrt{T} + m(\mathbf{w})
			\big]^2
			L^2 \, T \mathbf{w}
		\Big)} \\
		&\qquad \cdot
		\left[
			\left[
				\int_{\R^d}  
					\Norm{x}^{\mathbf{w} p} \,
				\nu(dx)
			\right]^{\nicefrac{1}{(\mathbf{w}p)}}
			+ 
			L \big(   T + m(\mathbf{w})\sqrt{T}   \big)
		\right]^\mathbf{w}
	\Bigg) \\
&\quad 
	+
\frac{4c(p-1)^{1/2}}{n^{1/2}}
	\Bigg( 
		1
		+
		2^{\mathbf{v}/2} 
		\exp{\!
		\Big( 
			\big[
				\sqrt{T} + m(\mathbf{v}p)
			\big]^2
			L^2 \, T \mathbf{v}
		\Big)} \\
	&\qquad \cdot	
		\left[		
			\left[
				\int_{\R^d} 
					\Norm{x}^{\mathbf{v}p} \,
				\nu(dx)
			\right]^{ \nicefrac{1}{(\mathbf{v}p)}}
			+ 
			L \big(   T + m(\mathbf{v}p)\sqrt{T}   \big) 
		\right]^\mathbf{v}
	\Bigg).
\end{split}
\end{equation}
The proof of Proposition~\ref{quantitative} is thus completed.
\end{proof}

\begin{cor}
\label{quantitative_nicer}
Let $d, n \in \N$, $T\in (0,\infty)$, $ \varepsilon, c, L, C \in [0,\infty)$, $\mathbf{v}, p \in [2, \infty)$, 
	let $\langle \cdot, \cdot \rangle \colon \R^d \times \R^d \to \R$ be the $d$-dimensional Euclidean scalar product,
let $\left\| \cdot \right\| \colon \R^d \to [0,\infty)$ be the $d$-dimensional Euclidean norm, 
let $\HSNorm{\cdot} \colon \R^{d \times d} \to [0,\infty)$ be the Hilbert-Schmidt norm on $\R^{d\times d}$,
let $\nu \colon \mathcal{B}(\R^d) \to [0,1]$ be a probability measure,
assume that
\begin{equation}\label{quantitative_nicerConstant}
		C 
	=
	(p-1)^{1/2}
	\exp \!
	\big(
	3\mathbf{v} (1+ L^2 T (\sqrt{T} + \mathbf{v} p )^2 ) 
	\big) \allowbreak
	\big(
	1
	+
	\left[ \!
	\textint_{\R^d}  
	\Norm{x}^{ p \mathbf{v} } \,
	\nu(dx)
	\right]^{ \!\nicefrac{1}{ p }}
	\big),
\end{equation}
let $\varphi \colon \R^d \to \R$ be a continuous function,  
let $\phi \colon \R^d \to \R$ be a $\mathcal{B}(\R^d)/ \mathcal{B}(\R)$-measurable function which satisfies 
for all $x \in \R^d$ that
\begin{equation}\label{quantitative_nicer:growth}
	\left| 
		\phi(x)
	\right| 
\leq 
	c \, (1+\| x \|^{\mathbf{v}})
\qandq
	\left| 
		\varphi(x)-\phi(x)
	\right|
\leq 
	\varepsilon (1+ \| x \|^{\mathbf{v}}),
\end{equation}
and let
$\mu \colon \R^d \to \R^d$ and $\sigma \colon \R^d \to \R^{d \times d}$ be functions which satisfy
for all $x,y \in \R^d$, $\lambda \in \R$ that
\begin{equation}\label{quantitative_nicer:mu}
\mu(\lambda x+y) + \lambda \mu(0) = \lambda\mu(x)+\mu(y),
\end{equation}
\begin{equation}\label{quantitative_nicer:sigma}
\sigma(\lambda x+y) + \lambda \sigma(0) = \lambda\sigma(x)+\sigma(y),
\end{equation}
and
$
		\norm{\mu(x)} + \HSNorm{{\sigma(x)}} 
\leq  
	L  (1 + \norm{x})
$.
Then 
\begin{enumerate}[(i)]

\item \label{quantitative_nicer:item1}
there exists a unique continuous function $u \colon [0,T] \allowbreak\times \R^d \to \R$
which satisfies that
$
\inf_{q \in (0,\infty)} \allowbreak
\sup_{(t, x) \in [0, T] \times \R^d} \allowbreak
\frac{ | u(t, x) | }{ 1 + \norm{x}^q }
<
\infty
$,
which satisfies
for all $x \in \R^d$ that 
$
u(0,x) = \varphi(x)
$, 
and which satisfies that $u|_{(0,T) \times \R^d}$ is a viscosity solution of
\begin{equation}
\begin{split}
	(\tfrac{\partial }{\partial t}u)(t,x) 
&= 
	\tfrac{1}{2} 
	\operatorname{Trace}\! \big( 
		\sigma(x)[\sigma(x)]^{\ast}(\operatorname{Hess}_x u )(t,x)
	\big)  
	+
	\langle (\nabla_x u)(t,x),\mu(x)\rangle
\end{split}
\end{equation}
for $(t,x) \in (0,T) \times \R^d$ and

\item \label{quantitative_nicer:item2}
there exist $A_1,A_2,\ldots,A_{n} \in \R^{d \times d}$, $b_1,b_2,\ldots,b_{n} \in \R^d$ such that
\begin{equation}
\label{quantitative_nicer:concl1} 
\begin{split}
	&\left[
		\int_{\R^d}  
		\left|
			u(T,x) - \tfrac{1}{n} \big[   \smallsum_{i=1}^{n}\, \phi( A_i x + b_i)  \big] 
		\right|^p \,
		\nu(dx)
	\right]^{\nicefrac{1}{p}}\leq 
	(
	\varepsilon 
	+
	n^{-\nicefrac{1}{2}} \, c
	)
	C.
\end{split}
\end{equation}
\end{enumerate}
\end{cor}

\begin{proof}[Proof of Corollary~\ref{quantitative_nicer}]
Throughout this proof 
let $ r \in (0, \infty)$ be given by
$
	r = L \sqrt{T}(\sqrt{T} + \mathbf{v} p )
$.
Note that  \eqref{quantitative_nicer:growth} implies that $\varphi$ is an at most polynomially growing function. Combining this,
\eqref{quantitative_nicer:mu}, and \eqref{quantitative_nicer:sigma} with
item~\eqref{viscosity_affine_existence:item1} in Corollary~\ref{viscosity_affine_existence}
(with
$ d = d $,
$ m = d $,
$ T = T $,
$ \varphi = \varphi $,
$ \mu = \mu $,
$ \sigma = \sigma $
in the notation of Corollary~\ref{viscosity_affine_existence})
demonstrates that
there exists a unique continuous function $u\colon [0,T] \times \R^d \to \R$ which satisfies
for all $x \in \R^d$ that 
$
	u(0,x) = \varphi(x)
$, 
which satisfies that
$
	\inf_{q \in (0,\infty)} 
	\sup_{(t, x) \in [0, T] \times \R^d} 
	\frac{ | u(t, x) | }{ 1 + \norm{x}^q }
<
	\infty
$,
and which satisfies that $u|_{(0,T) \times \R^d}$ is a viscosity solution of
\begin{equation}
\begin{split}
	(\tfrac{\partial }{\partial t}u)(t,x) 
&= 
	\tfrac{1}{2} 
	\operatorname{Trace}\! \big( 
		\sigma(x)[\sigma(x)]^{\ast}(\operatorname{Hess}_x u )(t,x)
	\big)  
	+
	\langle (\nabla_x u)(t,x),\mu(x)\rangle
\end{split}
\end{equation}
for $(t,x) \in (0,T) \times \R^d$.
This proves item~\eqref{quantitative_nicer:item1}.
Furthermore, note that item~\eqref{quantitative:item2} in Proposition~\ref{quantitative} 
(with
$ d = d $,
$ n = n $,
$ p = p $,
$ T = T $,
$ c = c $,
$ \varepsilon = \varepsilon $,
$ L = L $,
$ \mathbf{v} =\mathbf{v} $,
$ \mathbf{w} =\mathbf{v} $,
$ \nu =\nu $,
$ \varphi = \varphi $,
$ \phi = \phi $,
$ \mu = \mu $,
$ \sigma = \sigma $
in the notation of Proposition~\ref{quantitative})
assures that 
there exist $A_1,A_2,\ldots,A_{n} \in \R^{d \times d}$, $b_1,b_2,\ldots,b_{n} \in \R^d$ such that
\begin{equation}
\label{quantitative_nicer:eq1} 
\begin{split}
	&\left[
		\int_{\R^d}  
		\left|
			u(T,x) - \tfrac{1}{n}\big[ \!  \smallsum_{i=1}^{n}\, \phi( A_i x + b_i)  \big] 
		\right|^p \,
		\nu(dx)
	\right]^{\nicefrac{1}{p}} \\
&\leq 
		\varepsilon 
	\Bigg(
		1 +
		2^{\mathbf{v}/2} 
		\exp{\!
		\left( 
			\big[
				\sqrt{T}+\max\{2, \mathbf{v} \}
			\big]^2
			L^2 \, T \mathbf{v}
		\right)} \\
		&\qquad \cdot
		\left[
			L 
			\big(
				T + \max\{2, \mathbf{v} \}\sqrt{T} 
			\big)
			+
			\Big[ \!
				\textint_{\R^d}  
					\Norm{x}^{\mathbf{v} p} \,
				\nu(dx)
			\Big]^{ \!\nicefrac{1}{(\mathbf{v}p)}}
		\right]^\mathbf{v}
	\Bigg) \\
&\quad 
	+
	n^{-\nicefrac{1}{2}} \, 
	4  \, c \, (p-1)^{1/2}	
	\Bigg( 
		1
		+
		2^{\mathbf{v}/2} 
		\exp{ \!
		\left( 
			\big[
				\sqrt{T} + \max\{2, \mathbf{v}p\}
			\big]^2
			L^2 \, T \mathbf{v}
		\right)} \\
	&\qquad \cdot	
		\left[
			L
			\big(
				T +  \max\{2, \mathbf{v}p\} \sqrt{T} 
			\big)		
			+
			\Big[
				\textint_{\R^d} 
					\Norm{x}^{\mathbf{v}p} \,
				\nu(dx)
			\Big]^{ \nicefrac{1}{(\mathbf{v}p)}}
		\right]^\mathbf{v} 
	\Bigg).
\end{split}
\end{equation}
Next note that the fact that $\mathbf{v}p \geq 2$ and the fact that $p \geq 1$ imply that
\begin{equation}
	\max\{2, \mathbf{v}\}
\leq
	\max\{2, \mathbf{v}p\}
=
	\mathbf{v}p.
\end{equation}
This demonstrates that
\begin{equation}
\label{quantitative_nicer:eq3}
\begin{split}
		&1 +
		2^{\mathbf{v}/2} 
		\exp{\!
		\left( 
			\big[
				\sqrt{T}+\max\{2, \mathbf{v} \}
			\big]^2
			L^2 \, T \mathbf{v}
		\right)} \\
		&\qquad \cdot
		\left[
			L 
			\big(
				T + \max\{2, \mathbf{v} \}\sqrt{T} 
			\big)
			+
			\Big[ \!
				\textint_{\R^d}  
					\Norm{x}^{\mathbf{v} p} \,
				\nu(dx)
			\Big]^{ \!\nicefrac{1}{(\mathbf{v}p)}}
		\right]^\mathbf{v} \\
&\leq
	1
	+
	2^{\mathbf{v}/2} 
	\exp{\!
	\left( 
		\big[
			\sqrt{T} + \max\{2, \mathbf{v}p\}
		\big]^2
		L^2 \, T \mathbf{v}
	\right)} \\
&\qquad \cdot	
	\left[
		L
		\big(
			T +  \max\{2, \mathbf{v}p\} \sqrt{T} 
		\big)		
		+
		\Big[
			\textint_{\R^d} 
				\Norm{x}^{\mathbf{v}p} \,
			\nu(dx)
		\Big]^{ \nicefrac{1}{(\mathbf{v}p)}}
	\right]^\mathbf{v} \\
&=
	1
	+
	2^{\mathbf{v}/2} 
	\exp{\!
	\left(
		L^2 \, T
		\big[
			\sqrt{T} + \mathbf{v}p
		\big]^2
		 \mathbf{v}
	\right)} \\
&\qquad \cdot	
	\left[
		L \sqrt{T}
		\big(
			\sqrt{T} +  \mathbf{v}p 
		\big)		
		+
		\Big[
			\textint_{\R^d} 
				\Norm{x}^{\mathbf{v}p} \,
			\nu(dx)
		\Big]^{ \nicefrac{1}{(\mathbf{v}p)}}
	\right]^\mathbf{v} \\
&=
	1 +
	2^{\mathbf{v}/2} 
	\exp{ \!
	\left( 
		r^2 \mathbf{v}
	\right)} 
	\left[
		r
		+
		\Big[ \!
			\textint_{\R^d}  
				\Norm{x}^{\mathbf{v} p} \,
			\nu(dx)
		\Big]^{ \!\nicefrac{1}{(\mathbf{v}p)}}
	\right]^\mathbf{v}.
\end{split}
\end{equation}
In addition, note that the fact that for all $x \in \R$ it holds that
$
	1 + x \leq \exp(x)
$
and the fact that for all $y \in (0, \infty)$ it holds that
$
	1 + y + y^2 \leq \frac{3}{2} ( 1 + y^2 )
$
ensure that
\begin{equation}
\label{quantitative_nicer:eq4}
\begin{split}
	&1 +
	2^{\mathbf{v}/2} 
	\exp{ \!
	\left( 
		r^2 \mathbf{v}
	\right)} 
	\left[
		r
		+
		\Big[ \!
			\textint_{\R^d}  
				\Norm{x}^{\mathbf{v} p} \,
			\nu(dx)
		\Big]^{ \!\nicefrac{1}{(\mathbf{v}p)}}
	\right]^\mathbf{v} \\
&\leq
	1 +
	\exp (\nicefrac{\mathbf{v}}{2}) 
	\exp{ \!
	\left( 
		r^2 \mathbf{v}
	\right)} 
	\left[
		(1 + r)
		\max 
		\left\{
			1
			,
			\Big[ \!
				\textint_{\R^d}  
					\Norm{x}^{\mathbf{v} p} \,
				\nu(dx)
			\Big]^{ \!\nicefrac{1}{(\mathbf{v}p)}}
		\right\}
	\right]^\mathbf{v} \\
	&\leq
	1 +
	\exp (\nicefrac{\mathbf{v}}{2}) 
	\exp{ \!
		\left( 
		r^2 \mathbf{v}
		\right)} 
	\left[
	\exp(r)
	\max 
	\left\{
	1
	,
	\Big[ \!
	\textint_{\R^d}  
	\Norm{x}^{\mathbf{v} p} \,
	\nu(dx)
	\Big]^{ \!\nicefrac{1}{(\mathbf{v}p)}}
	\right\}
	\right]^\mathbf{v} \\
&\leq
	2
	\exp{ \!
	\left( 
		\mathbf{v}
		+
		r^2 \mathbf{v}
	\right)} 
	\exp(r \mathbf{v})
	\max 
	\left\{
		1
		,
		\Big[ \!
			\textint_{\R^d}  
				\Norm{x}^{\mathbf{v} p} \,
			\nu(dx)
		\Big]^{ \!\nicefrac{1}{p}}
	\right\}\\
&\leq
	2\exp{ \!
	\left( 
		(1 + r + r^2)\mathbf{v}
	\right)} 
	\left(
		1
		+
		\Big[ \!
			\textint_{\R^d}  
				\Norm{x}^{\mathbf{v} p} \,
			\nu(dx)
		\Big]^{ \!\nicefrac{1}{p}}
	\right) \\
&\leq
	2\exp{ \!
	\left( 
		\tfrac{3}{2}(1 + r^2)\mathbf{v}
	\right)} 
	\left(
		1
		+
		\Big[ \!
			\textint_{\R^d}  
				\Norm{x}^{\mathbf{v} p} \,
			\nu(dx)
		\Big]^{ \!\nicefrac{1}{p}}
	\right).
\end{split}
\end{equation}
Moreover, note that the fact that $\mathbf{v} \geq 2$ implies that $8=2^3\leq \exp(3)\le \exp
\big(
\tfrac{3}{2}\mathbf{v}
\big)\le 	\exp
\big(
\tfrac{3}{2}(1 + r^2)\mathbf{v}
\big).$ Hence, we obtain that
\begin{equation}
\begin{split}
&8\exp{ \!
	\left( 
	\tfrac{3}{2}(1 + r^2)\mathbf{v}
	\right)}
\leq
\exp\!
\left(
\tfrac{3}{2}(1 + r^2)\mathbf{v}
\right)
\exp{ \!
	\left( 
	\tfrac{3}{2}(1 + r^2)\mathbf{v}
	\right)} 
=
\exp \!
\left(
3(1 + r^2)\mathbf{v}
\right).
\end{split}
\end{equation}
Combining this, \eqref{quantitative_nicerConstant}, \eqref{quantitative_nicer:eq1}, \eqref{quantitative_nicer:eq3}, and \eqref{quantitative_nicer:eq4}
establishes that 
\begin{equation}
\begin{split}
	&\left[
		\int_{\R^d}  
		\left|
			u(T,x) - \tfrac{1}{n}\big[  \smallsum_{i=1}^{n}\, \phi( A_i x + b_i)  \big] 
		\right|^p \,
		\nu(dx)
	\right]^{\nicefrac{1}{p}} \\ 
&\leq 
	\varepsilon 
	\Bigg(
	1 +
		2^{\mathbf{v}/2} 
		\exp{ \!
		\left( 
			r^2 \mathbf{v}
		\right)} 
		\left[
			r
			+
			\Big[ \!
				\textint_{\R^d}  
					\Norm{x}^{\mathbf{v} p} \,
				\nu(dx)
			\Big]^{ \!\nicefrac{1}{(\mathbf{v}p)}}
		\right]^\mathbf{v} 
	\Bigg)\\
&\quad 
	+
	n^{-\nicefrac{1}{2}} \, 
	4  \, c \, (p-1)^{1/2}	
	\Bigg( 
		1 +
		2^{\mathbf{v}/2} 
		\exp{ \!
		\left( 
			r^2 \mathbf{v}
		\right)} 
		\left[
			r
			+
			\Big[ \!
				\textint_{\R^d}  
					\Norm{x}^{\mathbf{v} p} \,
				\nu(dx)
			\Big]^{ \!\nicefrac{1}{(\mathbf{v}p)}}
		\right]^\mathbf{v}
	\Bigg) \\
&\leq
	\left[\varepsilon + n^{-\nicefrac{1}{2}}\,4  \, c \, (p-1)^{1/2}\right]
	\, 2\exp{ \!
	\left( 
		\tfrac{3}{2}(1 + r^2)\mathbf{v}
	\right)} 
	\left(
		1
		+
		\Big[ \!
			\textint_{\R^d}  
				\Norm{x}^{\mathbf{v} p} \,
			\nu(dx)
		\Big]^{ \!\nicefrac{1}{p}}
	\right) \\
&\leq
	(\varepsilon + n^{-\nicefrac{1}{2}}\, c)  (p-1)^{1/2}\,
	 8 \exp{ \!
	\left( 
		\tfrac{3}{2}(1 + r^2)\mathbf{v}
	\right)} 
	\left(
		1
		+
		\Big[ \!
			\textint_{\R^d}  
				\Norm{x}^{\mathbf{v} p} \,
			\nu(dx)
		\Big]^{ \!\nicefrac{1}{p}}
	\right) \\
&\leq
	(\varepsilon + n^{-\nicefrac{1}{2}}\, c) \,
	 (p-1)^{1/2}  \exp{ \!
	\left( 
		3(1 + r^2)\mathbf{v}
	\right)} 
	\left(
		1
		+
		\Big[ \!
			\textint_{\R^d}  
				\Norm{x}^{\mathbf{v} p} \,
			\nu(dx)
		\Big]^{ \!\nicefrac{1}{p}}
	\right) \\
&=
	(\varepsilon  + n^{-\nicefrac{1}{2}}\,c)
	C.
\end{split}
\end{equation}
The proof of Corollary~\ref{quantitative_nicer} is thus completed.
\end{proof}

\subsection{Cost estimates}

\begin{prop}
\label{qualitative}
Let $d \in \N$, $T, \varepsilon \in (0,\infty)$, $c, L, C \in [0,\infty)$, $\mathbf{v}, p \in [2, \infty)$, 
$
	n \in \N \cap [ c^2 C^2 \varepsilon^{-2}, \infty)
$, 
	let $\langle \cdot, \cdot \rangle \colon \R^d \times \R^d \to \R$ be the $d$-dimensional Euclidean scalar product,
let $\left\| \cdot \right\| \colon \R^d \to [0,\infty)$ be the $d$-dimensional Euclidean norm, 
let $\HSNorm{\cdot} \colon \R^{d \times d} \to [0,\infty)$ be the Hilbert-Schmidt norm on $\R^{d\times d}$,
let $\nu \colon \mathcal{B}(\R^d) \to [0,1]$ be a probability measure,
assume that
\begin{equation}
		C 
	=
	2 (p-1)^{1/2}
	\exp \!
	\big(
	3\mathbf{v} (1+ L^2 T (\sqrt{T} + \mathbf{v} p )^2 ) 
	\big) \allowbreak
	\big(
	1
	+
	\left[ \!
	\textint_{\R^d}  
	\Norm{x}^{ p \mathbf{v} } \,
	\nu(dx)
	\right]^{ \!\nicefrac{1}{ p }}
	\big),
\end{equation}
let 
$\varphi \colon \R^d \to \R$ be a continuous function, 
let $\phi \colon \R^d \to \R$ be a $\mathcal{B}(\R^d)/ \mathcal{B}(\R)$-measurable function which satisfies for all $x\in\R^d$ that
\begin{equation}\label{qualitative:growth}
	\left| 
		\phi(x)
	\right| 
\leq 
	c  (1+\| x \|^{\mathbf{v}})
\qandq
	\left| 
		\varphi(x)-\phi(x)
	\right| 
\leq 
	C^{-1} \varepsilon  (1+ \| x \|^{\mathbf{v}})
\end{equation}
and let
$\mu \colon \R^d \to \R^d$ and
$\sigma \colon \R^d \to \R^{d \times d}$
be functions which satisfy 
for all $x,y \in \R^d$, $\lambda \in \R$ that
\begin{equation}\label{qualitative:mu}
\mu(\lambda x+y) + \lambda \mu(0) = \lambda\mu(x)+\mu(y),
\end{equation}
\begin{equation}\label{qualitative:sigma}
\sigma(\lambda x+y) + \lambda \sigma(0) = \lambda\sigma(x)+\sigma(y),
\end{equation}
and
$
		\norm{\mu(x)} + \HSNorm{{\sigma(x)}} 
\leq  
	L  (1 + \norm{x}).
$
Then
\begin{enumerate}[(i)]
\item \label{qualitative:item1}
there exists a unique continuous function $u \colon [0,T]\allowbreak \times \R^d \to \R$
which satisfies that
$
	\inf_{q \in (0,\infty)} \allowbreak
	\sup_{(t, x) \in [0, T] \times \R^d} \allowbreak
	\frac{ | u(t, x) | }{ 1 + \norm{x}^q }
<
	\infty
$, which satisfies
for all $x \in \R^d$ that 
$u(0,x) = \varphi(x)$, 
and which satisfies that $u|_{(0,T) \times \R^d}$ is a viscosity solution of
\begin{equation}
\begin{split}
	(\tfrac{\partial }{\partial t}u)(t,x) 
&= 
	\tfrac{1}{2} 
	\operatorname{Trace}\! \big( 
		\sigma(x)[\sigma(x)]^{\ast}(\operatorname{Hess}_x u )(t,x)
	\big)  
	+
	\langle (\nabla_x u)(t,x),\mu(x)\rangle
\end{split}
\end{equation}
for $(t,x) \in (0,T) \times \R^d$
and

\item \label{qualitative:item2}
there exist $A_1,A_2,\ldots,A_{n} \in \R^{d \times d}$, $b_1,b_2,\ldots,b_{n} \in \R^d$ such that
\begin{equation}
\label{qualitative:concl1} 
	\left[
		\int_{\R^d}  
		\left|
			u(T,x) - \tfrac{1}{n}\big[  \smallsum_{i=1}^{n}  \phi( A_i x + b_i)  \big] 
		\right|^p \,
		\nu(dx)
	\right]^{\nicefrac{1}{p}} 
\leq
		\varepsilon.
\end{equation}
\end{enumerate}
\end{prop}

\begin{proof}[Proof of Proposition~\ref{qualitative}]
Note that  \eqref{qualitative:growth} implies that $\varphi$ is an at most polynomially growing function. Combining this,
\eqref{qualitative:mu}, and \eqref{qualitative:sigma} with item~\eqref{viscosity_affine_existence:item1} in Corollary~\ref{viscosity_affine_existence}
(with
$ d = d $,
$ m = d $,
$ T = T $,
$ \varphi = \varphi $,
$ \mu = \mu $,
$ \sigma = \sigma $
in the notation of Corollary~\ref{viscosity_affine_existence})
establishes that
there exists a unique continuous function $u\colon [0,T] \times \R^d \to \R$ which satisfies
for all $x \in \R^d$ that 
$u(0,x) = \varphi(x)$, 
which satisfies that
$
	\inf_{q \in (0,\infty)} 
	\sup_{(t, x) \in [0, T] \times \R^d} 
	\frac{ | u(t, x) | }{ 1 + \norm{x}^q }
<
	\infty
$,
and which satisfies that $u|_{(0,T) \times \R^d}$ is a viscosity solution of
\begin{equation}
\begin{split}
	(\tfrac{\partial }{\partial t}u)(t,x) 
&= 
	\tfrac{1}{2} 
	\operatorname{Trace}\! \big( 
		\sigma(x)[\sigma(x)]^{\ast}(\operatorname{Hess}_x u )(t,x)
	\big)  
	+
	\langle (\nabla_x u)(t,x),\mu(x)\rangle
\end{split}
\end{equation}
for $(t,x) \in (0,T) \times \R^d$.
This proves item~\eqref{qualitative:item1}.
Next note that Corollary~\ref{quantitative_nicer} 
(with 
$d = d$, 
$n = n$, 
$T = T$, 
$\varepsilon = C^{-1} \varepsilon $, 
$c = c$, 
$L = L$, 
$\mathbf{v} = \mathbf{v}$, 
$p = p$,
$ \nu = \nu $,
$ \phi = \phi $,
$ \mu = \mu $,
$ \sigma = \sigma $
in the notation of Corollary~\ref{quantitative_nicer})
assures that
there exist $A_1,A_2,\ldots,A_{n} \in \R^{d \times d}$, $b_1,b_2,\ldots,b_{n} \in \R^d$ such that
\begin{equation}
\begin{split}
	&\left[
		\int_{\R^d}  
		\left|
			u(T,x) - \tfrac{1}{n}\big[ \smallsum_{i=1}^{n}\, \phi( A_i x + b_i)   \big] 
		\right|^p \,
		\nu(dx)
	\right]^{\nicefrac{1}{p}} \leq
	\left(
	C^{-1} \varepsilon 
	+
	n^{-\nicefrac{1}{2}} \,c
	\right)
	\frac{C}{2}.
\end{split}
\end{equation}
The hypothesis that 
$
	n \geq  c^2 C^2 \varepsilon^{-2}
$
hence assures that
\begin{equation}
\begin{split}
	&\left[
		\int_{\R^d}  
		\left|
			u(T,x) - \tfrac{1}{n} \big[ \smallsum_{i=1}^{n}\, \phi( A_i x + b_i) \big] 
		\right|^p \,
		\nu(dx)
	\right]^{\nicefrac{1}{p}} \\
&\leq
	\left(
	C^{-1} \varepsilon 
	+
	( c^2 C^2 \varepsilon^{-2})^{-\nicefrac{1}{2}} \,c
	\right) \frac{C}{2}\\
&=
	\left(
	C^{-1} \varepsilon 
	+
	C^{-1} \varepsilon 
	\right)\frac{C}{2} 
=
	\varepsilon.
\end{split}
\end{equation}
This establishes item~\eqref{qualitative:item2}.
The proof of Proposition~\ref{qualitative} is thus completed.
\end{proof}

\subsection[Representation properties for ANNs]{Representation properties for artificial neural networks}

\begin{setting}
\label{setting_NN}
For every $l \in \N$ let $\mathcal{M}_l$ be the set of all Borel measurable functions from $\R^l$ to $\R$,
let 
\begin{equation}
\begin{split}
	\mathcal{N}
&=
	\cup_{\mathcal{L} \in \{2, 3, \ldots \}}
	\cup_{ (l_0,l_1,\ldots, l_\mathcal{L}) \in ((\N^{\mathcal{L}}) \times \{ 1 \} ) }
		\left(
			\times_{k = 1}^\mathcal{L} (\R^{l_k \times l_{k-1}} \times \R^{l_k})
		\right),
\end{split}
\end{equation}
let $\mathbf{A}_l\colon \R^l\to\R^l$, $l \in \N$, and $\mathbf{a} \in \mathcal{M}_1$ be functions which satisfy 
for all $l \in \N$, $x = (x_1,x_2, \ldots, x_l) \in \R^l$ that
\begin{equation}
\label{setting_NN:ass1}
	\mathbf{A}_l(x)
=
	(\mathbf{a}(x_1), \mathbf{a}(x_2), \ldots, \mathbf{a}(x_l)),
\end{equation} 
and let 
$
	\mathcal{P}, \mathscr{P} \colon \mathcal{N} \to \N
$ and 
$
	\mathcal{R} \colon \mathcal{N} \to \cup_{l = 1}^\infty \mathcal{M}_l
$
be the functions which satisfy
for all $ \mathcal{L} \in \{2, 3, \ldots \}$, $ (l_0,l_1,\ldots, l_\mathcal{L}) \in ((\N^{\mathcal{L}}) \times \{ 1 \}) $, 
$
\Phi 
=
((W_1, B_1), \ldots, (W_\mathcal{L},\allowbreak B_\mathcal{L}))
\allowbreak=
( 
(W_k^{(i,j)})_{ i \in \{1, 2, \ldots, l_k \}, j \in \{1, 2, \ldots, l_{k-1} \}}, \allowbreak
(B_k^{(i)})_{i \in \{1, 2, \ldots, l_k \}} 
)_{k \in \{1, 2, \ldots, \mathcal{L} \} } 
\in  \allowbreak
( \times_{k = 1}^\mathcal{L} \allowbreak(\R^{l_k \times l_{k-1}}\allowbreak \times \R^{l_k}))$,
$x_0 \in \R^{l_0}, x_1 \in \R^{l_1}, \ldots, x_{\mathcal{L}-1} \in \R^{l_{\mathcal{L}-1}}$ 
with $\forall \, k \in \N \cap (0,\mathcal{L}) \colon x_k = \mathbf{A}_{l_k}(W_k x_{k-1} + B_k)$  
that
\begin{equation}
\label{setting_NN:ass2}
	\mathcal{R}(\Phi) \in \mathcal{M}_{l_0},
\qquad
	( \mathcal{R}(\Phi) ) (x_0) = W_\mathcal{L} x_{\mathcal{L}-1} + B_\mathcal{L},
\end{equation}
\begin{equation}
	\mathscr{P}(\Phi) 
=
	\sum_{k = 1}^\mathcal{L} 
	\sum_{i = 1}^{l_k}
	\left(
		\mathbbm{1}_{\R \backslash \{ 0 \}} (B_k^{(i)})
		+
		\smallsum\limits_{j = 1}^{l_{k-1}}
				\mathbbm{1}_{\R \backslash \{ 0 \}} (W_k^{(i,j)})
	\right),
\end{equation} 
and 
$
	\mathcal{P}(\Phi)
=
	\sum_{k = 1}^\mathcal{L} l_k(l_{k-1} + 1) 
$.

\end{setting}

\begin{lemma}
\label{multichannel_network}
Assume Setting~\ref{setting_NN} and
let $d, n \in \N$, $A_1,A_2,\ldots,A_{n} \in \R^{d \times d}$, $b_1,b_2,\ldots,b_{n} \in \R^d$, $\phi \in \mathcal{N}$ satisfy that 
$
	\mathcal{R}(\phi) \in \mathcal{M}_d
$.
Then
there exists $\psi \in \mathcal{N}$ which satisfies that 
$ \mathcal{P}(\psi) \leq n^2 \, \mathcal{P}(\phi) $, 
$ \mathscr{P}(\psi) \leq n \, \mathcal{P}(\phi) $, and
$ \mathcal{R}(\psi) \in \mathcal{M}_d $
and which satisfies for all $x \in \R^d$ that 
\begin{equation}
	(\mathcal{R}(\psi)) (x)
=
	\tfrac{1}{n}\big[  \smallsum_{i=1}^{n} (\mathcal{R}(\phi))( A_i x + b_i)  \big]. 
\end{equation}
\end{lemma}

\begin{proof}[Proof of Lemma~\ref{multichannel_network}]
Throughout this proof let $x \in \R^d$, 
for all $\mathcal{L} \in \{2, 3, \ldots \}$, 
$ (l_0,l_1,\ldots, l_\mathcal{L}) \in ((\N^{\mathcal{L}}) \times \{ 1 \}) $, 
$\Phi \in  (\times_{k = 1}^\mathcal{L} (\R^{l_k \times l_{k-1}} \times \R^{l_k}))$
let $W^{(\Phi )}_1 \in \R^{l_1 \times l_0}$, $W^{(\Phi )}_2 \in \R^{l_2 \times l_1}, \ldots,$   $W^{(\Phi )}_\mathcal{L} \in \R^{l_\mathcal{L} \times l_{\mathcal{L}-1}}$, 
$B^{(\Phi )}_1 \in \R^{l_1}$, $B^{(\Phi )}_2 \in \R^{l_2 }, \ldots$,   $B^{(\Phi )}_\mathcal{L} \in \R^{l_\mathcal{L}}$ 
satisfy that
\begin{equation}
\label{multichannel_network:setting0}
	\Phi
= 
	\big(
		(W^{(\Phi )}_1, B^{(\Phi )}_1), 
		(W^{(\Phi )}_2, B^{(\Phi )}_2), 
		\ldots, 
		(W^{(\Phi )}_\mathcal{L}, B^{(\Phi )}_\mathcal{L})
	\big),
\end{equation}
let $N \in \{2, 3, \ldots \}$, $(u_0, u_1, \ldots, u_N) \in (\{ d\} \times (\N^{N-1}) \times \{ 1 \})$ satisfy that 
$
	\phi \in  \times_{k = 1}^N (\R^{u_k \times u_{k-1}} \times \R^{u_k})
$
(i.e., $\phi$ corresponds to a fully connected feedforward artificial neural network with $N + 1$ layers with dimensions $(u_0, u_1, \ldots, u_N)$),
let
$
	\psi 
\in  
	(\R^{(n u_1) \times  u_0} \times \R^{n u_1}) 
	\times 
	(\times_{k = 2}^{N-1} (\R^{(n u_k) \times (n u_{k-1})} \times \R^{n u_k}))
	\times
	(\R^{u_N \times  (n u_{N-1})} \times \R^{u_N}) 
\subseteq 
	\mathcal{N}
$
(i.e., $\psi$ corresponds to a fully connected feedforward artificial neural network with $N + 1$ layers with dimensions $(u_0, nu_1, n u_2, \ldots, n u_{N-1}, u_N)$)
satisfy for all $k \in \{2, 3, \ldots, N-1 \}$ that
\begin{equation}
\label{multichannel_network:setting11}
	W^{(\psi)}_1
=
	\begin{pmatrix}
	 	W^{(\phi)}_1 A_1 \\
	 	W^{(\phi)}_1 A_2 \\
	 	\vdots \\
	 	W^{(\phi)}_1 A_n
	\end{pmatrix}
\in  
	\R^{(n u_1) \times u_0}
=
	\R^{(n u_1) \times d},
\end{equation}
\begin{equation}
\label{multichannel_network:setting12}
	B^{(\psi)}_1
=
	 \begin{pmatrix}
	 	W^{(\phi)}_1 b_1 + B^{(\phi)}_1\\
	 	W^{(\phi)}_1 b_2 + B^{(\phi)}_1 \\
	 	\vdots \\
	 	W^{(\phi)}_1 b_n + B^{(\phi)}_1
	 \end{pmatrix}
\in  
	\R^{n u_1},
\qquad 
	B^{(\psi)}_k
=
	 \begin{pmatrix}
	 	B^{(\phi)}_k\\
	 	B^{(\phi)}_k \\
	 	\vdots \\
	 	B^{(\phi)}_k
	 \end{pmatrix}
\in  
	\R^{n u_k},
\end{equation}
\begin{equation}
\label{multichannel_network:setting2}
	W^{(\psi)}_k
=
	\begin{pmatrix}
	 	W^{(\phi)}_k& 	0&						\cdots& 			0\\
	 	0&						W^{(\phi)}_k&  	\cdots&			\vdots\\
	 	\vdots& 				\vdots&				\ddots&			0\\
	 	0&	 					\cdots&				0&					W^{(\phi)}_k
	\end{pmatrix}
\in  
	\R^{(n u_k) \times (n u_{k-1})},
\end{equation}
\begin{equation}
\label{multichannel_network:setting3}
	W^{(\psi)}_N
= 
	\begin{pmatrix}
		\tfrac{1}{n}  W^{(\phi)}_N&
		\tfrac{1}{n}  W^{(\phi)}_N&
		\cdots&
		\tfrac{1}{n}  W^{(\phi)}_N
	\end{pmatrix}
\in 
	\R^{u_N \times (n u_{N-1})} =\R^{1 \times (n u_{N-1})},
\end{equation}
\begin{equation}
\label{multichannel_network:setting4}
\andq
	B^{(\psi)}_N
=
	B^{(\phi)}_N
\in
	\R^{u_N} = \R,
\end{equation}
let $y_{i, k} \in \R^{u_k}$, $i \in \{1, 2, \ldots, n \}$, $k \in \{0, 1, \ldots, N \}$, satisfy 
for all $i \in \{1, 2, \ldots, n \}$, $k \in \{1, 2, \ldots, N -1\}$  that
\begin{equation}
\label{multichannel_network:setting6}
	y_{i, 0} = A_i x + b_i,
\qquad
	y_{i, k} = \mathbf{A}_{u_k}(W^{(\phi)}_k y_{i, k-1} + B^{(\phi)}_k), 
\end{equation}
\begin{equation}
\label{multichannel_network:setting7}
\andq
	y_{i, N} = W^{(\phi)}_N y_{i, N-1} + B^{(\phi)}_N, 
\end{equation}
and let $z_0 \in \R^{u_0}, z_1 \in \R^{n u_1}, z_2 \in \R^{n u_2}, \ldots, z_{N-1} \in \R^{n u_{N-1}}, z_N \in \R^{u_N}$ satisfy that
\begin{equation}
\label{multichannel_network:setting8}
	z_0 = x,
\qquad
	z_k = \mathbf{A}_{n u_k}(W^{(\psi)}_k z_{k-1} + B^{(\psi)}_k),
\end{equation}
\begin{equation}
\label{multichannel_network:setting9}
\andq
	z_N = W^{(\psi)}_N z_{N-1} + B^{(\psi)}_N.
\end{equation}
Observe that \eqref{setting_NN:ass1} proves that for all $l,L\in\N$, $v=(v_1,v_2,\ldots, v_{L l})\in \R^{(L l)}$ it holds that 
\begin{equation}\label{multichannel_network:propertyActivationFunction}
\begin{split}
\mathbf{A}_{L l}(v)&=\big(\mathbf{a}(v_1),\mathbf{a}(v_2),\ldots, \mathbf{a}(v_{L l})\big)\\
&=\big(\mathbf{A}_l(v_1,v_2,\ldots, v_l), \mathbf{A}_l(v_{l+1},v_{l+2},\ldots, v_{2l}),\ldots,\\ &\hspace{5cm}\mathbf{A}_l(v_{(L-1)l+1},v_{(L-1)l+2},\ldots, v_{Ll})\big).
\end{split}
\end{equation} 
Furthermore, note that the fact that $u_0 = d$ ensures that 
$
	\mathcal{R}(\psi) \in \mathcal{M}_{u_0} = \mathcal{M}_d
$.
Next observe that \eqref{setting_NN:ass2}, \eqref{multichannel_network:setting0}, \eqref{multichannel_network:setting6}, and \eqref{multichannel_network:setting7} imply that 
for all $i \in \{1, 2, \ldots, n \}$ it holds that
\begin{equation}
\label{multichannel_network:eq01}
 y_{i, N} = (\mathcal{R}(\phi))( y_{i, 0}) = (\mathcal{R}(\phi))(A_i x + b_i).
\end{equation}
Moreover, observe that \eqref{setting_NN:ass2}, \eqref{multichannel_network:setting0}, \eqref{multichannel_network:setting8}, and \eqref{multichannel_network:setting9} ensure that 
\begin{equation}
\label{multichannel_network:eq02}
z_N = (\mathcal{R}(\psi))( z_0 ) =  (\mathcal{R}(\psi))( x ).
\end{equation}
Next we claim that for all $k \in \{1,2, \ldots, N-1 \}$ it holds that
\begin{equation}
\label{multichannel_network:eq03}
	z_k = \big(y_{1, k}, y_{2, k}, \ldots, y_{n, k} \big).
\end{equation} 
We now prove \eqref{multichannel_network:eq03} by induction on $k \in \{1,2, \ldots, N-1 \}$. 
For the base case $k = 1$ note that \eqref{multichannel_network:setting6} assures that 
for all $i \in \{1, 2, \ldots, n \}$ it holds that
\begin{equation}
\begin{split}
	y_{i, 1} 
&= 
	\mathbf{A}_{u_1}(W^{(\phi)}_1 y_{i, 0} + B^{(\phi)}_1)\\
&=
	\mathbf{A}_{u_1}(W^{(\phi)}_1 (A_i x + b_i) + B^{(\phi)}_1) \\
&=
	\mathbf{A}_{u_1}(W^{(\phi)}_1 A_i x + W^{(\phi)}_1 b_i + B^{(\phi)}_1). 
\end{split}
\end{equation}
This,  \eqref{multichannel_network:setting11}, \eqref{multichannel_network:setting12},  \eqref{multichannel_network:setting8}, and \eqref{multichannel_network:propertyActivationFunction} demonstrate that
\begin{equation}
\begin{split}
	z_1 
&= 
	\mathbf{A}_{n u_1}(W^{(\psi)}_1 x + B^{(\psi)}_1) \\
&=
	\mathbf{A}_{n u_1}
	\left(
		\begin{pmatrix}
			W^{(\phi)}_1 A_1 x \\
			W^{(\phi)}_1 A_2 x  \\
			\vdots \\
			W^{(\phi)}_1 A_n x 
		\end{pmatrix}
		+
		\begin{pmatrix}
			W^{(\phi)}_1 b_1 + B^{(\phi)}_1 \\
			W^{(\phi)}_1 b_2 + B^{(\phi)}_1 \\
			\vdots \\
			W^{(\phi)}_1 b_n + B^{(\phi)}_1 
		\end{pmatrix}
	\right) \\
&=
		\begin{pmatrix}
			\mathbf{A}_{u_1} ( W^{(\phi)}_1 A_1 x + W^{(\phi)}_1 b_1 + B^{(\phi)}_1) \\
			\mathbf{A}_{u_1} ( W^{(\phi)}_1 A_2 x + W^{(\phi)}_1 b_2 + B^{(\phi)}_1) \\
			\vdots \\
			\mathbf{A}_{u_1} ( W^{(\phi)}_1 A_n x + W^{(\phi)}_1 b_n + B^{(\phi)}_1)
		\end{pmatrix} 
=
	\begin{pmatrix}
		y_{1, 1} \\
		y_{2, 1} \\
		\vdots      \\
		y_{n, 1}
	\end{pmatrix}.
\end{split}
\end{equation}
This establishes \eqref{multichannel_network:eq03} in the base case $k = 1$.
For the induction step $\{1,2,\ldots, N-2 \} \ni k-1 \to k \in \{2, 3 \ldots, N -1 \}$ observe that \eqref{multichannel_network:setting12}, \eqref{multichannel_network:setting2}, \eqref{multichannel_network:setting6},  \eqref{multichannel_network:setting8}, and \eqref{multichannel_network:propertyActivationFunction} imply that
for all $k \in \{2, 3 \ldots, N \}$ with $z_{k-1} = (y_{1, k-1}, y_{2, k-1}, \ldots, y_{n, k-1} )$ it holds that
\begin{equation}
\begin{split}
	z_k
&=
	\mathbf{A}_{n u_k}
	\left(
		W^{(\psi)}_k 
		\begin{pmatrix}
			y_{1, k-1} \\
			y_{2, k-1} \\
			\vdots				\\
			y_{n, k-1}
		\end{pmatrix}
		 + B^{(\psi)}_k
	\right) \\
&=
	\mathbf{A}_{n u_k}
	\left(
		\begin{pmatrix}
			W^{(\phi)}_k y_{1, k-1} \\
			W^{(\phi)}_k y_{2, k-1} \\
			\vdots				\\
			W^{(\phi)}_k y_{n, k-1}
		\end{pmatrix}
		 + 
		 \begin{pmatrix}
			B^{(\phi)}_k \\
			B^{(\phi)}_k \\
			\vdots				\\
			B^{(\phi)}_k
		\end{pmatrix}
	\right) \\
&=
	\begin{pmatrix}
		\mathbf{A}_{u_k}( W^{(\phi)}_k y_{1, k-1} + B^{(\phi)}_k) \\
		\mathbf{A}_{u_k}( W^{(\phi)}_k y_{2, k-1} + B^{(\phi)}_k) \\
		\vdots				\\
		\mathbf{A}_{u_k}( W^{(\phi)}_k y_{n, k-1} + B^{(\phi)}_k)
	\end{pmatrix}
=
	\begin{pmatrix}
		y_{1, k} \\
		y_{2, k} \\
		\vdots      \\
		y_{n, k}
	\end{pmatrix}.
\end{split}
\end{equation}
Induction thus proves \eqref{multichannel_network:eq03}.
Next note that \eqref{multichannel_network:setting3}, \eqref{multichannel_network:setting4}, \eqref{multichannel_network:setting7}, \eqref{multichannel_network:setting9}, and \eqref{multichannel_network:eq03} demonstrate that
\begin{equation}
\begin{split}
	z_N 
&= 
	W^{(\psi)}_N z_{N-1} + B^{(\psi)}_N
=
		\begin{pmatrix}
	\tfrac{1}{n}  W^{(\phi)}_N&
	\tfrac{1}{n}  W^{(\phi)}_N&
	\cdots&
	\tfrac{1}{n}  W^{(\phi)}_N
	\end{pmatrix}
	\begin{pmatrix}
		y_{1, N-1} \\
		y_{2, N-1} \\
		\vdots      \\
		y_{n, N-1}
	\end{pmatrix}
	+
	B^{(\phi)}_N \\
&=
	\tfrac{1}{n}
	\big[
		\smallsum_{i = 1}^n W^{(\phi)}_N y_{i, N-1}
	\big]
	+
	B^{(\phi)}_N
=
	\tfrac{1}{n}
	\big[
		\smallsum_{i = 1}^n W^{(\phi)}_N y_{i, N-1}
		+
		B^{(\phi)}_N
	\big] \\
&=
	\tfrac{1}{n}
	\big[
		\smallsum_{i = 1}^n y_{i, N}
	\big].
\end{split}
\end{equation}
Combining this with  \eqref{multichannel_network:eq01} and \eqref{multichannel_network:eq02} establishes that
\begin{equation}
\label{multichannel_network:eq04}
	(\mathcal{R}(\psi))(x)
=
	z_N
=
	\tfrac{1}{n}
	\big[
		\smallsum_{i = 1}^n y_{i, N}
	\big]
=
	\tfrac{1}{n}
	\big[
		\smallsum_{i = 1}^n  (\mathcal{R}(\phi))(A_i x + b_i)
	\big].
\end{equation}
In addition, observe that the fact that
$
	\mathcal{P}(\phi) 
=
	\smallsum_{k = 1}^N u_k (u_{k-1} + 1)
$
assures that 
\begin{equation}
\label{multichannel_network:eq05}
\begin{split}
	\mathcal{P}(\psi) 
&=
	n u_1(u_0 + 1) 
	+
	\left[
		\smallsum\limits_{k = 2}^{N-1} n u_k(n u_{k-1} + 1) 
	\right]
	+
	u_N (n u_{N-1} + 1) \\
&\leq n^2 u_1(u_0 + 1) +	\left[
\smallsum\limits_{k = 2}^{N-1} n^2 u_k(u_{k-1} + 1) 
\right]
+
n^2 u_N ( u_{N-1} + 1) \\
&\leq
	n^2 \left[ \smallsum\limits_{k = 1}^N u_k (u_{k-1} + 1) \right]
=
	n^2 \, \mathcal{P}(\phi).
\end{split}
\end{equation}
Furthermore, note that \eqref{multichannel_network:setting11} -- \eqref{multichannel_network:setting4} and \eqref{multichannel_network:eq05} demonstrate that
\begin{equation}
\begin{split}
	\mathscr{P}(\psi) 
&\leq
	n u_1u_0  + n u_1
	+
	\left[ \smallsum\limits_{k = 2}^{N-1} n u_k u_{k-1} + n u_k \right] 
	+
	u_N n u_{N-1} 
	+
	u_N \\
&\leq
	n \left[ \smallsum\limits_{k = 1}^N u_k (u_{k-1} + 1) \right]
=
	n \, \mathcal{P}(\phi).
\end{split}
\end{equation}
Combining this, \eqref{multichannel_network:eq04}, and \eqref{multichannel_network:eq05} completes the proof of Lemma~\ref{multichannel_network}.
\end{proof}

\subsection[Cost estimates for ANNs]{Cost estimates for artificial neural networks}

\begin{samepage}
\begin{lemma}
\label{NN_approx_single}
Assume Setting~\ref{setting_NN},
let $d \in \N$, $T, \varepsilon \in (0,\infty)$, $ L, C, \mathbf{C} \in [0,\infty)$, $c \in [1,\infty)$, $\mathbf{v}, p \in [2, \infty)$, 
	let $\langle \cdot, \cdot \rangle \colon \R^d \times \R^d \to \R$ be the $d$-dimensional Euclidean scalar product,
let $\left\| \cdot \right\| \colon \R^d \to [0,\infty)$ be the $d$-dimensional Euclidean norm, 
let $\HSNorm{\cdot} \colon \R^{d \times d} \to [0,\infty)$ be the Hilbert-Schmidt norm on $\R^{d\times d}$,
let $\nu \colon \mathcal{B}(\R^d) \to [0,1]$ be a probability measure,
assume that
\begin{equation}
C 
=
2 (p-1)^{1/2}
\exp \!
\big(
3\mathbf{v} (1+ L^2 T (\sqrt{T} + \mathbf{v} p )^2 ) 
\big) \allowbreak
\big(
1
+
\left[ \!
\textint_{\R^d}  
\Norm{x}^{ p \mathbf{v} } \,
\nu(dx)
\right]^{ \!\nicefrac{1}{ p }}
\big),
\end{equation}
assume that $
\mathbf{C} = 4 (\max \{C, \varepsilon \})^4
$,
let 
$\varphi \colon \R^d \to \R$ be a continuous function,
let
$\mu \colon \R^d \to \R^d$ and
$\sigma \colon \R^d \to \R^{d \times d}$
be functions which satisfy 
for all $x,y \in \R^d$, $\lambda \in \R$ that
\begin{equation}\label{NN_approx_single:mu}
\mu(\lambda x+y) + \lambda \mu(0) = \lambda\mu(x)+\mu(y),
\end{equation}
\begin{equation}\label{NN_approx_single:sigma}
\sigma(\lambda x+y) + \lambda \sigma(0) = \lambda\sigma(x)+\sigma(y),
\end{equation}
and
$
\norm{\mu(x)} + \HSNorm{{\sigma(x)}} 
\leq  
L  (1 + \norm{x}),$
and let $\phi \in \mathcal{N}$ satisfy 
for all $x \in \R^d$ that
$
	\mathcal{R}(\phi) \in \mathcal{M}_d
$,
$	\left| 
		( \mathcal{R}(\phi) )  (x)
	\right| 
\leq 
	c  (1+\| x \|^{\mathbf{v}})
$,
and
\begin{equation}\label{NN_approx_single:growth}
\left| 
		\varphi(x) - ( \mathcal{R}(\phi) )(x)
	\right|
\leq 
	C^{-1} \varepsilon  (1+ \| x \|^{\mathbf{v}}).
\end{equation}
Then
\begin{enumerate}[(i)]

\item \label{NN_approx_single:item1}
there exists a unique continuous function $u\colon [0,T]\allowbreak \times \R^d \to \R$
which satisfies that
$
	\inf_{q \in (0,\infty)} \allowbreak
	\sup_{(t, x) \in [0, T] \times \R^d} \allowbreak
	\frac{ | u(t, x) | }{ 1 + \norm{x}^q }
<
	\infty
$,  which satisfies
for all $x \in \R^d$ that 
$u(0,x) = \varphi(x)$, 
and which satisfies that $u|_{(0,T) \times \R^d}$ is a viscosity solution of
\begin{equation}
\begin{split}
	(\tfrac{\partial }{\partial t}u)(t,x) 
&= 
	\tfrac{1}{2} 
	\operatorname{Trace}\! \big( 
		\sigma(x)[\sigma(x)]^{\ast}(\operatorname{Hess}_x u )(t,x)
	\big)  
	+
	\langle (\nabla_x u)(t,x),\mu(x)\rangle
\end{split}
\end{equation}
for $(t,x) \in (0,T) \times \R^d$
and

\item \label{NN_approx_single:item2}
there exists $\psi \in \mathcal{N}$ such that
$
	\mathcal{P}(\psi) 
\leq
	 c^4 \, \mathbf{C} \, \mathcal{P}(\phi) \, \varepsilon^{-4}
$,
$
	\mathscr{P}(\psi) 
\leq
	c^2 \, \mathbf{C} \, \mathcal{P}(\phi) \,\allowbreak \varepsilon^{-2}
$,
$
	\mathcal{R}(\psi) \in \mathcal{M}_d
$,
and
\begin{equation}
\label{NN_approx_single:concl1} 
	\left[
		\int_{\R^d}  
		\left|
			u(T,x) - ( \mathcal{R}(\psi) ) (x)
		\right|^p \,
		\nu(dx)
	\right]^{\nicefrac{1}{p}} 
\leq
		\varepsilon.
\end{equation}
\end{enumerate}
\end{lemma}
\end{samepage}

\begin{proof}[Proof of Lemma~\ref{NN_approx_single}]
Throughout this proof let $n = \min ( \N \cap [c^2 C^2 \varepsilon^{-2}, \infty))$.
Note that  \eqref{NN_approx_single:growth} implies that $\varphi$ is an at most polynomially growing function. Combining this,
\eqref{NN_approx_single:mu}, and \eqref{NN_approx_single:sigma} with
item~\eqref{viscosity_affine_existence:item1} in Corollary~\ref{viscosity_affine_existence} 
(with
$ d = d $,
$ m = d $,
$ T = T $,
$ \varphi = \varphi $, 
$ \mu = \mu $,
$ \sigma = \sigma $
in the notation of Corollary~\ref{viscosity_affine_existence})
establishes that 
there exists a unique continuous function $u\colon [0,T] \times \R^d \to \R$ which satisfies
for all $x \in \R^d$ that 
$u(0,x) = \varphi(x)$, 
which satisfies that
$
	\inf_{q \in (0,\infty)} 
	\sup_{(t, x) \in [0, T] \times \R^d} 
	\frac{ | u(t, x) | }{ 1 + \norm{x}^q }
<
	\infty
$,
and which satisfies that $u|_{(0,T) \times \R^d}$ is a viscosity solution of
\begin{equation}
\begin{split}
	(\tfrac{\partial }{\partial t}u)(t,x) 
&= 
	\tfrac{1}{2} 
	\operatorname{Trace}\! \big( 
		\sigma(x)[\sigma(x)]^{\ast}(\operatorname{Hess}_x u )(t,x)
	\big)  
	+
	\langle (\nabla_x u)(t,x),\mu(x)\rangle
\end{split}
\end{equation}
for $(t,x) \in (0,T) \times \R^d$.
This proves item~\eqref{NN_approx_single:item1}.
Next note that Proposition~\ref{qualitative} 
(with
$d = d$,
$T = T$,
$\varepsilon = \varepsilon $,
$ L = L $,
$ c = c $,
$ \mathbf{v} = \mathbf{v} $,
$ p = p $,
$ \nu = \nu $,
$ n = n $,
$ \varphi = \varphi $,
$ \phi = \mathcal{R}(\phi) $,
$ \mu = \mu $,
$ \sigma = \sigma $
in the notation of  Proposition~\ref{qualitative})
ensures that
there exist $A_1,A_2,\ldots,A_{n} \in \R^{d \times d}$, $b_1,b_2,\ldots,b_{n} \in \R^d$ such that
\begin{equation}
\label{NN_approx_single:eq1}
	\left[
		\int_{\R^d}  
		\left|
			u(T,x) - \tfrac{1}{n}\big[  \smallsum_{i=1}^{n}  (\mathcal{R}(\phi)) ( A_i x + b_i)  \big] 
		\right|^p \,
		\nu(dx)
	\right]^{\nicefrac{1}{p}} 
\leq
		\varepsilon.
\end{equation}
Moreover, observe that Lemma~\ref{multichannel_network} demonstrates that there exists $\psi \in \mathcal{N}$ 
which satisfies that
$ \mathcal{P}(\psi) \leq n^2 \, \mathcal{P}(\phi) $, 
$ \mathscr{P}(\psi) \leq n \, \mathcal{P}(\phi) $, and
$ \mathcal{R}(\psi) \in \mathcal{M}_d $ and which satisfies for all $x \in \R^d$ that 
\begin{equation}
	(\mathcal{R}(\psi)) (x)
=
	\tfrac{1}{n}\big[  \smallsum_{i=1}^{n} (\mathcal{R}(\phi))( A_i x + b_i)  \big]. 
\end{equation}
This and \eqref{NN_approx_single:eq1} assure that
\begin{equation}
\label{NN_approx_single:eq2}
\begin{split}
	&\left[
		\int_{\R^d}  
		\left|
			u(T,x) - (\mathcal{R}(\psi)) (x) 
		\right|^p \,
		\nu(dx)
	\right]^{\nicefrac{1}{p}} \\
&=
	\left[
		\int_{\R^d}  
		\left|
			u(T,x) - \tfrac{1}{n}\big[  \smallsum_{i=1}^{n}  (\mathcal{R}(\phi)) ( A_i x + b_i)  \big] 
		\right|^p \,
		\nu(dx)
	\right]^{\nicefrac{1}{p}} 
\leq
		\varepsilon.
\end{split}
\end{equation}
Moreover, note that the hypothesis that $
\mathbf{C} = 4 (\max \{C, \varepsilon \})^4
$ and $c\in [1,\infty)$ implies that
$
c^2\, \sqrt{\mathbf{C}} \, \varepsilon^{-2}\allowbreak
\geq
\sqrt{4\varepsilon^4} \, \varepsilon^{-2} 
=
2.
$
This ensures that
\begin{equation}
\label{NN_approx_single:eq3}
\begin{split}
	n
&\leq
	c^2 \, C^2 \, \varepsilon^{-2} + 1
\leq
	2\max \{c^2 \, C^2 \, \varepsilon^{-2}, 1\} \\
&=
	\max \{c^2 \, 2 \, C^2 \, \varepsilon^{-2}, 2\} \\
&\leq
	\max \{c^2 \, \sqrt{\mathbf{C}} \, \varepsilon^{-2}, 2\} \\
&=
	c^2 \, \sqrt{\mathbf{C}} \, \varepsilon^{-2}.
\end{split}
\end{equation}
This and the fact that $\mathcal{P}(\psi) \leq n^2 \, \mathcal{P}(\phi)$ imply that 
\begin{equation}
\label{NN_approx_single:eq4}
\begin{split}
	\mathcal{P}(\psi) 
&\leq 
	(c^2 \, \sqrt{\mathbf{C}} \, \varepsilon^{-2})^2 \, \mathcal{P}(\phi)
=
	c^4 \, \mathbf{C} \, \mathcal{P}(\phi) \, \varepsilon^{-4}.
\end{split}
\end{equation}
Furthermore, note that \eqref{NN_approx_single:eq3}, the fact that $\mathscr{P}(\psi) \leq n \, \mathcal{P}(\phi)$, and the fact that $\mathbf{C} \geq 1$ ensure that 
\begin{equation}
	\mathscr{P}(\psi) 
\leq 
	c^2 \, \sqrt{\mathbf{C}} \, \varepsilon^{-2} \, \mathcal{P}(\phi)
\leq
	c^2 \, \mathbf{C} \, \mathcal{P}(\phi) \, \varepsilon^{-2}.
\end{equation}
This, \eqref{NN_approx_single:eq2}, \eqref{NN_approx_single:eq4}, and 
the fact that $ \mathcal{R}(\psi) \in \mathcal{M}_d $ establish item~\eqref{NN_approx_single:item2}.
The proof of Lemma~\ref{NN_approx_single} is thus completed.
\end{proof}

\begin{prop}
\label{NN_approx_family}
Assume Setting~\ref{setting_NN},
let $d \in \N$, $T, a, r, R \in (0,\infty)$, $ L, \mathbf{C}, \mathbf{z} \in [0,\infty)$, $b, c \in [1,\infty)$, $\mathbf{v}, p \in [2, \infty)$, 
	let $\langle \cdot, \cdot \rangle \colon \R^d \times \R^d \to \R$ be the $d$-dimensional Euclidean scalar product,
let $\left\| \cdot \right\| \colon \R^d \to [0,\infty)$ be the $d$-dimen\-sional Euclidean norm, 
let $\HSNorm{\cdot} \colon \R^{d \times d} \to [0,\infty)$ be the Hilbert-Schmidt norm on $\R^{d\times d}$,
let $\nu \colon \mathcal{B}(\R^d) \to [0,1]$ be a probability measure,
assume that
\begin{multline}\label{NN_approx_familyDefinitionC}
\mathbf{C} 
= 
4 
\big[ \!
\max \{1, \tfrac{R}{r} \}\max \!
\big\{ 
2 (p-1)^{1/2}
\exp \!
\big(
3\mathbf{v} (1+ L^2 T (\sqrt{T} + \mathbf{v} p )^2 ) 
\big) \\
\big(
1
+
\left[ \!
\textint_{\R^d}  
\Norm{x}^{ p \mathbf{v} } \,
\nu(dx)
\right]^{ \!\nicefrac{1}{ p }}
\big),R
\big\}
\big]^{4 + \mathbf{z}},
\end{multline}
let 
$\varphi \colon \R^d \to \R$ be a continuous function, 
let
$\mu \colon \R^d \to \R^d$ and
$\sigma \colon \R^d \to \R^{d \times d}$
be functions which satisfy 
for all $x,y \in \R^d$, $\lambda \in \R$ that
\begin{equation}\label{NN_approx_family:mu}
\mu(\lambda x+y) + \lambda \mu(0) = \lambda\mu(x)+\mu(y),
\end{equation}
\begin{equation}\label{NN_approx_family:sigma}
\sigma(\lambda x+y) + \lambda \sigma(0) = \lambda\sigma(x)+\sigma(y),
\end{equation}
and
$
\norm{\mu(x)} + \HSNorm{{\sigma(x)}} 
\leq  
L  (1 + \norm{x}),$
and let $(\phi_\delta)_{\delta \in (0,r]} \subseteq \mathcal{N}$ satisfy 
for all $\delta \in (0,r]$, $x \in \R^d$ that
$
	\mathcal{P}(\phi_\delta) \leq a \, \delta^{-\mathbf{z}} 
$,
$
	\mathcal{R}(\phi_\delta) \in \mathcal{M}_d
$,
$	\left| 
		( \mathcal{R}(\phi_\delta) )  (x)
	\right| 
\leq 
	c \, (1+\| x \|^{\mathbf{v}})
$,
and
\begin{equation}
\label{NN_approx_family:ass1}
\left| 
		\varphi(x) - ( \mathcal{R}(\phi_\delta) )(x)
	\right|
\leq 
	b \, \delta \, (1+ \| x \|^{\mathbf{v}}).
\end{equation}
Then  
\begin{enumerate}[(i)]

\item \label{NN_approx_family:item1}
there exists a unique continuous function $u\colon [0,T] \allowbreak\times \R^d \to \R$ 
which satisfies that
$
	\inf_{q \in (0,\infty)}\allowbreak 
	\sup_{(t, x) \in [0, T] \times \R^d} \allowbreak
	\frac{ | u(t, x) | }{ 1 + \norm{x}^q }
<
	\infty
$,
which satisfies
for all $x \in \R^d$ that 
$u(0,x) = \varphi(x)$,
and which satisfies that $u|_{(0,T) \times \R^d}$ is a viscosity solution of
\begin{equation}
\begin{split}
	(\tfrac{\partial }{\partial t}u)(t,x) 
= 
	\tfrac{1}{2} 
	\operatorname{Trace}\! \big( 
		\sigma(x)[\sigma(x)]^{\ast}(\operatorname{Hess}_x u )(t,x)
	\big) 
	+
	\langle (\nabla_x u)(t,x),\mu(x)\rangle
\end{split}
\end{equation}
for $(t,x) \in (0,T) \times \R^d$
and

\item \label{NN_approx_family:item2}
there exist $(\psi_\varepsilon)_{\varepsilon \in (0,R]} \subseteq \mathcal{N}$ such that
for all  $\varepsilon \in (0,R]$ it holds that
$
	\mathcal{P}(\psi_\varepsilon) 
\leq
	 \mathbf{C} \, a \, b^\mathbf{z} \, c^4 \, \varepsilon^{-4-\mathbf{z}}
$,
$
	\mathscr{P}(\psi_\varepsilon) 
\leq
	\mathbf{C} \, a \, b^\mathbf{z} \, c^2 \, \varepsilon^{-2-\mathbf{z}}
$,
$
	\mathcal{R}(\psi_\varepsilon) \in \mathcal{M}_d
$,
and
\begin{equation}
\label{NN_approx_family:concl1} 
	\left[
		\int_{\R^d}  
		\left|
			u(T,x) - ( \mathcal{R}(\psi_\varepsilon) ) (x)
		\right|^p \,
		\nu(dx)
	\right]^{\nicefrac{1}{p}} 
\leq
		\varepsilon.
\end{equation}
\end{enumerate}
\end{prop}

\begin{proof}[Proof of Proposition~\ref{NN_approx_family}]
Throughout this proof let $\varepsilon \in (0,R]$ and let $C \in [0,\infty)$ be given by
\begin{equation}\label{NN_approx_familySimplificationC}
C 
=
2 (p-1)^{1/2}
\exp \!
\big(
3\mathbf{v} (1+ L^2 T (\sqrt{T} + \mathbf{v} p )^2 ) 
\big) \allowbreak
\big(
1
+
\left[ \!
\textint_{\R^d}  
\Norm{x}^{ p \mathbf{v} } \,
\nu(dx)
\right]^{ \!\nicefrac{1}{ p }}
\big).
\end{equation}
Note that  \eqref{NN_approx_family:ass1} implies that $\varphi$ is an at most polynomially growing function. Combining this,
\eqref{NN_approx_family:mu}, and \eqref{NN_approx_family:sigma} with
item~\eqref{viscosity_affine_existence:item1} in Corollary~\ref{viscosity_affine_existence}
(with
$ d = d $,
$ m = d $,
$ T = T $,
$ \varphi = \varphi $, 
$ \mu = \mu $,
$ \sigma = \sigma $
in the notation of Corollary~\ref{viscosity_affine_existence})
establishes that
there exists a unique continuous function $u\colon [0,T] \times \R^d \to \R$ which satisfies
for all $x \in \R^d$ that 
$u(0,x) = \varphi(x)$, 
which satisfies that
$
	\inf_{q \in (0,\infty)} 
	\sup_{(t, x) \in [0, T] \times \R^d} 
	\frac{ | u(t, x) | }{ 1 + \norm{x}^q }
<
	\infty
$,
and which satisfies that $u|_{(0,T) \times \R^d}$ is a viscosity solution of
\begin{equation}
\begin{split}
	(\tfrac{\partial }{\partial t}u)(t,x) 
= 
	\tfrac{1}{2} 
	\operatorname{Trace}\! \big( 
		\sigma(x)[\sigma(x)]^{\ast}(\operatorname{Hess}_x u )(t,x)
	\big) 
	+
	\langle (\nabla_x u)(t,x),\mu(x)\rangle
\end{split}
\end{equation}
for $(t,x) \in (0,T) \times \R^d$.
This proves item~\eqref{NN_approx_family:item1}.
Next observe that \eqref{NN_approx_family:ass1} ensures that 
for all $x \in \R^d$ it holds that
\begin{equation}
\begin{split}
\left| 
		\varphi(x) - \big( \mathcal{R}(\phi_{\min \{b^{-1} C^{-1} \varepsilon, r\}}) \big)(x)
	\right|
&\leq 
	\min \{ b^{-1} C^{-1} \varepsilon, r\} \, b (1+ \| x \|^{\mathbf{v}}) \\
&\leq
	b^{-1} C^{-1} \varepsilon \, b (1+ \| x \|^{\mathbf{v}}) \\
&=
	C^{-1} \varepsilon  (1+ \| x \|^{\mathbf{v}}).
\end{split}
\end{equation}
Lemma~\ref{NN_approx_single}
(with
$d = d$,
$T = T$,
$\varepsilon = \varepsilon$,
$L = L$,
$c = c$,
$\mathbf{v} = \mathbf{v}$,
$p = p$,
$\nu = \nu$,
$ \varphi = \varphi $,
$ \mu = \mu $,
$ \sigma = \sigma $,
$\phi = \phi_{\min \{b^{-1} C^{-1} \varepsilon, r\}}$
in the notation of Lemma~\ref{NN_approx_single}) hence assures that 
there exists $\psi \in \mathcal{N}$ such that
\begin{equation}
\label{NN_approx_family:eq1}
	\mathcal{P}(\psi) 
\leq
	c^4 \,  
	4 (\max \{ C, \varepsilon\})^4  \,  
	\mathcal{P}(\phi_{ \min \{b^{-1} C^{-1} \varepsilon, r\}})
	\, \varepsilon^{-4},
\end{equation}
\begin{equation}
\label{NN_approx_family:eq11}
	\mathscr{P}(\psi) 
\leq
	c^2 \,  \allowbreak
	4 (\max \{ C, \varepsilon\})^4  \,  
	\mathcal{P}(\phi_{ \min \{b^{-1} C^{-1} \varepsilon, r\}})
	\, \varepsilon^{-2},
\end{equation}
\begin{equation}
\label{NN_approx_family:eq2}
	\mathcal{R}(\psi) \in \mathcal{M}_d,
\qandq
	\left[
		\int_{\R^d}  
		\left|
			u(T,x) - ( \mathcal{R}(\psi) ) (x)
		\right|^p \,
		\nu(dx)
	\right]^{\nicefrac{1}{p}} 
\leq
		\varepsilon.
\end{equation}
Moreover, note that the fact that $b,C \geq 1$ and the fact that $\varepsilon R^{-1}\leq 1$ assures that
\begin{equation}
\begin{split}
	\min \{b^{-1} C^{-1} \varepsilon, r\} 
&\geq 
	\min \big\{b^{-1} C^{-1} \varepsilon, r \, b^{-1}C^{-1}  \varepsilon R^{-1}\big\} \\
&=
	 \min \{1, \tfrac{r}{R}\} \, b^{-1} C^{-1} \varepsilon.
\end{split}
\end{equation}  
This, the fact that $C \geq 1$, the fact that $\varepsilon \in (0,R]$, \eqref{NN_approx_familyDefinitionC}, and \eqref{NN_approx_familySimplificationC} ensure that
\begin{equation}
\label{NN_approx_family:eq3}
\begin{split}
	&4(\max \{ C, \varepsilon\})^4 (\min \{b^{-1} C^{-1} \varepsilon, r\})^{-\mathbf{z}} \\ 
&\leq
	4 (\max \{ C, R \})^4 
	\left( \min \{1, \tfrac{r}{R}\} \, b^{-1} C^{-1} \varepsilon\right)^{-\mathbf{z}}  \\
&=
	4 (\max \{ C,R\})^4  
	 (\max \{1, \tfrac{R}{r}\})^{\mathbf{z}} \,
	 b^{\mathbf{z}} \, C^{\mathbf{z}} \,\varepsilon^{-\mathbf{z}}\\
&\leq
	4 (\max \{ C,R\})^{4 + \mathbf{z}}
	 (\max \{1, \tfrac{R}{r}\})^{4 + \mathbf{z}} \,
	 b^{\mathbf{z}} \,\varepsilon^{-\mathbf{z}}\\
&=
	\mathbf{C} \, b^{\mathbf{z}} \, \varepsilon^{-\mathbf{z}}.
\end{split}
\end{equation}
Combining this with the hypothesis that 
for all $\delta \in (0,r]$ it holds that
$
	\mathcal{P}(\phi_\delta) \leq a \delta^{-\mathbf{z}}
$ 
and \eqref{NN_approx_family:eq1}
demonstrates that
\begin{equation}
\label{NN_approx_family:eq4}
\begin{split}
	\mathcal{P}(\psi)
&\leq
	c^4 \,  
	4 (\max \{ C, \varepsilon\})^4  \,  
	a \, (\min \{b^{-1} C^{-1} \varepsilon, r\})^{-\mathbf{z}} \,
	\varepsilon^{-4} \\
&\leq
	c^4 \, a \,
	\mathbf{C} \, b^{\mathbf{z}} \, \varepsilon^{-\mathbf{z}}
	\varepsilon^{-4} 
=
	\mathbf{C} \,
	a \, b^{\mathbf{z}} \, c^4 \,
	\varepsilon^{-4-\mathbf{z}}.
\end{split}
\end{equation}
Furthermore, observe that the hypothesis that 
for all $\delta \in (0,r]$ it holds that
$
	\mathcal{P}(\phi_\delta) \leq a \delta^{-\mathbf{z}} 
$,
\eqref{NN_approx_family:eq11}, and \eqref{NN_approx_family:eq3} 
demonstrate that
\begin{equation}
\begin{split}
	\mathscr{P}(\psi)
&\leq
	 c^2 \,  
	4 (\max \{ C, \varepsilon\})^4  \,  
	a \, (\min \{b^{-1} C^{-1} \varepsilon, r\})^{-\mathbf{z}} \,
	\varepsilon^{-2} \\
&\leq
	c^2 \,  a \,
	\mathbf{C} \, b^{\mathbf{z}} \, \varepsilon^{-\mathbf{z}} \,
	\varepsilon^{-2}  
=
	\mathbf{C} \, 
	a \, b^{\mathbf{z}} \, c^2 \, \varepsilon^{-2 -\mathbf{z}}.
\end{split}
\end{equation}
Combining this, \eqref{NN_approx_family:eq2}, and \eqref{NN_approx_family:eq4} establishes item~\eqref{NN_approx_family:item2}.
The proof of Proposition~\ref{NN_approx_family} is thus completed.
\end{proof}

\begin{samepage}
\begin{cor}
\label{NN_approx_family2}
Assume Setting~\ref{setting_NN},
let $d \in \N$, $T, r, R \in (0,\infty)$, $ L, \mathfrak{C}, v, w, z,\allowbreak \mathbf{z} \in [0,\infty)$, $\mathfrak{c} \in [1,\infty)$, $\mathbf{v}, p \in [2, \infty)$, 
	let $\langle \cdot, \cdot \rangle \colon \R^d \times \R^d \to \R$ be the $d$-dimensional Euclidean scalar product,
let $\left\| \cdot \right\| \colon \R^d \to [0,\infty)$ be the $d$-dimensional Euclidean norm, 
let $\HSNorm{\cdot} \colon \R^{d \times d} \to [0,\infty)$ be the Hilbert-Schmidt norm on $\R^{d\times d}$,
let $\nu \colon \mathcal{B}(\R^d) \to [0,1]$ be a probability measure,
assume that
\begin{multline}
\mathfrak{C} 
= 
\big[ 
\mathfrak{c}
\max \{1, \tfrac{R}{r} \}
\max \!
\big\{ 
2(p-1)^{1/2}
\exp \!
\big(
3\mathbf{v} (1+ L^2 T (\sqrt{T} + \mathbf{v} p )^2 ) 
\big)
\\\big(
1
+
\left[ \!
\textint_{\R^d}  
\Norm{x}^{ p \mathbf{v} } \,
\nu(dx)
\right]^{ \!\nicefrac{1}{ p }}
\big),R
\big\}
\big]^{5 + \mathbf{z}},
\end{multline}
let 
$\varphi \colon \R^d \to \R$ be a continuous function,
let
$\mu \colon \R^d \to \R^d$ and
$\sigma \colon \R^d \to \R^{d \times d}$
be functions which satisfy 
for all $x,y \in \R^d$, $\lambda \in \R$ that
\begin{equation}\label{NN_approx_family2:mu}
\mu(\lambda x+y) + \lambda \mu(0) = \lambda\mu(x)+\mu(y),
\end{equation}
\begin{equation}\label{NN_approx_family2:sigma}
\sigma(\lambda x+y) + \lambda \sigma(0) = \lambda\sigma(x)+\sigma(y),
\end{equation}
and
$
\norm{\mu(x)} + \HSNorm{{\sigma(x)}} 
\leq  
L  (1 + \norm{x}),$ 
and let $(\phi_\delta)_{\delta \in (0,r]} \subseteq \mathcal{N}$ satisfy 
for all $\delta \in (0,r]$, $x \in \R^d$ that
$
	\mathcal{P}(\phi_\delta) \leq \mathfrak{c} \, d^z \delta^{-\mathbf{z}} 
$,
$
	\mathcal{R}(\phi_\delta) \in \mathcal{M}_d
$,
$	\left| 
		( \mathcal{R}(\phi_\delta) )  (x)
	\right| 
\leq 
	\mathfrak{c} \, d^v (1+\| x \|^{\mathbf{v}})
$,
and
\begin{equation}
\label{NN_approx_family2:ass1}
\left| 
		\varphi(x) - ( \mathcal{R}(\phi_\delta) )(x)
	\right|
\leq 
	 \mathfrak{c} \, d^w \, \delta \, (1+ \| x \|^{\mathbf{v}}).
\end{equation}
Then  
\begin{enumerate}[(i)]

\item \label{NN_approx_family2:item1}
there exists a unique continuous function $u\colon [0,T] \allowbreak\times \R^d \to \R$
which satisfies that
$
	\inf_{q \in (0,\infty)} \allowbreak
	\sup_{(t, x) \in [0, T] \times \R^d} \allowbreak
	\frac{ | u(t, x) | }{ 1 + \norm{x}^q }
<
	\infty
$,
which satisfies
for all $x \in \R^d$ that 
$u(0,x) = \varphi(x)$, 
and which satisfies that $u|_{(0,T) \times \R^d}$ is a viscosity solution of
\begin{equation}
\begin{split}
	(\tfrac{\partial }{\partial t}u)(t,x) 
= 
	\tfrac{1}{2} 
	\operatorname{Trace}\! \big( 
		\sigma(x)[\sigma(x)]^{\ast}(\operatorname{Hess}_x u )(t,x)
	\big) 
	+
	\langle (\nabla_x u)(t,x),\mu(x)\rangle
\end{split}
\end{equation}
for $(t,x) \in (0,T) \times \R^d$
and

\item \label{NN_approx_family2:item2}
there exist $(\psi_\varepsilon)_{\varepsilon \in (0,R]} \subseteq \mathcal{N}$ such that
for all  $\varepsilon \in (0,R]$ it holds that
$
	\mathcal{P}(\psi_\varepsilon) 
\leq
	\mathfrak{C} \, d^{z + w \mathbf{z} + 4v} \, \varepsilon^{-4-\mathbf{z}}
$,
$
	\mathscr{P}(\psi_\varepsilon) 
\leq
	\mathfrak{C} \,d^{z + w \mathbf{z} + 2v} \, \varepsilon^{-2-\mathbf{z}}
$,
$
	\mathcal{R}(\psi_\varepsilon) \in \mathcal{M}_d
$,
and
\begin{equation}
\label{NN_approx_family2:concl1} 
	\left[
		\int_{\R^d}  
		\left|
			u(T,x) - ( \mathcal{R}(\psi_\varepsilon) ) (x)
		\right|^p \,
		\nu(dx)
	\right]^{\nicefrac{1}{p}} 
\leq
		\varepsilon.
\end{equation}
\end{enumerate}
\end{cor}
\end{samepage}

\begin{proof}[Proof of Corollary~\ref{NN_approx_family2}]
Throughout this proof let $C, \mathbf{C} \in [0,\infty)$ be given by
\begin{equation}
		C 
	=
	2(p-1)^{1/2}
	\exp \!
	\big(
	3\mathbf{v} (1+ L^2 T (\sqrt{T} + \mathbf{v} p )^2 ) 
	\big) \allowbreak
	\big(
	1
	+
	\left[ \!
	\textint_{\R^d}  
	\Norm{x}^{ p \mathbf{v} } \,
	\nu(dx)
	\right]^{ \!\nicefrac{1}{ p }}
	\big)
\end{equation}
and
\begin{equation}
		\mathbf{C}
	= 
	4 
	\big[ \!
	\max \{1, \tfrac{R}{r} \}
	\max 
	\{ 
	C
	,
	R 
	\}
	\big]^{4 + \mathbf{z}}.
\end{equation}
Note that  \eqref{NN_approx_family2:ass1} implies that $\varphi$ is an at most polynomially growing function. Combining this,
\eqref{NN_approx_family2:mu}, and \eqref{NN_approx_family2:sigma} with
item~\eqref{viscosity_affine_existence:item1} in Corollary~\ref{viscosity_affine_existence}
(with
$ d = d $,
$ m = d $,
$ T = T $,
$ \varphi = \varphi $, 
$ \mu = \mu $,
$ \sigma = \sigma $
in the notation of Corollary~\ref{viscosity_affine_existence}) 
establishes that 
there exists a unique continuous function $u\colon [0,T] \times \R^d \to \R$ 
which satisfies that
$
	\inf_{q \in (0,\infty)} 
	\sup_{(t, x) \in [0, T] \times \R^d} 
	\frac{ | u(t, x) | }{ 1 + \norm{x}^q }
<
	\infty
$, which satisfies
for all $x \in \R^d$ that 
$u(0,x) = \varphi(x)$, 
and which satisfies that $u|_{(0,T) \times \R^d}$ is a viscosity solution of
\begin{equation}
\begin{split}
	(\tfrac{\partial }{\partial t}u)(t,x) 
= 
	\tfrac{1}{2} 
	\operatorname{Trace}\! \big( 
		\sigma(x)[\sigma(x)]^{\ast}(\operatorname{Hess}_x u )(t,x)
	\big) 
	+
	\langle (\nabla_x u)(t,x),\mu(x)\rangle
\end{split}
\end{equation}
for $(t,x) \in (0,T) \times \R^d$.
This proves item~\eqref{NN_approx_family2:item1}.
Next observe that Proposition~\ref{NN_approx_family}
(with
$ d = d $,
$ T = T $,
$ a = \mathfrak{c} \, d^z $,
$ r = r $,
$ R = R $,
$ L = L $,
$ \mathbf{z} = \mathbf{z} $,
$ b = \mathfrak{c} \, d^w $,
$ c = \mathfrak{c} \, d^v $,
$ \mathbf{v} = \mathbf{v} $,
$p = p$,
$ \nu = \nu$,
$ \varphi = \varphi $,
$ \mu = \mu $,
$ \sigma = \sigma $,
$ (\phi_\delta)_{\delta \in (0,r]} = (\phi_\delta)_{\delta \in (0,r]}$
in the notation of Proposition~\ref{NN_approx_family}) proves that
there exist $(\psi_\varepsilon)_{\varepsilon \in (0,R]} \subseteq \mathcal{N}$ such that
for all  $\varepsilon \in (0,R]$ it holds that
\begin{equation}
\label{NN_approx_family2:eq1} 
	\mathcal{P}(\psi_\varepsilon) 
\leq
	\mathbf{C}  \mathfrak{c} \, d^z  ( \mathfrak{c} \, d^w)^\mathbf{z} ( \mathfrak{c} \, d^v)^4  \varepsilon^{-4-\mathbf{z}} , 
\end{equation}
\begin{equation}
\label{NN_approx_family2:eq2} 
	\mathscr{P}(\psi_\varepsilon) 
\leq
	 \mathbf{C} \mathfrak{c} \, d^z ( \mathfrak{c} \, d^w)^\mathbf{z} ( \mathfrak{c} \, d^v)^2  \varepsilon^{-2-\mathbf{z}} , 
\end{equation}
\begin{equation}
\label{NN_approx_family2:eq3} 
	\mathcal{R}(\psi_\varepsilon) \in \mathcal{M}_d,
\qandq
	\left[
		\int_{\R^d}  
		\left|
			u(T,x) - ( \mathcal{R}(\psi_\varepsilon) ) (x)
		\right|^p \,
		\nu(dx)
	\right]^{\nicefrac{1}{p}} 
\leq
		\varepsilon.
\end{equation}
Furthermore, note that the fact that $p, \mathbf{v} \geq 2$ assures that
\begin{equation}
	C
\geq
	2(p-1)^{1/2}
	\exp \!
	\big(
		3\mathbf{v} (1+ L^2 T (\sqrt{T} + \mathbf{v} p )^2 ) 
	\big)
\geq
	2 \exp(6)
\geq 
	4.
\end{equation}
This implies that
\begin{equation}
\label{NN_approx_family2:eq4}
\begin{split}
	\mathfrak{c}^{5 + \mathbf{z}} \,  \mathbf{C}
&=
	\mathfrak{c}^{5 + \mathbf{z}} \, 4 \big[ \!\max \{1, \tfrac{R}{r} \}\max \{ C ,R \}\big]^{4 + \mathbf{z}} \\
&\leq 
	\mathfrak{c}^{5 + \mathbf{z}} \,
	C
	\big[ \!\max \{1, \tfrac{R}{r} \}\max \{ C ,R \}\big]^{4 + \mathbf{z}} \\
&\leq
	\mathfrak{c}^{5 + \mathbf{z}} 
	\big[ \!\max \{1, \tfrac{R}{r} \}\max \{ C ,R \}\big]^{5 + \mathbf{z}} \\
&=
	\mathfrak{C}.
\end{split}
\end{equation}
This and \eqref{NN_approx_family2:eq1} demonstrate that for all $\varepsilon \in (0,R]$ it holds that
\begin{equation}
\label{NN_approx_family2:eq5}
	\mathcal{P}(\psi_\varepsilon) 
\leq
	  \mathbf{C} \, \mathfrak{c}^{5 + \mathbf{z}} d^{z+w\mathbf{z} + 4v}   \varepsilon^{-4-\mathbf{z}} 
\leq	
	  \mathfrak{C} \, d^{z+w\mathbf{z} + 4v} \varepsilon^{-4-\mathbf{z}}.
\end{equation}
Moreover, observe that \eqref{NN_approx_family2:eq2}, \eqref{NN_approx_family2:eq4}, and the fact that $\mathfrak{c} \geq 1$ assure that
\begin{equation}
\begin{split}
	\mathscr{P}(\psi_\varepsilon) 
&\leq
	\mathbf{C} \, \mathfrak{c}^{3 + \mathbf{z}}d^{z + w \mathbf{z} + 2v} \varepsilon^{-2-\mathbf{z}} \\
&\leq
	\mathbf{C} \, \mathfrak{c}^{5 + \mathbf{z}}d^{z + w \mathbf{z} + 2v} \varepsilon^{-2-\mathbf{z}}
\leq
	 \mathfrak{C} \, d^{z + w \mathbf{z} + 2v} \,  \varepsilon^{-2-\mathbf{z}} .
\end{split}
\end{equation}
Combining this, \eqref{NN_approx_family2:eq3}, and \eqref{NN_approx_family2:eq5} establishes item~\eqref{NN_approx_family2:item2}.
The proof of Corollary~\ref{NN_approx_family2} is thus completed.
\end{proof}

\begin{samepage}
\begin{prop}
\label{NN_approx_I}
Assume Setting~\ref{setting_NN},
let $\mathcal{I}$ be a set, 
let $\mathfrak{d} = (\mathfrak{d}_i)_{i \in \mathcal{I}} \colon \mathcal{I} \to \N$ be a function, 
for every $d \in \N$ 	let $\langle \cdot, \cdot \rangle_{\R^d} \colon \R^d \times \R^d \to \R$ be the $d$-dimensional Euclidean scalar product,
for every $d \in \N$ 
	let $\left\| \cdot \right\|_{\R^d} \colon \R^d \to [0,\infty)$ be the $d$-dimensional Euclidean norm and
	let $\HSNormIndex{\cdot}{d}{d}\colon \R^{d \times d} \to [0,\infty)$ be the Hilbert-Schmidt norm on $\R^{d\times d}$,
let $T, r, R \in (0,\infty)$, $ \mathfrak{C}, L, v, w, z,\mathbf{z},\theta  \in [0,\infty)$, $\mathfrak{c} \in [1,\infty)$, $\mathbf{v}, p \in [2, \infty)$, 
let $\nu_d \colon \mathcal{B}(\R^d) \to [0,1]$, $d \in \Image(\mathfrak{d})$, be probability measures,
assume that 
\begin{multline}\label{NN_approxSupConstant}
\mathfrak{C} 
= 
\big[ 
\mathfrak{c}
\max \{1, \tfrac{R}{r} \} \allowbreak
\max \!
\big\{
2(p-1)^{1/2}
\exp \!
\big(
3\mathbf{v} (1+ L^2 T (\sqrt{T} + \mathbf{v} p )^2 ) 
\big) \\
\big(
1
+ 
\sup\nolimits_{i \in \mathcal{I}} \big[(\mathfrak{d}_i)^{-\theta}
\left[ \!
\textint_{\R^{\mathfrak{d}_i}}  
\Norm{x}_{\R^{\mathfrak{d}_i} }^{ p \mathbf{v} } \,
\nu_{ \mathfrak{d}_i } (dx)
\right]^{ \!\nicefrac{1}{ p }}\big]
\big) ,R
\big\}
\big]^{5 + \mathbf{z}},
\end{multline}
let 
$\varphi_i \colon \R^{\mathfrak{d}_i} \to \R$, $i \in \mathcal{I}$, be continuous functions, 
let $\mu_i \colon \R^{\mathfrak{d}_i} \to \R^{\mathfrak{d}_i}$, $i \in \mathcal{I}$, and
$\sigma_i \colon \R^{\mathfrak{d}_i} \to \R^{\mathfrak{d}_i \times \mathfrak{d}_i}$, $i \in \mathcal{I}$, 
be functions which satisfy 
for all $i \in \mathcal{I}$, $x,y \in \R^{\mathfrak{d}_i}$, $\lambda \in \R$ that
\begin{equation}\label{NN_approx_I:mu}
\mu_i(\lambda x+y) + \lambda \mu_i(0) = \lambda\mu_i(x)+\mu_i(y),
\end{equation}
\begin{equation}\label{NN_approx_I:sigma}
\sigma_i(\lambda x+y) + \lambda \sigma_i(0) = \lambda\sigma_i(x)+\sigma_i(y),
\end{equation}
and 
$\norm{{\mu_i(x)}}_{\R^{\mathfrak{d}_i}} + \HSNormIndex{{\sigma_i(x)}}{\mathfrak{d}_i}{\mathfrak{d}_i}
\leq  
L  (1 + \norm{x}_{\R^{\mathfrak{d}_i}}),$
and let $(\phi_{i, \delta})_{i \in \mathcal{I}, \, \delta \in (0,r]} \subseteq \mathcal{N}$ satisfy 
for all $i \in \mathcal{I}$, $\delta \in (0,r]$, $x \in \R^{\mathfrak{d}_i}$ that
$
	\mathcal{P}(\phi_{i, \delta}) \leq \mathfrak{c} \, (\mathfrak{d}_i)^z \delta^{-\mathbf{z}} 
$,
$
	\mathcal{R}(\phi_{i, \delta}) \in \mathcal{M}_{\mathfrak{d}_i}
$,
$	\left| 
		( \mathcal{R}(\phi_{i, \delta}) )  (x)
	\right| 
\leq 
	\mathfrak{c} \, (\mathfrak{d}_i)^v (1+\| x \|_{\R^{\mathfrak{d}_i}}^{\mathbf{v}})
$,
and
\begin{equation}
\label{NN_approx_I:ass1}
\left| 
		\varphi_i(x) - ( \mathcal{R}(\phi_{i, \delta}) )(x)
	\right|
\leq 
	 \mathfrak{c} \, (\mathfrak{d}_i)^w \, \delta \, (1+ \| x \|_{\R^{\mathfrak{d}_i}}^{\mathbf{v}}).
\end{equation}
Then  
\begin{enumerate}[(i)]

\item \label{NN_approx_I:item1}
there exist unique continuous functions $u_i\colon [0,T]\allowbreak \times \R^{\mathfrak{d}_i} \to \R$, $i \in \mathcal{I}$, which satisfy
for all $i \in \mathcal{I}$, $x \in \R^{\mathfrak{d}_i}$ that 
$u_i(0,x) = \varphi_i(x)$,  
which satisfy for all $i \in \mathcal{I}$ that
$
	\inf_{q \in (0,\infty)} \allowbreak
	\sup_{(t, x) \in [0, T] \times \R^{\mathfrak{d}_i}} \allowbreak
	\frac{ | u_i(t, x) | }{ 1 + \norm{x}_{\R^{\mathfrak{d}_i}}^q }
\allowbreak<
	\infty
$,
and which satisfy for all $i \in \mathcal{I}$ that $u_i|_{(0,T) \times \R^{\mathfrak{d}_i}}$ is a viscosity solution of
\begin{equation}
\begin{split}
	(\tfrac{\partial }{\partial t}u_i)(t,x) 
&= 
	\tfrac{1}{2} 
	\operatorname{Trace}\! \big( 
		\sigma_i(x)[\sigma_i(x)]^{\ast}(\operatorname{Hess}_x u_i )(t,x)
	\big) 
	\\&\quad+
	\langle (\nabla_x u_i)(t,x),\mu_i(x)\rangle_{\R^{\mathfrak{d}_i}}
\end{split}
\end{equation}
for $(t,x) \in (0,T) \times \R^{\mathfrak{d}_i}$
and

\item \label{NN_approx_I:item2}
there exist $(\psi_{i, \varepsilon})_{i \in \mathcal{I}, \,\varepsilon \in (0,R]} \subseteq \mathcal{N}$ such that
for all $i \in \mathcal{I}$, $\varepsilon \in (0,R]$ it holds that
\begin{equation}
\mathcal{P}(\psi_{i,\varepsilon}) 
\leq
\mathfrak{C}\,(\mathfrak{d}_i)^{(5 + \mathbf{z})\theta+z + w \mathbf{z} + 4v} \, \varepsilon^{-4-\mathbf{z}},
\end{equation}
\begin{equation}
\mathscr{P}(\psi_{i,\varepsilon}) 
\leq
\mathfrak{C}\,(\mathfrak{d}_i)^{(5 + \mathbf{z})\theta+z + w \mathbf{z} + 2v} \, \varepsilon^{-2-\mathbf{z}} , 
\qquad
\mathcal{R}(\psi_{i,\varepsilon}) \in\mathcal{M}_{\mathfrak{d}_i},
\end{equation}
\begin{equation}
\andq
\left[
\int_{\R^d}  
\left|
u_i(T,x) - ( \mathcal{R}(\psi_{i,\varepsilon}) ) (x)
\right|^p \,
\nu_{\mathfrak{d}_i}(dx)
\right]^{\nicefrac{1}{p}} 
\leq
\varepsilon.
\end{equation}
\end{enumerate}
\end{prop}
\end{samepage}

\begin{proof}[Proof of Proposition~\ref{NN_approx_I}]
Throughout this proof let $i \in \mathcal{I},$ let $c_0\in(0,\infty)$ be given by 
\begin{equation}\label{NN_approxSmallConstant}
	c_0=2 (p-1)^{1/2}
	\exp \!
	\big(
	3\mathbf{v} (1+ L^2 T (\sqrt{T} + \mathbf{v} p )^2 ) 
	\big),
\end{equation}
and let $\mathcal{C} \in [0,\infty)$ be given by
\begin{equation}\label{NN_approxBigConstant}
		\mathcal{C}
	=
		\big[ 
			\mathfrak{c}
			\max \{1, \tfrac{R}{r} \} \allowbreak
			\max \!
			\big\{R, 
			c_0 \allowbreak
				\big(
					1
					+
					\left[ \!
						\textint_{\R^{ \mathfrak{d}_i }}  
							\Norm{x}_{\R^{ \mathfrak{d}_i }}^{ p \mathbf{v} } \,
						\nu_{ \mathfrak{d}_i } (dx)
					\right]^{ \!\nicefrac{1}{ p }}
				\big)
			\big\}
		\big]^{5 + \mathbf{z}}.
\end{equation}
Note that the fact that  $0<(\mathfrak{d}_i)^{-\theta}\le 1$ implies that
\begin{equation}
\begin{split}
&\max \!
\big\{ 
c_0
\big(
1
+
\left[ \!
\textint_{\R^{ \mathfrak{d}_i }}  
\Norm{x}_{\R^{ \mathfrak{d}_i }}^{ p \mathbf{v} } \,
\nu_{ \mathfrak{d}_i } (dx)
\right]^{ \!\nicefrac{1}{ p }}
\big)
,
R 
\big\}\\
&\le 	\max \!
\left\{ 
c_0 \allowbreak
\left(
(\mathfrak{d}_i)^\theta (\mathfrak{d}_i)^{-\theta}
+
(\mathfrak{d}_i)^\theta \sup_{j \in \mathcal{I}}\left[(\mathfrak{d}_j)^{-\theta}\left( \!
\textint_{\R^{ \mathfrak{d}_j }}  
\Norm{x}_{\R^{ \mathfrak{d}_j }}^{ p \mathbf{v} } \,
\nu_{ \mathfrak{d}_j } (dx)
\right)^{ \!\nicefrac{1}{ p }}
\right] \right)
,
R 
\right\}\\
&\le 	\max \!
\left\{ 
c_0 \allowbreak (\mathfrak{d}_i)^\theta 
\left(
1
+
\sup_{j \in \mathcal{I}}\left[(\mathfrak{d}_j)^{-\theta}\left( \!
\textint_{\R^{ \mathfrak{d}_j }}  
\Norm{x}_{\R^{ \mathfrak{d}_j }}^{ p \mathbf{v} } \,
\nu_{ \mathfrak{d}_j } (dx)
\right)^{ \!\nicefrac{1}{ p }}
\right] \right)
,
(\mathfrak{d}_i)^\theta (\mathfrak{d}_i)^{-\theta} R
\right\}\\
&\le (\mathfrak{d}_i)^\theta	\max \!
\left\{ 
c_0 \allowbreak 
\left(
1
+
\sup_{j \in \mathcal{I}}\left[(\mathfrak{d}_j)^{-\theta}\left( \!
\textint_{\R^{ \mathfrak{d}_j }}  
\Norm{x}_{\R^{ \mathfrak{d}_j }}^{ p \mathbf{v} } \,
\nu_{ \mathfrak{d}_j } (dx)
\right)^{ \!\nicefrac{1}{ p }}
\right] \right)
,
R
\right\}.
\end{split}
\end{equation}
Therefore, \eqref{NN_approxBigConstant}, \eqref{NN_approxSmallConstant}, and \eqref{NN_approxSupConstant} ensure that
\begin{equation}\label{NN_approx_ComparingConstants} 
	\begin{split}
			\mathcal{C}&
	=
	\big[ 
	\mathfrak{c}
	\max \{1, \tfrac{R}{r} \} \allowbreak
	\max \!
	\big\{ 
	c_0 \allowbreak
	\big(
	1
	+
	\left[ \!
	\textint_{\R^{ \mathfrak{d}_i }}  
	\Norm{x}_{\R^{ \mathfrak{d}_i }}^{ p \mathbf{v} } \,
	\nu_{ \mathfrak{d}_i } (dx)
	\right]^{ \!\nicefrac{1}{ p }}
	\big)
	,
	R 
	\big\}
	\big]^{5 + \mathbf{z}}\\
	&\le (\mathfrak{d}_i)^{\theta(5+\mathbf{z})}
	\bigg[ 
	\mathfrak{c}
	\max \{1, \tfrac{R}{r} \} \allowbreak
\\&\qquad\max \!
\left\{ 
c_0 \allowbreak 
\left(
1
+
\sup_{j \in \mathcal{I}}\left[(\mathfrak{d}_j)^{-\theta}\left( \!
\textint_{\R^{ \mathfrak{d}_j }}  
\Norm{x}_{\R^{ \mathfrak{d}_j }}^{ p \mathbf{v} } \,
\nu_{ \mathfrak{d}_j } (dx)
\right)^{ \!\nicefrac{1}{ p }}
\right] \right)
,
R
\right\}
	\bigg]^{5 + \mathbf{z}}\\
	&\le (\mathfrak{d}_i)^{\theta(5+\mathbf{z})}
\bigg[ 
\mathfrak{c}
\max \{1, \tfrac{R}{r} \}\max \!
\bigg\{ 
2(p-1)^{1/2}
\exp \!
\big(
3\mathbf{v} (1+ L^2 T (\sqrt{T} + \mathbf{v} p )^2 ) 
\big) \allowbreak 
\\&\qquad\left(
1
+
\sup_{j \in \mathcal{I}}\left[(\mathfrak{d}_j)^{-\theta}\left( \!
\textint_{\R^{ \mathfrak{d}_j }}  
\Norm{x}_{\R^{ \mathfrak{d}_j }}^{ p \mathbf{v} } \,
\nu_{ \mathfrak{d}_j } (dx)
\right)^{ \!\nicefrac{1}{ p }}
\right] \right)
,
R
\bigg\}
\bigg]^{5 + \mathbf{z}}\\	
&= (\mathfrak{d}_i)^{\theta(5+\mathbf{z})} \mathfrak{C}.
	\end{split}
\end{equation}
Moreover, observe that  \eqref{NN_approx_I:ass1} implies that $\varphi_i$ is an at most polynomially growing function. Combining this,
\eqref{NN_approx_I:mu}, and \eqref{NN_approx_I:sigma} with
item~\eqref{viscosity_affine_existence:item1} in Corollary~\ref{viscosity_affine_existence}
(with
$ d = \mathfrak{d}_i $,
$ m = \mathfrak{d}_i $,
$ T = T $,
$ \varphi = \varphi_i $, 
$ \mu = \mu_i $,
$ \sigma = \sigma_i $
in the notation of Corollary~\ref{viscosity_affine_existence})
establishes that there exists a unique continuous function $u_i\colon [0,T] \times \R^{\mathfrak{d}_i} \to \R$ which satisfies
for all $x \in \R^{\mathfrak{d}_i}$ that 
$u_i(0,x) = \varphi_i(x)$,  
which satisfies that
$
	\inf_{q \in (0,\infty)} 
	\sup_{(t, x) \in [0, T] \times \R^{\mathfrak{d}_i}} 
	\frac{ | u_i(t, x) | }{ 1 + \norm{x}_{\R^{\mathfrak{d}_i}}^q }
<
	\infty
$,
and which satisfies that $u_i|_{(0,T) \times \R^{\mathfrak{d}_i}}$ is a viscosity solution of
\begin{equation}
\begin{split}
	(\tfrac{\partial }{\partial t}u_i)(t,x) 
= 
	\tfrac{1}{2} 
	\operatorname{Trace}\! \big( 
		\sigma_i(x)[\sigma_i(x)]^{\ast}(\operatorname{Hess}_x u_i )(t,x)
	\big) 
	+
	\langle (\nabla_x u_i)(t,x),\mu_i(x)\rangle_{\R^{\mathfrak{d}_i}}
\end{split}
\end{equation}
for $(t,x) \in (0,T) \times \R^{\mathfrak{d}_i}$.
This proves item~\eqref{NN_approx_I:item1}.
Next observe that Corollary~\ref{NN_approx_family2} 
(with
$ d = \mathfrak{d}_i $,
$ T = T $,
$ r = r $,
$ R = R $,
$ v = v $,
$ w = w $,
$ z = z $,
$ \mathbf{z} = \mathbf{z} $,
$ \mathfrak{c} = \mathfrak{c} $,
$ \mathbf{v} = \mathbf{v} $,
$p = p$,
$ \nu = \nu_{\mathfrak{d}_i}$,
$ \varphi = \varphi_i $,
$ \mu = \mu_i $,
$ \sigma = \sigma_i $,
$ (\phi_\delta)_{\delta \in (0,r]} = (\phi_{i, \delta})_{\delta \in (0,r]}$
in the notation of Corollary~\ref{NN_approx_family2})
demonstrates that there exists $(\psi_{i,\varepsilon})_{\varepsilon \in (0,R]} \subseteq \mathcal{N}$ such that
for all $\varepsilon \in (0,R]$ it holds that
\begin{equation}
	\mathcal{P}(\psi_{i,\varepsilon}) 
\leq
	  \mathcal{C}\,(\mathfrak{d}_i)^{z + w \mathbf{z} + 4v} \, \varepsilon^{-4-\mathbf{z}},
\end{equation}
\begin{equation}
	\mathscr{P}(\psi_{i,\varepsilon}) 
\leq
	\mathcal{C}\,(\mathfrak{d}_i)^{z + w \mathbf{z} + 2v} \, \varepsilon^{-2-\mathbf{z}} , 
\qquad
	\mathcal{R}(\psi_{i,\varepsilon}) \in\mathcal{M}_{\mathfrak{d}_i},
\end{equation}
\begin{equation}
\andq
	\left[
		\int_{\R^d}  
		\left|
			u_i(T,x) - ( \mathcal{R}(\psi_{i,\varepsilon}) ) (x)
		\right|^p \,
		\nu_{\mathfrak{d}_i}(dx)
	\right]^{\nicefrac{1}{p}} 
\leq
		\varepsilon.
\end{equation}
Combining this with \eqref{NN_approx_ComparingConstants} 
establishes item~\eqref{NN_approx_I:item2}.
The proof of Proposition~\ref{NN_approx_I} is thus completed.
\end{proof}

\begin{samepage}
\begin{cor}
\label{NN_approx_N}
Assume Setting~\ref{setting_NN},
let $T, r, R \in (0,\infty)$, $ \mathfrak{C}, L, v, w, z, \mathbf{z},\theta \in [0,\infty)$, $\mathfrak{c} \in [1,\infty)$, $\mathbf{v}, p \in [2, \infty)$, 
for every $d \in \N$ 
let $\langle \cdot, \cdot \rangle_{\R^d} \colon \R^d \times \R^d \to \R$ be the $d$-dimensional Euclidean scalar product,
for every $d \in \N$ 
	let $\left\| \cdot \right\|_{\R^d} \colon \R^d \to [0,\infty)$ be the $d$-dimensional Euclidean norm and
	let $\HSNormIndex{\cdot}{d}{d} \colon \R^{d \times d} \to [0,\infty)$ be the Hilbert-Schmidt norm on $\R^{d\times d}$,
let $\nu_d \colon \mathcal{B}(\R^d) \to [0,1]$, $d \in \N$, be probability measures,
assume that 
\begin{multline}
\mathfrak{C} 
= 
\big[ 
\mathfrak{c}
\max \{1, \tfrac{R}{r} \} \allowbreak
\max \!
\big\{ 
2(p-1)^{1/2}
\exp \!
\big(
3\mathbf{v} (1+ L^2 T (\sqrt{T} + \mathbf{v} p )^2 ) 
\big)\\
\big(
1
+ \allowbreak
\sup\nolimits_{d\in\N} \big[d^{-\theta}
\left[ \!
\textint_{\R^{d}}  
\Norm{x}_{\R^{d} }^{ p \mathbf{v} } \,
\nu_{d} (dx)
\right]^{ \!\nicefrac{1}{ p }}\big]
\big) ,R
\big\}
\big]^{5 + \mathbf{z}},
\end{multline}
let 
$\varphi_d \colon \R^{d} \to \R$, $d \in \N$, be continuous functions,   
let 
$\mu_d \colon \R^{d} \to \R^{d}$, $d \in \N$, 
and
$\sigma_d \colon \R^{d} \to \R^{d \times d}$, $d \in \N$, 
be functions which satisfy 
for all $d \in \N$, $x,y \in \R^{d}$, $\lambda \in \R$ that
\begin{equation}
\mu_d(\lambda x+y) + \lambda \mu_d(0) = \lambda\mu_d(x)+\mu_d(y),
\end{equation}
\begin{equation}
\sigma_d(\lambda x+y) + \lambda \sigma_d(0) = \lambda\sigma_d(x)+\sigma_d(y),
\end{equation}
and 
$\norm{{\mu_d(x)}}_{\R^{d}} + \HSNormIndex{{\sigma_d(x)}}{d}{d}
\leq  
L  (1 + \norm{x}_{\R^{d}}),$
and let $(\phi_{d, \delta})_{d \in \N, \, \delta \in (0,r]} \subseteq \mathcal{N}$ satisfy 
for all $d \in \N$, $\delta \in (0,r]$, $x \in \R^{d}$ that
$
	\mathcal{P}(\phi_{d, \delta}) \leq \mathfrak{c} \, d^z \delta^{-\mathbf{z}} 
$,
$
	\mathcal{R}(\phi_{d, \delta}) \in\mathcal{M}_{d}
$,
$	\left| 
		( \mathcal{R}(\phi_{d, \delta}) )  (x)
	\right| 
\leq 
	\mathfrak{c} \, d^v (1+\| x \|_{\R^d}^{\mathbf{v}})
$,
and
\begin{equation}
\label{NN_approx_N:ass1}
\left| 
		\varphi_d(x) - ( \mathcal{R}(\phi_{d, \delta}) )(x)
	\right|
\leq 
	\mathfrak{c} \, d^w \, \delta \, (1+ \| x \|_{\R^d}^{\mathbf{v}}).
\end{equation}
Then  
\begin{enumerate}[(i)]
\item \label{NN_approx_N:item1}
there exist unique continuous functions $u_d\colon [0,T]\allowbreak \times \R^{d} \to \R$, $d \in \N$, which satisfy
for all $d \in \N$, $x \in \R^{d}$ that 
$u_d(0,x) = \varphi_d(x)$,  
which satisfy for all $d \in \N$ that
$
	\inf_{q \in (0,\infty)} \allowbreak
	\sup_{(t, x) \in [0, T] \times \R^d} \allowbreak
	\frac{ | u_d(t, x) | }{ 1 + \norm{x}_{\R^d}^q }
<
	\infty
$,
and which satisfy for all $d \in \N$ that $u_d|_{(0,T) \times \R^{d}}$ is a viscosity solution of
\begin{equation}
\begin{split}
	(\tfrac{\partial }{\partial t}u_d)(t,x) 
&= 
	\tfrac{1}{2} 
	\operatorname{Trace}\! \big( 
		\sigma_d(x)[\sigma_d(x)]^{\ast}(\operatorname{Hess}_x u_d )(t,x)
	\big) 
	\\&+
	\langle (\nabla_x u_d)(t,x),\mu_d(x)\rangle_{\R^d}
\end{split}
\end{equation}
for $(t,x) \in (0,T) \times \R^{d}$
and

\item \label{NN_approx_N:item2}
there exist $(\psi_{d, \varepsilon})_{d \in \N, \,\varepsilon \in (0,R]} \subseteq \mathcal{N}$ such that
for all $d \in \N$, $\varepsilon \in (0,R]$ it holds that
\begin{equation}
		\mathcal{P}(\psi_{d, \varepsilon}) 
	\leq
	\mathfrak{C}  \, d^{(5 + \mathbf{z})\theta+z + w \mathbf{z} + 4v} \, \varepsilon^{-4-\mathbf{z}},
\end{equation}
\begin{equation}
	\mathscr{P}(\psi_{d, \varepsilon}) 
\leq
\mathfrak{C}\, d^{(5 + \mathbf{z})\theta+z + w \mathbf{z} + 2v} \, \varepsilon^{-2-\mathbf{z}},\qquad
 \mathcal{R}(\psi_{d, \varepsilon}) \in \mathcal{M}_d,
\end{equation}
\begin{equation}
\label{NN_approx_N:concl1} 
	\andq\left[
		\int_{\R^d}  
		\left|
			u_d(T,x) - ( \mathcal{R}(\psi_{d, \varepsilon}) ) (x)
		\right|^p \,
		\nu_{d}(dx)
	\right]^{\nicefrac{1}{p}} 
\leq
		\varepsilon.
\end{equation}
\end{enumerate}
\end{cor}
\end{samepage}

\begin{proof}[Proof of Corollary~\ref{NN_approx_N}]
Observe that Proposition~\ref{NN_approx_I}
(with 
$ \mathcal{I} = \N $,
$ \mathfrak{d} = \id_{\N} $,
$ T = T $,
$ r = r $,
$ R = R $,
$ L = L $,
$ v = v $,
$ w = w $,
$ z = z $,
$ \mathbf{z} = \mathbf{z} $,
$\theta=\theta$,
$ \mathfrak{c} = \mathfrak{c} $,
$ \mathbf{v} = \mathbf{v} $,
$p = p$,
$ (\nu_d)_{d \in \Image( \mathfrak{d})} = (\nu_d)_{d \in \N}$,
$ (\varphi_i)_{i \in \mathcal{I}} = (\varphi_d)_{d \in \N} $,
$ (\mu_i)_{i \in \mathcal{I}} = (\mu_d)_{d \in \N} $,
$ (\sigma_i)_{i \in \mathcal{I}} = (\sigma_d)_{d \in \N} $,
$ (\phi_{i, \delta})_{i \in \mathcal{I}, \, \delta \in (0,r]} = (\phi_{d, \delta})_{d \in \N, \, \delta \in (0,r]}$
in the notation of Proposition~\ref{NN_approx_I}) establishes items~\eqref{NN_approx_N:item1}--\eqref{NN_approx_N:item2}.
The proof of Corollary~\ref{NN_approx_N} is thus completed.
\end{proof}

\subsection[ANNs with continuous activation functions]{Artificial neural networks with continuous activation functions}\label{subsectionContinuousANN}
In this subsection we establish in Theorem~\ref{cont_NN_approx} below the main result of this article. Theorem~\ref{cont_NN_approx} proves, roughly speaking, that fully-connected artificial neural networks overcome the curse of dimensionality in the numerical approximation of Black-Scholes PDEs (see \eqref{cont_NN_approx:concl1}  in item~\eqref{cont_NN_approx:item2} in Theorem~\ref{cont_NN_approx} for details). In Theorem~\ref{cont_NN_approx} the approximation error between the solution of the PDE and the artificial neural network is measured by means of $L^p$-norms with respect to the general probability measures $\nu_d$, $d \in \N$, in Theorem~\ref{cont_NN_approx}. 
 To make Theorem~\ref{cont_NN_approx} easier accessible, we derive a simplified and specialized version of Theorem~\ref{cont_NN_approx} in Corollary~\ref{cont_NN_approxSimple} below. In particular, in Corollary~\ref{cont_NN_approxSimple} below we specialize Theorem~\ref{cont_NN_approx} to the case where the general probability measures $ \nu_d $, $ d\in \N $, are nothing else but the continuous uniform distribution on $[0,1]^d$. Our proof of Corollary~\ref{cont_NN_approxSimple} uses the elementary estimate in Lemma~\ref{lem:uniformDistribution} below. For the sake of completeness we also present in this subsection a detailed proof of Lemma~\ref{lem:uniformDistribution}.
\begin{theorem}
\label{cont_NN_approx}
Let $T, r, R \in (0,\infty)$, $v, w, z, \mathbf{z},\theta\in[0,\infty)$, $\mathfrak{c} \in [1,\infty)$, $\mathbf{v}, p \in [2, \infty)$, 
for every $d \in \N$ let $\langle \cdot, \cdot \rangle_{\R^d} \colon \R^d \times \R^d \to \R$ be the $d$-dimensional Euclidean scalar product,
for every $d\in\N$
	let $\left\| \cdot \right\|_{\R^d} \colon \R^d \to [0,\infty)$ be the $d$-dimensional Euclidean norm and
	let $\HSNormIndex{\cdot}{d}{d} \colon \R^{d \times d} \to [0,\infty)$ be the Hilbert-Schmidt norm on $\R^{d\times d}$,
let $\nu_d \colon \mathcal{B}(\R^d) \to [0,1]$, $d \in \N$, be probability measures,
let 
$\varphi_d \colon \R^{d} \to \R$, $d \in \N$, be continuous functions, 
let 
$\mu_d \colon \R^{d} \to \R^{d}$, $d \in \N$, 
and
$\sigma_d \colon \R^{d} \to \R^{d \times d}$, $d \in \N$ ,
be functions which satisfy 
for all $d \in \N$, $x,y \in \R^{d}$, $\lambda \in \R$ that
\begin{equation}
	\mu_d(\lambda x+y) + \lambda \mu_d(0) = \lambda\mu_d(x)+\mu_d(y),
\end{equation}
\begin{equation}
		\sigma_d(\lambda x+y) + \lambda \sigma_d(0) = \lambda\sigma_d(x)+\sigma_d(y),
\end{equation}
and
	$\norm{\mu_d(x)}_{\R^{d}} + \HSNormIndex{{\sigma_d(x)}}{d}{d}
\leq  
	\mathfrak{c} (1 + \norm{x}_{\R^{d}})$,
let 
\begin{equation}
\begin{split}
	\mathcal{N}
&=
	\cup_{\mathcal{L} \in \{2,3, \ldots \}}
	\cup_{ (l_0,l_1,\ldots, l_\mathcal{L}) \in ((\N^{\mathcal{L}}) \times \{ 1 \} ) }
		\left(
			\times_{k = 1}^\mathcal{L} (\R^{l_k \times l_{k-1}} \times \R^{l_k})
		\right),
\end{split}
\end{equation}
assume that 
$\sup_{d \in \N} \left[ d^{-\theta p}
\textint_{\R^d}  
\Norm{x}_{\R^d}^{ p \mathbf{v} } \,
\nu_{ d } (dx)\right]
< 
\infty
,$
let $\mathbf{A}_d  \in C(\R^d, \R^d)$, $d \in \N$, and $\mathbf{a} \in C(\R, \R)$ be functions which satisfy 
for all $d \in \N$, $x = (x_1,x_2, \ldots, x_d) \allowbreak\in \R^d$ that
\begin{equation}
	\mathbf{A}_d(x)
=
	(\mathbf{a}(x_1), \mathbf{a}(x_2), \ldots, \mathbf{a}(x_d)),
\end{equation} 
let 
$
	\mathcal{P}, \mathscr{P} \colon \mathcal{N} \to \N
$ 
and 
$
	\mathcal{R} \colon 
	\mathcal{N} 
	\to 
	\cup_{d = 1}^\infty C(\R^d, \R)
$ 
be the functions which satisfy
for all $ \mathcal{L} \in \{2, 3, \ldots \}$, $ (l_0,l_1,\ldots, l_\mathcal{L}) \in ((\N^{\mathcal{L}}) \times \{ 1 \}) $, 
%
$
	\Phi 
=
	((W_1, B_1), \ldots,\allowbreak (W_\mathcal{L}, B_\mathcal{L}))
=
	( 
		(W_k^{(i,j)})_{ i \in \{1, 2, \ldots, l_k \},  j \in \{1, 2, \ldots, l_{k-1} \}}, \allowbreak
		(B_k^{(i)})_{i \in \{1, 2, \ldots, l_k \}} 
	)_{k \in \{1, 2, \ldots, \mathcal{L} \} } 
\allowbreak\in  
	( \times_{k = 1}^\mathcal{L} \allowbreak(\R^{l_k \times l_{k-1}} \times \R^{l_k}))
$,
$x_0 \in \R^{l_0}, x_1 \in \R^{l_1}, \ldots, x_{\mathcal{L}-1} \in \R^{l_{\mathcal{L}-1}}$ 
with $\forall \, k \in \N \cap (0,\mathcal{L}) \colon x_k = \mathbf{A}_{l_k}(W_k x_{k-1} + B_k)$
that
\begin{equation}
	\mathcal{R}(\Phi) \in C(\R^{l_0}, \R),
\qquad
	( \mathcal{R}(\Phi) ) (x_0) = W_\mathcal{L} x_{\mathcal{L}-1} + B_\mathcal{L},
\end{equation}
\begin{equation}
	\mathscr{P}(\Phi) 
=
	\sum_{k = 1}^\mathcal{L} 
	\sum_{i = 1}^{l_k}
	\left(
		\mathbbm{1}_{\R \backslash \{ 0 \}} (B_k^{(i)})
		+
		\smallsum\limits_{j = 1}^{l_{k-1}}
				\mathbbm{1}_{\R \backslash \{ 0 \}} (W_k^{(i,j)})
	\right),
\end{equation} 
and
$
	\mathcal{P}(\Phi)
=
	\sum_{k = 1}^\mathcal{L} l_k(l_{k-1} + 1) 
$,
and let $(\phi_{d, \delta})_{d \in \N, \, \delta \in (0,r]} \subseteq \mathcal{N}$ satisfy 
for all $d \in \N$, $\delta \in (0,r]$, $x \in \R^{d}$ that
$
	\mathcal{P}(\phi_{d, \delta}) \leq \mathfrak{c} \, d^z \delta^{-\mathbf{z}} 
$,
$
	\mathcal{R}(\phi_{d, \delta}) \in C(\R^{d}, \R)
$,
$	\left| 
		( \mathcal{R}(\phi_{d, \delta}) )  (x)
	\right| 
\leq 
	\mathfrak{c} \, d^v (1+\| x \|_{\R^d}^{\mathbf{v}})
$,
and
\begin{equation}
\label{cont_NN_approx:ass1}
	\left| 
		\varphi_d(x) - ( \mathcal{R}(\phi_{d, \delta}) )(x)
	\right|
\leq 
	\mathfrak{c} \, d^w \, \delta \, (1+ \| x \|_{\R^d}^{\mathbf{v}}).
\end{equation}
Then  
\begin{enumerate}[(i)]

\item \label{cont_NN_approx:item1}
there exist unique continuous functions $u_d\colon [0,T]\allowbreak \times \R^{d} \to \R$, $d \in \N$, which satisfy
for all $d \in \N$, $x \in \R^{d}$ that 
$u_d(0,x) = \varphi_d(x)$,
which satisfy 
for all $d \in \N$ that
$
	\inf_{q \in (0,\infty)} \allowbreak
	\sup_{(t, x) \in [0, T] \times \R^d} \allowbreak
	\frac{ | u_d(t, x) | }{ 1 + \norm{x}_{\R^d}^q }
\allowbreak<
	\infty
$,
and which satisfy for all $d \in \N$ that $u_d|_{(0,T) \times \R^{d}}$ is a viscosity solution of
\begin{equation}
\begin{split}
	(\tfrac{\partial }{\partial t}u_d)(t,x) 
&= 
	\tfrac{1}{2} 
	\operatorname{Trace}\! \big( 
		\sigma_d(x)[\sigma_d(x)]^{\ast}(\operatorname{Hess}_x u_d )(t,x)
	\big) 
	\\&+
	\langle (\nabla_x u_d)(t,x),\mu_d(x)\rangle_{\R^d}
\end{split}
\end{equation}
for $(t,x) \in (0,T) \times \R^{d}$
and

\item \label{cont_NN_approx:item2}
there exist $\mathfrak{C} \in (0,\infty)$, $(\psi_{d, \varepsilon})_{d \in \N, \,\varepsilon \in (0,R]} \subseteq \mathcal{N}$ such that
for all $d \in \N$, $\varepsilon \in (0,R]$ it holds that
$
	\mathcal{P}(\psi_{d, \varepsilon}) 
\leq
	\mathfrak{C}  \, d^{(5 + \mathbf{z})\theta+z + w \mathbf{z} + 4v} \, \varepsilon^{-4-\mathbf{z}}
$,
$
	\mathscr{P}(\psi_{d, \varepsilon}) 
\leq
	\mathfrak{C} \, d^{(5 + \mathbf{z})\theta+z + w \mathbf{z} + 2v} \, \varepsilon^{-2-\mathbf{z}}
$,
$
	\mathcal{R}(\psi_{d, \varepsilon}) \in C(\R^{d}, \R)
$,
and
\begin{equation}
\label{cont_NN_approx:concl1} 
	\left[
		\int_{\R^d}  
		\left|
			u_d(T,x) - ( \mathcal{R}(\psi_{d, \varepsilon}) ) (x)
		\right|^p \,
		\nu_{d}(dx)
	\right]^{\nicefrac{1}{p}} 
\leq
		\varepsilon.
\end{equation}
\end{enumerate}
\end{theorem}

\begin{proof}[Proof of Theorem~\ref{cont_NN_approx}]
Throughout this proof 
let $\mathfrak{C} \in (0,\infty)$ be given by
\begin{equation}
\begin{split}
	\mathfrak{C} 
&= 
	\Big[ 
		\mathfrak{c}
		\max \{1, \tfrac{R}{r} \} \allowbreak
		\max \!
		\Big\{R, 
			2(p-1)^{1/2}
			\exp \!
			\big(
				3\mathbf{v} (1+ \mathfrak{c}^2 T (\sqrt{T} + \mathbf{v} p )^2 ) 
			\big) \\
&\qquad \cdot
			\big(
			1
			+ \allowbreak
			\sup\nolimits_{d\in\N} \big[d^{-\theta}
			\left[ \!
			\textint_{\R^{d}}  
			\Norm{x}_{\R^{d} }^{ p \mathbf{v} } \,
			\nu_{d} (dx)
			\right]^{ \!\nicefrac{1}{ p }}\big]
			\big)
		\Big\}
	\Big]^{5 + \mathbf{z}}
\end{split}
\end{equation} 
and 
for every $d \in \N$ 
let $\mathcal{M}_d$ be the set of all Borel measurable functions from $\R^d$ to $\R$.
Note that Corollary~\ref{NN_approx_N}
(with 
$ T = T $,
$ r = r $,
$ R = R $,
$ L = \mathfrak{c} $,
$ v = v $,
$ w = w $,
$ z = z $,
$ \mathbf{z} = \mathbf{z} $,
$\theta=\theta$,
$ \mathfrak{c} = \mathfrak{c} $,
$ \mathbf{v} = \mathbf{v} $,
$p = p$,
$ (\nu_d)_{d \in \N} = (\nu_d)_{d \in \N}$,
$ (\varphi_d)_{d \in \N} = (\varphi_d)_{d \in \N} $,
$ (\mu_d)_{d \in \N} = (\mu_d)_{d \in \N} $,
$ (\sigma_d)_{d \in \N} = (\sigma_d)_{d \in \N} $,
$ (\phi_{d, \delta})_{d \in \N, \, \delta \in (0,r]} = (\phi_{d, \delta})_{d \in \N, \, \delta \in (0,r]}$
in the notation of Corollary~\ref{NN_approx_N})
demonstrates that
there exist unique continuous functions $u_d\colon [0,T] \times \R^{d} \to \R$, $d \in \N$, 
which satisfy
for all $d \in \N$, $x \in \R^{d}$ that 
$u_d(0,x) = \varphi_d(x)$,
which satisfy 
for all $d \in \N$ that
$
	\inf_{q \in (0,\infty)} 
	\sup_{(t, x) \in [0, T] \times \R^d} 
	\frac{ | u_d(t, x) | }{ 1 + \norm{x}_{\R^d}^q }
<
	\infty
$,
and which satisfy for all $d \in \N$ that $u_d|_{(0,T) \times \R^{d}}$ is a viscosity solution of
\begin{equation}
\label{cont_NN_approx:eq1}
\begin{split}
	(\tfrac{\partial }{\partial t}u_d)(t,x) 
= 
	\tfrac{1}{2} 
	\operatorname{Trace}\! \big( 
		\sigma_d(x)[\sigma_d(x)]^{\ast}(\operatorname{Hess}_x u_d )(t,x)
	\big) 
	+
	\langle (\nabla_x u_d)(t,x),\mu_d(x)\rangle_{\R^d}
\end{split}
\end{equation}
for $(t,x) \in (0,T) \times \R^{d}$
and
that
there exist $(\psi_{d, \varepsilon})_{d \in \N, \,\varepsilon \in (0,R]} \subseteq \mathcal{N}$ such that
for all $d \in \N$, $\varepsilon \in (0,R]$ it holds that
$
	\mathcal{P}(\psi_{d, \varepsilon}) 
\leq
	\mathfrak{C} \, d^{(5 + \mathbf{z})\theta+z + w \mathbf{z} + 4v} \, \varepsilon^{-4-\mathbf{z}}
$,
$
	\mathscr{P}(\psi_{d, \varepsilon}) 
\leq
	\mathfrak{C} \, d^{(5 + \mathbf{z})\theta+z + w \mathbf{z} + 2v} \, \varepsilon^{-2-\mathbf{z}}
$,
$
	\mathcal{R}(\psi_{d, \varepsilon}) \in \mathcal{M}_d
$,
and
\begin{equation}
\label{cont_NN_approx:eq2}
	\left[
		\int_{\R^d}  
		\left|
			u_d(T,x) - ( \mathcal{R}(\psi_{d, \varepsilon}) ) (x)
		\right|^p \,
		\nu_{d}(dx)
	\right]^{\nicefrac{1}{p}} 
\leq
		\varepsilon.
\end{equation}
The fact that
$
	\Image(\mathcal{R}) \subseteq \cup_{d = 1}^\infty C(\R^d, \R)
$
hence demonstrates that 
for all $d \in \N$, $\varepsilon \in (0,R]$ it holds that 
$
	\mathcal{R}(\psi_{d, \varepsilon}) \in C(\R^d, \R) \subseteq  \mathcal{M}_d
$.
Combining this with \eqref{cont_NN_approx:eq1} and \eqref{cont_NN_approx:eq2} establishes items~\eqref{cont_NN_approx:item1}--\eqref{cont_NN_approx:item2}.
The proof of Theorem~\ref{cont_NN_approx} is thus completed.
\end{proof}

\begin{lemma}\label{lem:uniformDistribution}
	Let $d\in \N$, $p\in [2,\infty)$, $u\in\R$, $v\in(u,\infty)$, and let $\norm{\cdot} \colon \R^d \to [0,\infty)$ be the $d$-dimensional Euclidean norm. 
	Then 	
	\begin{equation}
		\frac{1}{(v-u)^d}\int_{[u,v]^d} \norm{x}^p\, dx
		\le d^{\nicefrac{p}{2}} \max\big\{\vert u\vert^p, \vert v\vert^p\big\}.
	\end{equation} 
\end{lemma}

\begin{proof}[Proof of Lemma~\ref{lem:uniformDistribution}]
	Observe that the H\"older inequality implies that for all $x=(x_1,x_2,\allowbreak\dots, x_d)\in\R^d$ it holds that 
	\begin{equation}\label{uniformDistribution:normEst}
	\begin{split}
			\norm{x}^2&=\sum_{i=1}^d \vert x_i\vert^2
	\le \left(\sum_{i=1}^d \vert x_i\vert^p\right)^{\!\nicefrac{2}{p}} \left(\sum_{i=1}^d 1\right)^{\!1-\nicefrac{2}{p}}
	\\&=\left(\sum_{i=1}^d \vert x_i\vert^p\right)^{\!\nicefrac{2}{p}} d^{1-\nicefrac{2}{p}}.
	\end{split}
	\end{equation}
	Next note that the Fubini theorem ensures that
	\begin{equation}\label{uniformDistribution:intEst}
	\begin{split}
	&\int_{[u,v]^d} \left(\sum_{i=1}^d \vert x_i\vert^p\right) d(x_1,x_2,\dots,x_d)
	= \sum_{i=1}^d \int_{[u,v]^d} \vert x_i\vert^p\, d(x_1,x_2,\dots,x_d)
	\\&= \sum_{i=1}^d \left(\int_u^v \vert x_i\vert^p\, dx_i \right)\left(\int_u^v 1 \,dt\right)^{\!d-1}
	=  d \left(\int_u^v \vert t\vert^p\, dt\right) (v-u)^{d-1}
	\\&\le d (v-u)^{d} \sup_{t\in[u,v]}\big[ \vert t\vert^p\big]
	=d (v-u)^{d} \max\big\{\vert u\vert^p, \vert v\vert^p\big\}.
	\end{split}
	\end{equation}
	Combining this with \eqref{uniformDistribution:normEst} demonstrates that
	\begin{equation}
	\begin{split}
	&\frac{1}{(v-u)^d}\int_{[u,v]^d} \norm{x}^p\, dx
	\\&\le \frac{1}{(v-u)^d}\, d^{\nicefrac{p}{2}-1}\int_{[u,v]^d} \left(\sum_{i=1}^d \vert x_i\vert^p\right) d(x_1,x_2,\dots,x_d)
	\\&\le d^{\nicefrac{p}{2}} \max\big\{\vert u\vert^p, \vert v\vert^p\big\}.
	\end{split}
	\end{equation}
	The proof of Lemma~\ref{lem:uniformDistribution} is thus completed.
\end{proof}

\begin{cor}
	\label{cont_NN_approxSimple}
	Let $T, r, \mathfrak{c},p \in (0,\infty)$,  
	for every $d\in\N$
	let $\left\| \cdot \right\|_{\R^d} \colon \R^d \to [0,\infty)$ be the $d$-dimensional Euclidean norm and
	let $\HSNormIndex{\cdot}{d}{d} \colon \R^{d \times d} \to [0,\infty)$ be the Hilbert-Schmidt norm on $\R^{d\times d}$,
	let 
	$\varphi_d \colon \R^{d} \to \R$, $d \in \N$, be continuous functions, 
	let 
	$\mu_d \colon \R^{d} \to \R^{d}$, $d \in \N$, 
	and
	$\sigma_d \colon \R^{d} \to \R^{d \times d}$, $d \in \N$,
	be functions which satisfy 
	for all $d \in \N$, $x,y \in \R^{d}$, $\lambda \in \R$ that
	\begin{equation}
	\mu_d(\lambda x+y) + \lambda \mu_d(0) = \lambda\mu_d(x)+\mu_d(y),
	\end{equation}
	\begin{equation}
	\sigma_d(\lambda x+y) + \lambda \sigma_d(0) = \lambda\sigma_d(x)+\sigma_d(y),
	\end{equation}
	and
	$\norm{\mu_d(x)}_{\R^{d}} + \HSNormIndex{{\sigma_d(x)}}{d}{d}
	\leq  
	\mathfrak{c} (1 + \norm{x}_{\R^{d}})$,
	let 
	\begin{equation}
	\begin{split}
	\mathcal{N}
	&=
	\cup_{\mathcal{L} \in \{2,3, \ldots \}}
	\cup_{ (l_0,l_1,\ldots, l_\mathcal{L}) \in ((\N^{\mathcal{L}}) \times \{ 1 \} ) }
	\left(
	\times_{k = 1}^\mathcal{L} (\R^{l_k \times l_{k-1}} \times \R^{l_k})
	\right),
	\end{split}
	\end{equation}
	let $\mathbf{A}_d  \in C(\R^d, \R^d)$, $d \in \N$, and $\mathbf{a} \in C(\R, \R)$ be functions which satisfy 
	for all $d \in \N$, $x = (x_1,x_2, \ldots, x_d) \in \R^d$ that
	\begin{equation}
	\mathbf{A}_d(x)
	=
	(\mathbf{a}(x_1), \mathbf{a}(x_2), \ldots, \mathbf{a}(x_d)),
	\end{equation} 
	let 
	$
\mathcal{P}\colon \mathcal{N} \to \N
$ 
	and 
	$
	\mathcal{R} \colon 
	\mathcal{N} 
	\to 
	\cup_{d = 1}^\infty C(\R^d, \R)
	$ 
	be the functions which satisfy
	for all $ \mathcal{L} \in \{2, 3, \ldots \}$, $ (l_0,l_1,\ldots, l_\mathcal{L}) \in ((\N^{\mathcal{L}}) \times \{ 1 \}) $, 
	%
	$
	\Phi 
	=
	((W_1, B_1), \ldots,\allowbreak (W_\mathcal{L}, B_\mathcal{L}))
	=
	( 
	(W_k^{(i,j)})_{ i \in \{1, 2, \ldots, l_k \},  j \in \{1, 2, \ldots, l_{k-1} \}}, \allowbreak
	(B_k^{(i)})_{i \in \{1, 2, \ldots, l_k \}} 
	)_{k \in \{1, 2, \ldots, \mathcal{L} \} } 
	\allowbreak\in  
	( \times_{k = 1}^\mathcal{L} \allowbreak(\R^{l_k \times l_{k-1}} \times \R^{l_k}))
	$,
	$x_0 \in \R^{l_0}, x_1 \in \R^{l_1}, \ldots, x_{\mathcal{L}-1} \in \R^{l_{\mathcal{L}-1}}$ 
	with $\forall \, k \in \N \cap (0,\mathcal{L}) \colon x_k = \mathbf{A}_{l_k}(W_k x_{k-1} + B_k)$
	that
	\begin{equation}
	\mathcal{R}(\Phi) \in C(\R^{l_0}, \R),
	\qquad
	( \mathcal{R}(\Phi) ) (x_0) = W_\mathcal{L} x_{\mathcal{L}-1} + B_\mathcal{L},
	\end{equation}
	and
	$
	\mathcal{P}(\Phi)
	=
	\sum_{k = 1}^\mathcal{L} l_k(l_{k-1} + 1) 
	$,
	and let $(\phi_{d, \delta})_{d \in \N, \, \delta \in (0,r]} \subseteq \mathcal{N}$ satisfy 
	for all $d \in \N$, $\delta \in (0,r]$, $x \in \R^{d}$ that
	$
	\mathcal{P}(\phi_{d, \delta}) \leq \mathfrak{c} \, d^{\mathfrak{c}} \delta^{-{\mathfrak{c}}} 
	$,
	$
	\mathcal{R}(\phi_{d, \delta}) \in C(\R^{d}, \R)
	$,
	$	\left| 
	( \mathcal{R}(\phi_{d, \delta}) )  (x)
	\right| 
	\leq 
	\mathfrak{c} \, d^{\mathfrak{c}} (1+\| x \|_{\R^d}^{\mathfrak{c}})
	$,
	and
	\begin{equation}
	\label{cont_NN_approxSimple:ass1}
	\left| 
	\varphi_d(x) - ( \mathcal{R}(\phi_{d, \delta}) )(x)
	\right|
	\leq 
	\mathfrak{c} \, d^{\mathfrak{c}} \, \delta \, (1+ \| x \|_{\R^d}^{\mathfrak{c}}).
	\end{equation}
	Then  
	\begin{enumerate}[(i)]
		
		\item \label{cont_NN_approxSimple:item1}
		there exist unique continuous functions $u_d\colon [0,T]\allowbreak \times \R^{d} \to \R$, $d \in \N$, which satisfy
		for all $d \in \N$, $x \in \R^{d}$ that 
		$u_d(0,x) = \varphi_d(x)$,
		which satisfy 
		for all $d \in \N$ that
		$
		\inf_{q \in (0,\infty)} \allowbreak
		\sup_{(t, x) \in [0, T] \times \R^d} \allowbreak
		\frac{ | u_d(t, x) | }{ 1 + \norm{x}_{\R^d}^q }
		\allowbreak<
		\infty
		$,
		and which satisfy for all $d \in \N$ that $u_d|_{(0,T) \times \R^{d}}$ is a viscosity solution of
		\begin{equation}
		\begin{split}
		(\tfrac{\partial }{\partial t}u_d)(t,x) 
		&= 
		\operatorname{Trace}\! \big( 
		\sigma_d(x)[\sigma_d(x)]^{\ast}(\operatorname{Hess}_x u_d )(t,x)
		\big) 
		\\&+
		(\tfrac{\partial }{\partial x}u_d)(t,x)\, \mu_d(x)
		\end{split}
		\end{equation}
		for $(t,x) \in (0,T) \times \R^{d}$
		and
		
		\item \label{cont_NN_approxSimple:item2}
		there exist $\mathfrak{C} \in (0,\infty)$, $(\psi_{d, \varepsilon})_{d \in \N, \,\varepsilon \in (0,1]} \subseteq \mathcal{N}$ such that
		for all $d \in \N$, $\varepsilon \in (0,1]$ it holds that
		$
		\mathcal{P}(\psi_{d, \varepsilon}) 
		\leq
		\mathfrak{C}  \, d^{\mathfrak{C}} \, \varepsilon^{-\mathfrak{C}}
		$,
		$
		\mathcal{R}(\psi_{d, \varepsilon}) \in C(\R^{d}, \R)
		$,
		and
		\begin{equation}
		\label{cont_NN_approxSimple:concl1} 
		\left[
		\int_{[0,1]^d}  
		\left|
		u_d(T,x) - ( \mathcal{R}(\psi_{d, \varepsilon}) ) (x)
		\right|^p \,
		dx
		\right]^{\nicefrac{1}{p}} 
		\leq
		\varepsilon.
		\end{equation}
	\end{enumerate}
\end{cor}

\begin{proof}[Proof of Corollary~\ref{cont_NN_approxSimple}]
	Throughout this proof let $m \colon (0,\infty) \to [2,\infty)$ be the function which satisfies for all $z \in (0,\infty)$ that 
	$
	m(z) = \max \{2, z \}
	$ and for every $d\in\N$ let $\nu_d \colon \mathcal{B}(\R^d)\allowbreak \to [0,\infty]$ be the $d$-dimensional Lebesgue measure.
Note that Lemma~\ref{lem:uniformDistribution} (with $d=d$, $p=m(\mathfrak{c})m(p)$, $u=0$, $v=1$, $\norm{\cdot}=\norm{\cdot}_{\R^d}$  in the notation of Lemma~\ref{lem:uniformDistribution}) implies for all $d\in\N$ that
\begin{equation}
\int_{[0,1]^d} \Norm{x}_{\R^d}^{(m(\mathfrak{c})\,m(p))}\, dx
\le d^{\nicefrac{1}{2}(m(p)\,m(\mathfrak{c}))}.
\end{equation} 	
This ensures that
	\begin{equation}
	\begin{split}
	&\sup_{d \in \N} \left[ d^{-\nicefrac{1}{2}(m(p)\,m(\mathfrak{c}))}
	\int_{\R^d}  
	\Norm{x}_{\R^d}^{(m(\mathfrak{c})\,m(p))} \,
	\nu_d|_{[0,1]^d} (dx)\right]
	\\&=\sup_{d \in \N} \left[ d^{-\nicefrac{1}{2}(m(p)\,m(\mathfrak{c}))}
	\int_{[0,1]^d} \Norm{x}_{\R^d}^{(m(\mathfrak{c})\,m(p))}\, dx\right]
	\le 1< 
	\infty
	.
	\end{split}
	\end{equation}	
Theorem~\ref{cont_NN_approx} 
	(with 
	$ T = T $,
	$ r = r $,
	$ R = 1 $,
	$ v = \mathfrak{c}	$, 
	$ w = \mathfrak{c} $,
	$ z = \mathfrak{c} $,
	$ \mathbf{z} = \mathfrak{c} $,
$\theta=m(\mathfrak{c})/{2}$,
	$ \mathfrak{c} = \max\{\sqrt{2}\,\mathfrak{c},1\} $,
	$ \mathbf{v} =m(\mathfrak{c})$,
	$p = m(p)$,
	$ (\nu_d)_{d \in \N} = (\nu_d|_{[0,1]^d})_{d \in \N}$,
	$ (\varphi_d)_{d \in \N} = (\varphi_d)_{d \in \N} $,
	$ (\mu_d)_{d \in \N} = (\mu_d)_{d \in \N} $,
	$ (\sigma_d)_{d \in \N} = (\sqrt{2}\,\sigma_d)_{d \in \N} $,
	$ (\phi_{d, \delta})_{d \in \N, \, \delta \in (0,r]} = (\phi_{d, \delta})_{d \in \N, \, \delta \in (0,r]}$
	in the notation of Theorem~\ref{cont_NN_approx}) 
	  hence ensures that 
	\begin{enumerate}[(A)]	
	\item \label{cont_NN_approxSimpleProof:item1}
	there exist unique continuous functions $u_d\colon [0,T]\allowbreak \times \R^{d} \to \R$, $d \in \N$, which satisfy
	for all $d \in \N$, $x \in \R^{d}$ that 
	$u_d(0,x) = \varphi_d(x)$,
	which satisfy 
	for all $d \in \N$ that
	$
	\inf_{q \in (0,\infty)} \allowbreak
	\sup_{(t, x) \in [0, T] \times \R^d} \allowbreak
	\frac{ | u_d(t, x) | }{ 1 + \norm{x}_{\R^d}^q }
	\allowbreak<
	\infty
	$,
	and which satisfy for all $d \in \N$ that $u_d|_{(0,T) \times \R^{d}}$ is a viscosity solution of
	\begin{equation}
	\begin{split}
	(\tfrac{\partial }{\partial t}u_d)(t,x) 
	&= 
	\operatorname{Trace}\! \big( 
	\sigma_d(x)[\sigma_d(x)]^{\ast}(\operatorname{Hess}_x u_d )(t,x)
	\big) 
	\\&+
	(\tfrac{\partial }{\partial x}u_d)(t,x)\, \mu_d(x)
	\end{split}
	\end{equation}
	for $(t,x) \in (0,T) \times \R^{d}$
	and	
	\item \label{cont_NN_approxSimpleProof:item2}
there exist $C \in (0,\infty)$, $(\psi_{d, \varepsilon})_{d \in \N, \,\varepsilon \in (0,1]} \subseteq \mathcal{N}$ such that
for all $d \in \N$, $\varepsilon \in (0,1]$ it holds that
$
\mathcal{P}(\psi_{d, \varepsilon}) 
\leq
C  \, d^{\nicefrac{1}{2}(5 + \mathfrak{c})m(\mathfrak{c})+\mathfrak{c} + \mathfrak{c}^2  + 4\mathfrak{c}} \, \varepsilon^{-4-\mathfrak{c}}
$,
$
\mathcal{R}(\psi_{d, \varepsilon}) \in C(\R^{d}, \R)
$,
and
\begin{equation}
\label{cont_NN_approx:concl1Proof} 
\left[
\int_{\R^d}  
\left|
u_d(T,x) - ( \mathcal{R}(\psi_{d, \varepsilon}) ) (x)
\right|^{m(p)} \,
\nu_{d}|_{[0,1]^d}(dx)
\right]^{\nicefrac{1}{m(p)}} 
\leq
\varepsilon.
\end{equation}
\end{enumerate}	 
This implies item \eqref{cont_NN_approxSimple:item1}. Moreover, note that
\eqref{cont_NN_approxSimpleProof:item2} and the H\"older inequality demonstrate that for all 
$
\mathfrak{C}\in \left[\max\!\left\{C,4+\mathfrak{c},\nicefrac{1}{2}(5 + \mathfrak{c})m({\mathfrak{c}})+\mathfrak{c} + \mathfrak{c}^2  + 4\mathfrak{c} \right\},\infty\right)$, $d \in \N$, $\varepsilon \in (0,1]$
it holds that 
\begin{equation}\label{cont_NN_approxSimple:numberOfParameters}
	\mathcal{P}(\psi_{d, \varepsilon}) 
	\leq
	C  \, d^{\nicefrac{1}{2}(5 + \mathfrak{c})m(\mathfrak{c})+\mathfrak{c} + \mathfrak{c}^2  + 4\mathfrak{c}} \, \varepsilon^{-4-\mathfrak{c}}
	\le \mathfrak{C}  \, d^{\mathfrak{C}} \, \varepsilon^{-(4+\mathfrak{c})}
		\le \mathfrak{C}  \, d^{\mathfrak{C}} \, \varepsilon^{-\mathfrak{C}}
\end{equation}
and 
\begin{equation}
\begin{split}
&\left[
\int_{[0,1]^d}  
\left|
u_d(T,x) - ( \mathcal{R}(\psi_{d, \varepsilon}) ) (x)
\right|^p \,
dx
\right]^{\nicefrac{1}{p}} 
\\&\le\left[
\int_{[0,1]^d}  
\left|
u_d(T,x) - ( \mathcal{R}(\psi_{d, \varepsilon}) ) (x)
\right|^{m(p)} \,
dx
\right]^{\nicefrac{1}{m(p)}} 
\\&= \left[
\int_{\R^d}  
\left|
u_d(T,x) - ( \mathcal{R}(\psi_{d, \varepsilon}) ) (x)
\right|^{m(p)} \,
\nu_{d}|_{[0,1]^d}(dx)
\right]^{\nicefrac{1}{m(p)}} 
\leq
\varepsilon.
\end{split}
\end{equation}
This establishes item~\eqref{cont_NN_approxSimple:item2}.
	The proof of Corollary~\ref{cont_NN_approxSimple} is thus completed.
\end{proof}
\section[ANN approximations for Black-Scholes PDEs]{Artificial neural network approximations for Black-Scholes partial differential equations}\label{blackScholesSubsection}

\subsection{Elementary properties of the Black-Scholes model}

In this subsection we establish in Lemma~\ref{BS_properties} below a few elementary properties of the coefficient functions in the Black-Scholes model. For the sake of completeness we also provide in this subsection a detailed proof of Lemma~\ref{BS_properties}.

\begin{setting}
\label{BS_setting}
Let $p \in [2,\infty)$, $T \in (0,\infty)$, $\theta\in [0,\infty)$,
$(\alpha_{d, i})_{d \in \N, i \in \{1, 2, \ldots, d \}}$, $(\beta_{d, i})_{d \in \N, i \in \{1, 2, \ldots, d \}} \subseteq \R$ satisfy that
$
\sup_{d \in \N, i \in \{1, 2, \ldots, d \}} 
		(| \alpha_{d, i} |
		+
		| \beta_{d, i} |)
<
	\infty
$,
for every $d \in \N$ let $\left\| \cdot \right\|_{\R^d} \colon \R^d \to [0,\infty)$ be the $d$-dimensional Euclidean norm, 
for every $d \in \N$ 
let $\langle \cdot, \cdot \rangle_{\R^d} \colon \R^d \times \R^d \to \R$ be the $d$-dimensional Euclidean scalar product, 
for every $d \in \N$ let $\HSNormIndex{\cdot}{d}{d}\colon \allowbreak \R^{d \times d} \to [0,\infty)$ be the Hilbert-Schmidt norm on $\R^{d\times d}$,
let $e_{d, i} \in \R^d$, $d \in \N$, $ i \in \{1, 2, \ldots, d \}$,   satisfy
for all $d \in \N$ that
$
	e_{d,1} = (1, 0, \ldots, 0)$,
	$e_{d,2}= (0, 1,0, \ldots, 0)$,
	$\ldots,$
	$e_{d,d} = (0, \ldots, 0, 1)$,
let 
$\mathbf{B}_d= (\mathbf{B}_d^{(i,j)})_{i,j \in \{1, 2, \ldots, d \}} \in \R^{d \times d}$, $d \in \N$,	
satisfy for all $d\in\N$, $ i \in \{1, 2, \ldots, d \}$  that 
$\langle e_{d,i} , \mathbf{B}_d^{} \mathbf{B}_d^* e_{d,i} \rangle_{\R^d}=1$,
let $\mu_d \colon \R^d \to \R^d$, $d \in \N$, and $\sigma_d \colon \R^d \to \R^{d \times d}$, $d \in \N$,  be the functions which satisfy 
for all $d \in \N$, $x = (x_1, x_2, \ldots, x_d) \in \R^d$ that 
\begin{equation}\label{BS_settingMuSigma}
	\mu_d(x) = (\alpha_{d, 1} x_1, \ldots, \alpha_{d, d} x_d)\qandqShort \sigma_d(x) = \diag(\beta_{d, 1} x_1, \ldots, \beta_{d, d} x_d) \mathbf{B}_d,
\end{equation}
let $\nu_d \colon \mathcal{B}(\R^d) \to [0,1]$, $d \in \N$, be probability measures which satisfy for all $q \in (0,\infty)$ that
\begin{equation}
		\sup_{d \in \N} \left[d^{-\theta q}
	\textint_{\R^d}  
	\Norm{x}_{\R^d}^{ q } \,
	\nu_{ d } (dx)\right]
	< 
	\infty,
\end{equation}
let 
\begin{equation}
\begin{split}
	\mathcal{N}
&=
	\cup_{\mathcal{L} \in \{2,3, \ldots \}}
	\cup_{ (l_0,l_1,\ldots, l_\mathcal{L}) \in ((\N^{\mathcal{L}}) \times \{ 1 \} ) }
		\left(
			\times_{k = 1}^\mathcal{L} (\R^{l_k \times l_{k-1}} \times \R^{l_k})
		\right),
\end{split}
\end{equation}
let $\mathbf{A}_d  \in C(\R^d, \R^d)$, $d \in \N$, be the functions which satisfy 
for all $d \in \N$, $x = (x_1,x_2, \ldots, x_d) \in \R^d$ that
\begin{equation}
\label{BS_setting:eq2}
	\mathbf{A}_d(x)
=
	(\max\{x_1, 0\}, \max\{x_2, 0\}, \ldots, \max\{x_d, 0\}),
\end{equation} 
and 
let 
$
	\mathscr{P},\mathcal{P}  \colon \mathcal{N} \to \N
$ 
and 
$
	\mathcal{R} \colon 
	\mathcal{N} 
	\to 
	\cup_{d = 1}^\infty C(\R^d, \R)
$ 
be the functions which satisfy
for all $ \mathcal{L} \in \{2, 3, \ldots \}$, $ (l_0,l_1,\ldots, l_\mathcal{L}) \in ((\N^{\mathcal{L}}) \times \{ 1 \}) $, 
$
	\Phi 
=\allowbreak
	((W_1, B_1), \ldots,\allowbreak (W_\mathcal{L}, B_\mathcal{L}))
=\allowbreak
	( 
		(W_k^{(i,j)}\allowbreak)_{ i \in \{1, 2, \ldots, l_k \},   j \in \{1, 2, \ldots, l_{k-1} \}}, \allowbreak
		(B_k^{(i)})_{i \in \{1, 2, \ldots, l_k \}} 
	)_{k \in \{1, 2, \ldots, \mathcal{L} \} } \allowbreak
\in  
	( \times_{k = 1}^\mathcal{L}\allowbreak (\R^{l_k \times l_{k-1}} \times \R^{l_k}))
$,
$x_0 \in \R^{l_0}, x_1 \in \R^{l_1}, \ldots, x_{\mathcal{L}-1} \in \R^{l_{\mathcal{L}-1}}$ 
with 
$
	\forall \, k \in \N \cap (0,\mathcal{L}) \colon x_k = \mathbf{A}_{l_k}(W_k x_{k-1} + B_k)
$
that 
\begin{equation}
\label{BS_setting:eq3}
	\mathcal{R}(\Phi) \in C(\R^{l_0}, \R),
\qquad
	( \mathcal{R}(\Phi) ) (x_0) = W_\mathcal{L} x_{\mathcal{L}-1} + B_\mathcal{L},
\end{equation}
\begin{equation}
	\mathscr{P}(\Phi) 
=
	\sum_{k = 1}^\mathcal{L} 
	\sum_{i = 1}^{l_k}
	\left(
		\mathbbm{1}_{\R \backslash \{ 0 \}} (B_k^{(i)})
		+
		\smallsum\limits_{j = 1}^{l_{k-1}}
				\mathbbm{1}_{\R \backslash \{ 0 \}} (W_k^{(i,j)})
	\right),
\end{equation} 
and
$
	\mathcal{P}(\Phi)
=
	\sum_{k = 1}^\mathcal{L} l_k(l_{k-1} + 1)
$.
\end{setting}

\begin{lemma}
\label{BS_properties}
Assume Setting~\ref{BS_setting}.
Then
\begin{enumerate}[(i)]
\item \label{BS_properties:item1}
it holds for all $d \in \N$, $x,y \in \R^{d}$, $\lambda \in \R$ that
\begin{equation}
	\mu_d(\lambda x+y) + \lambda \mu_d(0) = \lambda\mu_d(x)+\mu_d(y),
\end{equation}

\item \label{BS_properties:item2}
it holds for all $d \in \N$, $x,y \in \R^{d}$, $\lambda \in \R$ that
\begin{equation}
	\sigma_d(\lambda x+y) + \lambda \sigma_d(0) = \lambda\sigma_d(x)+\sigma_d(y),
\end{equation}
and
\item \label{BS_properties:item3}
for all $d \in \N$, $x \in \R^{d}$ it holds that
\begin{equation}
\begin{split}
	&\norm{\mu_d(x)}_{\R^{d}} + \HSNormIndex{{\sigma_d(x)}}{d}{d}
	\leq  
	2\left[	
	\sup_{d \in \N, i \in \{1, 2, \ldots, d \}} 
	(| \alpha_{d, i} |
	+ 
	| \beta_{d, i} |)
	\right]
	\norm{x}_{\R^{d}} \\
&\leq  
	2\left[	
		\sup_{d \in \N, i \in \{1, 2, \ldots, d \}} 
		(| \alpha_{d, i} |
		+ 
		| \beta_{d, i} |)
	\right]
	(1 + \norm{x}_{\R^{d}}) 
< 
	\infty.
\end{split}
\end{equation}
\end{enumerate}
\end{lemma}

\begin{proof}[Proof of Lemma~\ref{BS_properties}]
Throughout this proof let $L \in (0,\infty)$ be given by
\begin{equation}
\label{BS_properties:setting1}
	L 
= 	
	\sup_{d \in \N} \sup_{i \in \{1, 2, \ldots, d \}} 
		(| \alpha_{d, i} |
		+ 
		| \beta_{d, i} |).
\end{equation}
First, note that the fact that 
for all $d \in \N$, $x \in \R^{d}$ it holds that
$
	\mu_d(x) = \diag(\alpha_{d,1}, \ldots, \alpha_{d,d}) x
$
and Lemma~\ref{affine_property} prove item~\eqref{BS_properties:item1}.
Moreover, observe that \eqref{BS_settingMuSigma} implies that 
for all $d \in \N$, $x,y \in \R^{d}$, $\lambda \in \R$ it holds that $\sigma_d(0)=0$ and 
\begin{equation}
\sigma_d(\lambda x+y) = \lambda\sigma_d(x)+\sigma_d(y).
\end{equation}
This establishes item~\eqref{BS_properties:item2}.
In addition, note that 
for all $d \in \N$, $x = (x_1, \ldots, x_d) \in \R^{d}$ it holds that
\begin{equation}
\label{BS_properties:eq1}
\begin{split}
	\norm{\mu_d(x)}_{\R^{d}}
&=
	\norm{(\alpha_{d, 1} x_1, \ldots, \alpha_{d, d} x_d)}_{\R^{d}}
=
	\left[ 
		\smallsum_{i = 1}^d  |  \alpha_{d, i} x_i |^2
	\right]^{\nicefrac{1}{2}} \\
&\leq
	\left[ 
		\big( \! \max \{ |  \alpha_{d, 1}  |, \ldots,  |  \alpha_{d, d}  | \} \big)^2
		\smallsum_{i = 1}^d  | x_i |^2
	\right]^{\nicefrac{1}{2}} \\
&=
	\max \{ |  \alpha_{d, 1}  |, \ldots,  |  \alpha_{d, d}  | \}
	\norm{x}_{\R^{d}} 
\leq
	L \norm{x}_{\R^{d}} \\
&\leq
	 L (1 + \norm{x}_{\R^{d}})
< 
	\infty.
\end{split}
\end{equation}
Moreover, observe that the fact that 
for all $d \in \N$, $x = (x_1, \ldots, x_d) \in \R^{d}$ it holds that
$
	\sigma_d(x)
=
	(\beta_{d, i} x_i \mathbf{B}_{d}^{(i, j)})_{i,j \in \{1, 2, \ldots, d \}}
\in 
	\R^{d \times d}
$
assures that 
for all $d \in \N$, $x = (x_1, \ldots, x_d) \in \R^{d}$ 
it holds that
\begin{equation}
\label{BS_properties:eq2}
\begin{split}
	\HSNormIndex{{\sigma_d(x)}}{d}{d}^2
&=
	\smallsum\limits_{i,j = 1}^d  \big|  \beta_{d, i} x_i \mathbf{B}_{d}^{(i, j)}  \big|^2  \\
&=
	\smallsum\limits_{i = 1}^d  
	\left(
		 | x_i  |^2  |\beta_{d, i} |^2
		\smallsum_{j = 1}^d \big| \mathbf{B}_{d}^{(i, j)}  \big|^2  
	\right) \\
&\leq
	\left[
		\max\limits_{i \in \{1, 2, \ldots, d \} }
		\left(
			 |\beta_{d, i} |^2\smallsum_{j = 1}^d \big| \mathbf{B}_{d}^{(i, j)}  \big|^2  
		\right) 
	\right]
	\smallsum\limits_{i = 1}^d  
		 | x_i  |^2 \\
&=	 
	\left[ 
		\max\limits_{i \in \{1, 2, \ldots, d \} }
		\left(
			 | \beta_{d, i} |^2\smallsum_{j = 1}^d \big| \mathbf{B}_{d}^{(i, j)}  \big|^2  
		\right)
	\right]
	\norm{x}_{\R^d}^2.
\end{split}
\end{equation}
The fact that 
for all $d \in \N$, $i \in \{1, 2, \ldots, d \}$ it holds that
\begin{equation}
	\smallsum_{j = 1}^d \big| \mathbf{\mathbf{B}}_{d}^{(i, j)}  \big|^2 
=
	\langle \mathbf{B}_d^* e_{d,i} , \mathbf{B}_d^* e_{d,i} \rangle_{\R^d}
=
	\langle e_{d,i} , \mathbf{B}_d^{} \mathbf{B}_d^* e_{d,i} \rangle_{\R^d}
=
	1
\end{equation}
hence demonstrates that
for all $d \in \N$, $x = (x_1, \ldots, x_d) \in \R^{d}$ 
it holds that
\begin{equation}
\label{BS_properties:eq3}
\begin{split}
\HSNormIndex{{\sigma_d(x)}}{d}{d}
&\leq	
	\left[ 
		\max\limits_{i \in \{1, 2, \ldots, d \} }
			 | \beta_{d, i} |^2
	\right]^{ \nicefrac{1}{2}}
	\norm{x}_{\R^d} \\
&=	
	\left[ 
		\max\limits_{i \in \{1, 2, \ldots, d \} }
			 | \beta_{d, i} | 
	\right]
	\norm{x}_{\R^d} \\
&\leq
	L \norm{x}_{\R^{d}}
\leq
	 L (1 + \norm{x}_{\R^{d}})
< 
	\infty.
\end{split}
\end{equation}
Combining this and \eqref{BS_properties:eq1} assures that 
for all $d \in \N$, $x \in \R^{d}$ it holds that
\begin{equation}
\begin{split}
	&\norm{\mu_d(x)}_{\R^{d}} + \HSNormIndex{{\sigma_d(x)}}{d}{d}
	\leq  
	2L \norm{x}_{\R^{d}}
\leq  
	2L
	(1 + \norm{x}_{\R^{d}}) 
< 
	\infty.
\end{split}
\end{equation}
This establishes item~\eqref{BS_properties:item3}.
The proof of Lemma~\ref{BS_properties} is thus completed.
\end{proof}


\subsection{Transformations of viscosity solutions}

In this subsection we establish in Proposition~\ref{lemBlackScholesTimeChange}, Corollary~\ref{BS_endvalueOneDirection}, and Corollary~\ref{BS_endvalue} a few elementary and essentially well-known transformation results for viscosity solutions of certain second-order PDEs.


\begin{prop}\label{lemBlackScholesTimeChange}
	Let $d\in\N$, $a,\lambda\in\R$, $b\in (a,\infty)$, let $f\colon (a,b)\times \R^d\times \R\times \R^d\times \R^{d\times d}\to \R$ be a function which satisfies for all $t\in (a,b)$, $x\in\R^d$, $\alpha\in\R$, $\eta\in\R^d$, $A,B\in \{C\in \R^{d\times d}\colon C^*=C\}$ with $A\le B$ that
	 \begin{equation}
	 	f(t,x,\alpha,\eta,A)\le f(t,x,\alpha,\eta,B),
	 \end{equation}
	let 
	$u\colon(a,b) \times \R^{d} \to \R$ 
	be a continuous function 
	which satisfies that $u$ is a viscosity solution of
	\begin{equation}\label{lemBlackScholesTimeChange:Equation}
	\begin{split}
	&\lambda (\tfrac{\partial }{\partial t}u)(t,x) 
=
f\big(t,x,u(t,x),(\nabla_x u )(t,x),(\operatorname{Hess}_x u )(t,x)
	\big) 
	\end{split}
	\end{equation}
	for $(t,x) \in (a,b) \times \R^{d}$,
	let $R\colon[a,b] \to [a,b]$ be the function which satisfies for all  $t\in[a,b]$ that $R(t)=a+b-t$,
	let $U\colon(a,b) \times \R^{d} \to \R$ be the function which satisfies for all  $t\in(a,b)$, $x\in\R^d$ that $U(t,x)=u(R(t),x)$,
	and	let $F\colon (a,b)\times \R^d\times \R\times \R^d\times \R^{d\times d}\to \R$ be the function which satisfies for all $t\in(a,b)$, $x\in\R^d$, $\alpha\in \R$, $\eta\in\R^d$, $A\in \R^{d\times d}$ that
	\begin{equation}\label{lemBlackScholesTimeChange:DefF}
	F(t,x,\alpha,\eta,A)=f(R(t),x,\alpha,\eta,A).
	\end{equation} 
	Then
	$U\colon(a,b) \times \R^{d} \to \R$ is a continuous function  
	which satisfies that $U$ is a viscosity solution of
	\begin{equation}
	\begin{split}
	&-\lambda (\tfrac{\partial }{\partial t}U)(t,x) 
	=
	F\big(t,x,U(t,x),(\nabla_x U )(t,x),(\operatorname{Hess}_x U )(t,x)
	\big)
	\end{split}
	\end{equation}
	for $(t,x) \in (a,b) \times \R^{d}$.
\end{prop}

\begin{proof}[Proof of Proposition~\ref{lemBlackScholesTimeChange}]
	Throughout this proof let $(s,y)\in (a,b)\times \R^d$, let
\begin{equation}
	\Psi=(\Psi(t,x))_{(t,x)\in (a,b)\times \R^d}, \Phi=(\Phi(t,x))_{(t,x)\in (a,b)\times \R^d}\colon (a,b)\times \R^d\to \R
\end{equation}
	 be twice continuously differentiable functions which satisfy that
	  $\Phi\ge U$, $\Phi(s,y)= U(s,y)$, $\Psi\le U$, and $ \Psi(s,y)= U(s,y)$, and
	let 
	\begin{equation}
	\psi=(\psi(t,x))_{(t,x)\in (a,b)\times \R^d}, \varphi=(\varphi(t,x))_{(t,x)\in (a,b)\times \R^d}\colon (a,b)\times \R^d\to \R
	\end{equation}	
	be the functions which satisfy for all $(t,x)\in (a,b)\times \R^d$ that
	\begin{equation}\label{lemBlackScholesTimeChange:propertiesPsiPhi}
		\psi(t,x)=\Psi(R(t),x)\qandq \varphi(t,x)=\Phi(R(t),x).
	\end{equation}
	Observe that $R\colon [a,b]\to[a,b]$ is a bijective function which satisfies that $R|_{(a,b)}\colon (a,b)\to(a,b)$ is twice continuously differentiable and which satisfies for all $t\in [a,b],$ $r\in (a,b)$ that
	\begin{equation}\label{lemBlackScholesTimeChange:propertiesR}
		R(R(t))=t\qandq R'(r)=-1.
	\end{equation}
	Combining this and \eqref{lemBlackScholesTimeChange:propertiesPsiPhi} ensures for all $(t,x)\in (a,b)\times \R^d$ that
	\begin{equation}\label{lemBlackScholesTimeChange:propertiesPsiPhiInverse}
	\Psi(t,x)=\psi(R(t),x)\qandq \Phi(t,x)=\varphi(R(t),x).
	\end{equation}
	Therefore, we obtain that \eqref{lemBlackScholesTimeChange:propertiesPsiPhi}, \eqref{lemBlackScholesTimeChange:propertiesR}, and the hypothesis that for all $(t,x)\in (a,b)\times \R^d$ it holds that  $U(t,x)=u(R(t),x)$  imply that for all $(t,x)\in (a,b)\times \R^d$ it holds that $\psi\in C^2((a,b)\times \R^d,\R)$, 
		\begin{equation}\label{lemBlackScholesTimeChange:psiTestsuOne}
	\psi(t,x)=\Psi(R(t),x)\le U(R(t),x)=u(R(R(t)),x)=u(t,x),
	\end{equation}	
	and
	\begin{equation}\label{lemBlackScholesTimeChange:psiTestsuTwo}
		\psi(R(s),y)=\Psi(s,y)=U(s,y)=u(R(s),y).
	\end{equation}	
	Moreover, observe that \eqref{lemBlackScholesTimeChange:propertiesPsiPhi}, \eqref{lemBlackScholesTimeChange:propertiesR}, \eqref{lemBlackScholesTimeChange:propertiesPsiPhiInverse}, and the hypothesis that for all $(t,x)\in (a,b)\times \R^d$ it holds that  $U(t,x)=u(R(t),x)$ demonstrate that for all $(t,x)\in (a,b)\times \R^d$ it holds that $\varphi\in C^2((a,b)\times \R^d,\R)$, 
	\begin{equation}\label{lemBlackScholesTimeChange:varphiTestsuOne}
	\varphi(t,x)=\Phi(R(t),x)\ge U(R(t),x)=u(R(R(t)),x)=u(t,x),
	\end{equation}
	and
	\begin{equation}\label{lemBlackScholesTimeChange:varphiTestsuTwo}
		\varphi(R(s),y)=\Phi(s,y)=U(s,y)=u(R(s),y).
	\end{equation}
	Combining this, \eqref{lemBlackScholesTimeChange:psiTestsuOne}, and \eqref{lemBlackScholesTimeChange:psiTestsuTwo} 
	with \eqref{lemBlackScholesTimeChange:Equation}
	implies that 
	\begin{equation}\label{lemBlackScholesTimeChange:uSubsol}
	\begin{split}
	&\lambda(\tfrac{\partial }{\partial t}\varphi)(R(s),y) 
\le
	f\big(R(s),y,\varphi(R(s),y),(\nabla_x \varphi )(R(s),y),(\operatorname{Hess}_x \varphi )(R(s),y)
	\big)
	\end{split}
	\end{equation}
	and 	
	\begin{equation}\label{lemBlackScholesTimeChange:uSupersol}
	\begin{split}
	&\lambda(\tfrac{\partial }{\partial t}\psi)(R(s),y) 
	\ge
	f\big(R(s),y,\psi(R(s),y),(\nabla_x \psi )(R(s),y),(\operatorname{Hess}_x \psi )(R(s),y)
	\big).
	\end{split}
	\end{equation}
	This, 
	\eqref{lemBlackScholesTimeChange:DefF},  \eqref{lemBlackScholesTimeChange:propertiesR}, and \eqref{lemBlackScholesTimeChange:propertiesPsiPhiInverse} ensure that 
	\begin{equation}\label{lemBlackScholesTimeChange:USubsol}
	\begin{split}
	&-\lambda(\tfrac{\partial }{\partial t}\Phi)(s,y)
	=\lambda (\tfrac{\partial }{\partial t}\varphi)(R(s),y) 
	\\&\le
	f\big(R(s),y,\varphi(R(s),y),(\nabla_x \varphi )(R(s),y),(\operatorname{Hess}_x \varphi )(R(s),y)
	\big)
	\\&=	f\big(R(s),y,\Phi(s,y),(\nabla_x \Phi )(s,y),(\operatorname{Hess}_x \Phi )(s,y)
	\big)
	\\&=F\big(s,y,\Phi(s,y),(\nabla_x \Phi )(s,y),(\operatorname{Hess}_x \Phi )(s,y)
\big).
	\end{split}
	\end{equation}
	Moreover, observe that \eqref{lemBlackScholesTimeChange:DefF}, \eqref{lemBlackScholesTimeChange:propertiesR}, \eqref{lemBlackScholesTimeChange:propertiesPsiPhiInverse}, and \eqref{lemBlackScholesTimeChange:uSupersol}
	demonstrate that
		\begin{equation}\label{lemBlackScholesTimeChange:USupersol}
	\begin{split}
	&-\lambda(\tfrac{\partial }{\partial t}\Psi)(s,y)
	=\lambda(\tfrac{\partial }{\partial t}\psi)(R(s),y) 
	\\&\ge
	f\big(R(s),y,\psi(R(s),y),(\nabla_x \psi )(R(s),y),(\operatorname{Hess}_x \psi )(R(s),y)
	\big)
	\\&=	f\big(R(s),y,\Psi(s,y),(\nabla_x \Psi )(s,y),(\operatorname{Hess}_x \Psi )(s,y)
	\big)
	\\&= F\big(s,y,\Psi(s,y),(\nabla_x \Psi )(s,y),(\operatorname{Hess}_x \Psi )(s,y)
	\big).
	\end{split}
	\end{equation}
	Next note that the hypothesis that $u\colon (a,b)\times \R^d\to \R$ is a continuous function and the hypothesis that for all  $t\in(a,b)$, $x\in\R^d$ it holds that $U(t,x)=u(R(t),x)$ demonstrate that   $U\colon (a,b)\times \R^d\to \R$ is a continuous function.
	Combining this with \eqref{lemBlackScholesTimeChange:USubsol} and \eqref{lemBlackScholesTimeChange:USupersol} assures that $U\colon (a,b)\times \R^d\to \R$ is a continuous function which satisfies that $U$ is a viscosity subsolution and a viscosity supersolution of 
	\begin{equation}
	\begin{split}
	&-\lambda(\tfrac{\partial }{\partial t}U)(t,x) 
	=
	F\big(t,x,U(t,x),(\nabla_x U )(t,x),(\operatorname{Hess}_x U )(t,x)
	\big)
	\end{split}
	\end{equation}
	for $(t,x) \in (a,b) \times \R^{d}$. This proves that $U\colon (a,b)\times \R^d\to \R$ is a continuous function which satisfies that $U$ is a viscosity solution of 
		\begin{equation}
	\begin{split}
	&-\lambda(\tfrac{\partial }{\partial t}U)(t,x) 
	=
	F\big(t,x,U(t,x),(\nabla_x U )(t,x),(\operatorname{Hess}_x U )(t,x)
	\big)
	\end{split}
	\end{equation}
	for $(t,x) \in (a,b) \times \R^{d}$.
	The proof of Proposition~\ref{lemBlackScholesTimeChange} is thus completed.
\end{proof}

\begin{cor}\label{BS_endvalueOneDirection}
	Let $d\in\N$, $a,\mathfrak{a},\mathfrak{b},\lambda\in\R$, $b\in \R \backslash \{a\}$, let $\nu \colon \mathcal{B}(\R^d) \to [0,\infty]$ be a measure, let $\varphi\colon \R^d\to\R$ be a continuous function, let $\Phi\colon \R^d\to\R$ be a $\mathcal{B}(\R^d)\backslash\mathcal{B}(\R)$-measurable function, let $f\colon \R^d\times \R\times \R^d\times \R^{d\times d}\to \R$ be a function which satisfies for all $x\in\R^d$, $\alpha\in\R$, $\eta\in\R^d$, $A,B\in \{C\in \R^{d\times d}\colon C^*=C\}$ with $A\le B$ that
\begin{equation}
f(x,\alpha,\eta,A)\le f(x,\alpha,\eta,B),
\end{equation}
assume that $\mathfrak{a}=\min\{a,b\}$ and $\mathfrak{b}=\max\{a,b\}$,
and assume that
there exists a unique continuous function $u \colon [\mathfrak{a},\mathfrak{b}]\allowbreak \times \R^{d} \to \R$ which satisfies
for all $x \in \R^{d}$ that 
$u(b,x) = \varphi(x)$,
which satisfies that
\begin{equation}\label{BS_endvalue:eqGrowthU}
	\inf_{q \in (0,\infty)} 
	\sup_{(t, x) \in [\mathfrak{a}, \mathfrak{b}] \times \R^d} 
	\frac{ | u(t, x) | }{ 1 + \norm{x}_{\R^d}^q }
	<
	\infty,
\end{equation}
and
which satisfies that $u|_{(\mathfrak{a}, \mathfrak{b}) \times \R^{d}}$ is a viscosity solution of
	\begin{equation}\label{BS_endvalue:eqViscosityU}
\begin{split}
&\lambda (\tfrac{\partial }{\partial t}u)(t,x) 
=
f\big(x,u(t,x),(\nabla_x u )(t,x),(\operatorname{Hess}_x u )(t,x)
\big) 
\end{split}
\end{equation}
for $(t,x) \in (\mathfrak{a},\mathfrak{b}) \times \R^{d}$
and it holds that
\begin{equation}\label{BS_endvalue:eqApproxU}
\left[
\int_{\R^d}  
\left|
u(a,x) - \Phi(x)
\right|^p \,
\nu(dx)
\right]^{\nicefrac{1}{p}} 
\leq
\varepsilon.
\end{equation}
Then there exists a unique continuous function $v\colon [\mathfrak{a},\mathfrak{b}] \allowbreak \times \R^{d} \to \R$ which satisfies
for all $x \in \R^{d}$ that 
$v(a,x) = \varphi(x)$,
which satisfies that
\begin{equation}
	\inf_{q \in (0,\infty)} 
	\sup_{(t, x) \in [\mathfrak{a}, \mathfrak{b}] \times \R^d} 
	\frac{ | v(t, x) | }{ 1 + \norm{x}_{\R^d}^q }
	<
	\infty,
\end{equation}
which satisfies that $v|_{(\mathfrak{a}, \mathfrak{b}) \times \R^{d}}$ is a viscosity solution of
	\begin{equation}\label{BS_endvalue:eqViscosityV}
\begin{split}
&-\lambda (\tfrac{\partial }{\partial t}v)(t,x) 
=
f\big(x,v(t,x),(\nabla_x v )(t,x),(\operatorname{Hess}_x v )(t,x)
\big) 
\end{split}
\end{equation}
for $(t,x) \in (\mathfrak{a},\mathfrak{b}) \times \R^{d}$
and it holds that
\begin{equation}\label{BS_endvalue:eqApproxV}
\left[
\int_{\R^d}  
\left|
v(b,x) - \Phi(x)
\right|^p \,
\nu(dx)
\right]^{\nicefrac{1}{p}} 
\leq
\varepsilon.
\end{equation}
\end{cor}

\begin{proof}[Proof of Corollary~\ref{BS_endvalueOneDirection}]
	Throughout this proof let $v\colon [\mathfrak{a},\mathfrak{b}] \allowbreak \times \R^{d} \to \R$ be the function which satisfies
	for all $t\in[\mathfrak{a},\mathfrak{b}]$, $x \in \R^{d}$ that $v(t,x)=u(\mathfrak{a}+\mathfrak{b}-t,x)$.
	Note that for all $t\in[\mathfrak{a},\mathfrak{b}]$, $x \in \R^{d}$ it holds that 
	\begin{equation}\label{BS_endvalue:Transformation}
		v(t,x)=u(\mathfrak{a}+\mathfrak{b}-t,x)=u(a+b-t,x).
	\end{equation}
	This and \eqref{BS_endvalue:eqApproxU} ensure that 
	\begin{equation}\label{BS_endvalue:eqApproxVProof}
	\left[
	\int_{\R^d}  
	\left|
	v(b,x) - \Phi(x)
	\right|^p \,
	\nu(dx)
	\right]^{\nicefrac{1}{p}} 
	\leq
	\varepsilon.
	\end{equation} 
	Next note that \eqref{BS_endvalue:eqViscosityU} and Proposition~\ref{lemBlackScholesTimeChange} (with 
	$d=d$, $a=\mathfrak{a}$, $\lambda=\lambda$, $b=\mathfrak{b}$,  
	$f(t,x,\alpha,\eta,A)=f(x,\alpha,\eta,A)$,
	$u(t,x)=u(t,x)$, $U(t,x)=v(t,x)$
	for $t\in (\mathfrak{a},\mathfrak{b})$, $x\in\R^d$, $\alpha\in\R$, $\eta\in \R^d$, $A\in \R^{d\times d}$
	in the notation of Proposition~\ref{lemBlackScholesTimeChange})  demonstrate that $v|_{(\mathfrak{a}, \mathfrak{b}) \times \R^{d}}$ is a viscosity solution of  
	\begin{equation}\label{BS_endvalue:ViscosityVProof}
\begin{split}
&-\lambda(\tfrac{\partial }{\partial t}v)(t,x) 
=
f\big(x,v(t,x),(\nabla_x v )(t,x),(\operatorname{Hess}_x v )(t,x)
\big) 
\end{split}
\end{equation}
	for $(t,x) \in (\mathfrak{a},\mathfrak{b}) \times \R^{d}$.
	Furthermore, observe that \eqref{BS_endvalue:eqGrowthU}, \eqref{BS_endvalue:Transformation}, and the hypothesis that for all $x\in\R^d$ it holds that $u(b,x) = \varphi(x)$ 
	imply that for all $x\in\R^d$ it holds that $v(a,x) = \varphi(x)$ and
	\begin{equation}\label{BS_endvalue:GrowthV}
	\begin{split}
		\inf_{q \in (0,\infty)} 
	\sup_{(t, x) \in [\mathfrak{a},\mathfrak{b}] \times \R^d} 
	\frac{ | v(t, x) | }{ 1 + \norm{x}_{\R^d}^q }
	&=	\inf_{q \in (0,\infty)} 
	\sup_{(t, x) \in [\mathfrak{a},\mathfrak{b}] \times \R^d} 
	\frac{ | u(\mathfrak{a}+\mathfrak{b}-t, x) | }{ 1 + \norm{x}_{\R^d}^q }
	\\&=	\inf_{q \in (0,\infty)} 
	\sup_{(t, x) \in [\mathfrak{a},\mathfrak{b}] \times \R^d} 
	\frac{ | u(t, x) | }{ 1 + \norm{x}_{\R^d}^q }
	<
	\infty.
	\end{split}
	\end{equation}
	Next let $w\colon [\mathfrak{a},\mathfrak{b}] \allowbreak \times \R^{d} \to \R$ be a continuous function which satisfies
	for all $x \in \R^{d}$ that 
	$w(a,x) = \varphi(x)$,
	which satisfies that
	$
	\inf_{q \in (0,\infty)} 
	\sup_{(t, x) \in [\mathfrak{a},\mathfrak{b}] \times \R^d} 
	\frac{ | w(t, x) | }{ 1 + \norm{x}_{\R^d}^q }
	<
	\infty
	$,  
	which satisfies that $w|_{(\mathfrak{a}, \mathfrak{b}) \times \R^{d}}$ is a viscosity solution of
	\begin{equation}
\begin{split}
&-\lambda(\tfrac{\partial }{\partial t}w)(t,x) 
=
f\big(x,w(t,x),(\nabla_x w )(t,x),(\operatorname{Hess}_x w )(t,x)
\big) 
\end{split}
\end{equation}
	for $(t,x) \in (\mathfrak{a},\mathfrak{b}) \times \R^{d}$,
	and which satisfies that 
	\begin{equation}\label{BS_endvalue:eqApproxW}
	\left[
	\int_{\R^d}  
	\left|
	w(b,x) - \Phi(x)
	\right|^p \,
	\nu(dx)
	\right]^{\nicefrac{1}{p}} 
	\leq
	\varepsilon,
	\end{equation}
	and let $z\colon [\mathfrak{a},\mathfrak{b}] \allowbreak \times \R^{d} \to \R$ be the function which satisfies
	for all $t\in [\mathfrak{a},\mathfrak{b}]$, $x \in \R^{d}$ that 
	\begin{equation}
		z(t,x) = w(\mathfrak{a}+\mathfrak{b}-t,x)=w(a+b-t,x).
	\end{equation}
	Observe that $z$ is a continuous function which satisfies that for all $x \in \R^{d}$ it holds that $z(b,x) = \varphi(x)$, which satisfies that
	$
	\inf_{q \in (0,\infty)} 
	\sup_{(t, x) \in [\mathfrak{a},\mathfrak{b}] \times \R^d} 
	\frac{ | z(t, x) | }{ 1 + \norm{x}_{\R^d}^q }
	\allowbreak<
	\infty
	$,   which satisfies that 
	$z|_{(\mathfrak{a}, \mathfrak{b}) \times \R^{d}}$ is a viscosity solution of
	\begin{equation}
\begin{split}
&\lambda (\tfrac{\partial }{\partial t}z)(t,x) 
=
f\big(x,z(t,x),(\nabla_x z )(t,x),(\operatorname{Hess}_x z )(t,x)
\big) 
\end{split}
\end{equation}
	for $(t,x) \in (\mathfrak{a},\mathfrak{b}) \times \R^{d}$
%
	(cf. Proposition~\ref{lemBlackScholesTimeChange} (with 
	$d=d$, $a=\mathfrak{a}$, $\lambda=-\lambda$, $b=\mathfrak{b}$,  
	$f(t,x,\alpha,\eta,A)=f(x,\alpha,\eta,A)$,
	$u(t,x)=w(t,x)$, $U(t,x)=z(t,x)$
	for $t\in (\mathfrak{a},\mathfrak{b})$, $x\in\R^d$, $\alpha\in\R$, $\eta\in \R^d$, $A\in \R^{d\times d}$
	in the notation of Proposition~\ref{lemBlackScholesTimeChange})), and which satisfies that
	\begin{equation}\label{BS_endvalue:eqApproxZ}
	\left[
	\int_{\R^d}  
	\left|
	z(a,x) - \Phi(x)
	\right|^p \,
	\nu(dx)
	\right]^{\nicefrac{1}{p}} 
	\leq
	\varepsilon.
		\end{equation}
	Hence, we obtain that for all $t\in [\mathfrak{a},\mathfrak{b}]$, $x\in\R^d$ it holds that $z(t,x)=u(t,x)$. The fact that 
	for all $t\in [\mathfrak{a},\mathfrak{b}]$, $x \in \R^{d}$ it holds that  $v(t,x)=u(\mathfrak{a}+\mathfrak{b}-t,x)$ and
	$w(t,x) = z(\mathfrak{a}+\mathfrak{b}-t,x)$ therefore demonstrates that for all $t\in [\mathfrak{a},\mathfrak{b}]$, $x \in \R^{d}$ it holds that 
	\begin{equation}
		w(t,x)=z(\mathfrak{a}+\mathfrak{b}-t,x)=u(\mathfrak{a}+\mathfrak{b}-t,x)=v(t,x).
	\end{equation}
	Combining this,  the fact that for all $x\in\R^d$ it holds that $v(a,x) = \varphi(x)$, \eqref{BS_endvalue:eqApproxVProof}, \eqref{BS_endvalue:ViscosityVProof}, and \eqref{BS_endvalue:GrowthV}  completes the proof of Corollary~\ref{BS_endvalueOneDirection}.
\end{proof}

\begin{cor}
\label{BS_endvalue}
Assume Setting~\ref{BS_setting}, let $d \in \N$, $\varepsilon, T \in (0,\infty)$, $\varphi \in  C(\R^d, \R)$, and $\psi \in \mathcal{N}$.
Then the following two statements are equivalent:

\begin{enumerate}[(i)]
\item \label{BS_endvalue:item1}
There exists a unique continuous function $u\colon [0,T]\allowbreak \times \R^{d} \to \R$ which satisfies
for all $x \in \R^{d}$ that 
$u(T,x) = \varphi(x)$,
which satisfies that
$
	\inf_{q \in (0,\infty)} \allowbreak
	\sup_{(t, x) \in [0, T] \times \R^d} 
	\frac{ | u(t, x) | }{ 1 + \norm{x}_{\R^d}^q }
<
	\infty
$, 
which satisfies that $u|_{(0,T) \times \R^{d}}$ is a viscosity solution of
\begin{equation}
\begin{split}
	&(\tfrac{\partial }{\partial t}u)(t,x) 
	+
	\big\langle (   \nabla_x u )(t,x), \mu_d(x)\big\rangle_{\R^d} \\
	&+ 
	\tfrac{1}{2} 
	\operatorname{Trace}\! \big( 
		\sigma_d(x)[\sigma_d(x)]^{\ast}(\operatorname{Hess}_x u )(t,x)
	\big) 
=
	0
\end{split}
\end{equation}
for $(t,x) \in (0,T) \times \R^{d}$,
and which satisfies that
\begin{equation}
	\left[
		\int_{\R^d}  
		\left|
			u(0,x) - ( \mathcal{R}(\psi) ) (x)
		\right|^p \,
		\nu_{d}(dx)
	\right]^{\nicefrac{1}{p}} 
\leq
		\varepsilon.
\end{equation}

\item \label{BS_endvalue:item2}
There exists a unique continuous function $v\colon [0,T] \allowbreak \times \R^{d} \to \R$ which satisfies
for all $x \in \R^{d}$ that 
$v(0,x) = \varphi(x)$,
which satisfies that
$
	\inf_{q \in (0,\infty)} \allowbreak
	\sup_{(t, x) \in [0, T] \times \R^d} 
	\frac{ | v(t, x) | }{ 1 + \norm{x}_{\R^d}^q }
<
	\infty
$, 
which satisfies that $v|_{(0,T) \times \R^{d}}$ is a viscosity solution of
\begin{equation}
\begin{split}
	(\tfrac{\partial }{\partial t}v)(t,x) 
&= 
	\tfrac{1}{2} 
	\operatorname{Trace}\! \big( 
		\sigma_d(x)[\sigma_d(x)]^{\ast}(\operatorname{Hess}_x v )(t,x)
	\big) \\
	&+
	\big\langle (  \nabla_x v  )  (t,x) ,\mu_d(x)\big\rangle_{\R^d}
\end{split}
\end{equation}
for $(t,x) \in (0,T) \times \R^{d}$,
and which satisfies that
\begin{equation}
	\left[
		\int_{\R^d}  
		\left|
			v(T,x) - ( \mathcal{R}(\psi) ) (x)
		\right|^p \,
		\nu_{d}(dx)
	\right]^{\nicefrac{1}{p}} 
\leq
		\varepsilon.
\end{equation}
\end{enumerate}
\end{cor}

\begin{proof}[Proof of Corollary~\ref{BS_endvalue}]
Observe that Corollary~\ref{BS_endvalueOneDirection} (with 
$d=d$, $a=0$, $\lambda=-1$, $b=T$, $\nu=\nu_d$, $\varphi(x)=\varphi(x)$, $\Phi(x)= ( \mathcal{R}(\psi) ) (x)$,
$f(x,\alpha,\eta,A)=\langle \eta, \mu_d(x)\rangle_{\R^d}+ 
	\tfrac{1}{2} \operatorname{Trace}\! \big( 
	\sigma_d(x)[\sigma_d(x)]^{\ast}A
	\big)$,
$u(t,x)=u(t,x)$
for $t\in [0,T]$, $x\in\R^d$, $\alpha\in\R$, $\eta\in \R^d$, $A\in \R^{d\times d}$
in the notation of Corollary~\ref{BS_endvalue})
proves that item~\eqref{BS_endvalue:item1} implies item~\eqref{BS_endvalue:item2}.
Next note that  Corollary~\ref{BS_endvalueOneDirection} (with 
$d=d$, $a=T$, $\lambda=1$, $b=0$, $\nu=\nu_d$, $\varphi(x)=\varphi(x)$, $\Phi(x)= ( \mathcal{R}(\psi) ) (x)$,
$f(x,\alpha,\eta,A)=\langle \eta, \mu_d(x)\rangle_{\R^d}+ 
\tfrac{1}{2} \operatorname{Trace}\! \big( 
\sigma_d(x)[\sigma_d(x)]^{\ast}A
\big)$,
$u(t,x)=v(t,x)$
for $t\in [0,T]$, $x\in\R^d$, $\alpha\in\R$, $\eta\in \R^d$, $A\in \R^{d\times d}$
in the notation of Corollary~\ref{BS_endvalue})
proves that item~\eqref{BS_endvalue:item2} implies item~\eqref{BS_endvalue:item1}.
The proof of Corollary~\ref{BS_endvalue} is thus completed.
\end{proof}

\subsection[ANN approximations for basket call options]{Artificial neural network approximations for basket call options}\label{SectionCall}

In this subsection we establish in Proposition~\ref{Basket_call} below that ANN approximations overcome the curse of dimensionality in the numerical approximations of the Black-Scholes model in the case of basket call options. Our proof of Proposition~\ref{Basket_call} employs the elementary ANN representation result for the payoff functions associated to basket call options in Lemma~\ref{Basket_call_NN} below. For the sake of completeness we also provide in this subsection a detailed proof of Lemma~\ref{Basket_call_NN}.

\begin{lemma}
\label{Basket_call_NN}
Assume Setting~\ref{BS_setting} and
let $(c_{d, i})_{d \in \N, i \in \{1, 2, \ldots, d \}}, (K_d)_{d \in \N} \subseteq \R$.
Then 
there exists $(\phi_d)_{d \in \N} \subseteq \mathcal{N}$ such that 
for all $d \in \N$, $x = (x_1,x_2, \ldots, x_d) \in \R^d$ it holds that
$
	\mathcal{P}(\phi_d) 
\leq 
	4 d
$, 
$
	\mathcal{R}(\phi_d) \in C(\R^d, \R)
$,
and
\begin{equation}
	(\mathcal{R}(\phi_d))(x)
=
	\max \{  c_{d, 1} x_1 +  c_{d, 2} x_2 + \ldots +  c_{d, d} x_d - K_d, 0 \}.
\end{equation}
\end{lemma}

\begin{proof}[Proof of Lemma~\ref{Basket_call_NN}]
Throughout this proof let $(\phi_d)_{d \in \N} \subseteq \mathcal{N}$ satisfy for all $d \in \N$ that
\begin{equation}
	\phi_d
=
	\big(
	((c_{d,1}, c_{d,2}, \ldots, c_{d,d}), -K_d)
	,
	(1, 0)
	\big)
\in 
	(\R^{d \times 1} \times \R) \times (\R \times \R)
\end{equation}
(i.e., $\phi_d$ corresponds to a fully connected feedforward artificial neural network with 3 layers with dimensions $(d, 1, 1)$).
Note that \eqref{BS_setting:eq2} and \eqref{BS_setting:eq3} ensure that
for all $d \in \N$, $x = (x_1, x_2, \ldots, x_d) \in \R^d$ it holds that 
$\mathcal{R}(\phi_d) \in C(\R^d, \R)$ and
\begin{equation}
\label{Basket_call_NN:eq1}
\begin{split}
	(\mathcal{R}(\phi_d))(x)
&=
	1 \cdot
	\max \left\{ 
		\begin{pmatrix}
			c_{d,1} &c_{d,2} &\cdots &c_{d,d}
		\end{pmatrix}
		x
		+
		(-K_d)
		,
		0
		\right\}
	 + 0 \\
&=
	\max\{ c_{d, 1} x_1 +  c_{d, 2} x_2 + \ldots +  c_{d, d} x_d - K_d, 0 \}.
\end{split}
\end{equation}
Moreover, observe that for all $d \in \N$ it holds that
\begin{equation}
	\mathcal{P}(\phi_d) 
= 
	1 (d+1) + 1 (1 + 1)
=
	d + 3
\leq 
	4d.
\end{equation}
This and \eqref{Basket_call_NN:eq1} complete the proof of Lemma~\ref{Basket_call_NN}.
\end{proof}

\begin{prop}
\label{Basket_call}
Assume Setting~\ref{BS_setting},
let $(c_{d, i})_{d \in \N,i \in \{1, 2, \ldots, d \}} \subseteq [0,1]$, $(K_d)_{d \in \N} \allowbreak\in (0,\infty),$ and assume  
for all $d \in \N$ that $\sum_{i = 1}^d c_{d, i} = 1$.
Then  
\begin{enumerate}[(i)]
\item \label{Basket_call:item1}
there exist unique continuous functions $u_d\colon [0,T]\allowbreak \times \R^{d} \to \R$, $d \in \N$, which satisfy
for all $d \in \N$, $x = (x_1, x_2, \ldots, x_d) \in \R^{d}$ that 
$u_d(T,x) = \max \{  c_{d, 1} x_1 +  c_{d, 2} x_2 + \ldots +  c_{d, d} x_d - K_d, 0 \}$,
which satisfy 
for all $d \in \N$ that
$
	\inf_{q \in (0,\infty)} 
	\sup_{(t, x) \in [0, T] \times \R^d} 
	\frac{ | u_d(t, x) | }{ 1 + \norm{x}_{\R^d}^q }
<
	\infty
$,
and which satisfy for all $d \in \N$ that $u_d|_{(0,T) \times \R^{d}}$ is a viscosity solution of
\begin{equation}
\begin{split}
	&(\tfrac{\partial }{\partial t}u_d)(t,x) 
	+
	\big\langle(\nabla_x u_d)(t,x), \mu_d(x) \big\rangle_{\R^d}\\
	&+ 
	\tfrac{1}{2} 
	\operatorname{Trace}\! \big( 
		\sigma_d(x)[\sigma_d(x)]^{\ast}(\operatorname{Hess}_x u_d )(t,x)
	\big) 
=
	0
\end{split}
\end{equation}
for $(t,x) \in (0,T) \times \R^{d}$
and

\item \label{Basket_call:item2}
there exist $\mathfrak{C} \in (0,\infty)$, $(\psi_{d, \varepsilon})_{d \in \N, \,\varepsilon \in (0,1]} \subseteq \mathcal{N}$ such that
for all $d \in \N$, $\varepsilon \in (0,1]$ it holds that
$
	\mathcal{P}(\psi_{d, \varepsilon}) 
\leq
	\mathfrak{C} \, d^{5\theta+1} \,  \varepsilon^{-4} 
$,
$
	\mathscr{P}(\psi_{d, \varepsilon}) 
\leq
	\mathfrak{C} \, d^{5\theta+1} \, \varepsilon^{-2}
$,
$
	\mathcal{R}(\psi_{d, \varepsilon}) \in C(\R^{d}, \R)
$,
and
\begin{equation}
\label{Basket_call:concl1} 
	\left[
		\int_{\R^d}  
		\left|
			u_d(0,x) - ( \mathcal{R}(\psi_{d, \varepsilon}) ) (x)
		\right|^p \,
		\nu_{d}(dx)
	\right]^{\nicefrac{1}{p}} 
\leq
		\varepsilon.
\end{equation}
\end{enumerate}
\end{prop}

\begin{proof}[Proof of Proposition~\ref{Basket_call}]
Throughout this proof 
let $ \varphi_d \colon \R^d \to \R$, $d \in \N$, satisfy 
for all $d \in \N$, $x = (x_1, x_2, \ldots, x_d) \in \R^{d}$ that 
\begin{equation}
	\varphi_d(x) = \max \{  c_{d, 1} x_1 +  c_{d, 2} x_2 + \ldots +  c_{d, d} x_d - K_d, 0 \}
\end{equation}
and
let $(\chi_d)_{d \in \N}, (\phi_{d, \delta})_{d \in \N, \delta \in (0,1]} \subseteq \mathcal{N}$ satisfy 
for all $d \in \N$, $x \in \R^{d}$, $\delta \in (0,1]$ that
$
	\mathcal{P}(\chi_d) 
\leq 
	4 d
$, 
$
	\mathcal{R}(\chi_d) \in C(\R^d, \R)
$,
$
	(\mathcal{R}(\chi_d))(x)
=
	\varphi_d(x)
$
(cf.\ Lemma~\ref{Basket_call_NN}),
and 
$
	\phi_{d, \delta} = \chi_d
$.
Note that 
for all $d \in \N$, $\delta \in (0,1]$ it holds that
\begin{equation}
\label{Basket_call:eq1}
\mathcal{R}(\phi_{d, \delta}) = \mathcal{R}(\chi_{d})= \varphi_d\in C(\R^d, \R).
\end{equation}
This implies that
for all $d \in \N$, $x \in \R^{d}$, $\delta \in (0,1]$ it holds that
\begin{equation}\label{Basket_call:eqA}
\begin{split}
\left| 
\varphi_d(x) - ( \mathcal{R}(\phi_{d, \delta}) )(x)
\right| 
=
\left| 
\varphi_d(x) - \varphi_d(x)
\right| 
=
0
\leq 
d^0\delta^0 (1 + \norm{x}_{\R^d}^2).
\end{split}
\end{equation}
Moreover, observe that \eqref{Basket_call:eq1} and the hypothesis that for all $d \in \N$, $i \in \{1, 2, \ldots, d \}$  it holds that $c_{d, i}\ge 0$ and
$
	\sum_{i = 1}^d c_{d, i} = 1
$
 assure that
for all $d \in \N$, $x = (x_1, x_2, \ldots, x_d) \in \R^{d}$, $\delta \in (0,1]$ it holds that
\begin{equation}
\label{Basket_call:eq2}
\begin{split}
	| (\mathcal{R}(\phi_{d, \delta}))(x) |
&=
	| (\mathcal{R}(\chi_{d}))(x) |
=
	| \varphi_d (x) | \\
&=
	\max \{  c_{d, 1} x_1 +  c_{d, 2} x_2 + \ldots +  c_{d, d} x_d - K_d, 0 \} \\
&\leq
	 c_{d, 1} | x_1 | +  c_{d, 2} | x_2 | + \ldots +  c_{d, d} | x_d | \\
&\leq	
	\left[ \smallsum_{i = 1}^d c_{d, i} \right]\max \{ | x_1 |,   | x_2 | , \ldots ,  | x_d | \} \\
&\leq
	\norm{x}_{\R^d}
\leq
	d^0 (1 + \norm{x}_{\R^d}^2).
\end{split}
\end{equation}
In addition, observe that 
for all $d \in \N$, $\delta \in (0,1]$ it holds that
\begin{equation}
\label{Basket_call:eq3}
	\mathcal{P}(\phi_{d, \delta}) 
=
	\mathcal{P}(\chi_d) 
\leq 
	4 d
=
	4 d^1 \delta^{-0}.
\end{equation}
Combining this, \eqref{Basket_call:eq1}, \eqref{Basket_call:eqA}, \eqref{Basket_call:eq2}, 
the hypothesis that for all $q \in (0,\infty)$ it holds that
\begin{equation}
		\sup_{d \in \N}\left[ d^{-\theta q}
	\textint_{\R^d}  
	\Norm{x}_{\R^d}^{ q } \,
	\nu_{ d } (dx)\right]
	< 
	\infty,
\end{equation}
and Lemma~\ref{BS_properties} 
with 
Theorem~\ref{cont_NN_approx} 
(with 
$ T = T $,
$ r = 1 $,
$ R = 1 $,
$ v = 0 $,
$ w = 0 $,
$ z = 1 $,
$ \mathbf{z} = 0 $,
$\theta=\theta$,
$ \mathfrak{c} = 
	\max \{ 
		4 ,
		2 
		\left[
			\sup_{d \in \N, i \in \{1, 2, \ldots, d \}} 
				(| \alpha_{d, i} | + | \beta_{d, i} |) 
		\right]
	\}$,
$ \mathbf{v} = 2 $,
$p = p$,
$ \nu_d = \nu_d$,
$ \varphi_d = \varphi_d $,
$ \mu_d = \mu_d $,
$ \sigma_d = \sigma_d $,
$a(x)=\max \{x, 0 \}$,
$ \phi_{d, \delta} = \phi_{d, \delta}$
for $d\in\N$, $x\in \R$, $\delta\in (0,1]$ in the notation of Theorem~\ref{cont_NN_approx})
demonstrates that 
there exist unique continuous functions $v_d\colon [0,T] \times \R^{d} \to \R$, $d \in \N$, which satisfy
for all $d \in \N$, $x \in \R^{d}$ that 
$v_d(0,x) = \varphi_d(x)$,
which satisfy 
for all $d \in \N$ that
$
	\inf_{q \in (0,\infty)} 
	\sup_{(t, x) \in [0, T] \times \R^d} 
	\frac{ | v_d(t, x) | }{ 1 + \norm{x}_{\R^d}^q }
<
	\infty
$,
and which satisfy for all $d \in \N$ that $v_d|_{(0,T) \times \R^{d}}$ is a viscosity solution of
\begin{equation}
\label{Basket_call:eq4}
\begin{split}
	(\tfrac{\partial }{\partial t}v_d)(t,x) 
&= 
	\tfrac{1}{2} 
	\operatorname{Trace}\! \big( 
		\sigma_d(x)[\sigma_d(x)]^{\ast}(\operatorname{Hess}_x v_d )(t,x)
	\big) 
	\\&+
	\big\langle(\nabla_x v_d)(t,x), \mu_d(x) \big\rangle_{\R^d}
\end{split}
\end{equation}
for $(t,x) \in (0,T) \times \R^{d}$
and that
there exist $\mathfrak{C} \in (0,\infty)$, $(\psi_{d, \varepsilon})_{d \in \N, \,\varepsilon \in (0,1]} \subseteq \mathcal{N}$ such that
for all $d \in \N$, $\varepsilon \in (0,1]$ it holds that
$
	\mathcal{P}(\psi_{d, \varepsilon}) 
\leq
	\mathfrak{C}  \, d^{5\theta+1} \, \varepsilon^{-4}
$,
$
	\mathscr{P}(\psi_{d, \varepsilon}) 
\leq
	\mathfrak{C} \, d^{5\theta+1} \, \varepsilon^{-2}
$,
$
	\mathcal{R}(\psi_{d, \varepsilon}) \in C(\R^{d}, \R)
$,
and
\begin{equation}
\label{Basket_call:eq5}
	\left[
		\int_{\R^d}  
		\left|
			v_d(T,x) - ( \mathcal{R}(\psi_{d, \varepsilon}) ) (x)
		\right|^p \,
		\nu_{d}(dx)
	\right]^{\nicefrac{1}{p}} 
\leq
		\varepsilon.
\end{equation}
Corollary~\ref{BS_endvalue} hence assures that there exist unique continuous functions 
$
	u_d\colon [0,T] \times \R^{d} \to \R
$, $d \in \N$,
which satisfy that 
for all $d \in \N$, $x = (x_1, x_2, \ldots,\allowbreak x_d) \allowbreak\in \R^{d}$ it holds that 
\begin{equation}
\begin{split}
	u_d(T,x)
=  
	\varphi_d(x) 
= 
	\max \{  c_{d, 1} x_1 +  c_{d, 2} x_2 + \ldots +  c_{d, d} x_d - K_d, 0 \},
\end{split}
\end{equation}
which satisfy 
for all $d \in \N$ that
$
	\inf_{q \in (0,\infty)} 
	\sup_{(t, x) \in [0, T] \times \R^d} 
	\frac{ | u_d(t, x) | }{ 1 + \norm{x}_{\R^d}^q }
<
	\infty
$,
and which satisfy 
for all $d \in \N$ that $u_d|_{(0,T) \times \R^{d}}$ is a viscosity solution of
\begin{equation}
\label{Basket_call:eq6}
\begin{split}
	&(\tfrac{\partial }{\partial t}u_d)(t,x) 
	+
	\big\langle(\nabla_x u_d)(t,x), \mu_d(x) \big\rangle_{\R^d}
\\
	&+ 
	\tfrac{1}{2} 
	\operatorname{Trace}\! \big( 
		\sigma_d(x)[\sigma_d(x)]^{\ast}(\operatorname{Hess}_x u_d )(t,x)
	\big) 
=
	0
\end{split}
\end{equation}
for $(t,x) \in (0,T) \times \R^{d}$
and that it holds
for all $d \in \N$, $\varepsilon \in (0,1]$ that
\begin{equation}
\label{Basket_call:eq7}
\begin{split}
	&\left[
		\int_{\R^d}  
		\left|
			u_d(0,x) - ( \mathcal{R}(\psi_{d, \varepsilon}) ) (x)
		\right|^p \,
		\nu_{d}(dx)
	\right]^{\nicefrac{1}{p}}
\leq
	\varepsilon.
\end{split}
\end{equation}
Combining this with the fact that
for all $d \in \N$, $\varepsilon \in (0,1]$ it holds that
$
	\mathcal{P}(\psi_{d, \varepsilon}) 
\leq
	\mathfrak{C}  \, d^{5\theta+1} \, \varepsilon^{-4}
$,
$
	\mathscr{P}(\psi_{d, \varepsilon}) 
\leq
	\mathfrak{C} \, d^{5\theta+1} \, \varepsilon^{-2}
$,
and
$
	\mathcal{R}(\psi_{d, \varepsilon}) \in C(\R^{d}, \R)
$
establishes items~\eqref{Basket_call:item1}--\eqref{Basket_call:item2}.
The proof of Proposition~\ref{Basket_call} is thus completed.
\end{proof}

\subsection[ANN approximations for basket put  options]{Artificial neural network approximations for basket put  options}\label{SectionPut}

In this subsection we establish in Proposition~\ref{Basket_put} below that ANN approximations overcome the curse of dimensionality in the numerical approximation of the Black-Scholes model in the case of basket put options. Our proof of Proposition~\ref{Basket_put} employs the elementary ANN representation result for the payoff functions associated to basket put options in Lemma~\ref{Basket_put_NN} below. For the sake of completeness we also provide in this subsection a detailed proof of Lemma~\ref{Basket_put_NN}.

\begin{lemma}
\label{Basket_put_NN}
Assume Setting~\ref{BS_setting} and
let $(c_{d, i})_{d \in \N, i \in \{1, 2, \ldots, d \}} \subseteq \R$, $K \in  \R$.
Then 
there exists $(\phi_d)_{d \in \N} \subseteq \mathcal{N}$ such that 
for all $d \in \N$, $x = (x_1,x_2, \ldots, x_d) \in \R^d$ it holds that
$
	\mathcal{P}(\phi_d) 
\leq 
	4 d
$, 
$
	\mathcal{R}(\phi_d) \in C(\R^d, \R)
$,
and
\begin{equation}
	(\mathcal{R}(\phi_d))(x)
=
	 \max \{ K - (c_{d, 1} x_1 +  c_{d, 2} x_2 + \ldots +  c_{d, d} x_d),  0 \}.
\end{equation}
\end{lemma}

\begin{proof}[Proof of Lemma~\ref{Basket_put_NN}]
Note that Lemma~\ref{Basket_call_NN}
(with $c_{d, i}=-c_{d, i}$, $K_d=-K$
 for $d\in\N$, $i \in \{1, 2, \ldots, d \}$ 
in the notation of Lemma~\ref{Basket_call_NN})
demonstrates that there exists $(\phi_d)_{d \in \N} \subseteq \mathcal{N}$ such that 
for all $d \in \N$, $x = (x_1, x_2, \ldots, x_d) \in \R^d$ it holds that 
$
	\mathcal{P}(\phi_d) 
\leq 
	4d
$,
$
	\mathcal{R}(\phi_d) \in C(\R^d, \R)
$,
and
\begin{equation}
\label{Basket_put_NN:eq1}
\begin{split}
	(\mathcal{R}(\phi_d))(x)
&=
	\max\{ -c_{d, 1} x_1 -  c_{d, 2} x_2 - \ldots -  c_{d, d} x_d + K, 0 \} \\
&=
	\max \{ K - (c_{d, 1} x_1 +  c_{d, 2} x_2 + \ldots +  c_{d, d} x_d),  0 \}.
\end{split}
\end{equation}
The proof of Lemma~\ref{Basket_put_NN} is thus completed.
\end{proof}

\begin{prop}
\label{Basket_put}
Assume Setting~\ref{BS_setting} and
let $(c_{d, i})_{d \in \N, i \in \{1, 2, \ldots, d \}} \subseteq [0,1]$, $K \in (0,\infty)$ satisfy 
for all $d \in \N$ that $\sum_{i = 1}^d c_{d, i} = 1$.
Then  
\begin{enumerate}[(i)]
\item \label{Basket_put:item1}
there exist unique continuous functions $u_d\colon [0,T] \allowbreak\times \R^{d} \to \R$, $d \in \N$, which satisfy
for all $d \in \N$, $x = (x_1, x_2, \ldots, x_d) \in \R^{d}$ that 
$
	u_d(T,x) = \max \{ K - (c_{d, 1} x_1 +  c_{d, 2} x_2 + \ldots +  c_{d, d} x_d),  0 \}
$,
which satisfy 
for all $d \in \N$ that
$
	\inf_{q \in (0,\infty)} 
	\sup_{(t, x) \in [0, T] \times \R^d} 
	\frac{ | u_d(t, x) | }{ 1 + \norm{x}_{\R^d}^q }
<
	\infty
$,
and which satisfy for all $d \in \N$ that $u_d|_{(0,T) \times \R^{d}}$ is a viscosity solution of
\begin{equation}
\begin{split}
	&(\tfrac{\partial }{\partial t}u_d)(t,x) 
	+
	\big\langle(\nabla_x u_d)(t,x), \mu_d(x) \big\rangle_{\R^d}
 \\
	&+ 
	\tfrac{1}{2} 
	\operatorname{Trace}\! \big( 
		\sigma_d(x)[\sigma_d(x)]^{\ast}(\operatorname{Hess}_x u_d )(t,x)
	\big) 
=
	0
\end{split}
\end{equation}
for $(t,x) \in (0,T) \times \R^{d}$
and

\item \label{Basket_put:item2}
there exist $\mathfrak{C} \in (0,\infty)$, $(\psi_{d, \varepsilon})_{d \in \N, \,\varepsilon \in (0,1]} \subseteq \mathcal{N}$ such that
for all $d \in \N$, $\varepsilon \in (0,1]$ it holds that
$
	\mathcal{P}(\psi_{d, \varepsilon}) 
\leq
	\mathfrak{C}  \, d^{5\theta+1} \, \varepsilon^{-4}
$,
$
	\mathscr{P}(\psi_{d, \varepsilon}) 
\leq
	\mathfrak{C} \, d^{5\theta+1} \, \varepsilon^{-2}
$,
$
	\mathcal{R}(\psi_{d, \varepsilon}) \in C(\R^{d}, \R)
$,
and
\begin{equation}
\label{Basket_put:concl1} 
	\left[
		\int_{\R^d}  
		\left|
			u_d(0,x) - ( \mathcal{R}(\psi_{d, \varepsilon}) ) (x)
		\right|^p \,
		\nu_{d}(dx)
	\right]^{\nicefrac{1}{p}} 
\leq
		\varepsilon.
\end{equation}
\end{enumerate}
\end{prop}

\begin{proof}[Proof of Proposition~\ref{Basket_put}]
Throughout this proof 
let $ \varphi_d \colon \R^d \to \R$, $d \in \N$, satisfy 
for all $d \in \N$, $x = (x_1, x_2, \ldots, x_d) \in \R^{d}$ that 
\begin{equation}
		\varphi_d(x) = \max \{ K - (c_{d, 1} x_1 +  c_{d, 2} x_2 + \ldots +  c_{d, d} x_d),  0 \}
\end{equation}
and
let $(\chi_d)_{d \in \N}, (\phi_{d, \delta})_{d \in \N, \delta \in (0,1]} \subseteq \mathcal{N}$ satisfy 
for all $d \in \N$, $x \in \R^{d}$, $\delta \in (0,1]$ that
$
	\mathcal{P}(\chi_d) 
\leq 
	4 d
$, 
$
	\mathcal{R}(\chi_d) \in C(\R^d, \R)
$,
$
	(\mathcal{R}(\chi_d))(x)
=
	\varphi_d(x)
$
(cf.\ Lemma~\ref{Basket_put_NN}),
and 
$
	\phi_{d, \delta} = \chi_d
$.
Note that 
for all $d \in \N$, $\delta \in (0,1]$ it holds that
\begin{equation}
\label{Basket_put:eq1}
\mathcal{R}(\phi_{d, \delta}) = \mathcal{R}(\chi_{d})=\varphi_d \in C(\R^d, \R).
\end{equation}
This and the hypothesis that for all $d \in \N$, $i \in \{1, 2, \ldots, d \}$ it holds that $c_{d,i}\in [0,1]$ and
$
	\sum_{i = 1}^d c_{d, i} = 1
$
assure that
for all $d \in \N$, $x = (x_1, x_2, \ldots, x_d) \in \R^{d}$, $\delta \in (0,1]$ it holds that
\begin{equation}
\label{Basket_put:eq2}
\begin{split}
	| (\mathcal{R}(\phi_{d, \delta}))(x) |
&=
	| (\mathcal{R}(\chi_{d}))(x) |
=
	| \varphi_d (x) | \\
&=
	\max \{ K - (c_{d, 1} x_1 +  c_{d, 2} x_2 + \ldots +  c_{d, d} x_d),  0 \} \\
&\leq
	 K + c_{d, 1} | x_1 | +  c_{d, 2} | x_2 | + \ldots +  c_{d, d} | x_d | \\
&\leq	
	K + \left[ \smallsum_{i = 1}^d c_{d, i} \right]\max \{ | x_1 |,   | x_2 | , \ldots ,  | x_d | \} \\
&\leq
	K + \norm{x}_{\R^d}
\leq
	K + 1 + \norm{x}_{\R^d}^2 \\
&\leq
	(K+1) \, d^0 (1 + \norm{x}_{\R^d}^2).
\end{split}
\end{equation}
In addition, observe that 
for all $d \in \N$, $\delta \in (0,1]$ it holds that
\begin{equation}
\label{Basket_put:eq3}
	\mathcal{P}(\phi_{d, \delta}) 
=
	\mathcal{P}(\chi_d) 
\leq 
	4 d
=
	4 d^1 \delta^{-0}.
\end{equation}
Furthermore, note that \eqref{Basket_put:eq1} ensures that 
for all $d \in \N$, $x \in \R^{d}$, $\delta \in (0,1]$ it holds that
\begin{equation}
\begin{split}
	\left| 
		\varphi_d(x) - ( \mathcal{R}(\phi_{d, \delta}) )(x)
	\right| 
&=
	\left| 
		\varphi_d(x) - ( \mathcal{R}(\chi_{d}) )(x)
	\right| \\
&=
	\left| 
		\varphi_d(x) - \varphi_d(x)
	\right| 
=
	0
\leq 
	 d^0 \delta^0(1 + \norm{x}_{\R^d}^2).
\end{split}
\end{equation}
Combining this, \eqref{Basket_put:eq1}--\eqref{Basket_put:eq3}, 
the fact that $(\varphi_d)_{d \in \N}$ are continuous functions, 
the hypothesis that for all $q \in (0,\infty)$ it holds that
\begin{equation}
\sup_{d \in \N}\left[ d^{-\theta q}
\textint_{\R^d}  
\Norm{x}_{\R^d}^{ q } \,
\nu_{ d } (dx)\right]
< 
\infty,
\end{equation}
and Lemma~\ref{BS_properties} 
with 
Theorem~\ref{cont_NN_approx} 
(with 
$ T = T $,
$ r = 1 $,
$ R = 1 $,
$ v = 0 $,
$ w = 0 $,
$ z = 1 $,
$\theta=\theta$,
$ \mathbf{z} = 0 $,
$ \mathfrak{c} = 
\max \{ 
4 , K + 1,
2 
\left[
\sup_{d \in \N, i \in \{1, 2, \ldots, d \}} 
(| \alpha_{d, i} | + | \beta_{d, i} |) 
\right]
\}$,
$ \mathbf{v} = 2 $,
$p = p$,
$ \nu_d = \nu_d$,
$ \varphi_d = \varphi_d $,
$ \mu_d = \mu_d $,
$ \sigma_d = \sigma_d $,
$a(x)=\max \{x, 0 \}$,
$ \phi_{d, \delta} = \phi_{d, \delta}$
for $d\in\N$, $x\in \R$, $\delta\in (0,1]$
in the notation of Theorem~\ref{cont_NN_approx})
demonstrate that 
there exist unique continuous functions $v_d\colon [0,T]\allowbreak \times \R^{d} \to \R$, $d \in \N$, which satisfy
for all $d \in \N$, $x \in \R^{d}$ that 
$v_d(0,x) = \varphi_d(x)$,
which satisfy 
for all $d \in \N$ that
$
	\inf_{q \in (0,\infty)} 
	\sup_{(t, x) \in [0, T] \times \R^d} 
	\frac{ | v_d(t, x) | }{ 1 + \norm{x}_{\R^d}^q }
<
	\infty
$,
and which satisfy for all $d \in \N$ that $v_d|_{(0,T) \times \R^{d}}$ is a viscosity solution of
\begin{equation}
\label{Basket_put:eq4}
\begin{split}
	(\tfrac{\partial }{\partial t}v_d)(t,x) 
&= 
	\tfrac{1}{2} 
	\operatorname{Trace}\! \big( 
		\sigma_d(x)[\sigma_d(x)]^{\ast}(\operatorname{Hess}_x v_d )(t,x)
	\big) 
	\\&+
	\big\langle(\nabla_x v_d)(t,x), \mu_d(x) \big\rangle_{\R^d}
\end{split}
\end{equation}
for $(t,x) \in (0,T) \times \R^{d}$
and that
there exist $\mathfrak{C} \in (0,\infty)$, $(\psi_{d, \varepsilon})_{d \in \N, \,\varepsilon \in (0,1]} \subseteq \mathcal{N}$ such that
for all $d \in \N$, $\varepsilon \in (0,1]$ it holds that
$
	\mathcal{P}(\psi_{d, \varepsilon}) 
\leq
	 \mathfrak{C} \, d^{5\theta+1} \, \varepsilon^{-4}
$,
$
	\mathscr{P}(\psi_{d, \varepsilon}) 
\leq
	\mathfrak{C} \, d^{5\theta+1} \, \varepsilon^{-2} 
$,
$
	\mathcal{R}(\psi_{d, \varepsilon}) \in C(\R^{d}, \R)
$,
and
\begin{equation}
\label{Basket_put:eq5}
	\left[
		\int_{\R^d}  
		\left|
			v_d(T,x) - ( \mathcal{R}(\psi_{d, \varepsilon}) ) (x)
		\right|^p \,
		\nu_{d}(dx)
	\right]^{\nicefrac{1}{p}} 
\leq
		\varepsilon.
\end{equation}
Corollary~\ref{BS_endvalue} hence assures that there exist unique continuous functions 
$
	u_d\colon [0,T] \times \R^{d} \to \R
$, $d \in \N$,
which satisfy that 
for all $d \in \N$, $x = (x_1, x_2, \ldots,\allowbreak x_d) \allowbreak\in \R^{d}$ it holds that 
\begin{equation}
\begin{split}
	u_d(T,x)
=  
	\varphi_d(x) 
= 
	\max \{ K - (c_{d, 1} x_1 +  c_{d, 2} x_2 + \ldots +  c_{d, d} x_d),  0 \},
\end{split}
\end{equation}
which satisfy 
for all $d \in \N$ that
$
	\inf_{q \in (0,\infty)} 
	\sup_{(t, x) \in [0, T] \times \R^d} 
	\frac{ | u_d(t, x) | }{ 1 + \norm{x}_{\R^d}^q }
<
	\infty
$,
and which satisfy 
for all $d \in \N$ that $u_d|_{(0,T) \times \R^{d}}$ is a viscosity solution of
\begin{equation}
\label{Basket_put:eq6}
\begin{split}
	&(\tfrac{\partial }{\partial t}u_d)(t,x) 
	+
	\big\langle(\nabla_x u_d)(t,x), \mu_d(x) \big\rangle_{\R^d}
\\
	&+ 
	\tfrac{1}{2} 
	\operatorname{Trace}\! \big( 
		\sigma_d(x)[\sigma_d(x)]^{\ast}(\operatorname{Hess}_x u_d )(t,x)
	\big) 
=
	0
\end{split}
\end{equation}
for $(t,x) \in (0,T) \times \R^{d}$
and that it holds
for all $d \in \N$, $\varepsilon \in (0,1]$ that
\begin{equation}
\label{Basket_put:eq7}
\begin{split}
	&\left[
		\int_{\R^d}  
		\left|
			u_d(0,x) - ( \mathcal{R}(\psi_{d, \varepsilon}) ) (x)
		\right|^p \,
		\nu_{d}(dx)
	\right]^{\nicefrac{1}{p}}
\leq
	\varepsilon.
\end{split}
\end{equation}
Combining this with the fact that
for all $d \in \N$, $\varepsilon \in (0,1]$ it holds that
$
	\mathcal{P}(\psi_{d, \varepsilon}) 
\leq
	\mathfrak{C} \, d^{5\theta+1} \, \varepsilon^{-4}  
$,
$
	\mathscr{P}(\psi_{d, \varepsilon}) 
\leq
	\mathfrak{C} \, d^{5\theta+1} \, \varepsilon^{-2}
$,
and
$
	\mathcal{R}(\psi_{d, \varepsilon}) \in C(\R^{d}, \R)
$
establishes items~\eqref{Basket_put:item1}--\eqref{Basket_put:item2}.
The proof of Proposition~\ref{Basket_put} is thus completed.
\end{proof}

\subsection[ANN approximations for call on max options]{Artificial neural network approximations for call on max options}\label{SectionCallOnMax}

In this subsection we establish in Proposition~\ref{call_on_max} below that ANN approximations overcome the curse of dimensionality in the numerical approximation of the Black-Scholes model in the case of call on max options. Our proof of Proposition~\ref{call_on_max} employs the ANN representation result for the payoff functions associated to call on max options in Lemma~\ref{call_on_max_NN} below. Our proof of Lemma~\ref{call_on_max_NN}, in turn, uses the elementary and essentially well-known facts in Lemmas~\ref{aux_max}--\ref{ComplexityCompution:lem}. For the sake of completeness we also provide in this subsection detailed proofs of Lemmas~\ref{aux_max}--\ref{ComplexityCompution:lem}.

\begin{lemma}
\label{aux_max}
Let $x, y \in \R$ and let 
$
(\cdot)^{+} \colon \R \to [0,\infty)
$ 
be the function which satisfies for all $q \in \R$ that 
$
(q)^{+} = \max \{ q , 0 \}
$.
Then
\begin{enumerate}[(i)]
\item \label{aux_max:item1}
it holds that 
$
	x = (x)^{+} - (-x)^{+}
$
and

\item \label{aux_max:item2}
it holds that 
$
	\max \{ x , y \} = (x - y)^{+} + y
$.
\item \label{aux_max:item3}
it holds that 
$
\min \{ x , y \} = - (x - y)^{+} + x
$.
\end{enumerate}
\end{lemma}

\begin{proof}[Proof of Lemma~\ref{aux_max}]
Observe that
\begin{equation}
\begin{split}
(x)^{+} - (-x)^{+}&=\max \{ x , 0 \}-\max \{ -x , 0 \}
\\&=[x-0]\mathbbm{1}_{[0,\infty)}(x)+[0-(-x)]\mathbbm{1}_{(-\infty,0)}(x)
=x.
\end{split}
\end{equation}
This establishes item~\eqref{aux_max:item1}. Next note that 
\begin{equation}
	(x - y)^{+} + y=\max \{ x-y , 0 \}+y=\max \{ x , y \}.
\end{equation}
This proves item~\eqref{aux_max:item2}. Moreover, observe that
\begin{equation}
\begin{split}
-(x - y)^{+} + x&=-(\max \{ x-y , 0 \}-x)=-\max \{ -y , -x \}\\&=\min \{ y , x \}.
\end{split}
\end{equation}
This establishes item~\eqref{aux_max:item3}.
The proof of Lemma~\ref{aux_max} is thus completed.
\end{proof}


\begin{lemma}\label{ComplexityCompution:lem}
	Let $(a_n)_{n\in\N}\subseteq [0,\infty)$ be the sequence which satisfies for all $n\in\N$ that
	\begin{equation}
	\begin{split}
	a_n&=(2(n-1) + 1)(n + 1)\\
	&\quad	+
	\left[
	\smallsum_{k = 1}^{n-1} (2 (n-(k+1)) + 1) (2 (n-k) + 1 + 1)
	\right]
	+
	1(1+1).
	\end{split}
	\end{equation}
	Then it holds for all $n\in\N$ that 
	\begin{equation}
	a_n\le 
	6 n^3. 
	\end{equation}
\end{lemma}

\begin{proof}[Proof of Lemma~\ref{ComplexityCompution:lem}]
Observe that for all $n\in\N$ it holds that
	\begin{equation}
	\begin{split}
	a_n
	&=
	(2(n-1) + 1)(n + 1)  \\
	&\quad	
	+
	\left[
	\smallsum_{k = 1}^{n-1} (2 (n-(k+1)) + 1) (2 (n-k) + 1 + 1)
	\right]
	+
	1(1+1) \\
	&=
	(2n - 1)(n + 1) 
	+
	\left[
	\smallsum_{k = 1}^{n-1} (2 (n-k) - 1) (2 (n-k) + 2)
	\right]
	+
	2 \\
	&=
	2n^2 + n - 1 
	+
	\left[
	\smallsum_{k = 1}^{n-1} 4(n-k)^2 + 2 (n-k) - 2
	\right]
	+
	2 \\
	&=
	2n^2 + n + 1 
	+
	\left[
	\smallsum_{k = 1}^{n-1} 4n^2 - 8 n k + 4k^2 + 2 n- 2 k - 2
	\right] \\
	&=
	2n^2 + n + 1 
	+
	4n^2(n-1) - 
	8 n 
	\left[
	\smallsum_{k = 1}^{n-1} k
	\right]
	+ 
	4
	\left[
	\smallsum_{k = 1}^{n-1} k^2
	\right]  \\
	&\quad
	+ 
	2 n (n-1)
	-
	2
	\left[
	\smallsum_{k = 1}^{n-1} k
	\right] 
	- 
	2(n-1).
	\end{split}
	\end{equation}
	The fact that for all $n\in\N$ it holds that 
\begin{equation}
		\smallsum_{k = 1}^{n-1} k
		=
		\frac{(n-1)n}{2}
		\qandq 
			\smallsum_{k = 1}^{n-1} k^2
			=
			\frac{(n-1)n(2n - 1)}{6}
\end{equation}
	therefore assures for all $n\in\N$ that
	\begin{equation}
	\begin{split}
	a_n
	&=
	2n^2 + n + 1 
	+
	4n^2(n-1) 
	- 
	8 n 
	\left[
	\tfrac{(n-1)n}{2}
	\right]
	+ 
	4
	\left[
	\tfrac{(n-1)n(2n - 1)}{6}
	\right]  \\
	&\quad
	+ 
	2 n (n-1)
	-
	2
	\left[
	\tfrac{(n-1)n}{2}
	\right] 
	- 
	2 (n-1) \\
	&=
	2n^2 + n + 1 
	+ 4 n^3
	-  4 n^2 
	- 
	\left[
	4 n^3 - 4 n^2
	\right] \\
	&\quad
	+ 
	\left[
	\tfrac{8}{6} n^3 + \tfrac{4 (-3)}{6} n^2 + \tfrac{4}{6} n
	\right]  
	+ 2 n^2
	-  2 n
	- \left[ n^2 - n \right]
	- 2 n
	+ 2 \\
	&=
	\left(4 - 4 + \tfrac{4}{3} \right)  n^3
	+
	\left(2 - 4 + 4 -2 + 2 - 1\right) n^2 \\
	&\quad
	+
	\left(1 + \tfrac{2}{3} - 2  + 1 - 2\right)  n
	+
	1 + 2 \\
	&\leq
	\tfrac{4}{3}  n^3
	+
	n^2
	+
	3
	\leq
	(\tfrac{4}{3} + 1 + 3) n^3
	\leq 
	6 n^3.
	\end{split}
	\end{equation}
	The proof of Lemma~\ref{ComplexityCompution:lem} is thus completed.
\end{proof}

\begin{lemma}
\label{call_on_max_NN}
Assume Setting~\ref{BS_setting} and
let $(K_d)_{d \in \N}, (c_{d, i})_{d \in \N, i \in \{1, 2, \ldots, d \}} \subseteq \R$. 
Then 
there exists $(\phi_d)_{d \in \N} \subseteq \mathcal{N}$ such that 
for all $d \in \N$, $x = (x_1,x_2, \ldots, x_d) \in \R^d$ it holds that
$
	\mathcal{P}(\phi_d) 
\leq 
	6 d^3
$, 
$
	\mathcal{R}(\phi_d) \in C(\R^d, \R)
$,
and
\begin{equation}
	(\mathcal{R}(\phi_d))(x)
=
	 \max \{ \max\{c_{d, 1} x_1, c_{d, 2} x_2, \ldots,  c_{d, d} x_d \} - K_d ,  0 \}.
\end{equation}
\end{lemma}

\begin{proof}[Proof of Lemma~\ref{call_on_max_NN}] 
Throughout this proof 
let $d \in \N$,
let 
\begin{equation}
\label{call_on_max_NN:setting1}
\begin{split}
	\phi &=  ((W_1, B_1), (W_2, B_2), \ldots, (W_d, B_d), (W_{d+1}, B_{d+1})) \\
&\in 
	(\R^{(2(d-1) + 1) \times d} \times \R^{2(d-1) + 1} ) \\
&\quad	\times 
	(\times_{k = 1}^{d-1} (\R^{(2(d-k) - 1) \times (2(d-k) + 1) } \times \R^{2(d-k) - 1})) \\
&\quad	\times 
	(\R^{1 \times 1} \times \R^{1} )
\end{split}
\end{equation}
(i.e. $\phi$ corresponds to fully connected feedforward artificial neural network with $d+2$ layers with dimensions $(d, 2(d-1) + 1, 2(d-2) + 1, 2(d-3) + 1, \ldots, 3, 1, 1)$)
satisfy for all $k \in \{1, 2, \ldots, {d-2} \}$ that
\begin{equation}
\label{call_on_max_NN:setting2}
	W_1 
= 
	\begin{pmatrix}
		c_{d,1} &-c_{d,2} &0 &\cdots &0 \\
		0 &c_{d,2} &0 &\cdots &0 \\
		0 &-c_{d,2} &0 &\cdots &0 \\
		0 &0 &c_{d,3} &\cdots &0 \\
		0 &0 &-c_{d,3} &\cdots &0 \\
		\vdots &\vdots &\vdots &\ddots &\vdots \\
		0 &0 &0 &\cdots &c_{d, d} \\
		0 &0 &0 &\cdots &-c_{d, d} 
	\end{pmatrix}
\in 
	\R^{(2(d-1) + 1) \times d},
\end{equation}
\begin{equation}
\label{call_on_max_NN:setting3}
\begin{split}
	W_{k+1} 
&= 
	\begin{pmatrix}
		1 &1 &-1 &-1 &1 &0 &0 &\dots &0 &0 \\
		0 &0 &0 &1 &-1 &0 &0 &\dots &0 &0 \\
		0 &0 &0 &-1 &1 &0 &0 &\dots &0 &0 \\
		0 &0 &0 &0 &0 &1 &-1 &\dots &0 &0 \\
		0 &0 &0 &0 &0 &-1 &1 &\dots &0 &0 \\
		\vdots &\vdots &\vdots &\vdots &\vdots &\vdots &\vdots &\ddots &\vdots &\vdots \\
		0 &0 &0 &0 &0 &0 &0 &\dots &1 &-1 \\
		0 &0 &0 &0 &0 &0 &0 &\dots &-1 &1 \\
	\end{pmatrix} \\
&\in 
	\R^{(2(d-k) - 1) \times (2(d-k) + 1) },
\end{split}
\end{equation}
\begin{equation}
\label{call_on_max_NN:setting4}
	W_d
=
	\begin{pmatrix}
		1 &1 &-1	
	\end{pmatrix},
\qquad
	W_{d+1} 
= 
	\begin{pmatrix}
		1
	\end{pmatrix}
\in \R^{1 \times 1},
\qquad
	B_1 = 0,  
\end{equation}
\begin{equation}
\label{call_on_max_NN:setting5}
	B_2 = 0,
\quad\ldots,
	B_{d-1} = 0, 
\qquad
	B_{d} = -K_d,
\qandq 
B_{d+1} = 0,
\end{equation}
let $x = (x_1, x_2, \ldots, x_d) \in \R^d$, 
let $z_0 \in \R^d, z_1 \in \R^{2(d-1) - 1}, z_{2} \in \R^{2(d-2) + 1}, \ldots,\allowbreak z_{d-1} \allowbreak\in \R^{3}, z_d \in \R, z_{d+1} \in \R$ satisfy 
for all $k \in \{1, 2, \ldots, {d-1} \}$ that 
$
	z_0 = x
$,
$
	z_k = \mathbf{A}_{2(d-k) + 1} (W_k z_{k-1} + B_k)
$, 
$
	z_d = \mathbf{A}_{1} (W_d z_{d-1} + B_d)
$, and
\begin{equation}
\label{call_on_max_NN:setting6}
	z_{d+1} =  W_{d+1} z_d + B_{d+1},
\end{equation}
and let 
$
(\cdot)^{+} \colon \R \to [0,\infty)$ be the function which satisfies for all $q \in \R$ that 
$
(q)^{+} = \max \{ q , 0 \}
$.
Note that \eqref{BS_setting:eq3}, \eqref{call_on_max_NN:setting1}, and \eqref{call_on_max_NN:setting6} imply that 
$
	(\mathcal{R}(\phi)) \in C(\R^d, \R)
$
and
\begin{equation}
\label{call_on_max_NN:eq1}
	z_{d+1}
=
	(\mathcal{R}(\phi))(x).
\end{equation}
Next we claim that 
for all $k \in \{1, 2, \ldots, {d-1} \}$ it holds that
\begin{equation}
\label{call_on_max_NN:eq2}
	z_k
=
\begin{pmatrix}
	\big(\max \{c_{d, 1} x_1, c_{d, 2} x_2, \ldots, c_{d, k} x_k \} - c_{d, k+1} x_{k+1} \big)^{+} \\
	(c_{d, k+1} x_{k+1} )^{+} \\
	(-c_{d, k+1} x_{k+1} )^{+} \\
	(c_{d, k+2} x_{k+2} )^{+} \\
	(-c_{d, k+2} x_{k+2} )^{+} \\
	\vdots \\
	(c_{d, d} x_{d} )^{+} \\
	(-c_{d, d} x_{d} )^{+} 
\end{pmatrix}.
\end{equation}
We now prove \eqref{call_on_max_NN:eq2} by induction on $k \in \{1, 2, \ldots, {d-1} \}$.
For the base case $k = 1$ note that \eqref{BS_setting:eq2}, \eqref{call_on_max_NN:setting2}, \eqref{call_on_max_NN:setting4}, and \eqref{call_on_max_NN:setting6} assure that 
\begin{equation}
\begin{split}
	z_1
&=
	\mathbf{A}_{2(d-1) + 1} (W_1 z_0 + B_1) 
=
	\mathbf{A}_{2(d-1) + 1} (W_1 x) \\
&=
	\mathbf{A}_{2(d-1) + 1}
	\begin{pmatrix}
		c_{d, 1} x_1 - c_{d, 2} x_2 \\
		c_{d, 2} x_{2}   \\
		-c_{d, 2} x_{2}  \\
		c_{d, 3} x_{3}  \\
		-c_{d, 3} x_{3} \\
		\vdots \\
		c_{d, d} x_{d}\\
		-c_{d, d} x_{d} 
	\end{pmatrix}
=
	\begin{pmatrix}
		\big(\max \{c_{d, 1} x_1 \} - c_{d, 2} x_{2} \big)^{+} \\
		(c_{d, 2} x_{2} )^{+} \\
		(-c_{d, 2} x_{2} )^{+} \\
		(c_{d, 3} x_{3} )^{+} \\
		(-c_{d, 3} x_{3} )^{+} \\
		\vdots \\
		(c_{d, d} x_{d} )^{+} \\
		(-c_{d, d} x_{d} )^{+} 
	\end{pmatrix}.
\end{split}
\end{equation}
This establishes \eqref{call_on_max_NN:eq2} in the base case $k = 1$.
Next note that item \eqref{aux_max:item2} in Lemma~\ref{aux_max} implies that for all $a,b,c\in\R$ it holds that 
\begin{equation}\label{aux_max:Consequence}
(b-a)^++a-c=\max\{a,b\}-c.
\end{equation} 
For the induction step $ \{1, 2, \ldots,\allowbreak {d-2} \} \ni k \to k+1 \in \{2, 3, \ldots, {d-1} \}$ observe that 
\eqref{BS_setting:eq2}, \eqref{call_on_max_NN:setting3}, \eqref{call_on_max_NN:setting5}, \eqref{call_on_max_NN:setting6},
\eqref{aux_max:Consequence} (with $a=c_{d, k+1} x_{k+1}$, $b=\max \{c_{d, 1} x_1, c_{d, 2} x_2, \ldots, c_{d, k} x_k \}$, $c=c_{d, k+2} x_{k+2}$ in the notation of \eqref{aux_max:Consequence}),
and item \eqref{aux_max:item1} in Lemma~\ref{aux_max} demonstrate that 
for all $k \in \{1, 2, \ldots, {d-2} \}$ with 
\begin{equation}
	z_k
=
\begin{pmatrix}
	\big(\max \{c_{d, 1} x_1, c_{d, 2} x_2, \ldots, c_{d, k} x_k \} - c_{d, k+1} x_{k+1} \big)^{+} \\
	(c_{d, k+1} x_{k+1} )^{+} \\
	(-c_{d, k+1} x_{k+1} )^{+} \\
	(c_{d, k+2} x_{k+2} )^{+} \\
	(-c_{d, k+2} x_{k+2} )^{+} \\
	\vdots \\
	(c_{d, d} x_{d} )^{+} \\
	(-c_{d, d} x_{d} )^{+} 
\end{pmatrix}
\end{equation}
it holds that
\begin{equation}
\begin{split}
	&z_{k+1} 
= 
	\mathbf{A}_{2(d-(k+1)) + 1} (W_{k+1} z_k + B_{k+1})
= 
	\mathbf{A}_{2(d-(k+1)) + 1} (W_{k+1} z_k ) \\
&=
	\mathbf{A}_{2(d-(k+1)) + 1}
	\begin{pmatrix}
		\left[
			\substack{ (\max \{c_{d, 1} x_1, \ldots, c_{d, k} x_k \}  - c_{d, k+1} x_{k+1} )^{+} \\
				+ (c_{d, k+1} x_{k+1} )^{+} - (-c_{d, k+1} x_{k+1} )^{+} - (c_{d, k+2} x_{k+2} )^{+} + (-c_{d, k+2} x_{k+2} )^{+}}
		\right] \\
		(c_{d, k+2} x_{k+2} )^{+} - (-c_{d, k+2} x_{k+2} )^{+} \\
		-(c_{d, k+2} x_{k+2} )^{+} + (-c_{d, k+2} x_{k+2} )^{+} \\
		(c_{d, k+3} x_{k+3} )^{+} - (-c_{d, k+3} x_{k+3} )^{+} \\
		-(c_{d, k+3} x_{k+3} )^{+} + (-c_{d, k+3} x_{k+3} )^{+} \\
		\vdots\\
		(c_{d, d} x_{d} )^{+} - (-c_{d, d} x_{d} )^{+} \\
		-(c_{d, d} x_{d} )^{+} + (-c_{d, d} x_{d} )^{+} 
	\end{pmatrix}\\
&=
	\mathbf{A}_{2(d-(k+1)) + 1}
	\begin{pmatrix}
			\substack{ (\max \{c_{d, 1} x_1, \ldots, c_{d, k} x_k \}  - c_{d, k+1} x_{k+1} )^{+} + c_{d, k+1} x_{k+1} 
				 - c_{d, k+2} x_{k+2} } \\
		c_{d, k+2} x_{k+2}  \\
		-c_{d, k+2} x_{k+2} \\
		c_{d, k+3} x_{k+3} \\
		-c_{d, k+3} x_{k+3} \\
		\vdots\\
		c_{d, d} x_{d} \\
		-c_{d, d} x_{d}  
	\end{pmatrix}\\
&=
	\begin{pmatrix}
		 \big(\max \{c_{d, 1} x_1,  \ldots, c_{d, k} x_k,  c_{d, k+1} x_{k+1} \} - c_{d, k+2} x_{k+2} \big)^{+} \\
		(c_{d, k+2} x_{k+2} )^{+} \\
		(-c_{d, k+2} x_{k+2} )^{+} \\
		(c_{d, k+3} x_{k+3} )^{+} \\
		(-c_{d, k+3} x_{k+3} )^{+} \\
		\vdots \\
		(c_{d, d} x_{d} )^{+} \\
		(-c_{d, d} x_{d} )^{+} 
	\end{pmatrix}.
\end{split}
\end{equation}
Induction thus proves \eqref{call_on_max_NN:eq2}.
Next observe that \eqref{call_on_max_NN:setting4}, \eqref{call_on_max_NN:setting5}, \eqref{call_on_max_NN:setting6}, \eqref{call_on_max_NN:eq2}, and item \eqref{aux_max:item2} in Lemma~\ref{aux_max} imply that 
\begin{equation}
\begin{split}
	z_d 
&= 
	\mathbf{A}_{1} (W_d z_{d-1} + B_d) \\
&=
	\mathbf{A}_{1}
	\left(
		\begin{pmatrix}
			1 &1 &-1 
		\end{pmatrix}
		\begin{pmatrix}
		 	\big(\max \{c_{d, 1} x_1,  \ldots,  c_{d, d-1} x_{d-1} \} - c_{d, d} x_{d} \big)^{+} \\
			(c_{d, d} x_{d} )^{+} \\
			(-c_{d, d} x_{d} )^{+} 
		\end{pmatrix}
		-
		K_d
	\right) \\
&=
	\mathbf{A}_{1}
	\Big(
		 \big(\max \{c_{d, 1} x_1,  \ldots,  c_{d, d-1} x_{d-1} \} - c_{d, d} x_{d} \big)^{+} \\
&\quad		 +
		(c_{d, d} x_{d} )^{+}
		-
		(-c_{d, d} x_{d} )^{+} 
		-
		K_d
	\Big) \\
&=
	\mathbf{A}_{1}
	\Big(
		 \big(\max \{c_{d, 1} x_1,  \ldots,  c_{d, d-1} x_{d-1} \} - c_{d, d} x_{d} \big)^{+} 
		 +
		c_{d, d} x_{d}
		-
		K_d
	\Big) \\
&=
	\mathbf{A}_{1}
	\big(
		 \max \{c_{d, 1} x_1,  \ldots,  c_{d, d-1} x_{d-1}, c_{d, d} x_{d} \} 
		-
		K_d
	\big) \\
&=
	\max \{ \max\{c_{d, 1} x_1, \ldots,  c_{d, d} x_d \} - K_d ,  0 \}.
\end{split}
\end{equation}
Combining this with \eqref{call_on_max_NN:setting4}--\eqref{call_on_max_NN:eq1} establishes that
\begin{equation}
\label{call_on_max_NN:eq3}
\begin{split}
	(\mathcal{R}(\phi))(x)
&=
	z_{d+1} 
=  
	W_{d+1} z_d + B_{d+1} \\
&=
	z_d
=
	\max \{ \max\{c_{d, 1} x_1, \ldots,  c_{d, d} x_d \} - K_d ,  0 \}.
\end{split}
\end{equation}
In addition, observe that Lemma~\ref{ComplexityCompution:lem} implies that 
\begin{equation}
	\begin{split}
		\mathcal{P}(\phi)
	&=
		(2(d-1) + 1)(d + 1)  \\
	&\quad	
		+
		\left[
			\smallsum_{k = 1}^{d-1} (2 (d-(k+1)) + 1) (2 (d-k) + 1 + 1)
		\right]
		+
		1(1+1) \\
	&\leq 
		6 d^3.
	\end{split}
\end{equation}
Combining this, \eqref{call_on_max_NN:eq1}, and \eqref{call_on_max_NN:eq3} completes the proof of Lemma~\ref{call_on_max_NN}.
\end{proof}

\begin{prop}
\label{call_on_max}
Assume Setting~\ref{BS_setting} and let $(K_d)_{d \in \N}, (c_{d, i})_{d \in \N, i \in \{1, 2, \ldots, d \}} \subseteq [0,\infty)$ satisfy that
$
	\sup_{ d \in \N , i \in \{1, 2, \ldots, d \}} c_{d, i}
<
	\infty.
$
Then  
\begin{enumerate}[(i)]
\item \label{call_on_max:item1}
there exist unique continuous functions $u_d\colon [0,T] \allowbreak\times \R^{d} \to \R$, $d \in \N$, which satisfy
for all $d \in \N$, $x = (x_1, x_2, \ldots, x_d) \in \R^{d}$ that 
$
	u_d(T,x) = \max \{ \max\{c_{d, 1} x_1, c_{d, 2} x_2, \ldots,  c_{d, d} x_d \} - K_d ,  0 \}
$,
which satisfy 
for all $d \in \N$ that
$
	\inf_{q \in (0,\infty)} 
	\sup_{(t, x) \in [0, T] \times \R^d} 
	\frac{ | u_d(t, x) | }{ 1 + \norm{x}_{\R^d}^q }
<
	\infty
$,
and which satisfy for all $d \in \N$ that $u_d|_{(0,T) \times \R^{d}}$ is a viscosity solution of
\begin{equation}
\begin{split}
	&(\tfrac{\partial }{\partial t}u_d)(t,x) 
	+
	\big\langle(\nabla_x u_d)(t,x), \mu_d(x) \big\rangle_{\R^d}
\\
	&+ 
	\tfrac{1}{2} 
	\operatorname{Trace}\! \big( 
		\sigma_d(x)[\sigma_d(x)]^{\ast}(\operatorname{Hess}_x u_d )(t,x)
	\big) 
=
	0
\end{split}
\end{equation}
for $(t,x) \in (0,T) \times \R^{d}$
and

\item \label{call_on_max:item2}
there exist $\mathfrak{C} \in (0,\infty)$, $(\psi_{d, \varepsilon})_{d \in \N, \,\varepsilon \in (0,1]} \subseteq \mathcal{N}$ such that
for all $d \in \N$, $\varepsilon \in (0,1]$ it holds that
$
	\mathcal{P}(\psi_{d, \varepsilon}) 
\leq
	\mathfrak{C} \, d^{5\theta+3} \, \varepsilon^{-4}
$,
$
	\mathscr{P}(\psi_{d, \varepsilon}) 
\leq
	\mathfrak{C} \, d^{5\theta+3} \, \varepsilon^{-2}
$,
$
	\mathcal{R}(\psi_{d, \varepsilon}) \in C(\R^{d}, \R)
$,
and
\begin{equation}
\label{call_on_max:concl1} 
	\left[
		\int_{\R^d}  
		\left|
			u_d(0,x) - ( \mathcal{R}(\psi_{d, \varepsilon}) ) (x)
		\right|^p \,
		\nu_{d}(dx)
	\right]^{\nicefrac{1}{p}} 
\leq
		\varepsilon.
\end{equation}
\end{enumerate}
\end{prop}

\begin{proof}[Proof of Proposition~\ref{call_on_max}]
Throughout this proof 
let $ \varphi_d \colon \R^d \to \R$, $d \in \N$, satisfy 
for all $d \in \N$, $x = (x_1, x_2, \ldots, x_d) \in \R^{d}$ that 
\begin{equation}
		\varphi_d(x) = \max \{ \max\{c_{d, 1} x_1, c_{d, 2} x_2, \ldots,  c_{d, d} x_d \} - K_d ,  0 \},
\end{equation}
let $(\chi_d)_{d \in \N}, (\phi_{d, \delta})_{d \in \N, \delta \in (0,1]} \subseteq \mathcal{N}$ satisfy 
for all $d \in \N$, $x \in \R^{d}$, $\delta \in (0,1]$ that
$
	\mathcal{P}(\chi_d) 
\leq 
	6 d^3
$, 
$
	\mathcal{R}(\chi_d) \in C(\R^d, \R)
$,
$
	(\mathcal{R}(\chi_d))(x)
=
	\varphi_d(x)
$
(cf.\ Lemma~\ref{call_on_max_NN}),
and 
$
	\phi_{d, \delta} = \chi_d
$,
and let $C \in [0,\infty)$ be given by
$
	C = \sup_{ d \in \N , i \in \{1, 2, \ldots, d \}} c_{d, i}
$.
Note that 
for all $d \in \N$, $\delta \in (0,1]$ it holds that
\begin{equation}
\label{call_on_max:eq1}
\mathcal{R}(\phi_{d, \delta}) = \mathcal{R}(\chi_{d})=\varphi_d \in C(\R^d, \R).
\end{equation}
This and the fact that for all $d \in \N, i \in \{1, 2, \ldots, d \}$ it holds that $c_{d, i}\in [0,\infty)$ ensures that
for all $d \in \N$, $x = (x_1, x_2, \ldots, x_d) \in \R^{d}$, $\delta \in (0,1]$ it holds that
\begin{equation}
\label{call_on_max:eq2}
\begin{split}
	| (\mathcal{R}(\phi_{d, \delta}))(x) |
&=
	| (\mathcal{R}(\chi_{d}))(x) |
=
	| \varphi_d (x) | \\
&=
	\max \{ \max\{c_{d, 1} x_1, c_{d, 2} x_2, \ldots,  c_{d, d} x_d \} - K_d ,  0 \} \\
&\leq
	 \max \{c_{d, 1} | x_1 |,  c_{d, 2} | x_2 |, \ldots,  c_{d, d} | x_d |  \} \\
&\leq	
	C
	\max \{ | x_1 |,   | x_2 | , \ldots ,  | x_d | \} \\
&\leq
	C\norm{x}_{\R^d} 
\leq
	C d^0 (1 + \norm{x}_{\R^d}^2).
\end{split}
\end{equation}
In addition, observe that 
for all $d \in \N$, $\delta \in (0,1]$ it holds that
\begin{equation}
\label{call_on_max:eq3}
	\mathcal{P}(\phi_{d, \delta}) 
=
	\mathcal{P}(\chi_d) 
\leq 
	6 d^3
=
	6 d^3 \delta^{-0}.
\end{equation}
Furthermore, note that \eqref{call_on_max:eq1} ensures that 
for all $d \in \N$, $x \in \R^{d}$, $\delta \in (0,1]$ it holds that
\begin{equation}
\begin{split}
	\left| 
		\varphi_d(x) - ( \mathcal{R}(\phi_{d, \delta}) )(x)
	\right| 
&=
	\left| 
		\varphi_d(x) - ( \mathcal{R}(\chi_{d}) )(x)
	\right| \\
&=
	\left| 
		\varphi_d(x) - \varphi_d(x)
	\right| 
=
	0
\leq 
	d^0 \delta^0 (1 + \norm{x}_{\R^d}^2).
\end{split}
\end{equation}
Combining this, \eqref{call_on_max:eq1}, \eqref{call_on_max:eq2}, \eqref{call_on_max:eq3}, 
the fact that $(\varphi_d)_{d \in \N}$ are continuous functions, 
the hypothesis that for all $q \in (0,\infty)$ it holds that
\begin{equation}
\sup_{d \in \N}\left[ d^{-\theta q}
\textint_{\R^d}  
\Norm{x}_{\R^d}^{ q } \,
\nu_{ d } (dx)\right]
< 
\infty,
\end{equation}
and Lemma~\ref{BS_properties} 
with 
Theorem~\ref{cont_NN_approx} 
(with 
$ T = T $,
$ r = 1 $,
$ R = 1 $,
$ v = 0 $,
$ w = 0 $,
$ z = 3 $,
$ \mathbf{z} = 0 $,
$\theta=\theta$,
$ \mathfrak{c} = 
	\max \{ 
		6 , C,
		2 
		\left[
			\sup_{d \in \N, i \in \{1, 2, \ldots, d \}} 
				(| \alpha_{d, i} | + | \beta_{d, i} |) 
		\right]
	\}$,
$ \mathbf{v} = 2 $,
$p = p$,
$ \nu_d = \nu_d$,
$ \varphi_d = \varphi_d $,
$ \mu_d = \mu_d $,
$ \sigma_d = \sigma_d $,
$a(x)=\max \{x, 0 \}$,
$ \phi_{d, \delta} = \phi_{d, \delta}$
for $d\in\N$, $x\in \R$, $\delta\in (0,1]$
in the notation of Theorem~\ref{cont_NN_approx})
demonstrates that 
there exist unique continuous functions $v_d\colon [0,T] \times \R^{d} \to \R$, $d \in \N$, which satisfy
for all $d \in \N$, $x \in \R^{d}$ that 
$v_d(0,x) = \varphi_d(x)$,
which satisfy 
for all $d \in \N$ that
$
	\inf_{q \in (0,\infty)} 
	\sup_{(t, x) \in [0, T] \times \R^d} 
	\frac{ | v_d(t, x) | }{ 1 + \norm{x}_{\R^d}^q }
<
	\infty
$,
and which satisfy for all $d \in \N$ that $v_d|_{(0,T) \times \R^{d}}$ is a viscosity solution of
\begin{equation}
\label{call_on_max:eq4}
\begin{split}
	(\tfrac{\partial }{\partial t}v_d)(t,x) 
&= 
	\tfrac{1}{2} 
	\operatorname{Trace}\! \big( 
		\sigma_d(x)[\sigma_d(x)]^{\ast}(\operatorname{Hess}_x v_d )(t,x)
	\big) 
	\\&+
	\big\langle(\nabla_x v_d)(t,x), \mu_d(x) \big\rangle_{\R^d}
\end{split}
\end{equation}
for $(t,x) \in (0,T) \times \R^{d}$
and that
there exist $\mathfrak{C} \in (0,\infty)$, $(\psi_{d, \varepsilon})_{d \in \N, \,\varepsilon \in (0,1]} \subseteq \mathcal{N}$ such that
for all $d \in \N$, $\varepsilon \in (0,1]$ it holds that
$
	\mathcal{P}(\psi_{d, \varepsilon}) 
\leq
	\mathfrak{C} \, d^{5\theta+3} \, \varepsilon^{-4}
$,
$
	\mathscr{P}(\psi_{d, \varepsilon}) 
\leq
	\mathfrak{C} \, d^{5\theta+3} \, \varepsilon^{-2}
$,
$
	\mathcal{R}(\psi_{d, \varepsilon}) \in C(\R^{d}, \R)
$,
and 
\begin{equation}
\label{call_on_max:eq5}
	\left[
		\int_{\R^d}  
		\left|
			v_d(T,x) - ( \mathcal{R}(\psi_{d, \varepsilon}) ) (x)
		\right|^p \,
		\nu_{d}(dx)
	\right]^{\nicefrac{1}{p}} 
\leq
		\varepsilon.
\end{equation}
Corollary~\ref{BS_endvalue} hence assures that there exist unique continuous functions 
$
	u_d\colon [0,T] \times \R^{d} \to \R
$, $d \in \N$,
which satisfy that 
for all $d \in \N$, $x = (x_1, x_2, \ldots,\allowbreak x_d)\allowbreak \in \R^{d}$ it holds that 
\begin{equation}
\begin{split}
	u_d(T,x)
&=  
	\varphi_d(x) = 
	\max \{ \max\{c_{d, 1} x_1, c_{d, 2} x_2, \ldots,  c_{d, d} x_d \} - K_d ,  0 \},
\end{split}
\end{equation}
which satisfy 
for all $d \in \N$ that
$
	\inf_{q \in (0,\infty)} 
	\sup_{(t, x) \in [0, T] \times \R^d} 
	\frac{ | u_d(t, x) | }{ 1 + \norm{x}_{\R^d}^q }
<
	\infty
$,
and which satisfy 
for all $d \in \N$ that $u_d|_{(0,T) \times \R^{d}}$ is a viscosity solution of
\begin{equation}
\label{call_on_max:eq6}
\begin{split}
	&(\tfrac{\partial }{\partial t}u_d)(t,x) 
	+
	\big\langle(\nabla_x u_d)(t,x), \mu_d(x) \big\rangle_{\R^d}
	 \\
	&+ 
	\tfrac{1}{2} 
	\operatorname{Trace}\! \big( 
		\sigma_d(x)[\sigma_d(x)]^{\ast}(\operatorname{Hess}_x u_d )(t,x)
	\big) 
=
	0
\end{split}
\end{equation}
for $(t,x) \in (0,T) \times \R^{d}$
and that it holds
for all $d \in \N$, $\varepsilon \in (0,1]$ that
\begin{equation}
\label{call_on_max:eq7}
\begin{split}
	&\left[
		\int_{\R^d}  
		\left|
			u_d(0,x) - ( \mathcal{R}(\psi_{d, \varepsilon}) ) (x)
		\right|^p \,
		\nu_{d}(dx)
	\right]^{\nicefrac{1}{p}}
\leq
	\varepsilon.
\end{split}
\end{equation}
Combining this with the fact that
for all $d \in \N$, $\varepsilon \in (0,1]$ it holds that
$
	\mathcal{P}(\psi_{d, \varepsilon}) 
\leq
	\mathfrak{C} \, d^{5\theta+3} \, \varepsilon^{-4}
$,
$
	\mathscr{P}(\psi_{d, \varepsilon}) 
\leq
	\mathfrak{C} \, d^{5\theta+3} \, \varepsilon^{-2}
$,
and
$
	\mathcal{R}(\psi_{d, \varepsilon}) \in C(\R^{d}, \R)
$
establishes items~\eqref{call_on_max:item1}--\eqref{call_on_max:item2}.
The proof of Proposition~\ref{call_on_max} is thus completed.
\end{proof}

\subsection[ANN approximations for call on min options]{Artificial neural network approximations for call on min options}\label{SectionCallOnMin}
In this subsection we establish in Proposition~\ref{call_on_min} below that ANN approximations overcome the curse of dimensionality in the numerical approximation of the Black-Scholes model in the case of call on min options. Our proof of Proposition~\ref{call_on_min} employs the ANN representation result for the payoff functions associated to call on min options in Lemma~\ref{call_on_min_NN} below.
%
%
%

\begin{lemma}
\label{call_on_min_NN}
Assume Setting~\ref{BS_setting} and
let $(K_d)_{d \in \N}, (c_{d, i})_{d \in \N, i \in \{1, 2, \ldots, d \}} \subseteq \R$.
Then 
there exists $(\phi_d)_{d \in \N} \subseteq \mathcal{N}$ such that 
for all $d \in \N$, $x = (x_1,x_2, \ldots, x_d) \in \R^d$ it holds that
$
	\mathcal{P}(\phi_d) 
\leq 
	6 d^3
$, 
$
	\mathcal{R}(\phi_d) \in C(\R^d, \R)
$,
and
\begin{equation}
	(\mathcal{R}(\phi_d))(x)
=
	 \max \{ \min\{c_{d, 1} x_1, c_{d, 2} x_2, \ldots,  c_{d, d} x_d \} - K_d ,  0 \}.
\end{equation}
\end{lemma}

\begin{proof}[Proof of Lemma~\ref{call_on_min_NN}]
Throughout this proof 
let $d \in \N$,
let 
\begin{equation}
\label{call_on_min_NN:setting1}
\begin{split}
	\phi &=  ((W_1, B_1), (W_2, B_2), \ldots, (W_d, B_d), (W_{d+1}, B_{d+1})) \\
&\in 
	(\R^{(2(d-1) + 1) \times d} \times \R^{2(d-1) + 1} ) \\
&\quad	\times 
	(\times_{k = 1}^{d-1} (\R^{(2(d-k) - 1) \times (2(d-k) + 1) } \times \R^{2(d-k) - 1})) \\
&\quad	\times 
	(\R^{1 \times 1} \times \R^{1} )
\end{split}
\end{equation}
(i.e., $\phi$ corresponds to fully connected feedforward artificial neural network with $d+2$ layers with dimensions $(d, 2(d-1) + 1, 2(d-2) + 1, 2(d-3) + 1, \ldots, 3, 1, 1)$)
satisfy for all $k \in \{1, 2, \ldots, {d-2} \}$ that
\begin{equation}
\label{call_on_min_NN:setting2}
	W_1 
= 
	\begin{pmatrix}
		-c_{d,1} &c_{d,2} &0 &\cdots &0 \\
		0 &c_{d,2} &0 &\cdots &0 \\
		0 &-c_{d,2} &0 &\cdots &0 \\
		0 &0 &c_{d,3} &\cdots &0 \\
		0 &0 &-c_{d,3} &\cdots &0 \\
		\vdots &\vdots &\vdots &\ddots &\vdots \\
		0 &0 &0 &\cdots &c_{d, d} \\
		0 &0 &0 &\cdots &-c_{d, d} 
	\end{pmatrix}
\in 
	\R^{(2(d-1) + 1) \times d},
\end{equation}
\begin{equation}
\label{call_on_min_NN:setting3}
\begin{split}
	W_{k+1} 
&= 
	\begin{pmatrix}
		1 &-1 &1 &1 &-1 &0 &0 &\dots &0 &0 \\
		0 &0 &0 &1 &-1 &0 &0 &\dots &0 &0 \\
		0 &0 &0 &-1 &1 &0 &0 &\dots &0 &0 \\
		0 &0 &0 &0 &0 &1 &-1 &\dots &0 &0 \\
		0 &0 &0 &0 &0 &-1 &1 &\dots &0 &0 \\
		\vdots &\vdots &\vdots &\vdots &\vdots &\vdots &\vdots &\ddots &\vdots &\vdots \\
		0 &0 &0 &0 &0 &0 &0 &\dots &1 &-1 \\
		0 &0 &0 &0 &0 &0 &0 &\dots &-1 &1 \\
	\end{pmatrix} \\
&\in 
	\R^{(2(d-k) - 1) \times (2(d-k) + 1) },
\end{split}
\end{equation}
\begin{equation}
\label{call_on_min_NN:setting4}
	W_d
=
	\begin{pmatrix}
		-1 &1 &-1
	\end{pmatrix},
\qquad
	W_{d+1} 
= 
	\begin{pmatrix}
		1
	\end{pmatrix}
\in \R^{1 \times 1},
\qquad
	B_1 = 0,  
\end{equation}
\begin{equation}
\label{call_on_min_NN:setting5}
	B_2 = 0,
\quad\ldots,
	B_{d-1} = 0, 
\qquad
	B_{d} = -K_d,
\qandq 
B_{d+1} = 0,
\end{equation}
let $x = (x_1, x_2, \ldots, x_d) \in \R^d$, 
let $z_0 \in \R^d, z_1 \in \R^{2(d-1) - 1}, z_{2} \in \R^{2(d-2) + 1}, \ldots,\allowbreak z_{d-1} \in \R^{3}, z_d \in \R, z_{d+1} \in \R$ satisfy 
for all $k \in \{1, 2, \ldots, {d-1} \}$ that 
$
	z_0 = x
$,
$
	z_k = \mathbf{A}_{2(d-k) + 1} (W_k z_{k-1} + B_k)
$, 
$
	z_d = \mathbf{A}_{1} (W_d z_{d-1} + B_d)
$, and
\begin{equation}
\label{call_on_min_NN:setting6}
	z_{d+1} =  W_{d+1} z_d + B_{d+1},
\end{equation}
and let 
$
(\cdot)^{+} \colon \R \to [0,\infty)$ be the function which satisfies for all $q \in \R$ that 
$
(q)^{+} = \max \{ q , 0 \}
$.
Note that \eqref{BS_setting:eq3}, \eqref{call_on_min_NN:setting1}, and \eqref{call_on_min_NN:setting6} imply that 
$
	(\mathcal{R}(\phi)) \in C(\R^d, \R)
$
and
\begin{equation}
\label{call_on_min_NN:eq1}
	z_{d+1}
=
	(\mathcal{R}(\phi))(x).
\end{equation}
Next we claim that 
for all $k \in \{1, 2, \ldots, {d-1} \}$ it holds that
\begin{equation}
\label{call_on_min_NN:eq2}
	z_k
=
\begin{pmatrix}
	\big(c_{d, k+1} x_{k+1} - \min \{c_{d, 1} x_1, c_{d, 2} x_2, \ldots, c_{d, k} x_k \}\big)^{+} \\
	(c_{d, k+1} x_{k+1} )^{+} \\
	(-c_{d, k+1} x_{k+1} )^{+} \\
	(c_{d, k+2} x_{k+2} )^{+} \\
	(-c_{d, k+2} x_{k+2} )^{+} \\
	\vdots \\
	(c_{d, d} x_{d} )^{+} \\
	(-c_{d, d} x_{d} )^{+} 
\end{pmatrix}.
\end{equation}
We now prove \eqref{call_on_min_NN:eq2} by induction on $k \in \{1, 2, \ldots, {d-1} \}$.
For the base case $k = 1$ note that \eqref{BS_setting:eq2}, \eqref{call_on_min_NN:setting2}, \eqref{call_on_min_NN:setting4}, and \eqref{call_on_min_NN:setting6} assure that 
\begin{equation}
\begin{split}
	z_1
&=
	\mathbf{A}_{2(d-1) + 1} (W_1 z_0 + B_1) 
=
	\mathbf{A}_{2(d-1) + 1} (W_1 x) \\
&=
	\mathbf{A}_{2(d-1) + 1}
	\begin{pmatrix}
		-c_{d, 1} x_1 + c_{d, 2} x_2 \\
		c_{d, 2} x_{2}   \\
		-c_{d, 2} x_{2}  \\
		c_{d, 3} x_{3}  \\
		-c_{d, 3} x_{3} \\
		\vdots \\
		c_{d, d} x_{d}\\
		-c_{d, d} x_{d} 
	\end{pmatrix}
=
	\begin{pmatrix}
		\big(c_{d, 2} x_{2} - \min \{c_{d, 1} x_1 \} \big)^{+} \\
		(c_{d, 2} x_{2} )^{+} \\
		(-c_{d, 2} x_{2} )^{+} \\
		(c_{d, 3} x_{3} )^{+} \\
		(-c_{d, 3} x_{3} )^{+} \\
		\vdots \\
		(c_{d, d} x_{d} )^{+} \\
		(-c_{d, d} x_{d} )^{+} 
	\end{pmatrix}.
\end{split}
\end{equation}
This establishes \eqref{call_on_min_NN:eq2} in the base case $k = 1$. 
Next note that item \eqref{aux_max:item3} in Lemma~\ref{aux_max} implies that for all $a,b,c\in\R$ it holds that 
\begin{equation}\label{aux_min:Consequence}
(a-b)^+-a+c=c-\min\{a,b\}.
\end{equation}
For the induction step $ \{1, 2, \ldots,\allowbreak {d-2} \} \ni k \to k+1 \in \{2, 3, \ldots, {d-1} \}$ observe that 
\eqref{BS_setting:eq2}, \eqref{call_on_min_NN:setting3}, \eqref{call_on_min_NN:setting5}, \eqref{call_on_min_NN:setting6},
\eqref{aux_min:Consequence} (with $a=c_{d, k+1} x_{k+1}$, $b=\min \{c_{d, 1} x_1, c_{d, 2} x_2, \ldots, c_{d, k} x_k \}$, $c=c_{d, k+2} x_{k+2}$ in the notation of \eqref{aux_min:Consequence}),
 and item \eqref{aux_max:item1} in Lemma~\ref{aux_max} demonstrate that 
for all $k \in \{1, 2, \ldots, {d-2} \}$ with 
\begin{equation}
	z_k
=
\begin{pmatrix}
	\big(c_{d, k+1} x_{k+1} - \min \{c_{d, 1} x_1, c_{d, 2} x_2, \ldots, c_{d, k} x_k \} \big)^{+} \\
	(c_{d, k+1} x_{k+1} )^{+} \\
	(-c_{d, k+1} x_{k+1} )^{+} \\
	(c_{d, k+2} x_{k+2} )^{+} \\
	(-c_{d, k+2} x_{k+2} )^{+} \\
	\vdots \\
	(c_{d, d} x_{d} )^{+} \\
	(-c_{d, d} x_{d} )^{+} 
\end{pmatrix}
\end{equation}
it holds that
\begin{equation}
\begin{split}
	&z_{k+1} 
= 
	\mathbf{A}_{2(d-(k+1)) + 1} (W_{k+1} z_k + B_{k+1})
= 
	\mathbf{A}_{2(d-(k+1)) + 1} (W_{k+1} z_k ) \\
&=
	\mathbf{A}_{2(d-(k+1)) + 1}
	\begin{pmatrix}
		\left[
			\substack{ (c_{d, k+1} x_{k+1} - \min \{c_{d, 1} x_1, c_{d, 2} x_2, \ldots, c_{d, k} x_k \}  )^{+} \\
				- (c_{d, k+1} x_{k+1} )^{+} + (-c_{d, k+1} x_{k+1} )^{+} + (c_{d, k+2} x_{k+2} )^{+} - (-c_{d, k+2} x_{k+2} )^{+}}
		\right] \\
		(c_{d, k+2} x_{k+2} )^{+} - (-c_{d, k+2} x_{k+2} )^{+} \\
		-(c_{d, k+2} x_{k+2} )^{+} + (-c_{d, k+2} x_{k+2} )^{+} \\
		(c_{d, k+3} x_{k+3} )^{+} - (-c_{d, k+3} x_{k+3} )^{+} \\
		-(c_{d, k+3} x_{k+3} )^{+} + (-c_{d, k+3} x_{k+3} )^{+} \\
		\vdots\\
		(c_{d, d} x_{d} )^{+} - (-c_{d, d} x_{d} )^{+} \\
		-(c_{d, d} x_{d} )^{+} + (-c_{d, d} x_{d} )^{+} 
	\end{pmatrix}\\
&=
	\mathbf{A}_{2(d-(k+1)) + 1}
	\begin{pmatrix}
			\substack{ (c_{d, k+1} x_{k+1} - \min \{c_{d, 1} x_1, c_{d, 2} x_2, \ldots, c_{d, k} x_k \})^{+} - c_{d, k+1} x_{k+1} 
				 + c_{d, k+2} x_{k+2} } \\
		c_{d, k+2} x_{k+2}  \\
		-c_{d, k+2} x_{k+2} \\
		c_{d, k+3} x_{k+3} \\
		-c_{d, k+3} x_{k+3} \\
		\vdots\\
		c_{d, d} x_{d} \\
		-c_{d, d} x_{d}  
	\end{pmatrix}\\
&=
	\begin{pmatrix}
		 \big(c_{d, k+2} x_{k+2} - \min \{c_{d, 1} x_1,  \ldots, c_{d, k} x_k,  c_{d, k+1} x_{k+1} \} \big)^{+} \\
		(c_{d, k+2} x_{k+2} )^{+} \\
		(-c_{d, k+2} x_{k+2} )^{+} \\
		(c_{d, k+3} x_{k+3} )^{+} \\
		(-c_{d, k+3} x_{k+3} )^{+} \\
		\vdots \\
		(c_{d, d} x_{d} )^{+} \\
		(-c_{d, d} x_{d} )^{+} 
	\end{pmatrix}.
\end{split}
\end{equation}
Induction thus proves \eqref{call_on_min_NN:eq2}.
Next observe that \eqref{call_on_min_NN:setting3}, \eqref{call_on_min_NN:setting5}, \eqref{call_on_min_NN:setting6}, \eqref{call_on_min_NN:eq2}, and item \eqref{aux_max:item3} in Lemma~\ref{aux_max} imply that 
\begin{equation}
\begin{split}
	z_d 
&= 
	\mathbf{A}_{1} (W_d z_{d-1} + B_d) \\
&=
	\mathbf{A}_{1}
	\left(
		\begin{pmatrix}
			-1 &1 &-1 
		\end{pmatrix}
		\begin{pmatrix}
		 	\big(c_{d, d} x_{d} - \min \{c_{d, 1} x_1,  \ldots,  c_{d, d-1} x_{d-1} \} \big)^{+} \\
			(c_{d, d} x_{d} )^{+} \\
			(-c_{d, d} x_{d} )^{+} 
		\end{pmatrix}
		-
		K_d
	\right) \\
&=
	\mathbf{A}_{1}
	\Big(
		 -\big(c_{d, d} x_{d} - \min \{c_{d, 1} x_1,  \ldots,  c_{d, d-1} x_{d-1} \} \big)^{+} \\
&\quad		 
		+
		(c_{d, d} x_{d} )^{+}
		-
		(-c_{d, d} x_{d} )^{+} 
		-
		K_d
	\Big) \\
&=
	\mathbf{A}_{1}
	\Big(
		 -\big(c_{d, d} x_{d} - \min \{c_{d, 1} x_1,  \ldots,  c_{d, d-1} x_{d-1} \} \big)^{+}
		 +
		c_{d, d} x_{d}
		-
		K_d
	\Big) \\
&=
	\mathbf{A}_{1}
	\big(
		 \min \{c_{d, 1} x_1,  \ldots,  c_{d, d-1} x_{d-1}, c_{d, d} x_{d} \} 
		-
		K_d
	\big) \\
&=
	\max \{ \min\{c_{d, 1} x_1, \ldots,  c_{d, d} x_d \} - K_d ,  0 \}.
\end{split}
\end{equation}
Combining this with \eqref{call_on_min_NN:setting4}--\eqref{call_on_min_NN:eq1} establishes that
\begin{equation}
\label{call_on_min_NN:eq3}
\begin{split}
	(\mathcal{R}(\phi))(x)
&=
	z_{d+1} 
=  
	W_{d+1} z_d + B_{d+1} \\
&=
	z_d
=
	\max \{ \min\{c_{d, 1} x_1, \ldots,  c_{d, d} x_d \} - K_d ,  0 \}.
\end{split}
\end{equation}In addition, observe that Lemma~\ref{ComplexityCompution:lem} implies that 
\begin{equation}
\begin{split}
\mathcal{P}(\phi)
&=
(2(d-1) + 1)(d + 1)  \\
&\quad	
+
\left[
\smallsum_{k = 1}^{d-1} (2 (d-(k+1)) + 1) (2 (d-k) + 1 + 1)
\right]
+
1(1+1) \\
&\leq 
6 d^3.
\end{split}
\end{equation}
Combining this, \eqref{call_on_min_NN:eq1}, and \eqref{call_on_min_NN:eq3} completes the proof of Lemma~\ref{call_on_min_NN}.
\end{proof}

\begin{prop}
\label{call_on_min}
Assume Setting~\ref{BS_setting} and let $(K_d)_{d \in \N}, (c_{d, i})_{d \in \N, i \in \{1, 2, \ldots, d \}} \subseteq [0,\infty)$ satisfy that
$
	\sup_{ d \in \N , i \in \{1, 2, \ldots, d \}} c_{d, i}
<
	\infty.
$
Then  
\begin{enumerate}[(i)]
\item \label{call_on_min:item1}
there exist unique continuous functions $u_d\colon [0,T]\allowbreak \times \R^{d} \to \R$, $d \in \N$, which satisfy
for all $d \in \N$, $x = (x_1, x_2, \ldots, x_d) \in \R^{d}$ that 
$
	u_d(T,x) = \max \{ \min\{c_{d, 1} x_1, c_{d, 2} x_2, \ldots,  c_{d, d} x_d \} - K_d ,  0 \}
$,
which satisfy 
for all $d \in \N$ that
$
	\inf_{q \in (0,\infty)} 
	\sup_{(t, x) \in [0, T] \times \R^d} 
	\frac{ | u_d(t, x) | }{ 1 + \norm{x}_{\R^d}^q }
<
	\infty
$,
and which satisfy for all $d \in \N$ that $u_d|_{(0,T) \times \R^{d}}$ is a viscosity solution of
\begin{equation}
\begin{split}
	&(\tfrac{\partial }{\partial t}u_d)(t,x) 
	+
	\big\langle(\nabla_x u_d)(t,x), \mu_d(x) \big\rangle_{\R^d}
\\
	&+ 
	\tfrac{1}{2} 
	\operatorname{Trace}\! \big( 
		\sigma_d(x)[\sigma_d(x)]^{\ast}(\operatorname{Hess}_x u_d )(t,x)
	\big) 
=
	0
\end{split}
\end{equation}
for $(t,x) \in (0,T) \times \R^{d}$
and

\item \label{call_on_min:item2}
there exist $\mathfrak{C} \in (0,\infty)$, $(\psi_{d, \varepsilon})_{d \in \N, \,\varepsilon \in (0,1]} \subseteq \mathcal{N}$ such that
for all $d \in \N$, $\varepsilon \in (0,1]$ it holds that
$
	\mathcal{P}(\psi_{d, \varepsilon}) 
\leq
	\mathfrak{C} \, d^{5\theta+3} \, \varepsilon^{-4}
$,
$
	\mathscr{P}(\psi_{d, \varepsilon}) 
\leq
	\mathfrak{C} \, d^{5\theta+3} \, \varepsilon^{-2}
$,
$
	\mathcal{R}(\psi_{d, \varepsilon}) \in C(\R^{d}, \R)
$,
and
\begin{equation}
\label{call_on_min:concl1} 
	\left[
		\int_{\R^d}  
		\left|
			u_d(0,x) - ( \mathcal{R}(\psi_{d, \varepsilon}) ) (x)
		\right|^p \,
		\nu_{d}(dx)
	\right]^{\nicefrac{1}{p}} 
\leq
		\varepsilon.
\end{equation}
\end{enumerate}
\end{prop}

\begin{proof}[Proof of Proposition~\ref{call_on_min}]
Throughout this proof 
let $ \varphi_d \colon \R^d \to \R$, $d \in \N$, satisfy 
for all $d \in \N$, $x = (x_1, x_2, \ldots, x_d) \in \R^{d}$ that 
\begin{equation}
		\varphi_d(x) = \max \{ \min\{c_{d, 1} x_1, c_{d, 2} x_2, \ldots,  c_{d, d} x_d \} - K_d ,  0 \},
\end{equation}
let $(\chi_d)_{d \in \N}, (\phi_{d, \delta})_{d \in \N, \delta \in (0,1]} \subseteq \mathcal{N}$ satisfy 
for all $d \in \N$, $x \in \R^{d}$, $\delta \in (0,1]$ that
$
	\mathcal{P}(\chi_d) 
\leq 
	6 d^3
$, 
$
	\mathcal{R}(\chi_d) \in C(\R^d, \R)
$,
$
	(\mathcal{R}(\chi_d))(x)
=
	\varphi_d(x)
$
(cf.\ Lemma~\ref{call_on_min_NN}),
and 
$
	\phi_{d, \delta} = \chi_d
$,
and let $C \in [0,\infty)$ be given by
$
	C = \sup_{ d \in \N , i \in \{1, 2, \ldots, d \}} c_{d, i}
$.
Note that 
for all $d \in \N$, $\delta \in (0,1]$ it holds that
\begin{equation}
\label{call_on_min:eq1}
\mathcal{R}(\phi_{d, \delta}) = \mathcal{R}(\chi_{d})=\varphi_d \in C(\R^d, \R).
\end{equation}
This ensures that
for all $d \in \N$, $x = (x_1, x_2, \ldots, x_d) \in \R^{d}$, $\delta \in (0,1]$ it holds that
\begin{equation}
\label{call_on_min:eq2}
\begin{split}
	| (\mathcal{R}(\phi_{d, \delta}))(x) |
&=
	| (\mathcal{R}(\chi_{d}))(x) |
=
	| \varphi_d (x) | \\
&=
	\max \{ \min\{c_{d, 1} x_1, c_{d, 2} x_2, \ldots,  c_{d, d} x_d \} - K_d ,  0 \} \\
&\leq
	 \max \{c_{d, 1} | x_1 |,  c_{d, 2} | x_2 |, \ldots,  c_{d, d} | x_d |  \} \\
&\leq	
	C
	\max \{ | x_1 |,   | x_2 | , \ldots ,  | x_d | \} \\
&\leq
	C\norm{x}_{\R^d} 
\leq
	C d^0 (1 + \norm{x}_{\R^d}^2).
\end{split}
\end{equation}
In addition, observe that 
for all $d \in \N$, $\delta \in (0,1]$ it holds that
\begin{equation}
\label{call_on_min:eq3}
	\mathcal{P}(\phi_{d, \delta}) 
=
	\mathcal{P}(\chi_d) 
\leq 
	6 d^3
=
	6 d^3 \delta^{-0}.
\end{equation}
Furthermore, note that \eqref{call_on_min:eq1} ensures that 
for all $d \in \N$, $x \in \R^{d}$, $\delta \in (0,1]$ it holds that
\begin{equation}
\begin{split}
	\left| 
		\varphi_d(x) - ( \mathcal{R}(\phi_{d, \delta}) )(x)
	\right| 
&=
	\left| 
		\varphi_d(x) - ( \mathcal{R}(\chi_{d}) )(x)
	\right| \\
&=
	\left| 
		\varphi_d(x) - \varphi_d(x)
	\right| 
=
	0
\leq 
	d^0 \delta^0 (1 + \norm{x}_{\R^d}^2).
\end{split}
\end{equation}
Combining this, \eqref{call_on_min:eq1}, \eqref{call_on_min:eq2}, \eqref{call_on_min:eq3}, 
the fact that $(\varphi_d)_{d \in \N}$ are continuous functions, 
the hypothesis that for all $q \in (0,\infty)$ it holds that
\begin{equation}
\sup_{d \in \N}\left[ d^{-\theta q}
\textint_{\R^d}  
\Norm{x}_{\R^d}^{ q } \,
\nu_{ d } (dx)\right]
< 
\infty,
\end{equation}
and Lemma~\ref{BS_properties} 
with 
Theorem~\ref{cont_NN_approx} 
(with 
$ T = T $,
$ r = 1 $,
$ R = 1 $,
$ v = 0 $,
$ w = 0 $,
$ z = 3 $,
$ \mathbf{z} = 0 $,
$\theta=\theta$,
$ \mathfrak{c} = 
	\max \{ 
		6 , C,
		2 
		\left[
			\sup_{d \in \N, i \in \{1, 2, \ldots, d \}} 
				(| \alpha_{d, i} | + | \beta_{d, i} | )
		\right]
	\}$,
$ \mathbf{v} = 2 $,
$p = p$,
$ \nu_d = \nu_d$,
$ \varphi_d = \varphi_d $,
$ \mu_d = \mu_d $,
$ \sigma_d = \sigma_d $,
$a(x)=\max \{x, 0 \}$,
$ \phi_{d, \delta} = \phi_{d, \delta}$
for $d\in\N$, $x\in \R$, $\delta\in (0,1]$
in the notation of Theorem~\ref{cont_NN_approx})
demonstrates that 
there exist unique continuous functions $v_d\colon [0,T] \times \R^{d} \to \R$, $d \in \N$, which satisfy
for all $d \in \N$, $x \in \R^{d}$ that 
$v_d(0,x) = \varphi_d(x)$,
which satisfy 
for all $d \in \N$ that
$
	\inf_{q \in (0,\infty)} 
	\sup_{(t, x) \in [0, T] \times \R^d} 
	\frac{ | v_d(t, x) | }{ 1 + \norm{x}_{\R^d}^q }
<
	\infty
$,
and which satisfy for all $d \in \N$ that $v_d|_{(0,T) \times \R^{d}}$ is a viscosity solution of
\begin{equation}
\label{call_on_min:eq4}
\begin{split}
	(\tfrac{\partial }{\partial t}v_d)(t,x) 
&= 
	\tfrac{1}{2} 
	\operatorname{Trace}\! \big( 
		\sigma_d(x)[\sigma_d(x)]^{\ast}(\operatorname{Hess}_x v_d )(t,x)
	\big)\\& 
	+
	\big\langle(\nabla_x v_d)(t,x), \mu_d(x) \big\rangle_{\R^d}
\end{split}
\end{equation}
for $(t,x) \in (0,T) \times \R^{d}$
and that
there exist $\mathfrak{C} \in (0,\infty)$, $(\psi_{d, \varepsilon})_{d \in \N, \,\varepsilon \in (0,1]} \subseteq \mathcal{N}$ such that
for all $d \in \N$, $\varepsilon \in (0,1]$ it holds that
$
	\mathcal{P}(\psi_{d, \varepsilon}) 
\leq
	\mathfrak{C} \, d^{5\theta+3} \, \varepsilon^{-4}
$,
$
	\mathscr{P}(\psi_{d, \varepsilon}) 
\leq
	\mathfrak{C} \, d^{5\theta+3} \, \varepsilon^{-2}
$,
$
	\mathcal{R}(\psi_{d, \varepsilon}) \in C(\R^{d}, \R)
$,
and
\begin{equation}
\label{call_on_min:eq5}
	\left[
		\int_{\R^d}  
		\left|
			v_d(T,x) - ( \mathcal{R}(\psi_{d, \varepsilon}) ) (x)
		\right|^p \,
		\nu_{d}(dx)
	\right]^{\nicefrac{1}{p}} 
\leq
		\varepsilon.
\end{equation}
Corollary~\ref{BS_endvalue} hence assures that there exist unique continuous functions 
$
	u_d\colon [0,T] \times \R^{d} \to \R
$, $d \in \N$,
which satisfy that 
for all $d \in \N$, $x = (x_1, x_2, \ldots,\allowbreak x_d) \allowbreak\in \R^{d}$ it holds that 
\begin{equation}
\begin{split}
	u_d(T,x)
&=  
	\varphi_d(x)= 
	\max \{ \min\{c_{d, 1} x_1, c_{d, 2} x_2, \ldots,  c_{d, d} x_d \} - K_d ,  0 \},
\end{split}
\end{equation}
which satisfy 
for all $d \in \N$ that
$
	\inf_{q \in (0,\infty)} 
	\sup_{(t, x) \in [0, T] \times \R^d} 
	\frac{ | u_d(t, x) | }{ 1 + \norm{x}_{\R^d}^q }
<
	\infty
$,
and which satisfy 
for all $d \in \N$ that $u_d|_{(0,T) \times \R^{d}}$ is a viscosity solution of
\begin{equation}
\label{call_on_min:eq6}
\begin{split}
	&(\tfrac{\partial }{\partial t}u_d)(t,x) 
	+
	\big\langle(\nabla_x u_d)(t,x), \mu_d(x) \big\rangle_{\R^d}
	 \\
	&+ 
	\tfrac{1}{2} 
	\operatorname{Trace}\! \big( 
		\sigma_d(x)[\sigma_d(x)]^{\ast}(\operatorname{Hess}_x u_d )(t,x)
	\big) 
=
	0
\end{split}
\end{equation}
for $(t,x) \in (0,T) \times \R^{d}$
and that it holds
for all $d \in \N$, $\varepsilon \in (0,1]$ that
\begin{equation}
\label{call_on_min:eq7}
\begin{split}
	&\left[
		\int_{\R^d}  
		\left|
			u_d(0,x) - ( \mathcal{R}(\psi_{d, \varepsilon}) ) (x)
		\right|^p \,
		\nu_{d}(dx)
	\right]^{\nicefrac{1}{p}}
\leq
	\varepsilon.
\end{split}
\end{equation}
Combining this with the fact that
for all $d \in \N$, $\varepsilon \in (0,1]$ it holds that
$
	\mathcal{P}(\psi_{d, \varepsilon}) 
\leq
	\mathfrak{C} \, d^{5\theta+3} \, \varepsilon^{-4}
$,
$
	\mathscr{P}(\psi_{d, \varepsilon}) 
\leq
	\mathfrak{C} \, d^{5\theta+3} \, \varepsilon^{-2}
$,
and
$
	\mathcal{R}(\psi_{d, \varepsilon}) \in C(\R^{d}, \R)
$
establishes items~\eqref{call_on_min:item1}--\eqref{call_on_min:item2}.
The proof of Proposition~\ref{call_on_min} is thus completed.
\end{proof}
\subsubsection*{Acknowledgements}
This project has been partially supported through the SNSF-Research project 200020{\_}175699 ``Higher order numerical approximation methods for stochastic partial differential equations''.
The third author gratefully acknowledges the Cluster of Excellence EXC 2044-390685587, Mathematics M\"unster: Dynamics-Geometry-Structure funded by the Deutsche Forschungsgemeinschaft (DFG, German Research Foundation).
Finally, we gratefully acknowledge financial support by the Deutsche
Forschungsgemeinschaft (DFG, German Research Foundation) through CRC 1173 and
we gratefully acknowledge a Research Travel Grant by the Karlsruhe House of Young Scientists (KHYS) supporting the stay of the second author at ETH Zurich.

\bibliographystyle{acm}
\bibliography{Lp_Neural_Net_approximation_bibfile}
\end{document}